\documentclass[11pt]{amsart}
\usepackage{amsmath,amsthm,amsfonts}
\usepackage{latexsym,txfonts}
\usepackage{amssymb}
\usepackage{slashed}
\usepackage{pifont}
\usepackage{enumitem}
\usepackage{cite}
\usepackage{verbatim}

\usepackage{fullpage}

\numberwithin{equation}{section}

\newtheorem{thm}{Theorem}[section]
\newtheorem{cor}[thm]{Corollary}
\newtheorem{lem}[thm]{Lemma}
\newtheorem{prop}[thm]{Proposition}
\newtheorem{rem}[thm]{Remark}
\newtheorem{defn}[thm]{Definition}

\newcommand{\cA}{ {\mathcal A} }
\newcommand{\R}{\mathbb{R}}
\newcommand{\Z}{\mathbb{Z}}
\newcommand{\bR}{\mathbb{R}}
\newcommand{\bS}{\mathbb{S}}
\newcommand{\bZ}{\mathbb{Z}}

\begin{document}

\title[Concentration Compactness for Critical Radial Wave Maps]{Concentration Compactness for Critical Radial Wave Maps}

\begin{abstract}
 We consider radially symmetric, energy critical wave maps from $(1+2)$-dimensional Minkowski space into the unit sphere $\mathbb{S}^m$, $m \geq 1$, and prove global regularity and scattering for classical smooth data of finite energy. In addition, we establish a priori bounds on a suitable scattering norm of the radial wave maps and exhibit concentration compactness properties of sequences of radial wave maps with uniformly bounded energies. This extends and complements the beautiful classical work of Christodoulou-Tahvildar-Zadeh~\cite{CTZ1, CTZ2} and Struwe~\cite{Struwe_radial1, Struwe_radial2} as well as of Nahas~\cite{Nahas} on radial wave maps in the case of the unit sphere as the target. The proof is based upon the concentration compactness/rigidity method of Kenig-Merle~\cite{Kenig_Merle1, Kenig_Merle2} and a ``twisted'' Bahouri-G\'erard type profile decomposition~\cite{Bahouri_Gerard}, following the implementation of this strategy by the second author and Schlag~\cite{KS} for energy critical wave maps into the hyperbolic plane as well as by the last two authors~\cite{KL} for the energy critical Maxwell-Klein-Gordon equation.
\end{abstract}

\author[]{Elisabetta Chiodaroli}
\address{B\^atiment des Math\'ematiques \\ EPFL \\ Station 8 \\ 1015 Lausanne \\ Switzerland}
\email{elisabetta.chiodaroli@epfl.ch}

\author[]{Joachim Krieger}
\address{B\^atiment des Math\'ematiques \\ EPFL \\ Station 8 \\ 1015 Lausanne \\ Switzerland}
\email{joachim.krieger@epfl.ch}

\author[]{Jonas L\"uhrmann}
\address{Department of Mathematics \\ Johns Hopkins University \\ 3400 N. Charles Street \\ Baltimore, MD 21218 \\ USA}
\email{luehrmann@math.jhu.edu}

\maketitle

\setcounter{tocdepth}{1}
\tableofcontents

\section{Introduction}

We consider wave maps $\phi \colon \bR^{1+2} \to \bS^m$ from $(1+2)$-dimensional Minkowski space $\bR^{1+2}$ into the $m$-dimensional unit sphere $\bS^m \hookrightarrow \bR^{m+1}$, $m \geq 1$, satisfying the equation 
\begin{equation*} \tag{WM}
 \Box \phi = - \phi \partial_\alpha \phi^\dagger \partial^\alpha \phi
\end{equation*}
for radially symmetric initial data 
\[
 (\phi, \partial_t \phi)|_{t=0} = (\phi_0, \phi_1) \colon \bR^2 \to T\bS^m.
\]

Here we use the extrinsic formulation of the wave maps equation, viewing the $\bR^{m+1}$-valued function $\phi$ as a column vector and denoting its transpose by $\phi^\dagger$. In particular, we note that any initial data $(\phi_0, \phi_1)$ must satisfy the consistency conditions 
\[
 \phi_0^\dagger \phi_0 = 1, \quad \phi_0^\dagger \phi_1 = 0.
\]
Greek indices such as $\alpha$ are implicitly assumed to run from $0$ to $2$ and we use the standard conventions for summing and raising/lowering indices with respect to the metric $\text{diag}(-1, +1, +1)$ on Minkowski space $\bR^{1+2}$. We denote the d'Alembertian by $\Box = -\partial_t^2 + \Delta$ and introduce the shorthand notation $\phi[t] = ( \phi(t), \partial_t \phi(t) )$ for $t \in \bR$. An initial data pair $(\phi_0, \phi_1) \colon \bR^2 \to T\bS^m$ is called classical if it is smooth and constant in the exterior of a compact set.

\medskip 

The wave maps equation (WM) admits a non-negative conserved energy functional 
\[
 E[\phi] := \frac{1}{2} \int_{\bR^2} \sum_{\alpha = 0}^2 | \partial_\alpha \phi |^2 \, dx,
\]
which is invariant under the scaling of the equation 
\[
 \phi(t,x) \mapsto \phi(\lambda t, \lambda x), \quad \lambda > 0.
\]
The Cauchy problem for (WM) on $\bR^{1+2}$ is therefore energy critical. 

\medskip 

The main result of this article asserts that for radially symmetric, classical initial data, the unique, smooth solutions to (WM) exist globally in time and scatter to finite energy free waves. Moreover, we establish a priori bounds on a suitable scattering norm of the solutions and we exhibit concentration compactness properties of sequences of radially symmetric wave maps into the unit sphere with uniformly bounded energies.

\begin{thm} \label{thm:main_theorem}
 There exists a non-decreasing function $K \colon [0,\infty) \to [0, \infty)$ with the following property: Let $(\phi_0, \phi_1)$ be radially symmetric, classical initial data of energy $E$. Then there exists a global, unique, smooth solution $\phi \colon \mathbb{R}^{1+2} \to \mathbb{S}^{m}$ to (WM) with initial data $\phi[0] = (\phi_0, \phi_1)$ satisfying the a priori bound
  \[
   \| \phi \|_{S} \leq K(E),
  \]
  where the $S$ norm is defined in Section~\ref{sec:function_spaces} below. In particular, $\phi$ scatters to finite energy free waves as $t \to \pm \infty$ in the sense that there exist $(f_\pm, g_\pm) \in \dot{H}^1_x \times L^2_x$ such that
  \[
   \lim_{t \to \pm \infty} \, \bigl\| \nabla_{t,x} \phi - \nabla_{t,x} S(t)(f_\pm, g_\pm) \bigr\|_{L^2_x} = 0,
  \]
  where $S(t)(f_\pm, g_\pm) = \cos(t |\nabla|) f_\pm + \frac{\sin(t |\nabla|)}{|\nabla|} g_\pm$ denotes the free wave propagator.
\end{thm}

We emphasize that global regularity and scattering for energy critical radial wave maps into arbitrary compact target manifolds has already been established in by now classical works of Christodoulou-Tahvildar-Zadeh~\cite{CTZ1, CTZ2} and Struwe~\cite{Struwe_radial1, Struwe_radial2} as well as of Nahas~\cite{Nahas}. We shall next give a brief overview of the history of the wave maps problem and motivate why we are revisiting the beautiful classical results on radial wave maps. Then we conclude this introduction with an overview of the proof of Theorem~\ref{thm:main_theorem}.

\subsection{History and motivation}

We note that the wave maps problem has been the subject of a fascinating and vast body of literature over the past decades that we cannot adequately review here in its entirety. Our primary focus shall be on energy critical wave maps from $(1+2)$-dimensional Minkowski space.

The study of energy critical radial wave maps was begun in the seminal work of Christodoulou-Tahvildar-Zadeh~\cite{CTZ1} where global regularity is proven for arbitrary compact target manifolds for radially symmetric initial data with sufficiently small energy. This small energy global regularity result is then strengthened to a large energy global regularity result by excluding concentration of energy provided the target manifold, unlike the sphere, satisfies a suitable convexity condition. In a subsequent paper~\cite{CTZ2}, Christodoulou-Tahvildar-Zadeh also obtain pointwise scattering bounds on the solutions under the same convexity condition on the target manifold. Struwe~\cite{Struwe_radial1, Struwe_radial2} then established large energy global regularity for radially symmetric wave maps into arbitrary compact target manifolds by combining the results from~\cite{CTZ1} with a careful blowup analysis from~\cite{Struwe}. More recently, Nahas~\cite{Nahas} also proved scattering for radially symmetric wave maps into arbitrary compact target manifolds.

For wave maps without any symmetry assumptions the sub-critical local well-posedness theory was developed by Klainerman-Machedon~\cite{KlMa1, KlMa2, KlMa3} and Klainerman-Selberg~\cite{Klainerman_Selberg1, Klainerman_Selberg2}, making crucial use of the null structure in the quadratic derivative nonlinearity of the wave maps equation.

A very important step in extending the small energy global regularity result of Christodoulou-Tahvildar-Zadeh~\cite{CTZ1} to the non-radial case was achieved by Tataru~\cite{Tataru1}, establishing global regularity for energy critical wave maps into arbitrary compact target manifolds for initial data which is small in the scale-invariant homogeneous Besov space $\dot{B}_{2,1}^1(\bR^2) \times \dot{B}_{2,1}^0(\bR^2)$. This work introduced an important functional framework for the study of the energy critical wave maps equation, in particular the delicate null frame spaces. The final breakthrough to prove small energy global regularity for energy critical wave maps into the unit sphere $\bS^m$, $m \geq 1$, was achieved by Tao~\cite{Tao2} through the key realization that certain non-perturbative terms in the wave maps nonlinearity can be cast into a better form by exploiting the gauge invariance of the wave maps problem. Small energy global regularity was then extended to other target manifolds by the second author~\cite{Krieger} for the hyperbolic plane $\mathbb{H}^2$ and by Tataru~\cite{Tataru2} for arbitrary target manifolds that can be isometrically embedded into Euclidean space. The key role that the gauge structure plays in the study of the wave maps equation at the critical regularity to renormalize the equation into a better form was also further clarified in the works of Klainerman-Rodnianski~\cite{Klainerman_Rodnianski}, Shatah-Struwe~\cite{Shatah_Struwe}, Nahmod-Stefanov-Uhlenbeck~\cite{NSU} and the second author~\cite{Krieger_3d} on global regularity for wave maps from higher-dimensional Minkowski space for small critical Sobolev data. 

For large energies, depending on the geometry of the target manifold, blowup can occur for energy critical wave maps. Indeed, the blowup analysis of Struwe~\cite{Struwe} for energy critical equivariant wave maps showed that singularity formation must be tied to the existence of non-trivial finite energy harmonic maps from $\bR^2$ into the target manifold. Later, the second author joint with Schlag and Tataru~\cite{KST}, Rapha\"el-Rodnianski~\cite{Raphael_Rodnianski} and Rodnianski-Sterbenz~\cite{Rodnianski_Sterbenz} constructed examples of equivariant wave maps into $\bS^2$ that blow up in finite time via the concentration of a non-trivial harmonic map.

These developments culminated in the threshold conjecture that for energy critical wave maps global regularity is expected to hold for initial data with energy less than the energy of any non-trivial harmonic map into the target manifold. This conjecture was established independently around the same time by the second author and Schlag~\cite{KS} for the hyperbolic plane $\mathbb{H}^2$ as the target, by Tao~\cite{Tao_large_WM} for all hyperbolic spaces $\mathbb{H}^d$, $d \geq 1$, and by Sterbenz-Tataru~\cite{Sterbenz_Tataru1, Sterbenz_Tataru2} for any target manifold that can be isometrically embedded into Euclidean space. 

\medskip 

Our motivation for this article is essentially twofold. On the one hand we analyze to what extent the complicated function spaces introduced in the seminal works of Tataru~\cite{Tataru1} and Tao~\cite{Tao2} can be replaced by a simpler functional framework in the radial context to settle the perturbative theory for the wave maps equation (WM) similarly to the strategy in Tao~\cite{Tao2}. Here our goal was to avoid Fourier localization as much as possible and apart from a basic spatial frequency localization, our framework in particular avoids modulation localizations as in \cite{Tataru1} and \cite{Tao2}. More specifically, our space $S$ to control the regularity of radial wave maps is built from dyadic subspaces $S_k$ in the sense that
\[
 \| \phi \|_{S}^2 := \sum_{k \in \bZ} \| \phi_k \|_{S_k}^2.
\]
The dyadic subspace $S_k$ is defined entirely in physical space and essentially has three parts of the following schematic form
\[
 \| \phi \|_{S_k} = \| \phi_k \|_{Str} + \sup_{\ell \in \bZ} \| \chi_{\{ r \sim 2^{\ell} \}} r^{-\frac{1}{2}} \nabla_{t,x} \phi_k \|_{L^2_t L^2_x} + \sum_\pm \| (\partial_t \pm \partial_r) \phi_k \|_{Z_k^\pm},
\]
see Definition~\ref{def:S_norm} for the precise definition. The first part of the $S_k$ norm consists of Strichartz-type norms where we crucially exploit the larger range of admissible Strichartz pairs in the radial context, see for instance Sterbenz~\cite{Sterbenz} and Fang-Wang~\cite{Fang_Wang}. The second part is a local energy decay norm which quite naturally replaces the important $\dot{X}^{1, \frac{1}{2}, \infty}$-type space from the non-radial context. Finally, the third part of the $S_k$ norm is formed by certain atomic spaces $Z_k^{\pm}$. These basically correspond to an ``incoming-outgoing'' decomposition of free radial waves and involve $L^2_{t \pm r} L^\infty_{t \mp r}$-type norms which naturally replace the more complicated null frame spaces~\cite{Tataru1} from the non-radial context.

\medskip

On the other hand, in this work we go beyond the results of Christodoulou-Tahvildar-Zadeh~\cite{CTZ1, CTZ2}, Struwe~\cite{Struwe_radial1, Struwe_radial2} and Nahas~\cite{Nahas} in the case of radially symmetric wave maps into the unit sphere $\bS^m$, $m \geq 1$, as the target manifold and prove, in addition to global regularity and scattering, a priori bounds on the scattering norm $S$ of the solutions to (WM) and we exhibit concentration compactness properties of sequences of radially symmetric wave maps into the unit sphere with uniformly bounded finite energies. To this end we use a version of the concentration compactness/rigidity method of Kenig-Merle~\cite{Kenig_Merle1, Kenig_Merle2} and a modified Bahouri-G\'erard type profile decomposition~\cite{Bahouri_Gerard}, following the implementation of this strategy by the second author and Schlag~\cite{KS} for (non-radial) energy critical wave maps into the hyperbolic plane $\mathbb{H}^2$ and the implementation by the last two authors~\cite{KL} for the related energy critical Maxwell-Klein-Gordon equation. Executing the concentration compactness/rigidity strategy in the context of energy critical wave maps is compounded by the presence of non-perturbative terms in the wave maps nonlinearity, which have to be dealt with via renormalization, and by certain strong low-high interactions in the wave maps nonlinearity. These difficulties will be explained in more detail in the next subsection. However, at this point we emphasize a key difference between our work for the unit sphere $\bS^m \hookrightarrow \bR^{m+1}$, $m \geq 1$, as the target and the work of the second author and Schlag~\cite{KS} for the $\mathbb{H}^2$ target regarding the renormalization procedure. The fact that the gauge group for the target $\mathbb{H}^2$ is abelian was exploited heavily in~\cite{KS} in order to implement a global-in-frequency gauge change by passing to the Coulomb gauge in the intrinsic setting. The construction of the Coulomb gauge is elementary and explicit in the abelian case, but it becomes problematic for large energies in the non-abelian case. Since in our work the gauge group is no longer abelian for the targets $\bS^m$ with $m \geq 3$, we instead employ a version of the gauge construction of Sterbenz-Tataru~\cite{Sterbenz_Tataru1} in the extrinsic setting which deals separately with each frequency level. The latter is a refinement and further development of the microlocal gauge introduced by Tao~\cite{Tao1, Tao2} and is better suited for large energies.

\subsection{Overview of the proof of Theorem~\ref{thm:main_theorem}}

Here we provide an outline of the main ideas of the proof of Theorem~\ref{thm:main_theorem}. Our goal is to show that there exists a non-decreasing function $K \colon [0,\infty) \to [0, \infty)$ with the following property: Let $(\phi_0, \phi_1) \colon \bR^2 \to T\bS^m$ be radially symmetric, classical initial data of energy $E$. Then there exists a global, unique, smooth wave map $\phi \colon \bR^{1+2} \to \bS^{m}$ with initial data $\phi[0] = (\phi_0, \phi_1)$ satisfying the a priori bound 
\[
 \| \phi \|_{S} \leq K(E).
\]
Once we have established this a priori bound, the scattering assertion of Theorem~\ref{thm:main_theorem} is an immediate consequence.

\medskip

Following the general philosophy of the concentration compactness/rigidity scheme, we argue by contradiction and assume that Theorem~\ref{thm:main_theorem} fails. Then the existence of the function $K(\cdot)$ yielding the a priori bounds must fail at some finite energy level. Correspondingly, the following set of energies must be non-empty
\[
 {\mathcal E} := \biggl\{ E \, : \, \sup_{\{ \phi \, : \, E[\phi] \leq E \}} \|\phi\|_{S[I]} = +\infty \biggr\},
\]
where the supremum is taken over all radially symmetric wave maps $\phi \colon I \times \bR^2 \to \bS^m$ with classical radially symmetric initial data of energy $E[\phi] \leq E$ and defined on some time interval $I$. We shall prove a small energy global regularity result together with a priori bounds on the scattering norm $S$ for radial solutions to (WM) in Theorem~\ref{thm:small_energy_global}. For this reason the infimum of the set ${\mathcal E}$ has to be strictly positive
\[
 E_{crit} := \inf {\mathcal E} > 0.
\]
Thus, we may pick a sequence of radial smooth wave maps $\phi^n \colon I^n \times \bR^2 \to \bS^{m}$, $n \geq 1$, with maximal intervals of existence $I^n$ such that
\[
 \lim_{n\to\infty} E[\phi^n] = E_{crit}, \quad \lim_{n\to\infty} \| \phi^n \|_{S[I^n]} = + \infty.
\]
In the following we call such a sequence of wave maps essentially singular. Our goal is now to rule out the existence of such an essentially singular sequence of wave maps, hence proving Theorem~\ref{thm:main_theorem}. This will be achieved in the following two main steps.
\begin{itemize}
 \item Extracting an energy class, radially symmetric, minimal blowup solution $\phi^\infty$ to (WM) of energy $E_{crit}$ from the essentially singular sequence $\{ \phi^n \}_{n \geq 1}$. As a key tool we use a ``twisted'' profile decomposition that takes into account the strong low-high interactions in the wave maps nonlinearity. We develop this modified Bahouri-G\'erard type nonlinear profile decomposition for our setting by following the procedure introduced by the second author and Schlag~\cite{KS} for energy critical wave maps into $\mathbb{H}^2$. The minimal blowup solution can be thought of as a ``minimal counterexample'' to Theorem~\ref{thm:main_theorem} and its orbit must therefore possess a strong compactness property modulo the symmetries of the equation. We note that the heart of this paper resides in this extraction procedure which is carried out in Section~\ref{sec:Concentration_Compactness_Step}. All sections leading up to it lay the groundwork.
 \item Ruling out the minimal blowup solution $\phi^\infty$ via a version of the Kenig-Merle rigidity argument~\cite{Kenig_Merle1} as in~\cite{KS}. Here we utilize the strong compactness property of $\phi^\infty$ against general properties of radial wave maps into the unit sphere. In particular, we invoke the non-existence of non-trivial, finite energy, radial harmonic maps into the unit sphere. This step is accomplished in Section~\ref{sec:rigidity_argument}.
\end{itemize}

\medskip 

The severe difficulties with the extraction of a minimal blowup solution to (WM) can be highlighted by comparing with the situation for the energy critical, defocusing nonlinear wave equation $\Box u = u^5$ on $\bR^{1+3}$. In this context Bahouri-G\'erard~\cite{Bahouri_Gerard} introduced a highly influential nonlinear profile decomposition. It basically asserts that a sequence of solutions to the quintic nonlinear wave equation with uniformly bounded energies can be decomposed into a sum of nonlinear solutions, which are referred to as the nonlinear profiles and which are rescaled and translated in space-time according to the non-compact symmetries of the equation, and an error term, which can be made small in a suitable norm. The ability to extract a minimal blowup solution ultimately relies on the asymptotic decoupling of different nonlinear profiles. In the quintic nonlinearity, the interactions of two different nonlinear profiles with essential frequency supports at divergent scales are asymptotically negligible. This reduces to consider diagonal frequency interactions. But then two different nonlinear profiles living at the same frequency scale must concentrate in divergent regions of space-time so that their interactions in the quintic nonlinearity again vanish asymptotically. 

\medskip 

In contrast, for energy critical wave maps frequency diagonalization appears to partially fail at the critical regularity due to strong low-high interactions in the wave maps nonlinearity. In order to gain a better understanding of these difficulties in our context of radial wave maps into the unit sphere, we now take a closer look at the perturbative theory for the frequency localized wave maps equation (WM). 

\medskip 

\noindent {\it Renormalization and perturbative theory.} More precisely, we study the evolution of a single dyadic frequency block $\phi_k$, say $k = 0$, satisfying the equation
\[
 \Box \phi_0 = - P_0 \bigl( \phi \partial_\alpha \phi^\dagger \partial^\alpha \phi \bigr).
\]
Upon decomposing each input of the nonlinearity into its Littlewood-Paley pieces, the nonlinear wave equation for $\phi_0$ assumes the following form
\begin{align*}
 \Box \phi_0 &= - 2 \phi_{\leq -10} \partial_\alpha \phi_{\leq -10}^\dagger \partial^\alpha \phi_0 \\
 &\quad - \sum_{k_2 > -10} \sum_{k_3 = k_2 + O(1)} P_0 \bigl( \phi_{\leq 20} \partial_\alpha \phi_{k_2}^\dagger \partial^\alpha \phi_{k_3} \bigr) \\
 &\quad - 2 \sum_{k_2 \leq -10} \phi_{\leq k_2-10} \bigl( P_0 \bigl( \partial_\alpha \phi_{k_2}^\dagger \partial^\alpha \phi_{-5 < \cdot < 5} \bigr) - \partial_\alpha \phi_{k_2}^\dagger \partial^\alpha \phi_0 \bigr) \\
 &\quad + \text{error},
\end{align*}
where we are only listing those interaction terms that are more difficult to deal with and the other relatively insignificant interaction terms are just denoted by $\text{error}$. The difficulties with the high-high interactions of the second term and the third term on the right-hand side are a more peculiar feature of our framework. They are non-perturbative in the sense that they cannot be treated with the multilinear estimates within our functional framework due to its relevatively elementary but appealing construction - for instance, we cannot gain in the high-high interactions in our null form estimates. But we note that these two interaction terms could be easily handled with the full power of the more sophisticated functional framework from Tao~\cite{Tao2}. Fortunately, this feature of our setting can be dealt with quite efficiently by passing to the ``nonlinearly modified variable''
\[
 \overline{\phi}_0 := \phi_0 + \frac{1}{2} \sum_{k_2 > -10} \sum_{k_3 = k_2 + O(1)} P_0 \bigl( \phi_{\leq 20} \phi_{k_2}^\dagger \phi_{k_3} \bigr) + \sum_{k_2 \leq -10} \phi_{\leq k_2 - 10} \bigl( P_0 \bigl( \phi_{k_2}^\dagger \phi_{-5 < \cdot < 5} \bigr) - \phi_{k_2}^\dagger \phi_0 \bigr),
\]
which has the effect of either distributing derivatives to different inputs with a better frequency balance or of turning a trilinear interaction term into an easier quintilinear one upon reinserting the wave maps equation. The wave equation for the new variable $\overline{\phi}_0$ then takes on the form
\begin{align*}
 \Box \overline{\phi}_0 = - 2 \phi_{\leq -10} \partial_\alpha \phi_{\leq -10}^\dagger \partial^\alpha \phi_0 + \text{error}
\end{align*}
and we are left to treat the more severe low-low-high interaction term on the right-hand side. Following Tao~\cite{Tao1, Tao2} we exploit the geometry of the wave maps problem and use the orthogonality relation $\phi^\dagger \partial^\alpha \phi = 0$ which remains approximately preserved upon frequency localization. Then we arrive at a better equation for $\phi_0$, namely
\begin{equation} \label{equ:intro_nlw_for_phi_bar}
 \Box \overline{\phi}_0 = - 2 \bigl( \phi_{\leq -10} \partial_\alpha \phi_{\leq -10}^\dagger - \partial_\alpha \phi_{\leq -10} \phi_{\leq -10}^\dagger \bigr) \partial^\alpha \phi_0 + \text{error},
\end{equation}
where the matrix $\bigl( \phi_{\leq -10} \partial_\alpha \phi_{\leq -10}^\dagger - \partial_\alpha \phi_{\leq -10} \phi_{\leq -10}^\dagger \bigr)$ is anti-symmetric. Now we have the following trilinear estimate from Proposition~\ref{prop:trilinear_estimate} at our disposal 
\begin{equation} \label{equ:intro_trilinear_estimate}
 \bigl\| \phi_{k_1} \partial_\alpha \phi_{k_2}^\dagger \partial^\alpha \phi_0 \bigr\|_{L^1_t L^2_x} \lesssim 2^{- \delta (k_1 - k_2)} \| \phi_{k_1} \|_{S_{k_1}} \| \phi_{k_2} \|_{S_{k_2}} \| \phi_0 \|_{S_0}, \quad k_2 + O(1) \leq k_1 \leq 0,
\end{equation}
for some absolute constant $\delta > 0$. The key exponential gain in this estimate allows us to handle those parts of the low-low-high interaction term where a derivative falls on the lowest frequency. At the same time we make the fundamental observation that we do not gain exponentially in the largest frequency difference in this estimate, resulting in the strong low-high interactions alluded to before. We note that the direct analogue of this estimate for the non-radial setting is due to Tao~\cite{Tao2} and is much more difficult to achieve than within our functional framework for the radial case. In view of \eqref{equ:intro_trilinear_estimate}, we split the interaction term on the right-hand side of \eqref{equ:intro_nlw_for_phi_bar} into two parts
\begin{equation} \label{equ:intro_nlw_for_phi_bar_two_parts}
 \Box \overline{\phi}_0 = 2 A_{\alpha; \leq -10} \partial^\alpha \phi_0 - 2 \sum_{k_2 \leq -10} \sum_{k_2 - 10 < k_1 \leq -10} \bigl( \phi_{k_1} \partial_\alpha \phi_{k_2}^\dagger - \partial_\alpha \phi_{k_2} \phi_{k_1}^\dagger \bigr) \partial^\alpha \phi_0 + \text{error},
\end{equation}
where we are introducing the connection form 
\begin{equation*}
 A_{\alpha; \leq -10} = - \sum_{k_2 \leq -10} \bigl( \phi_{\leq k_2 - 10} \partial_\alpha \phi_{k_2}^\dagger - \partial_\alpha \phi_{k_2} \phi_{\leq k_2 - 10}^\dagger \bigr).
\end{equation*}
While in the second interaction term on the right-hand side of \eqref{equ:intro_nlw_for_phi_bar_two_parts} a derivative falls on the lowest frequency and this part can therefore be handled with the trilinear estimate, this is not the case for the first interaction term. Thus, the latter term is non-perturbative and has to be renormalized into a better form. Following the method first introduced by Tao~\cite{Tao1, Tao2}, we define suitable $SO(m+1)$-valued gauge transformations $U_{\leq -10}$ and pass to the new variable $U_{\leq -10} \overline{\phi}_0$ satisfying the nonlinear wave equation 
\begin{equation} \label{equ:intro_nlw_for_gauged_phi_bar}
 \begin{aligned}
  \Box \bigl( U_{\leq -10} \overline{\phi}_0 \bigr) &= U_{\leq -10} \bigl( \Box \overline{\phi}_0 - 2 A_{\alpha; \leq -10} \partial^\alpha \phi_0 \bigr) \\
  &\quad + 2 \bigl(\partial_\alpha U_{\leq -10} + U_{\leq -10} A_{\alpha; \leq -10} \bigr) \partial^\alpha \phi_0 \\
  &\quad + 2 \partial_\alpha U_{\leq -10} \partial^\alpha ( \overline{\phi}_0 - \phi_0 ) \\
  &\quad + \bigl( \Box U_{\leq -10} \bigr) \overline{\phi}_0,
 \end{aligned}
\end{equation}
where the key difficulty now resides in the second interaction term on the right-hand side. More precisely, we implement a version of the gauge construction due to Sterbenz-Tataru~\cite{Sterbenz_Tataru1}, which is better suited for large energies, and define the gauge transformations $U_{\leq h}$ for $h \in \bR$ as solutions to the ODE 
\[
 \frac{d}{d h} U_{\leq h} = U_{\leq h} B_h, \quad \lim_{h \to -\infty} U_{\leq h} = Id,
\]
where the anti-symmetric matrix $B_h$ is of the schematic form
\[
 B_h = \phi_{\leq h-10} \phi_h^\dagger - \phi_h \phi_{\leq h-10}^\dagger.
\]
In particular, the anti-symmetry of $B_h$ ensures that the gauge transformations $U_{\leq h}$ as solutions to the above ODE are exactly orthogonal. This choice of gauge transformation effectively transfers a derivative to the lowest frequency in the key difficult term $2 \bigl(\partial_\alpha U_{\leq -10} + U_{\leq -10} A_{\alpha; \leq -10} \bigr) \partial^\alpha \phi_0$ so that it can basically be handled with the trilinear estimate~\eqref{equ:intro_trilinear_estimate}. Finally, we will have to transfer the bounds on the variable $U_{\leq -10} \overline{\phi}_0$ back to the frequency localized wave map~$\phi_0$. In particular, we note that we have to introduce certain frequency and spatial truncations in the definitions of the gauge transformations and the new variable $\overline{\phi}_0$ in order to have additional sources of smallness at our disposal in the large energy setting. These modifications will be explained as we go in the later sections. This essentially settles the perturbative theory for (WM). 

\medskip 

We now describe the extraction procedure of the minimal blowup solution $\phi^\infty$ to (WM) which necessitates the development of a ``twisted'' Bahouri-G\'erard type profile decomposition to take into account the effect of the strong low-high interactions described above. This undertaking is additionally compounded by the fact that we have to work at the level of the gauged variables $U_{\leq -10} \overline{\phi}_0$ because only these satisfy a nonlinear wave equation with good perturbative properties. As in~\cite{KS} we use a finite induction on frequency procedure to carefully disentangle the low-high frequency interactions.

\medskip 

\noindent {\it Decomposition into frequency atoms and evolving the lowest frequency non-atomic part.} The first step consists in decomposing the essentially singular sequence of data $\{ \phi^n[0] \}_{n \geq 1}$ into frequency atoms using the M\'etivier-Schochet procedure~\cite{Metivier-Schochet} as in Bahouri-G\'erard~\cite{Bahouri_Gerard}. Roughly speaking, the basic idea then goes as follows. Ultimately, we would like to conclude that upon passing to a subsequence, if necessary, the essentially singular sequence of data $\{ \phi^n[0] \}_{n \geq 1}$ consists of exactly one frequency atom wich in turn consists of exactly one concentration profile (to be defined precisely in Subsection~\ref{subsec:adding_in_first_atom}) of asymptotic energy $E_{crit}$. In this scenario, the sequence $\{ \phi^n[0] \}_{n \geq 1}$ has sufficient compactness properties that allow us to pass to a certain limit whose wave maps evolution will be the desired minimal blowup solution to (WM). In order to rule out all other possible scenarios, we seek to prove uniform in~$n$, finite, global $S$ norm bounds on (a subsequence of) the sequence of wave maps evolutions $\{ \phi^n \}_{n \geq 1}$, which would contradict that the sequence is essentially singular. To this end we first achieve control over the wave maps evolutions of certain low frequency truncations of the essentially singular sequence of data $\{ \phi^n[0] \}_{n \geq 1}$. Using a finite inductive procedure over the increasing size of the frequency supports of these low frequency truncations, we then conclude uniform in $n$, finite, global $S$ norm bounds on the actual essentially singular sequence $\{ \phi^n \}_{n \geq 1}$. 

\medskip 

Using a version of the M\'etivier-Schochet procedure~\cite{Metivier-Schochet} we obtain a decomposition into frequency atoms (of a subsequence) of the essentially singular sequence of data
\begin{equation*} 
 \phi^n[0] = \sum_{a=1}^\Lambda \phi^{na}[0] + \phi^{n\Lambda}[0],
\end{equation*}
where the frequency atoms $\phi^{na}[0]$ and the remainder term $\phi^{n \Lambda}[0]$ have disjoint frequency supports. The frequency atoms are sharply localized around frequency scales $(\lambda_n^a)^{-1}$, more precisely they have frequency support on $\bigl\{ |\xi| \in [ (\lambda_n^a)^{-1} R_n^{-1}, (\lambda_n^a)^{-1} R_n ] \bigr\}$ for some sequence $R_n \to \infty$ growing sufficently slowly. The frequency scales diverge from each other as $n \to \infty$ in the sense that
\[
 \lim_{n \to \infty} \, \frac{\lambda_n^a}{\lambda_n^b} + \frac{\lambda_n^b}{\lambda_n^a} = + \infty, \quad a \neq b.
\]
We may assume that the atoms are ordered in terms of the increasing size of their frequency support scales~$(\lambda_n^a)^{-1}$ and we introdue the notation $\mu_n^a = - \log ( \lambda_n^a )$ for the corresponding dyadic frequency support scales. Moreover, the remainder term $\phi^{n \Lambda}[0]$ satisfies the important Besov norm smallness
\[
 \lim_{\Lambda\to\infty} \limsup_{n\to\infty} \| \phi^{n\Lambda}[0] \|_{\dot{B}^1_{2,\infty} \times \dot{B}^0_{2,\infty}} = 0.
\]

\medskip 

As described above, ultimately we would like to conclude that there is exactly one atom in the decomposition~\eqref{equ:frequency_atom_decomposition}, i.e. $\Lambda = 1$, which is of asymptotic energy $E_{crit}$. If this is the case, we proceed directly to the next stage below where we consider the evolution of the first ``large'' frequency atom. Otherwise, we now start a finite inductive procedure to conclude that the sequence $\{ \phi^n \}_{n \geq 1}$ cannot be essentially singular. To this end we fix an integer $\Lambda_0$ sufficiently large such that upon passing to a subsequence, if necessary,
\[
 \sum_{a > \Lambda_0} \limsup_{n\to\infty} \| \phi^{na}[0] \|_{\dot{H}^1_x \times L^2_x}^2 \leq \varepsilon_0,
\]
where $\varepsilon_0 > 0$ is a sufficently small constant that plays the role of a perturbative threshold in the key bootstrap argument in Proposition~\ref{prop:bootstrap_bounds_to_next_level_in_evolving_lowest_frequency_nonatomic_part}. In particular, $\varepsilon_0$ will be chosen sufficently small depending only on the size of $E_{crit}$ and it will be chosen to be less than the small energy global regularity threshold established in Theorem~\ref{thm:small_energy_global}. Then we observe that due to the sharp frequency localizations of the atoms $\phi^{na}[0]$, $1 \leq a \leq \Lambda_0$, the remainder term $\phi^{n \Lambda_0}[0]$ gets split into $\Lambda_0 + 1$ ``frequency shells''
\begin{equation*} 
 \phi^{n \Lambda_0}[0] = \phi^{n \Lambda_0^{(0)}}[0] + \phi^{n \Lambda_0^{(1)}}[0] + \ldots + \phi^{n \Lambda_0^{(\Lambda_0)}}[0],
\end{equation*}
where $\phi^{n \Lambda_0^{(0)}}[0]$ shall denote the lowest frequency component.

\medskip

Our first step now consists in showing that the lowest frequency ``non-atomic'' component $\phi^{n \Lambda_0^{(0)}}[0]$ can be globally evolved and satisfies finite $S$ norm bounds just in terms of $E_{crit}$ uniformly for all sufficiently large $n$. Since the component $\phi^{n \Lambda_0^{(0)}}[0]$ may still have large energy, in order to be able to infer these $S$ norm bounds, we approximate $\phi^{n \Lambda_0^{(0)}}[0]$ by a finite number of delicately chosen low frequency truncations $P_{\leq b_L} \phi^{n \Lambda_0^{(0)}}[0]$.
However, up to this point we have totally ignored that the frequency truncations $P_{\leq b_L} \phi^{n \Lambda_0^{(0)}}[0] = P_{\leq b_L} \phi^n[0]$ are not ``geometric'' in the sense that they are not actual maps $\bR^2 \to T\bS^m$ into the unit sphere and therefore do not constitute suitable initial data for the wave maps equation (WM). To overcome this issue we just project the frequency truncations back to the sphere, using the normal projection operator, and denote the resulting initial data by $\Pi_{\leq b_L} \phi^n[0]$. This operation is well-defined and the frequency localization properties are approximately preserved up to exponential tails if around the frequency cut-offs a certain Besov norm smallness condition is satisfied, see Proposition~\ref{prop:data_approximate_preservation_of_frequency_localization}. This, in particular, forces us to carry out a further refined frequency atom decomposition of the lowest frequency non-atomic part $\phi^{n \Lambda_0^{(0)}}[0]$ to carefully pick these low frequency cut-offs $b_L$. Using a finite induction procedure we then obtain uniform in $n$, a priori bounds on the $S$ norms of the global evolutions of the ``geometric'' lowest frequency non-atomic parts $\Pi_{\leq \mu_n^1 - \log R_n} \phi^n[0]$ via an iterative bootstrap argument, which is accomplished in Proposition~\ref{prop:evolving_lowest_frequency_nonatomic_part} in Subsection~\ref{subsec:lowest_freq_non_atomic}.

\medskip 

\noindent {\it Selecting concentration profiles and adding the first large frequency atom.} Having established control over the global evolution of the lowest frequency non-atomic part $\Pi_{\leq \mu_n^1 - \log(R_n)} \phi^n[0]$ in the previous step, we now ``add in'' the first large frequency atom $\phi^{n1}[0] = P_{[\mu_n^1 - \log(R_n), \mu_n^1 + \log(R_n)]} \phi^n [0]$ in the sense that we now attempt to globally evolve the geometric initial data
\[
 \Pi_{\leq \mu_n^1 + \log(R_n)} \phi^n[0].
\]
In this paragraph we shall slightly abuse notation and write for simplicity
\[
 \phi^n[0] \equiv \Pi_{\leq \mu_n^1 + \log(R_n)} \phi^n[0],
\]
denoting the evolution of this data by $\phi^n$. By rescaling we may assume that $\mu_n^1 \equiv 0$ and we shall later denote by $\mu_n = \mu_n^1 - \log(R_n)$ the frequency cut-off delimiting the essential frequency supports of the lowest frequency non-atomic part and of the first large frequency atom. Moreover, we use the notation
\[
 \phi^n[0] \equiv \Pi_{\leq \mu_n^1 - \log(R_n)} \phi^n[0] + \bigl( \Pi_{\leq \mu_n^1 + \log(R_n)} \phi^n - \Pi_{\leq \mu_n^1 - \log(R_n)} \phi^n \bigr)[0] \equiv u^n[0] + \epsilon^n[0],
\]
where $u^n$ denotes the global evolution of the lowest frequency non-atomic part established in the previous stage. Since only the gauged variables satisfy a nonlinear wave equation with good perturbative properties, we now enact a Bahouri-G\'erard type profile decomposition at the level of the gauged variables $U_{<k}^{(\phi^n)} \overline{\phi^n_k}$, which then has to be translated back to the level of the coordinate functions $\phi^n$. Here, $U^{(\phi^n)}$ denote the gauge transformations defined by the evolutions $\phi^n$. In order to take into account the strong coupling between the very low frequencies coming from $u^n$ and the high frequencies, we have to extract the concentration profiles for $\epsilon^n$ with respect to a suitably modified linear wave operator to match the proper asymptotic evolution of the variables $U_{<k}^{(\phi^n)} \overline{\phi_k^n}$ as $t \to \pm \infty$. In view of \eqref{equ:intro_nlw_for_phi_bar_two_parts} and \eqref{equ:intro_nlw_for_gauged_phi_bar} we use the following very natural linear magnetic-type wave equation to select our concentration profiles 
\begin{equation} \label{equ:intro_magnetic_linear_wave}
 \boxed{ \Box \tilde{\epsilon} = 2 U_{<\mu_n}^{(u^{n})} \cA_{\alpha, low}(u^{n}) \bigl( U_{<\mu_n}^{(u^{n})} \bigr)^\dagger \partial^{\alpha} \tilde{\epsilon} }
\end{equation}
where 
\[
 \cA_{\alpha, low}(u^n) = - \sum_{k_2 < \mu_n} \sum_{k_2 - 10 < k_1 < \mu_n} ( u^n_{k_1} ) (\partial_\alpha u_{k_2}^n )^\dagger - (\partial_\alpha u^n_{k_2}) (u_{k_1}^n)^\dagger + \bigl( U_{< \mu_n}^{(u^n)} \bigr)^\dagger \partial_\alpha U_{< \mu_n}^{(u^n)} + A_{\alpha; < \mu_n}^{(u^n)}.
\]
The gauge transformations $U_{<\mu_n}^{(u^n)}$ and the anti-symmetric matrix $\cA_{\alpha, low}(u^{n})$ are defined purely in terms of the low frequency wave maps $u^n$ over which we already have global control at this stage. Moreover, the anti-symmetry of $\cA_{\alpha, low}(u^n)$ and the much lower essential frequency support of $u^n$ are key for proving asymptotic energy conservation for the flow associated with~\eqref{equ:intro_magnetic_linear_wave}, see Lemma~\ref{lem:soloftherightoneasymptotics}. We note that the linear magnetic-type wave equation~\eqref{equ:intro_magnetic_linear_wave} is the direct analogue of the linear magnetic wave equation in \cite[Definition 9.18]{KS} for the extraction of the concentration profiles in the context of energy critical wave maps into the hyperbolic plane $\mathbb{H}^2$.

A pleasant feature of the linear magnetic-type wave equation \eqref{equ:intro_magnetic_linear_wave} is that it is in fact independent of the frequency level $k \in \bZ$ one works at. Rougly speaking, to obtain the profile decomposition at the level of $\phi^n$ one then evolves the data $\tilde{\epsilon}_k[0] := U_{<k}^{(\phi^n)} \epsilon^n_k[0]$ with respect to the flow of~\eqref{equ:intro_magnetic_linear_wave} for each frequency $k \in \bZ$, selecting appropriate concentration times and profiles, and passes back to the level of $\phi^n$ by multiplying with $\bigl( U_{<k}^{(\phi^n)} \bigr)^\dagger$. Summing over all frequencies $k \in \bZ$ then roughly furnishes the desired profiles, see equation~\eqref{eq:profiles} and the remarks following it for the precise definitions. Provided that all concentration profiles have energy strictly less than $E_{crit}$ we may then carefully construct the global evolution of the data $\Pi_{\leq \mu_n^1 + \log(R_n)} \phi^n[0]$ and obtain uniform in $n$, a priori bounds on the $S$ norms, where a lot of work is required to accomplish that the profile decomposition ansatz is ``sufficiently geometric''. This whole step is carried out in Theorem~\ref{thm:profiledecomp} in Subsection~\ref{subsec:adding_in_first_atom}.

\medskip 

\noindent {\it Conclusion of the induction on frequency process.} We now continue this induction on frequency process and by proceeding as in Subsection~\ref{subsec:lowest_freq_non_atomic} obtain that the data $\Pi_{\leq \mu_n^2 - \log(R_n)} \phi^n[0]$ can be globally evolved with uniform $S$ norm bounds. Then we ``add in'' the second frequency atom $P_{[\mu_n^2 - \log(R_n), \mu_n^2 + \log(R_n)]} \phi^n[0]$ in the sense that by proceeding analogously to Subsection~\ref{subsec:adding_in_first_atom} we may establish the global evolution of the corresponding geometric data $\Pi_{\leq \mu_n^2 + \log(R_n)} \phi^n[0]$ with uniform $S$ norm bounds under the assumption that the associated profiles all have energy strictly less than $E_{crit}$.

All in all, we may continue this procedure $\Lambda_0$ many times and establish the global evolution with uniform in $n$, a priori $S$ norm bounds of (a subsequence of) the essentially singular sequence of data $\{ \phi^n[0] \}_{n \geq 1}$, which would however be a contradiction, unless (a subsequence of) the sequence $\{ \phi^n[0] \}_{n \geq 1}$ is composed of exactly one frequency atom that consists of exactly one profile of asymptotic energy $E_{crit}$. Thus, we must be in the latter scenario and (a subsequence of) $\{ \phi^n[0] \}_{n \geq 1}$ has sufficent compactness properties that allow us to pass to a certain limit whose wave maps evolution will be the desired minimal blowup solution to (WM) as detailed in Subsection~\ref{subsec:conclusion}. Here we note that the minimal blowup solution will merely have energy class regularity but that a strong local well-posedness theory for (WM) is only available at sub-critical regularities. For this reason we actually have to introduce a concept of energy class radial wave maps, which we achieve in Subsection~\ref{subsec:energy_class_evolution} by regularization and reduction to the small energy case via finite speed of propagation, analogously to the procedures in \cite{KS} and \cite{KL}.

The outcome of this last stage is that there exists a non-trivial, energy class, radially symmetric, minimal blowup solution $\phi^\infty \colon I \times \bR^2 \to \bS^{m}$ to (WM) of energy $E_{crit}$ and with maximal interval of existence~$I$. Moreover, there exists a continuous function $\lambda \colon I \to (0, \infty)$ so that the family of functions
\begin{equation*}
 \Bigl\{ \bigl( \phi^\infty(t, \lambda(t)^{-1} \cdot ), \lambda(t)^{-1} \partial_t \phi^\infty(t, \lambda(t)^{-1} \cdot \bigr) \colon t \in I \Bigr\}
\end{equation*}
is pre-compact in $\dot{H}^1_x \times L^2_x$.

\medskip 

\noindent {\it Rigidity argument.} Finally, we have to rule out the existence of such a minimal blowup solution $\phi^\infty$ to (WM). To this end we closely mimic the Kenig-Merle rigidity argument~\cite{Kenig_Merle1} as implemented in~\cite{KS}. In particular, we invoke the non-existence of non-trivial, finite energy, radially symmetric harmonic maps into the unit sphere. This finishes the outline of the main ideas entering the proof of Theorem~\ref{thm:main_theorem}.

\medskip 

We expect that the restriction to the unit sphere $\bS^m$, $m \geq 1$, as the target manifold in our work is not necessary and that our method extends to arbitrary compact target manifolds since a similar gauge construction as for the unit sphere works upon establishing additional estimates on the second fundamental form as in Tataru~\cite{Tataru2} and Sterbenz-Tataru~\cite{Sterbenz_Tataru1}.

Moreover, we point out that neither the gauge construction nor the selection process for the concentration profiles in our work crucially hinge on the radial symmetry assumption. We therefore suspect that the method of this article combined with the sophisticated functional framework from Tao~\cite{Tao2} ought to allow for a similar result in the non-radial case upon restricting to energy levels below the energy of any non-trivial harmonic map from $\bR^2$ to $\bS^m$.

\medskip 

\noindent {\it Organization of the paper.} In Section~\ref{sec:preliminaries} we introduce some notation and several basic definitions. In Section~\ref{sec:function_spaces} we present our precise functional framework. In Section~\ref{sec:multilinear_estimates} we collect the most important multilinear estimates to handle the wave maps nonlinearity in the radial case. In Section~\ref{sec:decomposition_nonlinearity} we carefully analyze the structure of the frequency localized wave maps nonlinearity and introduce the renormalization procedure to deal with the non-perturbative terms. Moreover, we establish a small energy global regularity result for~(WM). In Section~\ref{sec:breakdown_criterion} we show that the $S$ norm provides sufficient control on radial wave maps with classical initial data to infer long time existence and scattering. In Section~\ref{sec:Concentration_Compactness_Step} we begin with the actual proof of Theorem~\ref{thm:main_theorem} and accomplish the most difficult step of extracting a minimal blowup solution with the strong compactness property. In Section~\ref{sec:rigidity_argument} we rule out the existence of the minimal blowup solution and thus finish the proof of Theorem~\ref{thm:main_theorem}.

\section{Preliminaries} \label{sec:preliminaries}

\noindent {\it Notation and conventions.} We write $A \lesssim B$ to denote $A \leq C B$ for some absolute constant $C > 0$ that may depend on fixed parameters and we shall use the notation $A \ll B$ to indicate that the implicit constant in the estimate is small. Moreover, we borrow from Tao~\cite{Tao2} a convenient notation to describe multilinear expressions of product type. For scalar functions $\phi^{(1)}(t,x), \ldots, \phi^{(n)}(t,x)$ we denote by $L( \phi^{(1)}, \ldots, \phi^{(n)} )(t,x)$ any multilinear expression of the form
\begin{align*}
 L(\phi^{(1)}, \ldots, \phi^{(n)})(t,x) := \int K(y_1, \ldots, y_n) \phi^{(1)}(t, x-y_1) \cdots \phi^{(n)}(t, x-y_n) \, dy_1 \ldots dy_n,
\end{align*}
where the kernel $K$ is a measure with bounded mass. We extend this notation to the case when $\phi^{(1)}, \ldots, \phi^{(n)}$ take values as $(m+1)$-dimensional vectors or as $(m+1)\times(m+1)$ matrices.

\medskip 

\noindent {\it Littlewood-Paley projections.} We denote by $\varphi$ a non-negative smooth cut-off function satisfying $\varphi(y) = 1$ for $y \leq 1$ and $\varphi(y) = 0$ for $y > 2$. Then we set $\varphi_0(y) = \varphi(y) - \varphi(2 y)$ and $\varphi_k(y) = \varphi_0(2^{-k} y)$ for $k \in \bZ$. We define the dyadic Littlewood-Paley projection operators $P_k$ for $k \in \Z$ by
\begin{equation*}
 \widehat{P_k f}(\xi) = \varphi_k( |\xi| ) \hat{f}(\xi).
\end{equation*}
We often write $f_k = P_k f$. Occasionally, we also need to use continuous Littlewood-Paley projections $P_h$ for $h \in \bR$. We recall the following Leibniz rule for the Littlewood-Paley projections $P_k$, see \cite[Lemma 2]{Tao2}.
\begin{lem}
 It holds that
 \begin{equation} \label{equ:leibniz_rule}
  P_k (f g) = f P_k g + L ( \nabla_x f, 2^{-k} g ).
 \end{equation}
\end{lem}

\medskip 

\noindent {\it Spatial cut-offs.} We will also make use of cut-off functions for the radial variable $r \equiv |x|$. For $\ell \in \bZ$ we denote by $\chi_{\{ r \sim 2^\ell \}}$ a smooth non-negative bump function supported in $\{ r \sim 2^{\ell} \}$ such that we have a smooth finite partition of unity 
\[
 \sum_{\ell \in \bZ} \chi_{\{ r \sim 2^{\ell} \}}(r) = 1 \quad \text{for } r > 0.
\]
Moreover, we denote by $\chi_{\{ r \leq 2^{\ell} \}}$ a smooth cut-off function to $\{ r \lesssim 2^{\ell} \}$. Analogously, we define the cut-off functions $\chi_{\{ r > 2^{\ell}\}}$ and $\chi_{\{ 2^{\ell_1} \leq r \leq 2^{\ell_2}\}}$.

\medskip 

\noindent {\it Frequency envelopes.} We shall use the tool of frequency envelopes from \cite{Tao1} to track the frequency distribution of certain norms. A sequence $\{ c_k \}_{k\in\bZ} \in \ell^2(\bZ)$ of positive real numbers is a frequency envelope if
\[
 2^{-\sigma |k-k'|} c_{k'} \lesssim c_k \lesssim 2^{+ \sigma |k-k'|} c_{k'}
\]
for all $k, k' \in \bZ$, where $\sigma > 0$ is a small absolute constant. We say that an initial data set $\phi[0]$ lies underneath the envelope $\{ c_k \}_{k\in\bZ}$ if 
\[
 \| P_k \phi[0] \|_{\dot{H}^1_x \times L^2_x} \leq c_k
\]
for all $k \in \bZ$. Given an initial data set $\phi[0]$ we may obtain such an envelope by defining
\[
 c_k := \sum_{\ell \in \bZ} 2^{-\sigma |k-\ell|} \| P_{\ell} \phi[0] \|_{\dot{H}^1_x \times L^2_x}.
\]

\section{Function spaces} \label{sec:function_spaces}

In this section we introduce the functional framework used for the proof of Theorem~\ref{thm:main_theorem}. The space $S$ containing the radial wave map $\phi$ is built from dyadic subspaces $S_k$ in the sense that
\[
 \| \phi \|_S^2 := \sum_{k \in \bZ} \| \phi_k \|_{S_k}^2,
\]
where the space $S_k$ is defined as follows.
\begin{defn} \label{def:S_norm}
 Let $k \in \bZ$ and let $\phi$ be a radially symmetric function on $\bR^{1+2}$ with Fourier support in $\{ |\xi| \sim 2^k \}$. Then we define 
 \begin{equation*}
  \begin{aligned}
   \| \phi \|_{S_k} &:= \|\nabla_{t,x} \phi\|_{L^\infty_t L^2_x} + \sup_{ 2 \leq q,p \leq \infty, \frac{1}{q} + \frac{1}{p} < \frac{1}{2} } 2^{(\frac{1}{q} + \frac{2}{p} - 1) k} \| \nabla_{t,x} \phi \|_{L^q_t L^p_x} + \sup_{0 < \lambda < 1} 2^{(\frac{1}{2} - \lambda) k} \| \chi_{\{ r > 2^{-k} \}} r^{-\lambda} \phi \|_{L^2_t L^\infty_x} \\
   &\qquad + \sup_{\ell \in \bZ} \| \chi_{\{ r \sim 2^\ell \}} r^{-\frac{1}{2}} \nabla_{t,x} \phi \|_{L^2_t L^2_x} + \| (\partial_t + \partial_r) \phi \|_{Z^+_k} + \| (\partial_t - \partial_r) \phi \|_{Z^{-}_k},
  \end{aligned}
 \end{equation*}
 where $Z^{\pm}_k$ are atomic spaces defined below in Definition~\ref{def:Z_spaces}.
\end{defn}

We note that the space $S_k$ scales like free waves with $\dot{H}^1_x \times L^2_x$ initial data. The restrictions of the spaces $S$ and $S_k$ to a time interval $I$ are denoted by $S[I]$, respectively $S_k[I]$, with the induced norms. Next we provide the definition of the atomic spaces $Z_k^{\pm}$.

\begin{defn} \label{def:Z_spaces}
 Let $k \in \bZ$ and let $\psi$ be a radially symmetric function on $\bR^{1+2}$ with Fourier support in $\{ |\xi| \sim 2^k \}$. We introduce the auxiliary norm 
 \begin{equation*}
  \| \psi \|_{Y_k} := \sup_{0 < \lambda < 1} 2^{-(\frac{1}{2} + \lambda) k} \| \chi_{\{ r > 2^{-k}\}} r^{-\lambda} \psi \|_{L^2_t L^\infty_x} + \sup_{\ell \in \bZ} \| \chi_{\{ r \sim 2^\ell\}} r^{-\frac{1}{2}} \psi \|_{L^2_t L^2_x}. 
 \end{equation*}
 Then we define 
 \begin{align*}
   \| \psi \|_{Z_k^+} &:= \inf_{ \psi \, = \, \psi^{(+)} + \psi^{(-)} } \biggl\{ \Bigl( \sup_{0 < \lambda \leq \frac{1}{2}} 2^{-(\frac{1}{2} - \lambda)k} \| r^\lambda \psi^{(+)} \|_{L^2_{t+r} L^\infty_{t-r}} + \| \psi^{(+)} \|_{Y_k} \Bigr) \\
   &\qquad \qquad \qquad \quad + \Bigl( \sup_{ 2 \leq q,p \leq \infty, \frac{1}{q} + \frac{1}{p} < \frac{1}{2} } \sup_{ 0 < \lambda < 1} 2^{(\frac{1}{q} + \frac{2}{p} - 1)k} 2^{\lambda k} \sum_{\ell \in \bZ} \| \chi_{\{ r \sim 2^\ell \}} r^{+\lambda} \psi^{(-)} \|_{L^q_t L^p_x} + \| \psi^{(-)} \|_{Y_k} \Bigr) \biggr\}
   \intertext{and} 
   \| \psi \|_{Z_k^-} &:= \inf_{ \psi \, = \, \psi^{(+)} + \psi^{(-)} } \biggl\{ \Bigl( \sup_{ 2 \leq q,p \leq \infty, \frac{1}{q} + \frac{1}{p} < \frac{1}{2} } \sup_{ 0 < \lambda < 1} 2^{(\frac{1}{q} + \frac{2}{p} - 1)k} 2^{\lambda k} \sum_{\ell \in \bZ} \| \chi_{\{ r \sim 2^\ell \}} r^\lambda \psi^{(+)} \|_{L^q_t L^p_x} + \| \psi^{(+)} \|_{Y_k} \Bigr) \\
   &\qquad \qquad \qquad \quad + \Bigl( \sup_{0 < \lambda \leq \frac{1}{2}} 2^{-(\frac{1}{2} - \lambda)k} \| r^{+\lambda} \psi^{(-)} \|_{L^2_{t-r} L^\infty_{t+r}} + \| \psi^{(-)} \|_{Y_k} \Bigr) \biggr\}.
 \end{align*}
\end{defn}

We will place the nonlinearities in the simple $L^1_t L^2_x$ space. The remainder of this section is devoted to the proof of the following key energy estimate connecting the $S_k$ space and the $L^1_t L^2_x$ space.

\begin{lem}[Energy estimate] \label{lem:energy_estimate}
 Let $k \in \bZ$ and let $I$ be any time interval containing $0$. For any radially symmetric function $\phi$ on $I \times \bR^2$ with Fourier support in $\{ |\xi| \sim 2^k \}$, we have 
 \begin{equation} \label{equ:energy_estimate}
  \| \phi_k \|_{S_k[I]} \lesssim \| \nabla_{t,x} \phi_k(0) \|_{L^2_x} + \| \Box \phi_k \|_{L^1_t L^2_x[I]}.
 \end{equation}
\end{lem}

The proof of Lemma~\ref{lem:energy_estimate} is an immediate consequence of the next lemmas. Here we first note that in the radial context a significantly larger range of admissible Strichartz norms is at our disposal, see Sterbenz~\cite{Sterbenz} and Fang-Wang~\cite{Fang_Wang}.
\begin{lem}[Strichartz estimates]
 Let $2 \leq q, p \leq \infty$ with $(q,p) \neq (\infty, \infty)$ satisfy $\frac{1}{q} + \frac{1}{p} < \frac{1}{2}$. Let $k \in \bZ$ and let $I$ be any time interval containing $0$. For any radially symmetric function $\phi$ on $I \times \bR^2$ with Fourier support in $|\xi| \sim 2^k$, it holds that 
 \begin{equation*}
  2^{(\frac{1}{q} + \frac{2}{p} - 1) k} \| \nabla_{t,x} \phi_k \|_{L^q_t L^p_x[I]} \lesssim \| \nabla_{t,x} \phi(0) \|_{L^2_x} + \| \Box \phi \|_{L^1_t L^2_x[I]}.
 \end{equation*}
\end{lem}

Next we prove Strichartz estimates involving a radial weight. 

\begin{lem}[Weighted Strichartz estimates] \label{lem:weighted_strichartz_estimate}
 Let $2 \leq p < \infty$ and $\frac{1}{p} < \lambda < \frac{2}{p}$. Let $k \in \bZ$ and let $I$ be any time interval containing $0$. For any radially symmetric function $\phi$ on $I \times \bR^2$ with Fourier support in $\{ |\xi| \sim 2^k \}$, we have 
 \begin{equation*}
  2^{(\frac{1}{2} + \frac{2}{p} - \lambda) k} \| r^{-\lambda} \phi \|_{L^2_t L^p_x[I]} \lesssim \| \nabla_{t,x} \phi(0) \|_{L^2_x} + \| \Box \phi \|_{L^1_t L^2_x[I]}.
 \end{equation*}
\end{lem}
\begin{proof}
 We adapt the proof of Strichartz estimates under the spherical symmetry assumption by Sterbenz~\cite{Sterbenz} to incorporate the radial weight $r^{-\lambda}$. In view of Duhamel's formula we may assume without loss of generality that $\phi$ is of the form $\phi = e^{it|\nabla|} f_k$ for a radially symmetric function $f_k$ with Fourier support on $|\xi| \sim 2^k$. Moreover, by scaling invariance it suffices to consider the case $k=0$.

 First, we recall that for radially symmetric initial data, the free wave propagator in two space dimensions takes on the specific form
 \[
  (e^{it|\nabla|} f_0)(x) = \int_{\bR^2} e^{2 \pi i (t |\xi| + x \cdot \xi)} \hat{f}_0(\xi) \, d\xi = 2 \pi \int_0^\infty e^{2 \pi i t \rho} J_0(2 \pi r \rho) \varphi_0(\rho) \hat{f}_0(\rho) \rho \, d\rho,
 \]
 where $\varphi_0(\rho)$ is a smooth bump function supported on $\{ \rho \sim 1 \}$ and where
 \[
  J_0(y) = \frac{1}{2\pi} \int_0^{2 \pi} e^{i y \sin(\theta)} \, d\theta, \quad y \in \bR,
 \]
 is the Bessel function of order $0$. Moreover, we recall the following standard asymptotics for the Bessel function of order $0$, see e.g. \cite{Watson},
 \begin{equation} \label{equ:bessel_asymptotics}
  J_0(y) = y^{-\frac{1}{2}} e^{+iy} \beta_{+}(y) + y^{-\frac{1}{2}} e^{-iy} \beta_{-}(y), \quad y \geq 1,
 \end{equation}
 where the functions $\beta_{\pm}$ satisfy the symbol-type bounds 
 \begin{equation} \label{equ:bessel_symbol_bounds}
  | \beta_{\pm}^{(n)}(y) | \lesssim_n y^{-n}, \quad n \geq 0, \, y \geq 1.
 \end{equation}

 We now distinguish the two regimes $r \lesssim 1$ and $r \gg 1$. In the former case, we just use H\"older's inequality, the assumption $\lambda < \frac{2}{p}$ and a standard $TT^\ast$ estimate to obtain that
 \begin{align*}
  \bigl\| \chi_{\{ r \lesssim 1\}} r^{-\lambda} e^{it|\nabla|} f_0 \bigr\|_{L^2_t L^p_x} &\lesssim \bigl\| \chi_{\{ r \lesssim 1\}} r^{-\lambda} \bigr\|_{L^p_x} \bigl\| \chi_{\{ r \lesssim 1\}} \int_0^\infty e^{2 \pi i t \rho} J_0(2 \pi r \rho) \varphi_0(\rho) \hat{f}_0(\rho) \rho \, d\rho \Bigr\|_{L^2_t} \\
  &\lesssim \bigl\| \hat{f}_0(\rho) \rho \bigr\|_{L^2_\rho} \\
  &\lesssim \| f_0 \|_{L^2_x}.
 \end{align*}
 In the latter case $r \gg 1$, the asymptotics~\eqref{equ:bessel_asymptotics} yield that
 \[
  \chi_{\{ r \gg 1\}} r^{-\lambda} (e^{it|\nabla|} f_0)(r) \sim \sum_{\pm} \chi_{\{ r \gg 1\}} r^{-\frac{1}{2} - \lambda} \int_0^\infty e^{2\pi i(t \pm r) \rho} \beta_{\pm}(2 \pi r \rho) \varphi_0(\rho) \hat{f}_0(\rho) \rho^{\frac{1}{2}} \, d\rho.
 \]
 Noting that the function $\rho \mapsto \hat{f}_0(\rho) \rho^{\frac{1}{2}}$ is compactly supported in the interval $(0, 4)$, we may consider its Fourier series 
 \[
  \hat{f}_0(\rho) \rho^{\frac{1}{2}} = \sum_{n \in \bZ} c_n e^{i \frac{\pi}{2} n \rho}, \quad \rho \in (0,4),
 \]
 whose Fourier coefficients satisfy
 \[
  \sum_{n \in \bZ} |c_n|^2 \sim \bigl\| \hat{f}_0(\rho) \rho^{\frac{1}{2}} \bigr\|^2_{L^2_\rho} \sim \| f_0 \|^2_{L^2_x}.
 \]
 We can therefore write 
 \[
  \chi_{\{ r \gg 1\}} r^{-\lambda} (e^{it|\nabla|} f_0)(r) \sim \sum_\pm r^{-\frac{1}{2} - \lambda} \sum_{n \in \bZ} c_n \psi^{\pm}_n(t,r) 
 \]
 with 
 \[
  \psi_n^\pm(t,r) = \chi_{\{ r \gg 1\}} \int_0^\infty e^{2 \pi i (t \pm r + \frac{n}{4}) \rho} \beta_{\pm}(2\pi r \rho) \varphi_0(\rho) \, d\rho.
 \]
 From the symbol-type bounds \eqref{equ:bessel_symbol_bounds} we obtain by repeated integration by parts that 
 \[
  |\psi_n(t,r)| \lesssim_M \Bigl( 1 + \bigl| t \pm r + \frac{n}{4} \bigr| \Bigr)^{-M}.
 \]
 Thus, choosing $M \gg 1$ sufficiently large, we infer from an application of H\"older's inequality (in $n \in \bZ$) and the embedding $\ell^2 \hookrightarrow \ell^p$ for $p \geq 2$ that 
 \begin{align*}
  \bigl\| \chi_{\{ r \gg 1\}} r^{-\lambda} e^{it|\nabla|} f_0 \bigr\|_{L^p_x} &\lesssim \sum_\pm \, \biggl( \int_1^\infty \Bigl( \sum_{n\in\bZ} |c_n| |\psi_n^\pm(t,r)| \Bigr)^p r^{1 - p(\frac{1}{2}+\lambda)} \, dr \biggr)^{\frac{1}{p}} \\
  &\lesssim \sum_\pm \, \biggl( \int_1^\infty \Bigl( \sum_{n\in\bZ} \frac{|c_n|}{(1 + |t \pm r + \frac{n}{4}|)^M} \Bigr)^p r^{1 - p(\frac{1}{2}+\lambda)} \, dr \biggr)^{\frac{1}{p}} \\
  &\lesssim \biggl( \int_1^\infty \Bigl( \sum_{n\in\bZ} \frac{|c_n|^p}{(1 + | |t+\frac{n}{4}| - r|)^{(M-2)p}} \Bigr) r^{1 - p(\frac{1}{2}+\lambda)} \, dr \biggr)^{\frac{1}{p}} \\
  &\lesssim \biggl( \sum_{n \in \bZ} \frac{|c_n|^p}{ (1+|t+\frac{n}{4}|)^{p(\frac{1}{2} + \lambda) - 1} } \biggr)^{\frac{1}{p}} \\
  &\lesssim \biggl( \sum_{n \in \bZ} \frac{|c_n|^2}{ (1+|t+\frac{n}{4}|)^{2(\frac{1}{2} + \lambda - \frac{1}{p})} } \biggr)^{\frac{1}{2}}.
 \end{align*}
 Since by assumption $\lambda > \frac{1}{p}$, we obtain the desired estimate
 \[
  \bigl\| \chi_{\{ r \gg 1\}} r^{-\lambda} e^{it|\nabla|} f_0 \bigr\|_{L^2_t L^p_x} \lesssim \Bigl( \sum_{n \in \bZ} |c_n|^2 \Bigr)^{\frac{1}{2}} \sim \|f_0\|_{L^2_x}.
 \]
\end{proof}

Now we are in the position to deduce a weighted $L^2_t L^\infty_x$ Strichartz estimate.

\begin{lem}[Weighted endpoint Strichartz estimate]
 Let $0 < \lambda < 1$. Let $k \in \bZ$ and let $I$ be any time interval containing $0$. For any radially symmetric function $\phi$ on $I \times \bR^2$ with Fourier support in $\{ |\xi| \sim 2^k \}$, we have 
 \begin{equation*}
  2^{(\frac{1}{2} - \lambda) k} \| \chi_{\{ r \geq 2^{-k}\}} r^{-\lambda} \phi \|_{L^2_t L^\infty_x[I]} \lesssim \| \nabla_{t,x} \phi(0) \|_{L^2_x} + \| \Box \phi \|_{L^1_t L^2_x[I]}.
 \end{equation*}
\end{lem}
\begin{proof}
 We begin by writing 
 \[
  \| \chi_{\{ r \geq 2^{-k}\}} r^{-\lambda} \phi \|_{L^2_t L^\infty_x} \lesssim \sum_{\ell \geq -k} \| \chi_{\{ 2^\ell \leq r \leq 2^{\ell+1}\}} r^{-\lambda} P_k \phi \|_{L^2_t L^\infty_x}.
 \]
 Since the Littlewood-Paley projection operator $P_k$ lives at spatial scale $2^{-k}$, for $\ell \geq -k$ it approximately preserves the spatial localization enforced by the cutoff $\chi_{\{ 2^{\ell} \leq r \leq 2^{\ell+1}\}}$  up to exponentially decaying tails that can be dealt with easily. We may therefore replace the right-hand side of the previous line by
 \[
  \sum_{\ell \geq -k} \| \chi_{\{ 2^\ell \leq r \leq 2^{\ell+1}\}} r^{-\lambda} \tilde{P}_k \bigl( \chi_{\{ 2^{\ell-2} \leq r \leq 2^{\ell+3}\}} P_k \phi \bigr) \|_{L^2_t L^\infty_x},
 \]
 where $\tilde{P}_k$ is a fattened Littlewood-Paley projection operator such that $\tilde{P}_k P_k \phi = P_k \phi$. Applying Bernstein's inequality for some $2 \leq p < \infty$ with $\frac{1}{p} < \frac{\lambda}{2} < \frac{2}{p}$, we arrive at the bound
 \[
  \sum_{\ell \geq -k} 2^{-\lambda \ell} 2^{\frac{2}{p} k} \| \chi_{\{ 2^{\ell-2} \leq r \leq 2^{\ell+3}\}} P_k \phi \|_{L^2_t L^p_x} \lesssim \sum_{\ell \geq -k} 2^{-\frac{\lambda}{2} \ell} 2^{\frac{2}{p} k} \| r^{-\frac{\lambda}{2}} P_k \phi \|_{L^2_t L^p_x}.
 \]
 Next we invoke the weighted Strichartz estimates from Lemma~\ref{lem:weighted_strichartz_estimate} to obtain the bound 
 \[
  \sum_{\ell \geq -k} 2^{-\frac{\lambda}{2} \ell} 2^{\frac{2}{p} k} 2^{-(\frac{1}{2} + \frac{2}{p} - \frac{\lambda}{2})k} \bigl( \|\nabla_{t,x} \phi(0)\|_{L^2_x} + \|\Box \phi\|_{L^1_t L^2_x} \bigr) \lesssim 2^{-(\frac{1}{2} - \lambda) k} \bigl( \|\nabla_{t,x} \phi(0)\|_{L^2_x} + \|\Box \phi\|_{L^1_t L^2_x} \bigr),
 \]
 which finishes the proof.
\end{proof}

Moreover, we have the following local energy decay estimate.

\begin{lem}[Local energy decay estimate]
 Let $k \in \bZ$ and let $I$ be any time interval containing $0$. For any radially symmetric function $\phi$ on $I \times \bR^2$ with Fourier support in $\{ |\xi| \sim 2^k \}$, it holds that 
 \begin{equation*}
  \sup_{\ell \in \bZ} \| \chi_{\{ r \sim 2^\ell\}} r^{-\frac{1}{2}} \nabla_{t,x} \phi \|_{L^2_t L^2_x[I]} \lesssim \|\nabla_{t,x} \phi(0)\|_{L^2_x} + \| \Box \phi \|_{L^1_t L^2_x[I]}.
 \end{equation*}
\end{lem}
\begin{proof}
 We proceed similarly to the proof of Lemma~\ref{lem:weighted_strichartz_estimate}. In view of Duhamel's formula, we may again assume that $\phi$ is of the form $\phi = e^{it|\nabla|} f_k$ for a radially symmetric function $f_k$ with Fourier support on $|\xi| \sim 2^k$. Moreover, by scaling invariance it suffices to consider the case $k=0$. 

 We first prove the local energy decay estimate for the time derivative $\partial_t (e^{it|\nabla|} f_0)$. Due to the radial symmetry assumption, we have the specific representation formula
 \[
  \partial_t (e^{it|\nabla|} f_0)(r) = 4 \pi^2 i \int_0^\infty e^{2 \pi i t \rho} J_0(2 \pi r \rho) \varphi_0(\rho) \hat{f}_0(\rho) \rho^2 \, d\rho.
 \]
 We distinguish the regimes $r \lesssim 1$ and $r \gg 1$. Here we only treat the more difficult case $r \gg 1$. By the asymptotics~\eqref{equ:bessel_asymptotics} we have that 
 \[
  \chi_{\{r \gg 1\}} r^{-\frac{1}{2}} \partial_t (e^{it|\nabla|} f_0)(r) \sim \sum_\pm \chi_{\{r \gg 1\}} r^{-1} \int_0^\infty e^{2 \pi i (t\pm r) \rho} \beta_\pm(2\pi r \rho) \varphi_0(\rho) \hat{f}_0(\rho) \rho^{\frac{3}{2}} \, d\rho.
 \]
 The function $\rho \mapsto \hat{f}_0(\rho) \rho^{\frac{3}{2}}$ has compact support in the interval $(0,4)$ and can therefore be developed into a Fourier series 
 \[
  \hat{f}_0(\rho) \rho^{\frac{3}{2}} = \sum_{n \in \bZ} c_n e^{i \frac{\pi}{2} n \rho}, \quad \rho \in (0,4),
 \]
 where the Fourier coefficients satisfy
 \[
  \sum_{n \in \bZ} |c_n|^2 \sim \bigl\| \hat{f}_0(\rho) \rho^{\frac{3}{2}} \bigr\|_{L^2_\rho}^2 \lesssim \|f_0\|_{L^2_x}^2.
 \]
 Now we follow closely the arguments in the proof of Lemma~\ref{lem:weighted_strichartz_estimate} to find that 
 \begin{align*}
  \sup_{\ell \geq 0} \, \bigl\| \chi_{\{ r \sim 2^\ell\}} r^{-\frac{1}{2}} \partial_t (e^{it|\nabla|} f_0)(r) \bigr\|_{L^2_t L^2_x} \lesssim \sup_{\ell \geq 0} \, \biggl( \sum_\pm \int_{\bR} \int_{\{ r \sim 2^\ell \}} \sum_{n \in \bZ} \frac{|c_n|^2}{(1 + |t \pm r + \frac{n}{4}|)^2} r^{-1} \, dr \, dt \biggr)^{\frac{1}{2}}.
 \end{align*}
 Changing the order of integration and computing the time integral first, we see that the right-hand side obeys the desired bound
 \begin{align*}
  \sup_{\ell \geq 0} \, \biggl( \int_{\{ r \sim 2^\ell \}} \sum_{n \in \bZ} |c_n|^2 r^{-1} \, dr \biggr)^{\frac{1}{2}} \lesssim \biggl( \sum_{n \in \bZ} |c_n|^2 \biggr)^{\frac{1}{2}} \lesssim \|f_0\|_{L^2_x}.
 \end{align*}
 
 The proof of the local energy decay estimate for the spatial derivatives of $e^{it|\nabla|} f_0$ proceeds analogously by noting that the first derivative $J_0'$ of the Bessel function $J_0$ of order $0$ satisfies the same asymptotics~\eqref{equ:bessel_asymptotics}~as~$J_0$.
\end{proof}

Finally, we turn to the atomic $Z_k^{\pm}$ spaces.

\begin{lem}[The $Z^{\pm}_k$ spaces] \label{lem:Z_space}
 Let $k \in \bZ$ and let $I$ be any time interval containing $0$. For any radially symmetric function $\phi$ on $I \times \bR^2$ with Fourier support in $\{ |\xi| \sim 2^k \}$, we have that 
 \begin{align*}
  \| (\partial_t + \partial_r) \phi_k \|_{Z_k^+[I]} &\lesssim \| \nabla_{t,x} \phi_k(0) \|_{L^2_x} + \| \Box \phi_k \|_{L^1_t L^2_x[I]}
 \intertext{and} 
  \| (\partial_t - \partial_r) \phi_k \|_{Z_k^-[I]} &\lesssim \| \nabla_{t,x} \phi_k(0) \|_{L^2_x} + \| \Box \phi_k \|_{L^1_t L^2_x[I]}.
 \end{align*}
\end{lem}
\begin{proof}
 We only provide the proof of the estimate for the $Z_k^+$ space since the case of the $Z_k^-$ space can be dealt with analogously. Moreover, in view of Duhamel's formula we may assume that $\phi_k$ is of the form $e^{it |\nabla|} f_k$ for a radially symmetric function $f_k$ with Fourier support on $\{ |\xi| \sim 2^k \}$. By scaling invariance it suffices to consider the case $k=0$. 
 
 Then we observe that 
 \begin{align*}
  (\partial_t + \partial_r) (e^{i t |\nabla|} f_0)(r) = 4 \pi^2 \int_0^\infty e^{2 \pi i t \rho} \bigl( i J_0(2 \pi r \rho) + J_0'(2 \pi r \rho) \bigr) \varphi_0(\rho) \hat{f_0}(\rho) \rho^2 \, d\rho.
 \end{align*}
 By stationary phase it is easy to see that we have the asymptotics
 \begin{equation} \label{equ:atomic_asymptotics}
  i J_0(y) + J_0'(y) = \frac{i}{2 \pi} \int_0^{2 \pi} e^{i y \sin(\theta)} ( 1 + \sin(\theta) ) \, d\theta = y^{-\frac{1}{2}} e^{+iy} \beta_+(y) + y^{-\frac{3}{2}} e^{-iy} \beta_-(y), \quad y \geq 1,
 \end{equation}
 where the functions $\beta_{\pm}$ satisfy the symbol-type bounds 
 \begin{equation*} 
  | \beta_{\pm}^{(n)}(y) | \lesssim_n y^{-n}, \quad n \geq 0, \, y \geq 1.
 \end{equation*}
 We now distinguish the regimes $r \lesssim 1$ and $r \gg 1$. Here we only turn to the more delicate latter case $r \gg 1$. By the asymptotics~\eqref{equ:atomic_asymptotics} we see that $\chi_{\{ r \gg 1 \}} (\partial_t + \partial_r) (e^{i t |\nabla|} f_0)$ decomposes into two components
 \begin{equation} \label{equ:decomposition_for_Z_space}
  \begin{aligned}
   \chi_{\{ r \gg 1\}} (\partial_t + \partial_r) (e^{i t |\nabla|} f_0)(r) &\sim \chi_{\{ r \gg 1\}} r^{-\frac{1}{2}} \int_0^\infty e^{2 \pi i (t+r) \rho} \beta_+(2 \pi r \rho) \varphi_0(\rho) \hat{f_0}(\rho) \rho^{\frac{3}{2}} \, d\rho \\
   &\quad + \chi_{\{ r \gg 1\}} r^{-\frac{3}{2}} \int_0^\infty e^{2 \pi i (t-r) \rho} \beta_{-}(2 \pi r \rho) \varphi_0(\rho) \hat{f_0}(\rho) \rho^{\frac{1}{2}} \, d\rho.
  \end{aligned}
 \end{equation}
 It is easy to see that both components on the right-hand side have finite $Y_0$ norm. Moreover, the first component on the right-hand side of~\eqref{equ:decomposition_for_Z_space} can be placed into the weighted $L^2_{t+r} L^\infty_{t-r}$ component of the $Z_0^+$ space. More precisely, for any $0 < \lambda \leq \frac{1}{2}$ we obtain by a standard $T T^\ast$ argument that 
 \begin{align*}
  &\biggl\| r^{+\lambda} \chi_{\{ r \gg 1\}} r^{-\frac{1}{2}} \int_0^\infty e^{2 \pi i (t+r) \rho} \beta_+(2 \pi r \rho) \varphi_0(\rho) \hat{f_0}(\rho) \rho^{\frac{3}{2}} \, d\rho \biggr\|_{L^2_{t+r} L^\infty_{t-r}} \\
  &\lesssim \biggl\| \int_0^\infty e^{2 \pi i (t+r) \rho} \chi_{\{ r \gg 1\}} \beta_+(2 \pi r \rho) \varphi_0(\rho) \hat{f_0}(\rho) \rho^{\frac{3}{2}} \, d\rho \biggr\|_{L^2_{t+r} L^\infty_{t-r}} \\
  &\lesssim \bigl\| \hat{f_0}(\rho) \rho^{\frac{3}{2}} \bigr\|_{L^2_\rho} \\
  &\lesssim \| f_0 \|_{L^2_x}.
 \end{align*}
 On the other hand, the second component on the right-hand side of~\eqref{equ:decomposition_for_Z_space} satisfies for any exponent pair $(q, p)$ with $2 \leq q, p \leq \infty$ and $\frac{1}{q} + \frac{1}{p} < \frac{1}{2}$ and for any $0 < \lambda < 1$ that
 \begin{align*}
  &\sum_{\ell \in \bZ} \, \biggl\| \, \chi_{\{ r \sim 2^{\ell} \}} r^{+\lambda} \chi_{\{ r \gg 1\}} r^{-\frac{3}{2}} \int_0^\infty e^{2 \pi i (t-r) \rho} \beta_{-}(2 \pi r \rho) \varphi_0(\rho) \hat{f_0}(\rho) \rho^{\frac{1}{2}} \, d\rho \biggr\|_{L^q_t L^p_x} \\
  &\lesssim \sum_{\ell \gg 1} 2^{(\lambda - 1) \ell} \, \biggl\| r^{-\frac{1}{2}} \int_0^\infty e^{2 \pi i (t-r) \rho} \chi_{\{ r \gg 1\}} \beta_{-}(2 \pi r \rho) \varphi_0(\rho) \hat{f_0}(\rho) \rho^{\frac{1}{2}} \, d\rho \biggr\|_{L^q_t L^p_x} \\
  &\lesssim \biggl\| r^{-\frac{1}{2}} \int_0^\infty e^{2 \pi i (t-r) \rho} \chi_{\{ r \gg 1\}} \beta_{-}(2 \pi r \rho) \varphi_0(\rho) \hat{f_0}(\rho) \rho^{\frac{1}{2}} \, d\rho \biggr\|_{L^q_t L^p_x} \\
  &\lesssim \| f_0 \|_{L^2_x},
 \end{align*}
 where in the case of $(q, p) = (\infty, \infty)$ the last estimate just follows from H\"older's inequality while for all other admissible exponent pairs $(q,p)$ the last estimate follows from the proof of Strichartz estimates in the radial case as in Sterbenz~\cite{Sterbenz}. Putting things together, we conclude that
 \[
  \| (\partial_t + \partial_r)  (e^{i t |\nabla|} f_0) \|_{Z_0^+} \lesssim \| f_0 \|_{L^2_x},
 \]
 which finishes the proof of Lemma~\ref{lem:Z_space}.
\end{proof}

\section{Multilinear estimates} \label{sec:multilinear_estimates}

Here we collect several important multilinear estimates that will be of crucial use to estimate the wave maps nonlinearity. We begin with the following null form estimate.

\begin{prop}[Null form estimate] \label{prop:null_form_estimate}
For $\frac{3}{2} \leq p \leq \infty$ it holds that
\begin{equation} \label{equ:null_form_estimate}
 \bigl\| \partial_\alpha \phi_{k_1}^{(1)} \, \partial^\alpha \phi_{k_2}^{(2)} \bigr\|_{L^p_t L^p_x} \lesssim 2^{(1-\frac{3}{2p}) k_1} 2^{(1-\frac{3}{2p}) k_2} \| \phi_{k_1}^{(1)} \|_{S_{k_1}} \| \phi_{k_2}^{(2)} \|_{S_{k_2}}.
\end{equation}
\end{prop}
\begin{proof}
The assertion follows by interpolation between the simple $L^\infty_t L^\infty_x$ estimate
\[
 \bigl\| \partial_\alpha \phi_{k_1}^{(1)} \, \partial^\alpha \phi_{k_2}^{(2)} \bigr\|_{L^\infty_t L^\infty_x} \lesssim 2^{k_1} \bigl\| \nabla_{t,x} \phi_{k_1} \bigr\|_{L^\infty_t L^2_x} 2^{k_2} \bigl\| \nabla_{t,x} \phi_{k_2} \bigr\|_{L^\infty_t L^2_x} \lesssim 2^{k_1} 2^{k_2} \|\phi_{k_1}^{(1)}\|_{S_{k_1}} \|\phi_{k_2}^{(2)}\|_{S_{k_2}}
\]
and the following $L^{\frac{3}{2}}_t L^{\frac{3}{2}}_x$ estimate
\begin{equation} \label{equ:null_form_estimate_proof_L32}
 \bigl\| \partial_\alpha \phi_{k_1}^{(1)} \, \partial^\alpha \phi_{k_2}^{(2)} \bigr\|_{L^{\frac{3}{2}}_t L^{\frac{3}{2}}_x} \lesssim  \|\phi_{k_1}^{(1)}\|_{S_{k_1}} \|\phi_{k_2}^{(2)}\|_{S_{k_2}},
\end{equation}
which we now prove. Since in the radially symmetric setting it holds that
\[
 -2 \partial_\alpha \phi \partial^\alpha \psi = (\partial_t + \partial_r) \phi (\partial_t - \partial_r) \psi + (\partial_t - \partial_r) \phi (\partial_t + \partial_r) \psi,
\]
by symmetry it suffices to consider bounding the expression $(\partial_t + \partial_r) \phi_{k_1}^{(1)} (\partial_t - \partial_r) \phi_{k_2}^{(2)}$ in $L^{\frac{3}{2}}_t L^{\frac{3}{2}}_x$. To this end we decompose the two inputs into $(+)$ and $(-)$ components, i.e.
\[
 (\partial_t + \partial_r) \phi_{k_1}^{(1)} = (\partial_t + \partial_r) \phi_{k_1}^{(1,+)} + (\partial_t + \partial_r) \phi_{k_1}^{(1,-)}
\]
and
\[
 (\partial_t - \partial_r) \phi_{k_2}^{(2)} = (\partial_t - \partial_r) \phi_{k_2}^{(2,+)} + (\partial_t - \partial_r) \phi_{k_2}^{(2,-)}.
\]
Then we consider all possible interactions. In the case of $(+)/(+)$ interactions we estimate 
\begin{align*}
 &\bigl\| (\partial_t + \partial_r) \phi_{k_1}^{(1,+)} (\partial_t - \partial_r) \phi_{k_2}^{(2,+)} \bigr\|_{L^{\frac{3}{2}}_t L^{\frac{3}{2}}_x} \\ 
 &\lesssim \sup_{\ell \in \bZ} \, \bigl\| \chi_{\{ r \sim 2^\ell\}} r^{-\frac{1}{2}} (\partial_t + \partial_r) \phi_{k_1}^{(1,+)} \bigr\|_{L^2_t L^2_x} \sum_{\ell \in \bZ} \, \bigl\| \chi_{\{ r \sim 2^\ell\}} r^{+\frac{1}{2}} (\partial_t - \partial_r) \phi_{k_2}^{(2, +)} \bigr\|_{L^6_t L^6_x} \\
 &\lesssim \| \phi_{k_1}^{(1)} \|_{S_{k_1}} \| \phi_{k_2}^{(2)} \|_{S_{k_2}}
\end{align*}
and the case of $(-)/(+)$ as well as of $(-)/(-)$ interactions can be bounded analogously. We are therefore left with the treatment of the delicate $(+)/(-)$ interactions. By spherical symmetry, we may write
\begin{align*}
 &\bigl\| (\partial_t + \partial_r) \phi_{k_1}^{(1,+)} (\partial_t - \partial_r) \phi_{k_2}^{(2,-)} \bigr\|_{L^{\frac{3}{2}}_t L^{\frac{3}{2}}_x}^{\frac{3}{2}} \lesssim \iint r^{-\frac{1}{2}} \bigl| r^{\frac{1}{2}} (\partial_t + \partial_r) \phi_{k_1}^{(1,+)} \bigr|^{\frac{3}{2}} \bigl| r^{\frac{1}{2}} (\partial_t - \partial_r) \phi_{k_2}^{(2,-)} \bigr|^{\frac{3}{2}} \, dr \, dt.
\end{align*}
Changing variables from $(t,r)$ to $(t+r, t-r)$ and noting that $r = \frac{1}{2}( (t+r) - (t-r) )$, we can estimate the previous line by
\[
 \iint |(t+r)-(t-r)|^{-\frac{1}{2}} \, \Bigl( \bigl\| r^{\frac{1}{2}} (\partial_t + \partial_r) \phi_{k_1}^{(1,+)} \bigr\|_{L^\infty_{t-r}}^{\frac{3}{2}} \Bigr)(t+r) \, \Bigl( \bigl\| r^{\frac{1}{2}} (\partial_t - \partial_r) \phi_{k_2}^{(2,-)} \bigr\|_{L^\infty_{t+r}}^{\frac{3}{2}}\Bigr)(t-r) \, d(t+r) \, d(t-r).
\]
Next we apply H\"older's inequality followed by the Hardy-Littlewood-Sobolev inequality and obtain the desired bound
\begin{align*}
 &\Bigl\| \bigl\| r^{\frac{1}{2}} (\partial_t + \partial_r) \phi_{k_1}^{(1,+)} \bigr\|_{L^\infty_{t-r}}^{\frac{3}{2}} \Bigr\|_{L^{\frac{4}{3}}_{t+r}} \, \biggl\| \int |(t+r)-(t-r)|^{-\frac{1}{2}} \, \Bigl( \bigl\| r^{\frac{1}{2}} (\partial_t - \partial_r) \phi_{k_2}^{(2,-)} \bigr\|_{L^\infty_{t+r}}^{\frac{3}{2}}\Bigr)(t-r) \, d(t-r) \biggr\|_{L^4_{t+r}} \\
 &\lesssim \bigl\| r^{\frac{1}{2}} (\partial_t + \partial_r) \phi_{k_1}^{(1,+)} \bigr\|_{L^2_{t+r} L^\infty_{t-r}}^{\frac{3}{2}} \, \Bigl\| \bigl\| r^{\frac{1}{2}} (\partial_t - \partial_r) \phi_{k_2}^{(2,-)} \bigr\|_{L^\infty_{t+r}}^{\frac{3}{2}} \Bigr\|_{L^{\frac{4}{3}}_{t-r}} \\
 &\lesssim \bigl\| r^{\frac{1}{2}} (\partial_t + \partial_r) \phi_{k_1}^{(1,+)} \bigr\|_{L^2_{t+r} L^\infty_{t-r}}^{\frac{3}{2}} \, \bigl\| r^{\frac{1}{2}} (\partial_t - \partial_r) \phi_{k_2}^{(2,-)} \bigr\|_{L^2_{t-r} L^\infty_{t+r}}^{\frac{3}{2}} \\
 &\lesssim \| \phi_{k_1}^{(1)} \|_{S_{k_1}}^{\frac{3}{2}} \| \phi_{k_2}^{(2)} \|_{S_{k_2}}^{\frac{3}{2}}.
\end{align*}
This finishes the proof of the $L^{\frac{3}{2}}_t L^{\frac{3}{2}}_x$ estimate \eqref{equ:null_form_estimate_proof_L32} and thus concludes the proof of Proposition~\ref{prop:null_form_estimate}.
\end{proof}

We emphasize that the preceding proposition allows to estimate the null form $\partial_\alpha \phi \partial^\alpha \psi$ ``below'' $L^2_t L^2_x$, in fact in $L^{\frac{3}{2}}_t L^{\frac{3}{2}}_x$. This feature will be crucial in the concentration compactness step in Section~\ref{sec:Concentration_Compactness_Step} to gain smallness for certain multilinear expressions, which is explained in detail in Subsection~\ref{subsubsec:gain_smallness_null_form_below_L2}. It was first pointed out by Klainerman-Machedon~\cite{KlMa1} that the null form $\partial_\alpha \phi \partial^\alpha \psi$ can be estimated ``below'' $L^2_t L^2_x$ provided both factors are free radial waves. The corresponding improvements without the radiality assumption are due to Bourgain~\cite{Bourgain}, Wolff~\cite{Wolff} and Tao~\cite{Tao_bilinear_restriction}. Next we establish a weighted null form estimate.

\begin{prop}[Weighted null form estimate] \label{prop:weighted_null_form_estimate}
For $0 < \lambda \leq \frac{1}{2}$ it holds that 
\begin{equation} \label{equ:weighted_null_form_estimate}
 \bigl\| r^{+ \lambda} \partial_\alpha \phi_{k_1}^{(1)} \partial^\alpha \phi_{k_2}^{(2)} \bigr\|_{L^2_t L^2_x} \lesssim 2^{(\frac{1}{2} - \lambda) \min \{ k_1, k_2 \} } \| \phi_{k_1}^{(1)} \|_{S_{k_1}} \| \phi_{k_2}^{(2)} \|_{S_{k_2}}.
\end{equation}
\end{prop}
\begin{proof}
 Without loss of generality we may assume that $k_1 \leq k_2$. Then we can dispense with the case when the radial variable is restricted to the range $r \leq 2^{-k_1}$ because here we can just bound by
 \begin{align*}
  \bigl\| \chi_{\{ r \leq 2^{-k_1} \}} r^{+\lambda} \partial_\alpha \phi_{k_1}^{(1)} \partial^\alpha \phi_{k_2}^{(2)} \bigr\|_{L^2_t L^2_x} &\lesssim \sum_{\ell \leq -k_1} \bigl\| \chi_{\{ r \sim 2^\ell \}} r^{\frac{1}{2} + \lambda} \partial_\alpha \phi_{k_1}^{(1)} \bigr\|_{L^\infty_t L^\infty_x} \bigl\| \chi_{\{ r \sim 2^\ell \}} r^{-\frac{1}{2}} \partial^\alpha \phi_{k_2}^{(2)} \bigr\|_{L^2_t L^2_x} \\
  &\lesssim \sum_{\ell \leq -k_1} 2^{(\frac{1}{2} + \lambda) \ell} 2^{k_1} \| \nabla_{t,x} \phi_{k_1}^{(1)} \|_{L^\infty_t L^2_x} \sup_{\ell \in \bZ} \, \bigl\| \chi_{\{ r \sim 2^\ell \}} r^{-\frac{1}{2}} \nabla_{t,x} \phi_{k_2}^{(2)} \bigr\|_{L^2_t L^2_x} \\
  &\lesssim 2^{(\frac{1}{2} - \lambda) k_1} \| \phi_{k_1}^{(1)} \|_{S_{k_1}} \| \phi_{k_2}^{(2)} \|_{S_{k_2}}.
 \end{align*}
 Moreover, since in the radially symmetric setting the null form $\partial_\alpha \phi_{k_1}^{(1)} \partial^\alpha \phi_{k_2}^{(2)}$ can be written as a linear combination of terms $(\partial_t \pm \partial_r) \phi_{k_1}^{(1)} (\partial_t \mp \partial_r) \phi_{k_2}^{(2)}$, we are left to estimate
 \[
  \bigl\| \chi_{\{ r > 2^{-k_1}\}} r^{+\lambda} (\partial_t + \partial_r) \phi_{k_1}^{(1)} (\partial_t - \partial_r) \phi_{k_2}^{(2)} \bigr\|_{L^2_t L^2_x}.
 \]
 To this end we again decompose the two inputs into $(+)$ and $(-)$ components
 \[
  (\partial_t \pm \partial_r) \phi_{k_j}^{(j)} = (\partial_t \pm \partial_r) \phi_{k_j}^{(j,+)} + (\partial_t \pm \partial_r) \phi_{k_j}^{(j,-)}, \quad j = 1,2
 \]
 and then estimate all possible interactions. In the case of $(+)/(-)$ interactions, we have that
 \begin{align*}
  \bigl\| \chi_{\{ r > 2^{-k_1}\}} r^{+\lambda} (\partial_t + \partial_r) \phi_{k_1}^{(1,+)} (\partial_t - \partial_r) \phi_{k_2}^{(2,-)} \bigr\|_{L^2_t L^2_x} &\lesssim \bigl\| r^{+\lambda} (\partial_t + \partial_r) \phi_{k_1}^{(1,+)} \bigr\|_{L^2_{t+r} L^\infty_{t-r}} \bigl\| r^{+\frac{1}{2}} (\partial_t - \partial_r) \phi_{k_2}^{(2,-)} \bigr\|_{L^2_{t-r} L^\infty_{t+r}} \\
  &\lesssim 2^{(\frac{1}{2}-\lambda) k_1} \| \phi_{k_1}^{(1)} \|_{S_{k_1}} \| \phi_{k_2}^{(2)} \|_{S_{k_2}}.
 \end{align*}
 Then we bound the $(+)/(+)$ interactions by
 \begin{align*}
  &\bigl\| \chi_{\{ r > 2^{-k_1}\}} r^{+\lambda} (\partial_t + \partial_r) \phi_{k_1}^{(1,+)} (\partial_t - \partial_r) \phi_{k_2}^{(2,-)} \bigr\|_{L^2_t L^2_x} \\
  &\lesssim \bigl\| \chi_{\{ r > 2^{-k_1} \}} r^{-\lambda} (\partial_t + \partial_r) \phi_{k_1}^{(+)} \bigr\|_{L^2_t L^\infty_x} \sum_{\ell \in \bZ} \, \bigl\| \chi_{\{ r \sim 2^\ell\}} r^{+ 2 \lambda} (\partial_t - \partial_r) \phi_{k_2}^{(+)} \bigr\|_{L^\infty_t L^2_x} \\
  &\lesssim 2^{(\frac{1}{2} + \lambda) k_1} \|\phi_{k_1}^{(1)} \|_{S_{k_1}} 2^{-2\lambda k_2} \|\phi_{k_2}^{(2)}\|_{S_{k_2}} \\
  &\lesssim 2^{(\frac{1}{2} - \lambda) k_1} \|\phi_{k_1}^{(1)} \|_{S_{k_1}} \|\phi_{k_2}^{(2)}\|_{S_{k_2}}.
 \end{align*}
 In the case of $(-)/(-)$ interactions we estimate as follows
 \begin{align*}
  &\bigl\| \chi_{\{ r > 2^{-k_1}\}} r^{+\lambda} (\partial_t + \partial_r) \phi_{k_1}^{(1,-)} (\partial_t - \partial_r) \phi_{k_2}^{(2,-)} \bigr\|_{L^2_t L^2_x} \\
  &\lesssim \sum_{\ell \in \bZ} \, \bigl\| \chi_{\{ r \sim 2^\ell \}} r^{\frac{1}{2} + \lambda} (\partial_t + \partial_r) \phi_{k_1}^{(1,-)} \bigr\|_{L^\infty_t L^\infty_x} \sup_{\ell \in \bZ} \, \bigl\| \chi_{\{ r \sim 2^\ell \}} r^{-\frac{1}{2}} (\partial_t - \partial_r) \phi_{k_2}^{(2,+)} \bigr\|_{L^2_t L^2_x} \\
  &\lesssim 2^{(\frac{1}{2}-\lambda) k_1} \|\phi_{k_1}^{(1)}\|_{S_{k_1}} \|\phi_{k_2}^{(2)}\|_{S_{k_2}}
 \end{align*}
 and the remaining case of $(-)/(+)$ interactions can be treated in exactly the same manner.
\end{proof}

We conclude with the following delicate trilinear estimate.

\begin{prop}[Trilinear estimate] \label{prop:trilinear_estimate}
 We have that
 \begin{equation} \label{equ:trilinear_estimate}
  \bigl\| L \bigl( \phi_{k_1}^{(1)}, \partial_\alpha \phi_{k_2}^{(2)}, \partial^\alpha \phi_{k_3}^{(3)} \bigr) \bigr\|_{L^1_t L^2_x} \lesssim 2^{- \frac{1}{4} ( k_1 - \min \{ k_2, k_3 \} )} \|\phi_{k_1}^{(1)}\|_{S_{k_1}} \|\phi_{k_2}^{(2)}\|_{S_{k_2}} \|\phi_{k_3}^{(3)}\|_{S_{k_3}}
 \end{equation}
 for $k_1 \geq \min \{k_2, k_3\} + O(1)$.
\end{prop}
\begin{proof}
 We may assume without loss of generality that $k_2 \leq k_3$. Moreover, we may restrict the radial variable to the range $\{ r > 2^{-k_2} \}$, because otherwise we can just easily estimate 
 \begin{align*}
  &\bigl\| \chi_{\{ r \leq 2^{-k_2}\}} L \bigl( \phi_{k_1}^{(1)}, \partial_\alpha \phi_{k_2}^{(2)}, \partial^\alpha \phi_{k_3}^{(3)} \bigr) \bigr\|_{L^1_t L^2_x} \\ 
  &\lesssim \| \phi_{k_1}^{(1)} \|_{L^4_t L^\infty_x} \sum_{\ell \leq -k_2} \| \chi_{\{ r \sim 2^\ell \}} r^{+\frac{1}{2}} \nabla_{t,x} \phi_{k_2}^{(2)} \|_{L^4_t L^\infty_x} \sup_{\ell \in \bZ} \| \chi_{\{r \sim 2^{\ell}\}} r^{-\frac{1}{2}} \nabla_{t,x} \phi_{k_3}^{(3)} \|_{L^2_t L^2_x} \\
  &\lesssim 2^{-\frac{1}{4} (k_1 - k_2)} \|\phi_{k_1}^{(1)}\|_{S_{k_1}} \|\phi_{k_2}^{(2)}\|_{S_{k_2}} \|\phi_{k_3}^{(3)}\|_{S_{k_3}}.
 \end{align*}
 In the radially symmetric setting we have the identity
 \[
  -2 \partial_\alpha \phi \partial^\alpha \psi = (\partial_t + \partial_r) \phi (\partial_t - \partial_r) \psi + (\partial_t - \partial_r) \phi (\partial_t + \partial_r) \psi.
 \]
 By symmetry it therefore suffices to now bound the expression
 \[
  \bigl\| \chi_{\{r > 2^{-k_2}\}} L \bigl( \phi_{k_1}^{(1)}, (\partial_t + \partial_r) \phi_{k_2}^{(2)}, (\partial_t - \partial_r) \phi_{k_3}^{(3)} \bigr) \bigr\|_{L^1_t L^2_x}. 
 \]
 To this end we decompose $(\partial_t + \partial_r) \phi_{k_2}^{(2)}$ and $(\partial_t - \partial_r) \phi_{k_3}^{(3)}$ into $(+)$ and $(-)$ components 
 \[
  (\partial_t \pm \partial_t) \phi_{k_j}^{(j)} = (\partial_t \pm \partial_r) \phi_{k_j}^{(j,+)} + (\partial_t \pm \partial_r) \phi_{k_j}^{(j,-)} \quad \text{for } j = 2, 3
 \]
 and distinguish all possible interaction scenarios. In the case of $(+)/(-)$ interactions, we estimate
 \begin{align*}
  &\bigl\| \chi_{\{r > 2^{-k_2}\}} L \bigl( \phi_{k_1}^{(1)}, (\partial_t + \partial_r) \phi_{k_2}^{(2, +)}, (\partial_t - \partial_r) \phi_{k_3}^{(3, -)} \bigr) \bigr\|_{L^1_t L^2_x} \\
  &\lesssim \bigl\| \chi_{\{ r > 2^{-k_1}\}} r^{-\frac{1}{4}} \phi_{k_1}^{(1)} \bigr\|_{L^2_t L^\infty_x} \bigl\| r^{+\frac{1}{4}} (\partial_t + \partial_r) \phi_{k_2}^{(2,+)} \bigr\|_{L^2_{t+r} L^\infty_{t-r}} \bigl\| r^{+\frac{1}{2}} (\partial_t - \partial_r) \phi_{k_3}^{(3,-)} \bigr\|_{L^2_{t-r} L^\infty_{t+r}} \\
  &\lesssim 2^{- \frac{1}{4} (k_1 - k_2)} \|\phi_{k_1}^{(1)}\|_{S_{k_1}} \|\phi_{k_2}^{(2)}\|_{S_{k_2}} \|\phi_{k_3}^{(3)}\|_{S_{k_3}},
 \end{align*}
 where we used that $k_1 \geq k_2 + O(1)$ by assumption. Then we bound the $(-)/(+)$ interactions by
 \begin{align*}
  &\bigl\| \chi_{\{r > 2^{-k_2}\}} L \bigl( \phi_{k_1}^{(1)}, (\partial_t + \partial_r) \phi_{k_2}^{(2, -)}, (\partial_t - \partial_r) \phi_{k_3}^{(3, +)} \bigr) \bigr\|_{L^1_t L^2_x} \\
  &\lesssim \bigl\| \chi_{\{ r > 2^{-k_1}\}} r^{-\frac{1}{4}} \phi_{k_1}^{(1)} \bigr\|_{L^2_t L^\infty_x} \sum_{\ell \in \bZ} \, \bigl\| \chi_{\{ r \sim 2^\ell \}} r^{+\frac{3}{4}} (\partial_t + \partial_r) \phi_{k_2}^{(2,-)} \bigr\|_{L^\infty_t L^\infty_x} \sup_{\ell \in \bZ} \, \bigl\| \chi_{\{r \sim 2^\ell\}} r^{-\frac{1}{2}} \nabla_{t,x} \phi_{k_3}^{(3,+)} \bigr\|_{L^2_t L^2_x} \\
  &\lesssim 2^{-\frac{1}{4} (k_1 - k_2)} \|\phi_{k_1}^{(1)}\|_{S_{k_1}} \|\phi_{k_2}^{(2)}\|_{S_{k_2}} \|\phi_{k_3}^{(3)}\|_{S_{k_3}}
 \end{align*}
 and the case of $(-)/(-)$ interactions can be treated in exactly the same manner. Finally, we have to consider the case of $(+)/(+)$ interactions
 \begin{align*}
  &\bigl\| \chi_{\{r > 2^{-k_2}\}} L \bigl( \phi_{k_1}^{(1)}, (\partial_t + \partial_r) \phi_{k_2}^{(2, +)}, (\partial_t - \partial_r) \phi_{k_3}^{(3, +)} \bigr) \bigr\|_{L^1_t L^2_x} \\
  &\lesssim \bigl\| \chi_{\{ r > 2^{-k_1} \}} r^{-\frac{1}{4}} \phi_{k_1}^{(1)} \bigr\|_{L^2_t L^\infty_x} \bigl\| \chi_{\{ r > 2^{-k_2}\}} r^{-\frac{1}{4}} (\partial_t + \partial_r) \phi_{k_2}^{(2,+)} \bigr\|_{L^2_t L^\infty_x} \sum_{\ell \in \bZ} \, \bigl\| \chi_{\{ r \sim 2^\ell \}} r^{+\frac{1}{2}} (\partial_t - \partial_r) \phi_{k_3}^{(3,+)} \bigr\|_{L^\infty_t L^2_x} \\
  &\lesssim 2^{-\frac{1}{4}(k_1 - k_2)} 2^{-\frac{1}{2}(k_3-k_2)} \|\phi_{k_1}^{(1)}\|_{S_{k_1}} \|\phi_{k_2}^{(2)}\|_{S_{k_2}} \|\phi_{k_3}^{(3)}\|_{S_{k_3}},
 \end{align*}
 which finishes the proof.
\end{proof}

\section{Decomposition of the nonlinearity and renormalization} \label{sec:decomposition_nonlinearity}

In this section we study the structure of the nonlinearity in the frequency localized wave maps equation
\begin{equation}
 \Box \phi_k = - P_k \bigl( \phi \partial_\alpha \phi^\dagger \partial^\alpha \phi \bigr), \quad k \in \bZ.
\end{equation}
Our main tools to estimate this wave maps nonlinearity are the null form estimate~\eqref{equ:null_form_estimate} and the trilinear estimate~\eqref{equ:trilinear_estimate} from the previous section. However, these will turn out to be far from sufficient, even for small energies. Accordingly, using Littlewood-Paley theory we will ``peel off'' the ``good'' parts from the nonlinearity to isolate its non-perturbative parts. Then we will introduce the renormalization procedure to deal with the latter parts of the nonlinearity. For the sake of readability we decided to only introduce a simplified version of our renormalization procedure in this section. This version suffices for small energies, but in order to handle large energies the construction of our renormalization procedure will, roughly speaking, involve further frequency and spatial truncations. We will explain the necessary modifications as we go in the later sections of this paper. 

\medskip

Our starting point is the following decomposition of the wave maps nonlinearity at fixed frequency $2^k$,
\begin{align}
 &P_k \bigl( \phi \partial_\alpha \phi^\dagger \partial^\alpha \phi \bigr) \nonumber \\
 &= 2 \sum_{k_2 \leq k-10} \phi_{\leq k_2-10} \partial_\alpha \phi_{k_2}^\dagger \partial^\alpha \phi_k \label{equ:WM_nonlinearity_bad0} \\
 &\quad + \sum_{k_2 > k-10} \sum_{k_3 = k_2+O(1)} P_k \bigl( \phi_{\leq k_2 + 20} \partial_\alpha \phi_{k_2}^\dagger \partial^\alpha \phi_{k_3} \bigr) \label{equ:WM_nonlinearity_bad1} \\
 &\quad + 2 \sum_{k_2 \leq k-10} \phi_{\leq k_2-10} \bigl( P_k \bigl( \partial_\alpha \phi_{k_2}^\dagger \partial^\alpha \phi_{k-5 < \cdot < k+5} \bigr) - \partial_\alpha \phi_{k_2}^\dagger \partial^\alpha \phi_k \bigr). \label{equ:WM_nonlinearity_bad2} \\ 
 &\quad + \sum_{k_2 \leq k-10} \sum_{k_3 = k + O(1)} P_k L \bigl( \phi_{k_2 - 10 < \cdot \leq k-10}, \partial_\alpha \phi_{k_2}, \partial^\alpha \phi_{k_3} \bigr) \label{equ:WM_nonlinearity_derivative_on_lowest_freq1}  \\
 &\quad + \sum_{k_2 > k-10} \sum_{k_1 = k_2 + O(1)} \sum_{k_3 \leq k_2 + O(1)} P_k L \bigl( \phi_{k_1}, \partial_\alpha \phi_{k_2}, \partial^\alpha \phi_{k_3} \bigr) \label{equ:WM_nonlinearity_derivative_on_lowest_freq2} \\
 &\quad + \sum_{k_2 \leq k-10} \sum_{k_3 = k + O(1)} 2^{-k} P_k L \bigl( \nabla_x \phi_{\leq k_2 - 10}, \partial_\alpha \phi_{k_2}, \partial^\alpha \phi_{k_3} \bigr). \label{equ:WM_nonlinearity_derivative_transferred_to_lowest_freq} 
\end{align}
In order to arrive at this decomposition, we begin with
\[
 P_k \bigl( \phi \partial_\alpha \phi^\dagger \partial^\alpha \phi \bigr) = P_k \bigl( \phi \partial_\alpha \phi^\dagger_{\leq k-10} \partial^\alpha \phi_{\leq k-10} \bigr) + 2 P_k \bigl( \phi \partial_\alpha \phi_{\leq k-10}^\dagger \partial^\alpha \phi_{> k-10} \bigr) + P_k \bigl( \phi \partial_\alpha \phi^\dagger_{>k-10} \partial^\alpha \phi_{>k-10} \bigr)
\]
and then further decompose into
\begin{align}
 &P_k \bigl( \phi \partial_\alpha \phi^\dagger \partial^\alpha \phi \bigr) \nonumber \\
 &= P_k \bigl( \phi \partial_\alpha \phi^\dagger_{\leq k-10} \partial^\alpha \phi_{\leq k-10} \bigr) \label{equ:WM_nonlinearity_precise_decomp_term1} \\
 &\quad + 2 \sum_{k_1 > k-10} P_k \bigl( \phi_{k_1} \partial_\alpha \phi_{\leq k-10}^\dagger \partial^\alpha \phi_{k_1 + O(1)} \bigr) \label{equ:WM_nonlinearity_precise_decomp_term2} \\
 &\quad + 2 \sum_{k_2 \leq k-10} P_k \bigl( \phi_{k_2 -10 < \cdot \leq k-10} \partial_\alpha \phi_{k_2}^\dagger \partial^\alpha \phi_{k-5 < \cdot < k+5} \bigr) \label{equ:WM_nonlinearity_precise_decomp_term3} \\
 &\quad + 2 \sum_{k_2 \leq k-10} P_k \bigl( \phi_{\leq k_2-10} \partial_\alpha \phi_{k_2}^\dagger \partial^\alpha \phi_{k-5 < \cdot < k+5} \bigr) \label{equ:WM_nonlinearity_precise_decomp_term4} \\
 &\quad + \sum_{k_2 > k-10} P_k \bigl( \phi_{k_2-5 < \cdot < k_2+5} \partial_\alpha \phi_{k_2}^\dagger \partial^\alpha \phi_{\leq k_2 -10} \bigr) \label{equ:WM_nonlinearity_precise_decomp_term5} \\
 &\quad + \sum_{k_2 > k-10} P_k \bigl( \phi_{\leq k_2 + 20} \partial_\alpha \phi_{k_2}^\dagger \partial^\alpha \phi_{k_2-10 < \cdot \leq k_2+10} \bigr) \label{equ:WM_nonlinearity_precise_decomp_term6} \\
 &\quad + \sum_{k_2 > k-10} \sum_{k_3 > k_2 + 10} P_k \bigl( \phi_{k_3 - 5 < \cdot < k_3 + 5} \partial_\alpha \phi_{k_2}^\dagger \partial^\alpha \phi_{k_3} \bigr). \label{equ:WM_nonlinearity_precise_decomp_term7}
\end{align}
The first term~\eqref{equ:WM_nonlinearity_precise_decomp_term1} can be estimated in the same manner as the term~\eqref{equ:WM_nonlinearity_derivative_on_lowest_freq2} and is therefore not further included in the decomposition \eqref{equ:WM_nonlinearity_bad0}--\eqref{equ:WM_nonlinearity_derivative_transferred_to_lowest_freq}. The second term~\eqref{equ:WM_nonlinearity_precise_decomp_term2} is of type~\eqref{equ:WM_nonlinearity_derivative_on_lowest_freq2}, while the third term~\eqref{equ:WM_nonlinearity_precise_decomp_term3} is of type~\eqref{equ:WM_nonlinearity_derivative_on_lowest_freq1}. Using the Leibniz rule~\eqref{equ:leibniz_rule} for the projection $P_k$, we may write the fourth term~\eqref{equ:WM_nonlinearity_precise_decomp_term4} as
\begin{equation} \label{equ:WM_nonlinearity_precise_decomp_term3_expanded}
 \begin{aligned}
  & 2 \sum_{k_2 \leq k-10} P_k \bigl( \phi_{k_2 -10 < \cdot \leq k-10} \partial_\alpha \phi_{k_2}^\dagger \partial^\alpha \phi_{k-5 < \cdot < k+5} \bigr) \\
  &\quad = 2 \sum_{k_2 \leq k-10} \phi_{\leq k_2 -10} \partial_\alpha \phi_{k_2}^\dagger \partial^\alpha \phi_k \\
  &\quad \quad + 2 \sum_{k_2 \leq k-10} 2^{-k} P_k L \bigl( \nabla_x \phi_{\leq k_2 -10}, \partial_\alpha \phi_{k_2}, \partial^\alpha \phi_{k-5 < \cdot < k+5} \bigr) \\
  &\quad \quad + 2 \sum_{k_2 \leq k-10} \phi_{\leq k_2 -10} \bigl( P_k (\partial_\alpha \phi_{k_2}^\dagger \partial^\alpha \phi_{k-5 < \cdot < k+5} ) - \partial_\alpha \phi_{k_2}^\dagger \partial^\alpha \phi_k \bigr).
 \end{aligned}
\end{equation}
Then the first term on the right-hand side of~\eqref{equ:WM_nonlinearity_precise_decomp_term3_expanded} coincides exactly with the term~\eqref{equ:WM_nonlinearity_bad0}, the second term on the right-hand side is of type~\eqref{equ:WM_nonlinearity_derivative_transferred_to_lowest_freq} and the third term on the right-hand side coincides with the term~\eqref{equ:WM_nonlinearity_bad2}. Finally, the terms \eqref{equ:WM_nonlinearity_precise_decomp_term5} and \eqref{equ:WM_nonlinearity_precise_decomp_term7} are both of the type~\eqref{equ:WM_nonlinearity_derivative_on_lowest_freq2}, while the high-high interactions term~\eqref{equ:WM_nonlinearity_precise_decomp_term6} coincides with the term~\eqref{equ:WM_nonlinearity_bad1}.

\medskip

Let us now return to the decomposition \eqref{equ:WM_nonlinearity_bad0}--\eqref{equ:WM_nonlinearity_derivative_transferred_to_lowest_freq} of the wave maps nonlinearity at fixed frequency. In order to estimate the term~\eqref{equ:WM_nonlinearity_derivative_on_lowest_freq1} we will use the trilinear estimate~\eqref{equ:trilinear_estimate}, while we will bound the terms \eqref{equ:WM_nonlinearity_derivative_on_lowest_freq2} and \eqref{equ:WM_nonlinearity_derivative_transferred_to_lowest_freq} using a combination of the null form estimate~\eqref{equ:null_form_estimate} and Strichartz estimates. In contrast, the first three terms \eqref{equ:WM_nonlinearity_bad0} -- \eqref{equ:WM_nonlinearity_bad2} are not amenable to good bounds using the null form estimate and the trilinear estimate. Correspondingly, we have to introduce a renormalization procedure to deal with these non-perturbative terms. For the term~\eqref{equ:WM_nonlinearity_bad0} we will follow Tao's idea \cite{Tao1, Tao2} to apply a suitable gauge transformation that casts this part of the nonlinearity into a better form, while for the other two terms \eqref{equ:WM_nonlinearity_bad1}--\eqref{equ:WM_nonlinearity_bad2} we will pass to a ``nonlinearly modified'' version of $\phi_k$ that satisfies a better equation. 

We begin with the latter part of our renormalization procedure. Using the identity
\begin{equation} \label{equ:reinsert_wm_equation_identity}
 \begin{aligned}
  \Box \bigl( \phi^{(1)} \phi^{(2) \dagger} \phi^{(3)} \bigr) &= 2 \phi^{(1)} \partial_\alpha \phi^{(2) \dagger} \partial^\alpha \phi^{(3)} + 2 \partial_\alpha \phi^{(1)} \partial^\alpha \phi^{(2) \dagger} \phi^{(3)} + 2 \partial_\alpha \phi^{(1)} \phi^{(2) \dagger} \partial^\alpha \phi^{(3)} \\
  &\quad \quad + \Box \phi^{(1)} \phi^{(2) \dagger} \phi^{(3)} + \phi^{(1)} \Box \phi^{(2) \dagger} \phi^{(3)} + \phi^{(1)} \phi^{(2) \dagger} \Box \phi^{(3)},
 \end{aligned}
\end{equation}
we may rewrite the high-high interaction term~\eqref{equ:WM_nonlinearity_bad1} into the schematic form 
\begin{align}
 \sum_{k_2 > k-10} \sum_{k_3 = k_2+O(1)} P_k \bigl( \phi_{\leq k_2 + 20} \partial_\alpha \phi_{k_2}^\dagger \partial^\alpha \phi_{k_3} \bigr) &= \Box \biggl( \frac{1}{2} \sum_{k_2 > k-10} \sum_{k_3 = k_2+O(1)} P_k \bigl( \phi_{\leq k_2 + 20} \phi_{k_2}^\dagger \phi_{k_3} \bigr) \biggr) \nonumber \\
 &\quad + \sum_{k_2 > k-10} \sum_{k_3 = k_2+O(1)} P_k L \bigl( \partial_\alpha \phi_{\leq k_2+20}, \partial^\alpha \phi_{k_2}, \phi_{k_3} \bigr) \label{equ:high_high_reinserted1} \\
 &\quad + \sum_{k_2 > k-10} \sum_{k_3 = k_2+O(1)} P_k L \bigl( \Box \phi_{\leq k_2+20}, \phi_{k_2}, \phi_{k_3} \bigr) \label{equ:high_high_reinserted_quintilinear1} \\
 &\quad + \sum_{k_2 > k-10} \sum_{k_3 = k_2+O(1)} P_k L \bigl( \phi_{\leq k_2+20}, \Box \phi_{k_2}, \phi_{k_3} \bigr). \label{equ:high_high_reinserted_quintilinear2}
\end{align}
Here we note that all terms apart from the first one on the right-hand side of the previous equation now have a better structure in that a derivative falls on a lowest frequency or upon reinserting the wave maps equation these terms become quintilinear and thus easier to estimate. In particular, we note that the term~\eqref{equ:high_high_reinserted1} is of the same form as the term~\eqref{equ:WM_nonlinearity_derivative_on_lowest_freq2}.

Similarly, using the Leibniz rule~\eqref{equ:leibniz_rule} for the projection $P_k$ and the identity~\eqref{equ:reinsert_wm_equation_identity}, the other difficult term~\eqref{equ:WM_nonlinearity_bad2} can be rewritten into the schematic form
\begin{align}
 2 \sum_{k_2 \leq k-10} \phi_{\leq k_2-10} \bigl( P_k \bigl( \partial_\alpha \phi_{k_2}^\dagger \partial^\alpha \phi_{k-5 < \cdot < k+5} \bigr) - \partial_\alpha \phi_{k_2}^\dagger \partial^\alpha \phi_k \bigr) &= \Box \biggl( \sum_{k_2 \leq k-10} \phi_{\leq k_2 - 10} \bigl( P_k (\phi_{k_2}^\dagger \phi_{k-5 < \cdot < k+5} ) - \phi_{k_2}^\dagger \phi_k \bigr) \biggr) \nonumber \\
 &\quad + \sum_{k_2 \leq k-10} 2^{-k} L \bigl( \partial_\alpha \phi_{\leq k_2-10}, \nabla_x \partial^\alpha \phi_{k_2}, \phi_{k-5 < \cdot < k+5} \bigr) \label{equ:derivative_transferred_type2_reinserted1} \\
 &\quad + \sum_{k_2 \leq k-10} 2^{-k} L \bigl( \partial_\alpha \phi_{\leq k_2-10}, \nabla_x \phi_{k_2}, \partial^\alpha \phi_{k-5 < \cdot < k+5} \bigr) \label{equ:derivative_transferred_type2_reinserted2} \\
 &\quad + \sum_{k_2 \leq k-10} 2^{-k} L \bigl( \Box \phi_{\leq k_2-10}, \nabla_x \phi_{k_2}, \phi_{k-5 < \cdot < k+5} \bigr) \label{equ:derivative_transferred_type2_quintilinear1} \\
 &\quad + \sum_{k_2 \leq k-10} 2^{-k} L \bigl( \phi_{\leq k_2-10}, \nabla_x \Box \phi_{k_2}, \phi_{k-5 < \cdot < k+5} \bigr) \label{equ:derivative_transferred_type2_quintilinear2} \\
 &\quad + \sum_{k_2 \leq k-10} 2^{-k} L \bigl( \phi_{\leq k_2-10}, \nabla_x \phi_{k_2}, \Box \phi_{k-5 < \cdot < k+5} \bigr), \label{equ:derivative_transferred_type2_quintilinear3}
\end{align}
where again all terms apart from the first one on the right-hand side are now amenable to better estimates either because a derivative now falls on the lowest frequency or because upon reinserting the wave maps equation the term is now quintilinear. We observe that the terms \eqref{equ:derivative_transferred_type2_reinserted1} and \eqref{equ:derivative_transferred_type2_reinserted2} are essentially of the same form as the term~\eqref{equ:WM_nonlinearity_derivative_transferred_to_lowest_freq} and can be dealt with in the same manner. Hence, introducing the following ``nonlinearly modified'' version of~$\phi_k$,
\begin{equation} \label{equ:definition_phi_bar}
 \bar{\phi}_k := \phi_k + \frac{1}{2} \sum_{k_2 > k-10} \sum_{k_3 = k_2 + O(1)} P_k \bigl( \phi_{\leq k_2+20} \phi_{k_2}^\dagger \phi_{k_3} \bigr) + \sum_{k_2 \leq k-10} \phi_{\leq k_2 - 10} \bigl( P_k (\phi_{k_2}^\dagger \phi_{k-5 < \cdot < k+5} ) - \phi_{k_2}^\dagger \phi_k \bigr),
\end{equation}
our previous considerations imply that $\bar{\phi}_k$ satisfies a wave equation of the following schematic form
\begin{equation} \label{equ:phi_bar_nlw}
 \begin{aligned}
  \Box \bar{\phi}_k &= - 2 \sum_{k_2 \leq k-10} \phi_{\leq k_2-10} \partial_\alpha \phi_{k_2}^\dagger \partial^\alpha \phi_k \\
  &\quad + \sum_{k_2 \leq k-10} \sum_{k_3 = k + O(1)} P_k L \bigl( \phi_{k_2 - 10 < \cdot \leq k-10}, \partial_\alpha \phi_{k_2}, \partial^\alpha \phi_{k_3} \bigr) \\
  &\quad + \sum_{k_2 > k-10} \sum_{k_1 = k_2 + O(1)} \sum_{k_3 \leq k_2 + O(1)} P_k L \bigl( \phi_{k_1}, \partial_\alpha \phi_{k_2}, \partial^\alpha \phi_{k_3} \bigr) \\
  &\quad + \sum_{k_2 \leq k-10} \sum_{k_3 = k + O(1)} 2^{-k} P_k L \bigl( \nabla_x \phi_{\leq k_2 - 10}, \partial_\alpha \phi_{k_2}, \partial^\alpha \phi_{k_3} \bigr) \\
  &\quad + Q_k(\phi),
 \end{aligned}
\end{equation}
where $Q_k(\phi)$ comprises all the quintilinear terms \eqref{equ:high_high_reinserted_quintilinear1}--\eqref{equ:high_high_reinserted_quintilinear2} and \eqref{equ:derivative_transferred_type2_quintilinear1}--\eqref{equ:derivative_transferred_type2_quintilinear3}
\begin{align*}
 Q_k(\phi) &= \sum_{k_2 > k-10} \sum_{k_3 = k_2+O(1)} P_k L \bigl( P_{\leq k_2+20} ( \phi \partial_\alpha \phi^\dagger \partial^\alpha \phi ), \phi_{k_2}, \phi_{k_3} \bigr) \\
 &\quad + \sum_{k_2 > k-10} \sum_{k_3 = k_2+O(1)} P_k L \bigl( \phi_{\leq k_2+20}, P_{k_2} ( \phi \partial_\alpha \phi^\dagger \partial^\alpha \phi ), \phi_{k_3} \bigr) \\
 &\quad + \sum_{k_2 \leq k-10} 2^{-k} L \bigl( P_{\leq k_2-10}(\phi \partial_\alpha \phi^\dagger \partial^\alpha \phi), \nabla_x \phi_{k_2}, \phi_{k-5 < \cdot < k+5} \bigr) \\
 &\quad + \sum_{k_2 \leq k-10} 2^{-k} L \bigl( \phi_{\leq k_2-10}, \nabla_x P_{k_2} ( \phi \partial_\alpha \phi^\dagger \partial^\alpha \phi ), \phi_{k-5 < \cdot < k+5} \bigr) \\
 &\quad + \sum_{k_2 \leq k-10} 2^{-k} L \bigl( \phi_{\leq k_2-10}, \nabla_x \phi_{k_2}, P_{k-5 < \cdot < k+5} ( \phi \partial_\alpha \phi^\dagger \partial^\alpha \phi ) \bigr).
\end{align*}
All the quintilinear terms in $Q_k(\phi)$ can be easily estimated using just combinations of the null form estimate~\eqref{equ:null_form_estimate} and Strichartz estimates as well as the fact that $\|\phi\|_{L^\infty_t L^\infty_x} \lesssim 1$. Thus, we can overcome the difficulties with the terms \eqref{equ:WM_nonlinearity_bad1}--\eqref{equ:WM_nonlinearity_bad2} at the expense of passing from $\phi_k$ to the new variable $\bar{\phi}_k$. Note that $\bar{\phi}_k$ is also localized to frequency $2^k$. In the following we will sometimes use the notation $\bar{\phi} = \sum_{k\in\bZ} \bar{\phi}_k$. 

\medskip

Now we still have to deal with the more severe term~\eqref{equ:WM_nonlinearity_bad0} which comes up as the first term on the right-hand side of the equation~\eqref{equ:phi_bar_nlw} for $\bar{\phi}_k$. We follow Tao's idea~\cite{Tao1, Tao2} and apply a renormalization $U_{\leq k-10} \bar{\phi}_k$ so that $\Box \bigl( U_{\leq k-10} \bar{\phi}_k \bigr)$ takes on a much better form. More specifically, we implement a version of the gauge construction introduced by Sterbenz-Tataru~\cite{Sterbenz_Tataru1, Sterbenz_Tataru2}. To this end we pick a smooth cut-off function $\eta \in C^\infty(\bR)$ such that $\eta(y) = 0$ for $y \leq -\frac{2}{3}$ and $\eta(y) = 1$ for $y \geq -\frac{1}{3}$. Then we define for $h \in \bR$ the matrix
\begin{equation} \label{equ:matrix_gauge_transform}
 B_h := \frac{d}{dh} \sum_{\ell \in \bZ} \eta(h - \ell) (\phi_{\leq \ell - 10} \phi_{\ell}^\dagger - \phi_{\ell} \phi_{\leq \ell -10}^\dagger) = \sum_{\ell \in \bZ} \eta'(h - \ell) (\phi_{\leq \ell - 10} \phi_{\ell}^\dagger - \phi_{\ell} \phi_{\leq \ell -10}^\dagger).
\end{equation}
Now we define the gauge transformations $U_{\leq h}$ for $h \in \bR$ via the ODE
\begin{equation} \label{equ:ode_gauge_transform}
 \frac{d}{dh} U_{\leq h} = U_{\leq h} B_h, \quad \lim_{h \to -\infty} U_{\leq h} = Id.
\end{equation}
By the anti-symmetry of $B_h$, it follows that the gauge transformations $U_{\leq h}$ as solutions to the ODE~\eqref{equ:ode_gauge_transform} satisfy $U_{\leq h}^\dagger U_{\leq h} = Id$ and are thus exactly orthogonal. However, they do not have a sharp frequency localization. Fortunately, we will see that the renormalized quantity $U_{\leq k-10} \bar{\phi}_k$ is still approximately localized to frequency $2^k$ up to exponentially decaying tails.

\medskip

Next, we compute the equation that the renormalization $U_{\leq k-10} \bar{\phi}_k$ satisfies. Upon defining for $k \in \bZ$
\begin{equation}
 A_{\alpha; k} = \partial_\alpha \phi_k \phi_{\leq k - 10}^\dagger - \phi_{\leq k - 10} \partial_\alpha \phi_k^\dagger
\end{equation}
and
\begin{equation} \label{equ:definition_connection_form_A}
 A_{\alpha; \leq k} = \sum_{\ell \leq k} A_{\alpha; \ell},
\end{equation}
we find that 
\begin{equation} \label{equ:nlw_for_renormalized_phi_bar} 
 \begin{aligned}
  \Box \bigl( U_{\leq k-10} \bar{\phi}_k \bigr) &= U_{\leq k-10} \bigl( \Box \bar{\phi}_k - 2 A_{\alpha; \leq k-10} \partial^\alpha \phi_k \bigr) \\
  &\quad + 2 \bigl( \partial_\alpha U_{\leq k-10} + U_{\leq k-10} A_{\alpha; \leq k-10} \bigr) \partial^\alpha \phi_k \\
  &\quad + 2 \partial_\alpha U_{\leq k-10} \partial^\alpha (\bar{\phi}_k - \phi_k) \\
  &\quad + \bigl( \Box U_{\leq k-10} \bigr) \bar{\phi}_k.
 \end{aligned}
\end{equation}
In the following proposition we carefully analyze the structure of each term in the nonlinearity of the wave equation~\eqref{equ:nlw_for_renormalized_phi_bar} for $U_{\leq k-10} \bar{\phi}_k$.
\begin{prop} \label{prop:nlw_for_renormalized_phi_bar_schematic_identities}
 It holds that
 \begin{equation} \label{equ:nlw_for_renormalized_phi_bar_term1}
  \begin{aligned}
   \Box \bar{\phi}_k - 2 A_{\alpha; \leq k-10} \partial^\alpha \phi_k &= \sum_{k_2 \leq k-10} \sum_{k_3 = k + O(1)} P_k L \bigl( \phi_{k_2 - 10 < \cdot \leq k-10}, \partial_\alpha \phi_{k_2}, \partial^\alpha \phi_{k_3} \bigr) \\
   &\quad + \sum_{k_2 > k-10} \sum_{k_1 = k_2 + O(1)} \sum_{k_3 \leq k_2 + O(1)} P_k L \bigl( \phi_{k_1}, \partial_\alpha \phi_{k_2}, \partial^\alpha \phi_{k_3} \bigr) \\
   &\quad + \sum_{k_2 \leq k-10} \sum_{k_3 = k + O(1)} 2^{-k} P_k L \bigl( \nabla_x \phi_{\leq k_2 - 10}, \partial_\alpha \phi_{k_2}, \partial^\alpha \phi_{k_3} \bigr) \\
   &\quad + Q_k(\phi).
  \end{aligned}
 \end{equation}
 Moreover, we have the following schematic identities
 \begin{align}
  \bigl( \partial_\alpha U_{\leq k-10} + U_{\leq k-10} A_{\alpha; \leq k-10} \bigr) \partial^\alpha \phi_k &= \sum_{k_2 \leq k-10} U_{\leq k_2-1} L \bigl( \phi_{k_2-10 \leq \cdot \leq k-10}, \partial_\alpha \phi_{k_2}, \partial^\alpha \phi_k \bigr), \label{equ:schematic_form_nlw_for_renormalized_phi_bar_term2} \\
  \partial_\alpha U_{\leq k-10} \partial^\alpha (\bar{\phi}_k - \phi_k) &= \sum_{k_1 \leq k-10} U_{\leq k_1 - 1} L \bigl( \partial_\alpha L ( \phi_{\leq k_1-10}, \phi_{k_1} ), \partial^\alpha (\bar{\phi}_k - \phi_k) \bigr), \label{equ:schematic_form_nlw_for_renormalized_phi_bar_term3} \\
  \bigl( \Box U_{\leq k-10} \bigr) \bar{\phi}_k &= \sum_{k_1 \leq k-10} U_{\leq k_1 - 1} L \bigl( \Box L ( \phi_{\leq k_1-10}, \phi_{k_1} ), \bar{\phi}_k \bigr) \label{equ:schematic_form_nlw_for_renormalized_phi_bar_term4} \\
   &\quad + \sum_{k_1 \leq k_2 \leq k-10} U_{\leq k_1-1} L \bigl( \partial_\alpha L ( \phi_{\leq k_1-10}, \phi_{k_1} ), \partial^\alpha L ( \phi_{\leq k_2-10}, \phi_{k_2} ), \bar{\phi}_k \bigr). \nonumber
 \end{align}
\end{prop}
\begin{proof}
 We begin with the proof of the first identity~\eqref{equ:nlw_for_renormalized_phi_bar_term1}. In view of \eqref{equ:phi_bar_nlw} we only have to understand the structure of the term 
 \[
  \sum_{k_2 \leq k-10} \partial_\alpha \phi_{k_2} \phi_{\leq k_2-10}^\dagger \partial^\alpha \phi_k.
 \]
 But here we can argue as in Tao~\cite[p. 461]{Tao2} and use the geometric identity $\phi^\dagger \partial^\alpha \phi = 0$ to show that this term is in fact a sum of terms of the schematic forms \eqref{equ:WM_nonlinearity_derivative_on_lowest_freq1}--\eqref{equ:WM_nonlinearity_derivative_transferred_to_lowest_freq}. This proves~\eqref{equ:nlw_for_renormalized_phi_bar_term1}.
 
 \medskip
 
 Next we turn to the proof of the schematic identity~\eqref{equ:schematic_form_nlw_for_renormalized_phi_bar_term2}. Using the exact orthogonality of the gauge transformations $U_{\leq h}$, by telescoping we may write $U_{\leq k-10}$ as
 \[
  U_{\leq k-10} = U_{\leq k - M} \prod_{j=k-M+1}^{k-10} U_{\leq j-1}^\dagger U_{\leq j} 
 \]
 for some $M \gg 1$. Applying a derivative $\partial_\alpha$ and letting $M \to \infty$, we find that
 \[
  \partial_\alpha U_{\leq k-10} = \sum_{j \leq k-10} U_{\leq j-1} \partial_\alpha \bigl( U_{\leq j-1}^\dagger U_{\leq j} \bigr) U_{\leq j}^\dagger U_{\leq k-10}.
 \]
 From the defining ODE \eqref{equ:ode_gauge_transform} for the gauge transformations we have that 
 \begin{equation} \label{equ:integral_expansion1}
  U_{\leq j-1}^\dagger U_{\leq j} = \int_{j-1}^j U_{\leq j-1}^\dagger U_{\leq h} B_h \, dh + Id
 \end{equation}
 and hence 
 \[
  \partial_\alpha \bigl( U_{\leq j-1}^\dagger U_{\leq j} \bigr) = \partial_\alpha \int_{j-1}^j B_h \, dh + \partial_\alpha \int_{j-1}^j \bigl( U_{\leq j-1}^\dagger U_{\leq j} - Id \bigr) B_h \, dh.
 \]
 Thus, we may write 
 \begin{align*}
  \partial_\alpha U_{\leq k-10} + U_{\leq k-10} A_{\alpha; \leq k-10} &= \sum_{j \leq k-10} U_{\leq j-1} \biggl( \partial_\alpha \int_{j-1}^j B_h \, dh \biggr) U_{\leq j}^\dagger U_{\leq k-10} \\
  &\quad + \sum_{j \leq k-10} U_{\leq j-1} A_{\alpha; j} \\
  &\quad + \sum_{j \leq k-10} U_{\leq j-1} \bigl( U_{\leq j-1}^\dagger U_{\leq k-10} - Id \bigr) A_{\alpha; j} \\
  &\quad + \sum_{j \leq k-10} U_{\leq j-1} \biggl( \partial_\alpha \int_{j-1}^j \bigl( U_{\leq j-1}^\dagger U_{\leq h} - Id \bigr) B_h \, dh \biggr) U_{\leq j}^\dagger U_{\leq k-10}.  
 \end{align*}
 Now by the definition of $B_h$~\eqref{equ:matrix_gauge_transform} we have that 
 \[
  \partial_\alpha \int_{j-1}^j B_h \, dh = \partial_\alpha (\phi_{\leq j-10} \phi_j^\dagger - \phi_j \phi_{\leq j-10}^\dagger ) = - A_{\alpha; j} + (\partial_\alpha \phi_{\leq j-10} \phi_j^\dagger - \phi_j \partial_\alpha \phi_{\leq j-10}^\dagger).
 \]
 Combining the two previous identities we find that
 \begin{equation} \label{equ:nlw_for_renormalized_phi_bar_term2_expanded}
  \begin{aligned}
   \partial_\alpha U_{\leq k-10} + U_{\leq k-10} A_{\alpha; \leq k-10} &= \sum_{j \leq k-10} U_{\leq j-1} A_{\alpha; j} \bigl( - U_{\leq j}^\dagger U_{\leq k-10} + Id \bigr) \\
   &\quad + \sum_{j \leq k-10} U_{\leq j-1} \bigl( \partial_\alpha \phi_{\leq j-10} \phi_j^\dagger - \phi_j \partial_\alpha \phi_{\leq j-10}^\dagger \bigr) U_{\leq j}^\dagger U_{\leq k-10} \\
   &\quad + \sum_{j \leq k-10} U_{\leq j-1} \bigl( U_{\leq j-1}^\dagger U_{\leq k-10} - Id \bigr) A_{\alpha; j} \\
   &\quad + \sum_{j \leq k-10} U_{\leq j-1} \biggl( \partial_\alpha \int_{j-1}^j \bigl( U_{\leq j-1}^\dagger U_{\leq h} - Id \bigr) B_h \, dh \biggr) U_{\leq j}^\dagger U_{\leq k-10}.  
  \end{aligned}
 \end{equation}
 In view of the identity
 \[
  U_{\leq k_1}^\dagger U_{\leq k_2} - Id = \int_{k_1}^{k_2} U_{\leq k_1}^\dagger U_{\leq h} B_h \, dh \quad \text{for } k_1 < k_2,
 \]
 the exact orthogonality of the gauge transformations $U_{\leq h}$, the definition of $B_h$~\eqref{equ:matrix_gauge_transform} and the fact that $\|\phi\|_{L^\infty_t L^\infty_x} \lesssim 1$, it is now apparent that when it comes to estimates, the first three terms on the right-hand side of~\eqref{equ:nlw_for_renormalized_phi_bar_term2_expanded} lead to terms of the schematic form~\eqref{equ:schematic_form_nlw_for_renormalized_phi_bar_term2}. To see this also for the last term on the right-hand side of~\eqref{equ:nlw_for_renormalized_phi_bar_term2_expanded}, we note that by iteratively inserting the identity~\eqref{equ:integral_expansion1}, we obtain that
 \begin{align*}
  \partial_\alpha \int_{j-1}^j \bigl( U_{\leq j-1}^\dagger U_{\leq h} - Id \bigr) B_h \, dh &= \int_{j-1}^j \int_{j-1}^{h_1} U_{\leq j-1}^\dagger U_{\leq h_2} B_{h_2} \partial_\alpha B_{h_1} \, dh_2 \, dh_1 \\
  &\quad + \sum_{n=1}^\infty \int_{j-1}^j \int_{j-1}^{h_1} \cdots \int_{j-1}^{h_n} U_{\leq j-1}^\dagger U_{\leq h_{n+1}} \partial_\alpha B_{h_{n+1}} B_{h_n} \cdots B_{h_1} \, dh_{n+1} \cdots dh_1.
 \end{align*}
 Here the series may seem to diverge at first sight, but the integration over simplices yields the necessary compensating $\frac{1}{n!}$ decay. Thus, also the last term on the right-hand side of~\eqref{equ:nlw_for_renormalized_phi_bar_term2_expanded} is of the schematic form~\eqref{equ:schematic_form_nlw_for_renormalized_phi_bar_term2}, which finishes the proof of~\eqref{equ:schematic_form_nlw_for_renormalized_phi_bar_term2}.
 
 \medskip
 
 Similarly, we can prove~\eqref{equ:schematic_form_nlw_for_renormalized_phi_bar_term3} using the identities
 \[
  \partial_\alpha U_{\leq k-10} = \sum_{j \leq k-10} U_{\leq j-1} \partial_\alpha \bigl( U_{\leq j-1}^\dagger U_{\leq j} \bigr) U_{\leq j}^\dagger U_{\leq k-10}
 \]
 and 
 \begin{align*}
  \partial_\alpha \bigl( U_{\leq j-1}^\dagger U_{\leq j} \bigr) &= \int_{j-1}^j U_{\leq j-1}^\dagger U_{\leq h_1} \partial_\alpha B_{h_1} \, dh_1 \\
  &\quad + \sum_{n=1}^\infty \int_{j-1}^j \int_{j-1}^{h_1} \cdots \int_{j-1}^{h_n} U_{\leq j-1}^\dagger U_{\leq h_{n+1}} \partial_\alpha B_{h_{n+1}} B_{h_n} \cdots B_{h_1} \, dh_{n+1} \cdots dh_1.
 \end{align*}
 
 \medskip
 
 Finally, we turn to the proof of~\eqref{equ:schematic_form_nlw_for_renormalized_phi_bar_term4}. From the telescoping identity 
 \[
  U_{\leq k-10} = U_{\leq k-M} \prod_{j = k-M+1}^{k-10} U_{\leq j-1}^\dagger U_{\leq j}
 \]
 for $M \gg 1$, we obtain that
 \begin{equation} \label{equ:box_U_expansion}
  \begin{aligned}
   \Box U_{\leq k-10} &= \sum_{j \leq k-10} U_{\leq j-1} \Box \bigl( U_{\leq j-1}^\dagger U_{\leq j} \bigr) U_{\leq j}^\dagger U_{\leq k-10} \\
   &\quad + 2 \sum_{j_1 < j_2 \leq k-10} U_{\leq j_1 -1} \partial_\alpha \bigl( U_{\leq j_1-1}^\dagger U_{\leq j_1} \bigr) \cdots \partial^\alpha \bigl( U_{\leq j_2-1}^\dagger U_{\leq j_2} \bigr) U_{\leq j_2}^\dagger U_{\leq k-10}.
  \end{aligned}
 \end{equation}
 Then by iteratively inserting the identities
 \begin{equation*}
  \Box \bigl( U_{\leq j-1}^\dagger U_{\leq j} \bigr) = \int_{j-1}^j \Box \bigl( U_{\leq j-1}^\dagger U_{\leq h} \bigr) B_h \, dh + 2 \int_{j-1}^j \partial_\alpha \bigl( U_{\leq j-1}^\dagger U_{\leq h} \bigr) \partial^\alpha B_h \, dh + \int_{j-1}^j \bigl( U_{\leq j-1}^\dagger U_{\leq h} \bigr) \Box B_h \, dh
 \end{equation*}
 and
 \begin{equation*}
  \partial_\alpha \bigl( U_{\leq j-1}^\dagger U_{\leq j} \bigr) = \int_{j-1}^j \partial_\alpha \bigl( U_{\leq j-1}^\dagger U_{\leq h} \bigr) B_h \, dh + \int_{j-1}^j \bigl( U_{\leq j-1}^\dagger U_{\leq h} \bigr) \partial_\alpha B_h \, dh
 \end{equation*}
 into \eqref{equ:box_U_expansion}, we infer by similar arguments as before that $\bigl( \Box U_{\leq k-10} \bigr) \bar{\phi}_k$ is indeed of the schematic form~\eqref{equ:schematic_form_nlw_for_renormalized_phi_bar_term4}.
\end{proof}

Next, we prove several useful bounds on $\bar{\phi}_k$ and on the gauge transformations that will be needed in the sequel.
\begin{lem} \label{lem:bounds_on_phi_bar}
 Let $\phi \colon I \times \bR^2 \to \bS^{m}$ be a radial wave map of energy $E$. Then we have for all $k \in \bZ$ that
 \begin{equation} \label{equ:L_infty_bounds_phi_bar}
  \| \bar{\phi}_k \|_{L^\infty_t L^\infty_x} \lesssim 1 + E 
 \end{equation}
 and it holds that
 \begin{equation} \label{equ:kinetic_energy_bound_phi_bar}
  \sum_{k\in\bZ} \| \nabla_{t,x} \bar{\phi}_k \|_{L^2_x}^2 \lesssim E + E^3.
 \end{equation}
\end{lem}
\begin{proof}
 We start with the proof of~\eqref{equ:L_infty_bounds_phi_bar}. Since $\phi$ maps into the sphere and the Littlewood-Paley projections $P_k$ are disposable, we clearly have $\|\phi_k\|_{L^\infty_t L^\infty_x} \lesssim 1$. Using Bernstein's inequality, we easily bound the second component in the definition~\eqref{equ:definition_phi_bar} of $\bar{\phi}_k$ by
 \begin{align*}
  \biggl\| \sum_{k_2 > k-10} \sum_{k_3 = k_2 + O(1)} P_k \bigl( \phi_{\leq k_2+20} \phi_{k_2}^\dagger \phi_{k_3} \bigr) \biggr\|_{L^\infty_x} &\lesssim \sum_{k_2 > k-10} \sum_{k_3=k_2 + O(1)} 2^{2k} \|\phi_{k_2}\|_{L^2_x} \|\phi_{k_3}\|_{L^2_x} \\
  &\lesssim \sum_{k_2 > k-10} \|\nabla_x \phi_{k_2}\|_{L^2_x}^2 \\
  &\lesssim E,
 \end{align*}
 while for the third component in the definition~\eqref{equ:definition_phi_bar} we additionally use the Leibniz rule~\eqref{equ:leibniz_rule} for $P_k$ to find that
 \begin{align*}
  \biggl\| \sum_{k_2 \leq k-10} \phi_{\leq k_2-10} \bigl( P_k ( \phi_{k_2}^\dagger \phi_{k-5 < \cdot < k+5} ) - \phi_{k_2}^\dagger \phi_k \bigr) \biggr\|_{L^\infty_x} &= \biggl\| \sum_{k_2 \leq k-10} 2^{-k} L\bigl( \phi_{\leq k_2-10}, \nabla_x \phi_{k_2}, \phi_{k+O(1)} \bigr) \biggr\|_{L^\infty_x} \\
  &\lesssim \sum_{k_2 \leq k-10} 2^{-k + k_2} \|\nabla_x \phi_{k_2}\|_{L^2_x} \| \nabla_x \phi_k \|_{L^2_x} \\
  &\lesssim E.
 \end{align*}
 The proof of~\eqref{equ:kinetic_energy_bound_phi_bar} follows by analogous estimates.
\end{proof}

\begin{lem} \label{lem:gauge_transform_bounds}
 Let $\phi \colon I \times \bR^2 \to \bS^{m}$ be a radial wave map of energy $E$ and let $U_{\leq h}$ for $h \in \bR$ be the associated gauge transformations as defined in~\eqref{equ:ode_gauge_transform}. For $2 < p \leq \infty$ it holds that 
 \begin{equation} \label{equ:L_p_bound_nabla_U}
  \bigl\| \nabla_{t,x} U_{\leq k} \bigr\|_{L^\infty_t L^p_x} \lesssim 2^{(1-\frac{2}{p})k} (E^{\frac{1}{2}} + E).
 \end{equation}
 Moreover, we have for $2 \leq p \leq \infty$ that
 \begin{equation} \label{equ:L_p_bound_nabla2_U}
  \bigl\| \nabla_x \nabla_{t,x} U_{\leq k} \bigr\|_{L^\infty_t L^p_x} \lesssim 2^{(2-\frac{2}{p})k} (E^{\frac{1}{2}} + E).
 \end{equation}
\end{lem}
\begin{proof}
 By telescoping as in the proof of Proposition~\ref{prop:nlw_for_renormalized_phi_bar_schematic_identities}, we obtain the schematic identity 
 \[
  \nabla_{t,x} U_{\leq k} = \sum_{k_2 \leq k} U_{\leq k_2-1} \nabla_{t,x} L ( \phi_{\leq k_2-10}, \phi_{k_2} ) U_{\leq k_2}^\dagger U_{\leq k}.
 \]
 Hence we find for $2 < p \leq \infty$ that
 \begin{align*}
  \bigl\| \nabla_{t,x} U_{\leq k} \bigr\|_{L^p_x} &\lesssim \sum_{k_2 \leq k} \bigl\| L ( \nabla_{t,x} \phi_{\leq k_2-10}, \phi_{k_2} ) \bigr\|_{L^p_x} + \sum_{k_2 \leq k} \bigl\| L ( \phi_{\leq k_2-10}, \nabla_{t,x} \phi_{k_2} ) \bigr\|_{L^p_x} \\
  &\lesssim \sum_{k_2 \leq k} \sum_{k_1 \leq k_2 - 10} \| \nabla_{t,x} \phi_{k_1} \|_{L^\infty_x} \|\phi_{k_2}\|_{L^p_x} + \sum_{k_2 \leq k} \| \nabla_{t,x} \phi_{k_2} \|_{L^p_x} \\
  &\lesssim \sum_{k_2 \leq k} \sum_{k_1 \leq k_2 - 10} 2^{k_1} 2^{-\frac{2}{p} k_2} \|\nabla_{t,x} \phi_{k_1}\|_{L^2_x} \|\nabla_x \phi_{k_2}\|_{L^2_x}  + \sum_{k_2 \leq k} 2^{(1 - \frac{2}{p}) k_2} \|\nabla_{t,x} \phi_{k_2}\|_{L^2_x} \\
  &\lesssim 2^{(1-\frac{2}{p})k} (E + E^{\frac{1}{2}}),
 \end{align*}
 which yields \eqref{equ:L_p_bound_nabla_U}. The proof of the bound~\eqref{equ:L_p_bound_nabla2_U} proceeds similarly.
\end{proof}

We conclude this section by establishing a small energy global regularity result for the wave maps equation (WM) for radially symmetric, classical initial data. The proof is a fairly immediate consequence of the multilinear estimates from the previous section and of the careful decomposition of the wave maps nonlinearity as well as of the introduction of the renormalization procedure in this section.
\begin{thm} \label{thm:small_energy_global}
 There exists an absolute constant $\varepsilon > 0$ such that for any radially symmetric, classical initial data $(\phi_0, \phi_1) \colon \bR^2 \to T \bS^m$ with energy $E[\phi] < \varepsilon$, there exists a unique, radially symmetric, classical global solution $\phi \colon \bR^{1+2} \to \bS^m$ to (WM) with initial data $\phi[0] = (\phi_0, \phi_1)$ satisfying 
 \begin{equation} \label{equ:small_energy_S_norm_bound}
  \| \phi \|_{S[\bR]} \lesssim E[\phi]^{\frac{1}{2}}.
 \end{equation}
\end{thm}
\begin{proof}
 It is a standard argument to show that in order to conclude global regularity and the bound~\eqref{equ:small_energy_S_norm_bound}, it suffices to prove frequency envelope bounds for the local-in-time evolution of a classical radial wave map with initial data $(\phi_0, \phi_1)$, see~\cite{Tao1, Tao2}. Thus, let $\{ c_k \}_{k\in\bZ}$ be a frequency envelope covering the initial data $(\phi_0, \phi_1)$ and denote by $\phi$ the classical wave maps evolution with initial data $\phi[0] = (\phi_0, \phi_1)$ on a time interval $I = [-T, T]$ for some $0 < T < \infty$. We shall prove frequency envelope bounds for the evolution by a bootstrap argument. Making the bootstrap assumption 
 \[
  \| P_k \phi \|_{S_k[I]} \leq C c_k, \quad k \in \bZ,
 \]
 for some sufficently large, absolute constant $C > 0$, we shall now show that this implies the improved bound 
 \[
  \| P_k \phi \|_{S_k[I]} \leq \frac{C}{2} c_k, \quad k \in \bZ.
 \]
 By scaling invariance, it suffices to prove this bound for $k = 0$. To this end we first deduce a frequency envelope bound for the renormalized quantity $U_{\leq -10} \overline{\phi}_0$. This step is slightly compounded by the fact that the variable $U_{\leq -10}  \overline{\phi}_0$ is only approximately localized to frequency $\sim 0$ up to exponentially decaying tails. Correspondingly, we show that for any $0 \leq \alpha < 1$, it holds that
 \begin{equation} \label{equ:small_energy_global_est1}
  \sum_{k_0 \in \bZ} 2^{\alpha |k_0|} \bigl\| P_{k_0} \big( U_{\leq -10} \overline{\phi}_0 \big) \bigr\|_{S_{k_0}[I]} \ll C c_0.
 \end{equation}
 From this bound on $U_{\leq -10} \overline{\phi}_0$ we can then pass back to $\phi_0$ and recover the improved bound 
 \[
  \| P_0 \phi \|_{S_0[I]} \leq \frac{C}{2} c_0,
 \]
 as explained in detail in the proof of Proposition~\ref{prop:partial_fungibility} in the next section. In order to prove the bound~\eqref{equ:small_energy_global_est1} we show that for any $0 \leq \alpha < 1$, 
 \begin{equation} \label{equ:small_energy_global_est2}
  \sum_{k_0 \in \bZ} 2^{\alpha |k_0|} \bigl\| P_{k_0} \nabla_{t,x} \bigl( U_{\leq -10} \overline{\phi}_0 \bigr) \bigr\|_{L^2_x} \ll C c_0
 \end{equation}
 and that 
 \begin{equation} \label{equ:small_energy_global_est3}
  \sum_{k_0 \in \bZ} 2^{\alpha |k_0|} \bigl\| P_{k_0} \Box \bigl( U_{\leq -10} \overline{\phi}_0 \bigr) \bigr\|_{L^1_t L^2_x[I]} \lesssim \varepsilon C c_0.
 \end{equation}
 For sufficently small $\varepsilon > 0$ the energy estimate~\eqref{equ:energy_estimate} then yields the desired bound~\eqref{equ:small_energy_global_est1}. We note that the exponential factors $2^{\alpha |k_0|}$ in \eqref{equ:small_energy_global_est2} and \eqref{equ:small_energy_global_est3} can be controlled by playing out Bernstein's inequality and the bounds on the gauge transformations from Lemma~\ref{lem:gauge_transform_bounds}, see the proof of Proposition~\ref{prop:partial_fungibility} in the next section for such an argument. Then the bound~\eqref{equ:small_energy_global_est2} is straightforward to derive from the definition of $\overline{\phi}_0$ using the properties of frequency envelopes. Finally, in order to deduce the bound~\eqref{equ:small_energy_global_est3}, we recall that Proposition~\ref{prop:nlw_for_renormalized_phi_bar_schematic_identities} carefully lists each schematic term arising in the nonlinearity $\Box \bigl( U_{\leq -10} \overline{\phi}_0 \bigr)$. Using the properties of frequency envelopes we then prove~\eqref{equ:small_energy_global_est3} separately for each schematic term. Specifically, we may bound the first term on the right-hand side of~\eqref{equ:nlw_for_renormalized_phi_bar_term1} as well as the term~\eqref{equ:schematic_form_nlw_for_renormalized_phi_bar_term2} using the trilinear estimate~\eqref{equ:trilinear_estimate}, while all other terms on the right-hand side of~\eqref{equ:nlw_for_renormalized_phi_bar_term1} as well as the terms \eqref{equ:schematic_form_nlw_for_renormalized_phi_bar_term3}--\eqref{equ:schematic_form_nlw_for_renormalized_phi_bar_term4} can be estimated by a combination of the null form estimate~\eqref{equ:null_form_estimate} and Strichartz estimates.
\end{proof}

\section{Breakdown criterion} \label{sec:breakdown_criterion}

Here we show that the $S$ norm introduced in Definition~\ref{def:S_norm} provides sufficient control on a radially symmetric, classical solution to the wave maps equation (WM) in order to infer long time existence and scattering. The main result is summarized in the following proposition.

\begin{prop} \label{prop:breakdown_criterion}
 Let $I$ be the maximal time interval of existence of a radially symmetric, classical wave map $\phi \colon I \times \bR^2 \to \bS^m$. If $\| \phi \|_{S[I]} < \infty$, then we must have $I = \bR$. Moreover, $\phi$ then scatters to finite energy free waves as $t \to \pm \infty$ in the sense that there exist $(f_\pm, g_\pm) \in \dot{H}^1_x \times L^2_x$ such that
 \[
  \lim_{t \to \pm \infty} \, \bigl\| \nabla_{t,x} \phi - \nabla_{t,x} S(t)(f_\pm, g_\pm) \bigr\|_{L^2_x} = 0,
 \]
 where $S(t)(f_\pm, g_\pm) = \cos(t |\nabla|) f_\pm + \frac{\sin(t |\nabla|)}{|\nabla|} g_\pm$ denotes the free wave propagator.
\end{prop}

We will give the proof of Proposition~\ref{prop:breakdown_criterion} at the end of this section. The key ingredient for the proof is the next proposition which provides frequency envelope bounds for a radially symmetric, classical wave map $\phi \colon I \times \bR^2 \to \bS^m$ with finite $S$ norm. 

\begin{prop} \label{prop:finite_S_gives_freq_envelope_bounds}
 Let $\phi \colon I \times \bR^2 \to \bS^m$ be a radially symmetric, classical wave map defined on a time interval~$I$ containing time $t=0$. Let $\{ c_k \}_{k \in \bZ}$ be a frequency envelope covering the initial data $\phi[0]$. If $\| \phi \|_{S[I]} < \infty$, then there exists $C \equiv C \bigl( \|\phi\|_{S[I]} \bigr) > 0$ such that for all $k \in \bZ$,
 \begin{equation*}
  \| P_k \phi \|_{S_k[I]} \leq C c_k.
 \end{equation*}
\end{prop}

The proof of Proposition~\ref{prop:finite_S_gives_freq_envelope_bounds} relies on the following weak divisibility property of the $S$ norm.

\begin{prop} \label{prop:partial_fungibility}
 Let $\phi \colon I \times \bR^2 \to \bS^m$ be a radially symmetric, classical wave map of energy $E$ defined on a time interval $I$ and satisfying $\|\phi\|_{S[I]} \leq K$. Then there exists a partition $I = \cup_{j=1}^N I_j$ into $N \equiv N(K, E)$ consecutive time intervals $I_j$ such that
 \begin{equation} \label{equ:partial_fungibility}
  \|\phi\|_{S[I_j]} \lesssim C(E)\quad \text{for } j = 1, \ldots, N,
 \end{equation}
 where $C(E)$ is an absolute constant that just depends on the size of the energy $E$.
\end{prop}
\begin{proof}
 The proof proceeds in two steps. First we show that 
 \begin{equation} \label{equ:partial_fungibility_difference}
  \| \bar{\phi} - \phi \|_{S[I]} \lesssim C(E).
 \end{equation}
 Then we prove that the interval $I$ can be partitioned into $N(K,E)$ consecutive subintervals $I_j$ such that
 \begin{equation} \label{equ:partial_fungibility_phi_bar}
  \| \bar{\phi} \|_{S[I_j]} \lesssim C(E) \quad \text{for } j = 1, \ldots, N.
 \end{equation}
 The assertion \eqref{equ:partial_fungibility} then is an immediate consequence of the two previous bounds. Their proof, however, requires the introduction of certain frequency and spatial cutoffs in the definitions of the quantity $\bar{\phi}$ and of the gauge transformations $U_{\leq h}$ which we will explain as we go. We will be able to prove \eqref{equ:partial_fungibility_difference} directly, gaining smallness by choosing the frequency and spatial cutoff parameter sufficently large, while we will establish \eqref{equ:partial_fungibility_phi_bar} using a suitable divisibility argument.
 
 \medskip
 
 We begin with the proof of the bound \eqref{equ:partial_fungibility_difference}. To this end we recall that our motivation for the introduction of the quantity $\bar{\phi} = \sum_{k\in\bZ} \bar{\phi}_k$ was in the first place that the high-high interaction term~\eqref{equ:WM_nonlinearity_bad1}
 \begin{equation*}
  \sum_{k_2 > k-10} \sum_{k_3 = k_2 + O(1)} P_k \bigl( \phi_{\leq k_2 + 20} \partial_\alpha \phi_{k_2}^\dagger \partial^\alpha \phi_{k_3} \bigr)
 \end{equation*}
 as well as the term~\eqref{equ:WM_nonlinearity_bad2}
 \begin{equation*}
  \sum_{k_2 \leq k-10} \phi_{\leq k_2-10} \bigl( P_k ( \partial_\alpha \phi_{k_2}^\dagger \partial^\alpha \phi_{k-5 < \cdot < k+5} ) - \partial_\alpha \phi_{k_2}^\dagger \partial^\alpha \phi_k \bigr)
 \end{equation*}
 in the wave map nonlinearity localized to frequency $\sim 2^k$ turn out to be non-perturbative due to our choice to estimate the nonlinearity solely in $L^1_t L^2_x$ and to build the $S$ norm out of relatively elementary function spaces. However, we can in fact ``peel off'' some more good parts from these terms. More precisely, the following frequency and spatially truncated versions of \eqref{equ:WM_nonlinearity_bad1} 
 \begin{equation} \label{equ:WM_nonlinearity_bad1_peeled_off}
  \sum_{k-10 < k_2 \leq k+m} \sum_{k_3=k_2+O(1)} P_k \bigl( \chi_{\{ r \leq 2^{-k+m}\}} \phi_{\leq k_2+20} \partial_\alpha \phi_{k_2}^\dagger \partial^\alpha \phi_{k_3} \bigr)
 \end{equation}
 and of \eqref{equ:WM_nonlinearity_bad2} 
 \begin{equation} \label{equ:WM_nonlinearity_bad2_peeled_off}
  \quad 2 \sum_{k-m < k_2 \leq k-10} P_{k-10 < \cdot < k+10} \Bigl( \chi_{\{ r \leq 2^{-k+m} \}} \phi_{\leq k_2+20} \bigl( P_k ( \partial_\alpha \phi_{k_2}^\dagger \partial^\alpha \phi_{k-5 < \cdot < k+5} ) - \partial_\alpha \phi_{k_2}^\dagger \partial^\alpha \phi_k \bigr) \Bigr)
 \end{equation}
 for some large parameter $m \gg 1$, can still be suitably estimated, as we will show in detail in the second part of the proof of this proposition. Correspondingly, within this proof we may also work with the following modified definition 
 \begin{align}
  \bar{\phi}_k - \phi_k &:= \frac{1}{2} \sum_{k-10 < k_2 \leq k+m} \sum_{k_3 = k_2 + O(1)} P_k \bigl( \chi_{\{ r > 2^{-k+m} \}} \phi_{\leq k_2+20} \phi_{k_2}^\dagger \phi_{k_3} \bigr) \label{equ:modified_definition_phi_bar_A} \\
  &\quad + \frac{1}{2} \sum_{k_2 > k + m} \sum_{k_3 = k_2+O(1)} P_k \bigl( \phi_{\leq k_2+20} \phi_{k_2}^\dagger \phi_{k_3} \bigr) \label{equ:modified_definition_phi_bar_B} \\
  &\quad + \sum_{k-m < k_2 \leq k-10} P_{k-10 < \cdot < k+10} \Bigl( \chi_{\{ r > 2^{-k+m} \}} \phi_{\leq k_2+20} \bigl( P_k ( \phi_{k_2}^\dagger \phi_{k-5 < \cdot < k+5} ) - \phi_{k_2}^\dagger \phi_k \bigr) \Bigr) \label{equ:modified_definition_phi_bar_C} \\
  &\quad + \sum_{k_2 \leq k-m} P_{k-10 < \cdot < k+10} \Bigl( \phi_{\leq k_2+20} \bigl( P_k ( \phi_{k_2}^\dagger \phi_{k-5 < \cdot < k+5} ) - \phi_{k_2}^\dagger \phi_k \bigr) \Bigr). \label{equ:modified_definition_phi_bar_D}
 \end{align}
 Note that additional error terms will arise here in the equation for $\Box \bar{\phi}_k$ when a derivative hits the spatial cutoffs $\chi_{\{ r > 2^{-k+m}\}}$. However, these extra terms can be dealt with easily. We now prove that
 \begin{equation} \label{equ:partial_fungibility_difference_decay_in_m}
  \| \bar{\phi} - \phi \|_{S[I]} \lesssim 2^{-\alpha m} \| \phi \|_{S[I]}^\beta
 \end{equation}
 for suitable $\alpha, \beta > 0$, which immediately implies the desired bound~\eqref{equ:partial_fungibility_difference} upon choosing $m \gg 1$ sufficiently large depending on the size of $\|\phi\|_{S[I]}$. We start off with the term~\eqref{equ:modified_definition_phi_bar_A} in the expression for $\bar{\phi}_k - \phi_k$ and seek to show by direct estimation that
 \begin{equation} \label{equ:partial_fungibility_modified_definition_phi_bar_A}
  \biggl( \sum_{k\in\bZ} \, \Bigl\| \sum_{k-10 < k_2 \leq k+m} \sum_{k_3 = k_2 + O(1)} P_k \bigl( \chi_{\{ r > 2^{-k+m} \}} \phi_{\leq k_2+20} \phi_{k_2}^\dagger \phi_{k_3} \bigr) \Bigr\|_{S_k[I]}^2 \biggr)^{\frac{1}{2}} \lesssim 2^{-\alpha m} \| \phi \|_{S[I]}^\beta.
 \end{equation}
 To this end we prove \eqref{equ:partial_fungibility_modified_definition_phi_bar_A} separately for each part of the $S_k$ norm. We begin with the kinetic energy component. Suppose $\nabla_{t,x}$ hits $\phi_{k_2}^\dagger$. Using Strauss' improved Sobolev embedding in the radial case
 \begin{equation} \label{equ:strauss_sobolev_embedding}
  \bigl\| r^{+\frac{1}{2}} \phi_k \bigr\|_{L^\infty_x} \lesssim 2^{-\frac{1}{2} k} \| \phi_k \|_{\dot{H}^1_x},
 \end{equation}
 we may estimate 
 \begin{align*}
  &\biggl( \sum_{k\in\bZ} \, \Bigl\| \sum_{k-10 < k_2 \leq k+m} \sum_{k_3 = k_2 + O(1)} P_k \bigl( \chi_{\{ r > 2^{-k+m} \}} \phi_{\leq k_2+20} \nabla_{t,x} \phi_{k_2}^\dagger \phi_{k_3} \bigr) \Bigr\|_{L^\infty_t L^2_x[I]}^2 \biggr)^{\frac{1}{2}} \\
  &\lesssim \biggl( \sum_{k\in\bZ} \, \Bigl( \sum_{k-10 < k_2 \leq k+m} \sum_{k_3 = k_2 + O(1)} \bigl\| \nabla_{t,x} \phi_{k_2} \bigr\|_{L^\infty_t L^2_x[I]} \bigl\| \chi_{\{ r > 2^{-k+m}\}} \phi_{k_3} \bigr\|_{L^\infty_t L^\infty_x[I]} \Bigr)^2 \biggr)^{\frac{1}{2}} \\
  &\lesssim \biggl( \sum_{k\in\bZ} \, \Bigl( \sum_{k-10 < k_2 \leq k+m} \sum_{k_3 = k_2 + O(1)} \| \phi_{k_2} \|_{S_{k_2}[I]} 2^{-\frac{1}{2} m} 2^{-\frac{1}{2} (k_3 - k)} \| \phi_{k_3} \|_{S_{k_3}[I]} \Bigr)^2 \biggr)^{\frac{1}{2}} \\
  &\lesssim 2^{-\frac{1}{2} m} \| \phi \|_{S[I]}^2.
 \end{align*}
 Analogously, we can deal with the expressions arising when $\nabla_{t,x}$ hits $\phi_{\leq k_2+20}$ or the cutoff $\chi_{\{ r > 2^{-k+m} \}}$, which yields \eqref{equ:partial_fungibility_modified_definition_phi_bar_A} for the kinetic energy component of the $S_k$ norm. In a very similar fashion we may also prove \eqref{equ:partial_fungibility_modified_definition_phi_bar_A} for the Strichartz, weighted endpoint Strichartz and local energy decay components of the $S_k$ norm. Hence we are left to prove \eqref{equ:partial_fungibility_modified_definition_phi_bar_A} for the more delicate atomic space parts $Z_k^{\pm}$. It suffices to consider the $Z_k^{+}$ component. Assume that $(\partial_t + \partial_r)$ hits the input $\phi_{k_2}^\dagger$. Then we decompose $(\partial_t + \partial_r) \phi_{k_2}$ into its $(+)$ and $(-)$ components
 \[
  (\partial_t + \partial_r) \phi_{k_2} = (\partial_t + \partial_r) \phi_{k_2}^{(+)} + (\partial_t + \partial_r) \phi_{k_2}^{(-)}.
 \]
 We now show that in case of the $(\partial_t + \partial_r) \phi_{k_2}^{(+)}$ component, we can place the whole expression
 \begin{equation*}
  \sum_{k-10 < k_2 \leq k+m} \sum_{k_3 = k_2 +O(1)} P_k \bigl( \chi_{\{ r > 2^{-k+m} \}} \phi_{\leq k_2+20} (\partial_t + \partial_r) \phi_{k_2}^{(+) \dagger} \phi_{k_3} \bigr)
 \end{equation*}
 into the $(+)$ component of the $Z_k^{+}$ norm with the desired bound. Indeed, we have for the $L^2_{t+r} L^\infty_{t-r}$ part that
 \begin{align*}
  &\Biggl( \sum_{k\in\bZ} \, \biggl( \sup_{0 < \lambda \leq \frac{1}{2}} 2^{-(\frac{1}{2}-\lambda) k} \Bigl\| r^{+\lambda} \sum_{k-10 < k_2 \leq k+m} \sum_{k_3 = k_2 + O(1)} P_k \bigl( \chi_{\{ r > 2^{-k+m}\}} \phi_{\leq k_2+20} (\partial_t + \partial_r) \phi_{k_2}^{(+) \dagger} \phi_{k_3} \bigr) \Bigr\|_{L^2_{t+r} L^\infty_{t-r}[I]} \biggr)^2 \Biggr)^{\frac{1}{2}} \\
  &\lesssim \Biggl( \sum_{k\in\bZ} \, \biggl( \sup_{0 < \lambda \leq \frac{1}{2}} \sum_{k-10 < k_2 \leq k+m} \sum_{k_3 = k_2 + O(1)} 2^{-(\frac{1}{2} - \lambda) k} \bigl\| r^{+\lambda} (\partial_t + \partial_r) \phi_{k_2}^{(+)} \bigr\|_{L^2_{t+r} L^\infty_{t-r}} \bigl\| \chi_{\{ r > 2^{-k+m} \}} \phi_{k_3} \bigr\|_{L^\infty_t L^\infty_x} \biggr)^2 \Biggr)^{\frac{1}{2}} \\
  &\lesssim \Biggl( \sum_{k\in\bZ} \, \biggl( \sup_{0 < \lambda \leq \frac{1}{2}} \sum_{k-10 < k_2 \leq k+m} \sum_{k_3 = k_2 + O(1)} 2^{-\frac{1}{2} m} 2^{-\lambda (k_2-k)} \| \phi_{k_2} \|_{S_{k_2}} \| \phi_{k_3} \|_{S_{k_3}} \biggr)^2 \Biggr)^{\frac{1}{2}} \\
  &\lesssim 2^{-\frac{1}{2} m} m^{\frac{1}{2}} \|\phi\|_{S[I]}^2,
 \end{align*}
 where we again used Strauss' improved Sobolev embedding \eqref{equ:strauss_sobolev_embedding}. Similarly, we may bound the $Y_k$ norm. 
 
 Instead, when we face the $(\partial_t + \partial_r) \phi_{k_2}^{(-)}$ component, we can suitably place the whole resulting expression into the $(-)$ component of the $Z_k^+$ norm. More precisely, for the weighted $L^q_t L^p_x$ norm we find that 
 \begin{align*}
  &2^{(\frac{1}{q}+\frac{2}{p}-1) k} 2^{\lambda k} \sum_{\ell\in\bZ} \, \Bigl\| \chi_{\{ r \sim 2^\ell\}} r^{+\lambda} \sum_{k-10 < k_2 \leq k+m} \sum_{k_3 = k_2+O(1)} P_k \bigl( \chi_{\{ r > 2^{-k+m}\}} \phi_{\leq k_2+20} (\partial_t + \partial_r) \phi_{k_2}^{(-) \dagger} \phi_{k_3} \bigr) \Bigr\|_{L^q_t L^p_x} \\
  &\lesssim \sum_{k-10 < k_2 \leq k+m} \sum_{k_3 = k_2+O(1)} 2^{(\frac{1}{q}+\frac{2}{p}-1) k} 2^{\lambda k} \sum_{\ell\in\bZ} \, \bigl\| \chi_{\{ r \sim 2^\ell\}} r^{+\lambda} (\partial_t + \partial_r) \phi_{k_2}^{(-)} \bigr\|_{L^q_t L^p_x} \bigl\| \chi_{\{ r > 2^{-k+m}\}} \phi_{k_3} \bigr\|_{L^\infty_t L^\infty_x} \\
  &\lesssim 2^{-\frac{1}{2} m} \sum_{k-10 < k_2 \leq k+m} 2^{(\frac{1}{q} + \frac{2}{p} - \frac{1}{2} + \lambda)(k-k_2)} \|\phi_{k_2}\|_{S_{k_2}}^2
 \end{align*}
 and in view of the admissible ranges for $(q,p)$ and $\lambda$, upon square-summing over $k\in\bZ$, we may bound the whole expression by $2^{-\delta m} m^{\frac{1}{2}} \|\phi\|_{S[I]}^2$ for some small $\delta > 0$, which is of the desired form. We proceed analogously with the $Y_k$ norm. When $(\partial_t + \partial_r)$ hits the other inputs, the $Z_k^+$ norm can be bounded similarly. This finishes the proof of~\eqref{equ:partial_fungibility_modified_definition_phi_bar_A} and hence the proof of~\eqref{equ:partial_fungibility_difference_decay_in_m} for the term \eqref{equ:modified_definition_phi_bar_A}.
 
 In order to show \eqref{equ:partial_fungibility_difference_decay_in_m} for the term \eqref{equ:modified_definition_phi_bar_B} we can proceed analogously only that we have to introduce an additional splitting
 \begin{align*}
  \sum_{k_2 > k+m} \sum_{k_3 = k_2 + O(1)} P_k \bigl( \phi_{\leq k_2+20} \phi_{k_2}^\dagger \phi_{k_3} \bigr) &= \sum_{k_2 > k+m} \sum_{k_3 = k_2 + O(1)} P_k \bigl( \chi_{\{ r \leq 2^{-k}\}} \phi_{\leq k_2+20} \phi_{k_2}^\dagger \phi_{k_3} \bigr) \\
  &\quad + \sum_{k_2 > k+m} \sum_{k_3 = k_2 + O(1)} P_k \bigl( \chi_{\{ r > 2^{-k} \}} \phi_{\leq k_2+20} \phi_{k_2}^\dagger \phi_{k_3} \bigr).
 \end{align*}
 Suppose here that the input $\phi_{k_2}^\dagger$ is hit by a derivative, then we achieve an exponential gain in $-k_3$ and thus smallness, either by using Bernstein to place $\phi_{k_3}$ into $L^\infty_t L^2_x$ or by placing $\phi_{k_3}$ into $L^\infty_t L^\infty_x$ and using Strauss' improved Sobolev embedding~\eqref{equ:strauss_sobolev_embedding}.
 
 Finally, the proof of \eqref{equ:partial_fungibility_difference_decay_in_m} for the terms \eqref{equ:modified_definition_phi_bar_C} and \eqref{equ:modified_definition_phi_bar_D} works similarly to the above estimates and is left to the reader. This finishes the first part of the proof of Proposition~\ref{prop:partial_fungibility}, namely establishing the bound~\eqref{equ:partial_fungibility_difference}. 

 \medskip
 
 We now turn to the proof of the estimate~\eqref{equ:partial_fungibility_phi_bar}. This step is slightly compounded by the fact that the renormalized variable $U_{\leq k-10} \bar{\phi}_k$ is only approximately frequency localized to frequency $2^k$ up to exponentially decaying tails. We will first show via a divisibility argument that we can partition the time interval $I = \cup_j I_j$ into $N(K, E)$ consecutive intervals $I_j = [t_j, t_{j+1}]$ satisfying  
 \begin{equation} \label{equ:divisibility_bound_refined}
  \biggl( \sum_{k\in\bZ} \, \Bigl( \sum_{k_0 \in \bZ} 2^{\alpha |k-k_0|} \bigl\| P_{k_0} \bigl( U_{\leq k-10} \bar{\phi}_k \bigr) \bigr\|_{S_{k_0}[I_j]} \Bigr)^2 \biggr)^{\frac{1}{2}} \lesssim C(E)
 \end{equation}
 for any $0 \leq \alpha < 1$. Then we will infer the desired bound 
 \begin{equation*}
  \biggl( \sum_{k\in\bZ} \| \bar{\phi}_k \|_{S_k[I_j]}^2 \biggr)^{\frac{1}{2}} \lesssim C(E)
 \end{equation*}
 by decomposing the product $\bar{\phi}_k = U_{\leq k-10}^\dagger \bigl( U_{\leq k-10} \bar{\phi}_k \bigr)$ into a Littlewood-Paley trichotomy and invoking the bounds~\eqref{equ:divisibility_bound_refined} as well as by introducing a modified version of the gauge transformations~$U_{\leq h}$. 
 
 We begin with the proof of~\eqref{equ:divisibility_bound_refined}. Using the energy estimate~\eqref{equ:energy_estimate} we find that
 \begin{equation} \label{equ:divisibility_energy_estimate_applied}
  \begin{aligned}
   \biggl( \sum_{k\in\bZ} \, \Bigl( \sum_{k_0 \in \bZ} 2^{\alpha |k-k_0|} \bigl\| P_{k_0} \bigl( U_{\leq k-10} \bar{\phi}_k \bigr) \bigr\|_{S_{k_0}[I_j]} \Bigr)^2 \biggr)^{\frac{1}{2}} &\lesssim \biggl( \sum_{k\in\bZ} \, \Bigl( \sum_{k_0 \in \bZ} 2^{\alpha |k-k_0|} \bigl\| \nabla_{t,x} P_{k_0} \bigl( U_{\leq k-10} \bar{\phi}_k \bigr)(t_j) \bigr\|_{L^2_x} \Bigr)^2 \biggr)^{\frac{1}{2}} \\
   &\quad + \biggl( \sum_{k\in\bZ} \, \Bigl( \sum_{k_0 \in \bZ} 2^{\alpha |k-k_0|} \bigl\| P_{k_0} \Box \bigl( U_{\leq k-10} \bar{\phi}_k \bigr) \bigr\|_{L^1_t L^2_x[I_j]} \Bigr)^2 \biggr)^{\frac{1}{2}}.
  \end{aligned}
 \end{equation}
 For the first term on the right-hand side of~\eqref{equ:divisibility_energy_estimate_applied}, we now show in detail that it is bounded by $C(E)$ in the case when the derivative $\nabla_{t,x}$ falls on $\bar{\phi}_k$. The other case when it hits $U_{\leq k-10}$ can be treated similarly. Using Bernstein's inequality and the sharp localization of $\bar{\phi}_k$ to frequency $2^k$, we obtain that
 \begin{equation} \label{equ:divisibility_kinetic_energy_bound}
  \begin{aligned}
   \sum_{k\in\bZ} \Bigl( \sum_{k_0 \in \bZ} 2^{\alpha |k-k_0|} \bigl\| P_{k_0} \bigl( U_{\leq k-10} \nabla_{t,x} \bar{\phi}_k \bigr\|_{L^2_x} \Bigr)^2 &\lesssim \sum_{k \in \bZ} \Bigl( \sum_{k_0 \leq k-10} 2^{\alpha(k-k_0)} \bigl\| P_{k_0} \bigl( (P_{k+O(1)} U_{\leq k-10}) \nabla_{t,x} \bar{\phi}_k \bigr\|_{L^2_x} \Bigr)^2 \\
   &\quad + \sum_{k\in\bZ} \, \bigl\| P_{k+O(1)} \bigl( U_{\leq k-10} \nabla_{t,x} \bar{\phi}_k \bigr) \bigr\|_{L^2_x}^2 \\
   &\quad + \sum_{k\in\bZ} \Bigl( \sum_{k_0 > k+10} 2^{\alpha(k_0-k)} \bigl\| P_{k_0} \bigl( (P_{k_0+O(1)} U_{\leq k-10}) \nabla_{t,x} \bar{\phi}_k \bigr) \bigr\|_{L^2_x} \Bigr)^2 \\
   &\lesssim \sum_{k\in\bZ} \Bigl( \sum_{k_0 \leq k-10} 2^{\alpha (k-k_0)} 2^{k_0-2k} \bigl\| \nabla_x^2 U_{\leq k-10} \bigr\|_{L^2_x} \|\nabla_{t,x} \bar{\phi}_k\|_{L^2_x} \Bigr)^2 \\
   &\quad + \sum_{k\in\bZ} \| \nabla_{t,x} \bar{\phi}_k \|_{L^2_x}^2 \\
   &\quad + \sum_{k\in\bZ} \Bigl( \sum_{k_0 > k+10} 2^{\alpha (k_0-k)} 2^{k-2k_0} \bigl\| \nabla_x^2 U_{\leq k-10} \bigr\|_{L^2_x} \|\nabla_{t,x} \bar{\phi}_k \|_{L^2_x} \Bigr)^2.
  \end{aligned}
 \end{equation}
 Then the bounds on the gauge transformation from Lemma~\ref{lem:gauge_transform_bounds}, the kinetic energy bounds on $\bar{\phi}_k$ from Lemma~\ref{lem:bounds_on_phi_bar} and the fact that $\alpha < 1$ yield that the right-hand side of the previous line is bounded by
 \begin{align*}
  C(E) \, \Bigl( \sum_{k_0 \leq k-10} 2^{(\alpha -1)(k-k_0)} + \sum_{k_0 > k+10} 2^{(\alpha -2) (k_0-k)} \Bigr) \sum_{k\in\bZ} \|\nabla_{t,x} \bar{\phi}_k\|_{L^2_x}^2 \lesssim C(E).
 \end{align*}
 
 Next we prove by divisibility that on suitable time intervals $I_j$ the second term on the right-hand side of~\eqref{equ:divisibility_energy_estimate_applied}
 \begin{equation} \label{equ:divisibility_L1L2_with_tails}
  \biggl( \sum_{k\in\bZ} \, \Bigl( \sum_{k_0 \in \bZ} 2^{\alpha |k-k_0|} \bigl\| P_{k_0} \Box \bigl( U_{\leq k-10} \bar{\phi}_k \bigr) \bigr\|_{L^1_t L^2_x[I_j]} \Bigr)^2 \biggr)^{\frac{1}{2}}
 \end{equation}
 is also bounded by $C(E)$. Here we first recall from~\eqref{equ:nlw_for_renormalized_phi_bar} that the wave equation for the renormalized variable $U_{\leq k-10} \bar{\phi}_k$ is given by 
 \begin{equation} \label{equ:nlw_for_renormalized_phi_bar_in_fungibility_proof}
  \begin{aligned}
   \Box \bigl( U_{\leq k-10} \bar{\phi}_k \bigr) &= U_{\leq k-10} \bigl( \Box \bar{\phi}_k - 2 A_{\alpha; \leq k-10} \partial^\alpha \phi_k \bigr) \\
   &\quad + 2 \bigl( \partial_\alpha U_{\leq k-10} + U_{\leq k-10} A_{\alpha; \leq k-10} \bigr) \partial^\alpha \phi_k \\
   &\quad + 2 \partial_\alpha U_{\leq k-10} \partial^\alpha (\bar{\phi}_k - \phi_k) \\
   &\quad + \bigl( \Box U_{\leq k-10} \bigr) \bar{\phi}_k
  \end{aligned}
 \end{equation} 
 and that in Proposition~\ref{prop:nlw_for_renormalized_phi_bar_schematic_identities} we had carefully uncovered the schematic form of each term in the nonlinearity on the right-hand side of~\eqref{equ:nlw_for_renormalized_phi_bar_in_fungibility_proof}. In view of these identities we observe that we can treat the exponential tails in~\eqref{equ:divisibility_L1L2_with_tails} by playing out Bernstein's inequality and the bounds on the gauge transformations from Lemma~\ref{lem:gauge_transform_bounds}, analogously to how we proceeded in the estimate~\eqref{equ:divisibility_kinetic_energy_bound}. It therefore suffices in the following to prove by divisibility that we can partition the time interval $I = \cup_j I_j$ into $N(K,E)$ consecutive intervals $I_j$ satisfying
 \begin{equation}
  \sum_{k\in\bZ} \, \bigl\| \Box \bigl( U_{\leq k-10} \bar{\phi}_k \bigr) \bigr\|_{L^1_t L^2_x[I_j]}^2 \lesssim C(E).
 \end{equation}
 We now prove this bound seperately for each type of term appearing on the right-hand side of the wave equation~\eqref{equ:nlw_for_renormalized_phi_bar_in_fungibility_proof} for the renormalized variable $U_{\leq k-10} \bar{\phi}_k$. We start with the term $(\Box \bar{\phi}_k - 2 A_{\alpha; \leq k-10} \partial^\alpha \phi_k \bigr)$ and recall its basic decomposition~\eqref{equ:nlw_for_renormalized_phi_bar_term1}. First though, we have to deal with the two extra terms \eqref{equ:WM_nonlinearity_bad1_peeled_off} and \eqref{equ:WM_nonlinearity_bad2_peeled_off} that appear due to our modified definition of $\bar{\phi}$ in this proof. For the term~\eqref{equ:WM_nonlinearity_bad1_peeled_off} we use Bernstein's inequality to bound
 \begin{align*}
  &\sum_{k\in\bZ} \, \biggl\| \sum_{k-10 < k_2 \leq k+m} \sum_{k_3=k_2+O(1)} P_k \bigl( \chi_{\{ r \leq 2^{-k+m}\}} \phi_{\leq k_2+20} \partial_\alpha \phi_{k_2}^\dagger \partial^\alpha \phi_{k_3} \bigr) \biggr\|_{L^1_t L^2_x}^2 \\
  &\lesssim \sum_{k\in\bZ} \, \biggl( \sum_{k-10 < k_2 \leq k+m} \sum_{\ell \leq -k+m} 2^{k+\ell} \bigl\| \chi_{\{ r \sim 2^\ell \}} r^{-\frac{1}{2}} \nabla_{t,x} \phi_{k_2} \bigr\|_{L^2_t L^2_x}^2 \biggr)^2 \\
  &\lesssim \|\phi\|_S^2 \biggl( \sum_{k\in\bZ} \sum_{k-10 < k_2 \leq k+m} \sum_{\ell \leq -k+m} 2^{k+\ell} \bigl\| \chi_{\{ r \sim 2^\ell \}} r^{-\frac{1}{2}} \nabla_{t,x} \phi_{k_2} \bigr\|_{L^2_t L^2_x}^2 \biggr). 
 \end{align*}
 Then we observe that the last factor in the previous line has the divisibility property and satisfies
 \begin{align*}
  \sum_{k\in\bZ} \sum_{k-10 < k_2 \leq k+m} \sum_{\ell \leq -k+m} 2^{k+\ell} \bigl\| \chi_{\{ r \sim 2^\ell \}} r^{-\frac{1}{2}} \nabla_{t,x} \phi_{k_2} \bigr\|_{L^2_t L^2_x}^2 \lesssim C(m) \sum_{k\in\bZ} \|\phi_k\|_{S_k}^2 \lesssim C(m) \|\phi\|_S^2.
 \end{align*}
 It can thus be made sufficiently small when restricted to suitable time intervals noting that at this point the size of the parameter $m \gg 1$ has already been fixed. Similarly, we can gain smallness for the other additional term~\eqref{equ:WM_nonlinearity_bad2_peeled_off} on suitable time intervals. We continue with the first term on the right-hand side of~\eqref{equ:nlw_for_renormalized_phi_bar_term1}. Here we have to distinguish two cases 
 \begin{align*}
  &\sum_{k\in\bZ} \, \biggl\| \sum_{k_2 \leq k-10} \sum_{k_3 = k + O(1)} P_k L \bigl( \phi_{k_2 - 10 < \cdot \leq k-10}, \partial_\alpha \phi_{k_2}, \partial^\alpha \phi_{k_3} \bigr) \biggr\|_{L^1_t L^2_x}^2 \\ &\lesssim \sum_{k\in\bZ} \, \biggl\| \sum_{k_2 \leq k-10} \sum_{k_3 = k + O(1)} P_k \Bigl( \chi_{\{ r \leq 2^{-k_2} \}} L \bigl( \phi_{k_2 - 10 < \cdot \leq k-10}, \partial_\alpha \phi_{k_2}, \partial^\alpha \phi_{k_3} \bigr) \Bigr) \biggr\|_{L^1_t L^2_x}^2 \\
  &\quad + \sum_{k\in\bZ} \, \biggl\| \sum_{k_2 \leq k-10} \sum_{k_3 = k + O(1)} P_k \Bigl( \chi_{\{ r > 2^{-k_2} \}} L \bigl( \phi_{k_2 - 10 < \cdot \leq k-10}, \partial_\alpha \phi_{k_2}, \partial^\alpha \phi_{k_3} \bigr) \Bigr) \biggr\|_{L^1_t L^2_x}^2.
 \end{align*}
 In the first case we bound
 \begin{align*}
  &\sum_{k\in\bZ} \, \biggl\| \sum_{k_2 \leq k-10} \sum_{k_3 = k + O(1)} P_k \Bigl( \chi_{\{ r \leq 2^{-k_2} \}} L \bigl( \phi_{k_2 - 10 < \cdot \leq k-10}, \partial_\alpha \phi_{k_2}, \partial^\alpha \phi_{k_3} \bigr) \Bigr) \biggr\|_{L^1_t L^2_x}^2 \\
  &\lesssim \sum_{k\in\bZ} \, \biggl\| \sum_{k_2 \leq k-10} \sum_{k_1 \geq k_2-10} \sum_{\ell \leq -k_2} \chi_{\{ r \sim 2^\ell\}} r^{+\frac{1}{2}} |\phi_{k_1}| |\nabla_{t,x} \phi_{k_2}| \biggr\|_{L^2_t L^\infty_x}^2 \sup_{\ell \in \bZ} \, \bigl\| \chi_{\{r \sim 2^{\ell}\}} r^{-\frac{1}{2}} \nabla_{t,x} \phi_k \bigr\|_{L^2_t L^2_x}^2 \\
  &\lesssim \biggl\| \sum_{k_2 \in \bZ} \sum_{k_1 \geq k_2-10} \sum_{\ell \leq -k_2} \chi_{\{ r \sim 2^\ell\}} r^{+\frac{1}{2}} |\phi_{k_1}| |\nabla_{t,x} \phi_{k_2}| \biggr\|_{L^2_t L^\infty_x}^2 \|\phi\|_{S}^2 
 \end{align*}
 and note that the first factor on the right-hand side of the previous line has the divisibility property 
 \begin{align*}
  \biggl\| \sum_{k_2 \in \bZ} \sum_{k_1 \geq k_2-10} \sum_{\ell \leq -k_2} \chi_{\{ r \sim 2^\ell\}} r^{+\frac{1}{2}} |\phi_{k_1}| |\nabla_{t,x} \phi_{k_2}| \biggr\|_{L^2_t L^\infty_x} \lesssim \sum_{k_2 \in \bZ} \sum_{k_1 \geq k_2-10} 2^{-\frac{1}{2} k_2} \|\phi_{k_1}\|_{L^4_t L^\infty_x} \| \nabla_{t,x} \phi_{k_2} \|_{L^4_t L^\infty_x}  \lesssim \|\phi\|_S^2
 \end{align*}
 and thus yields smallness on suitable time intervals. In the second case we use the trilinear estimate~\eqref{equ:trilinear_estimate} and its proof to bound 
 \begin{align*}
  &\sum_{k\in\bZ} \, \biggl\| \sum_{k_2 \leq k-10} \sum_{k_3 = k + O(1)} P_k \Bigl( \chi_{\{ r > 2^{-k_2} \}} L \bigl( \phi_{k_2 - 10 < \cdot \leq k-10}, \partial_\alpha \phi_{k_2}, \partial^\alpha \phi_{k_3} \bigr) \Bigr) \biggr\|_{L^1_t L^2_x}^2 \\
  &\lesssim \sum_{k\in\bZ} \, \biggl( \sum_{k_2 \leq k-10} \sum_{k_1 \geq k_2 - 10} 2^{-\frac{1}{4}(k_1-k_2)} 2^{+\frac{1}{4} k_1} \bigl\| \chi_{\{ r > 2^{-k_1} \}} r^{-\frac{1}{4}}  \phi_{k_1} \bigr\|_{L^2_t L^\infty_x} \|\phi_{k_2}\|_{S_{k_2}} \|\phi_k\|_{S_k} \biggr)^2 \\
  &\lesssim \biggl( \sum_{k_1 \in \bZ} 2^{\frac{1}{2} k_1} \bigl\| \chi_{\{ r > 2^{-k_1} \}} r^{-\frac{1}{4}}  \phi_{k_1} \bigr\|_{L^2_t L^\infty_x}^2 \biggr) \, \|\phi\|_S^4
 \end{align*}
 and then the first factor on the right-hand side has the divisibility property to gain the desired smallness. Next we turn to the second term on the right-hand side of~\eqref{equ:nlw_for_renormalized_phi_bar_term1} and seek to achieve smallness for the expression
 \[
  \sum_{k\in\bZ} \, \biggl\| \sum_{k_2 > k-10} \sum_{k_1 = k_2 + O(1)} \sum_{k_3 \leq k_2 + O(1)} P_k L \bigl( \phi_{k_1}, \partial_\alpha \phi_{k_2}, \partial^\alpha \phi_{k_3} \bigr) \biggr\|_{L^1_t L^2_x}^2.
 \]
 To this end we note that by Bernstein's inequality and the null form estimate~\eqref{equ:null_form_estimate} it holds that
 \begin{align*}
  &\biggl\| \sum_{k_2 > k-10} \sum_{k_1 = k_2 + O(1)} \sum_{k_3 \leq k_2 + O(1)} P_k L \bigl( \phi_{k_1}, \partial_\alpha \phi_{k_2}, \partial^\alpha \phi_{k_3} \bigr) \biggr\|_{L^1_t L^2_x}^2 \\
  &\lesssim \sum_{k_2 > k-10} \sum_{k_1 = k_2 + O(1)} \sum_{k_3 \leq k_2 + O(1)} 2^{\frac{1}{5} k} \bigl\| \phi_{k_1} \bigr\|_{L^{\frac{5}{2}}_t L^\infty_x} \bigl\| \partial_\alpha \phi_{k_2} \partial^\alpha \phi_{k_3} \bigr\|_{L^{\frac{5}{3}}_t L^{\frac{5}{3}}_x} \\
  &\lesssim \sum_{k_2 > k-10} \sum_{k_3 \leq k_2 + O(1)} 2^{-\frac{1}{5}(k_2-k)} 2^{-\frac{1}{10}(k_2-k_3)} \|\phi_{k_2}\|_{S_{k_2}}^2 \|\phi_{k_3}\|_{S_{k_3}}.
 \end{align*}
 Hence, we can exploit the exponential gains in the frequency differences to achieve smallness when the inputs have disparate frequency supports and reduce to the case where all inputs are at frequency $2^{k+O(M)}$ where $M \gg 1$ is chosen sufficiently large depending on $\|\phi\|_S$. But then we obtain smallness by divisibility as in the previous step. The third term on the right-hand side of~\eqref{equ:nlw_for_renormalized_phi_bar_term1} can be treated in exactly the same manner and all quintilinear terms in the fourth term $Q_k(\phi)$ on the right-hand side of~\eqref{equ:nlw_for_renormalized_phi_bar_term1} are easier to control using just Strichartz estimates and the null form estimate~\eqref{equ:null_form_estimate}. We leave the details to the reader.
 
 We are left to consider all remaining terms on the right-hand side of the wave equation~\eqref{equ:nlw_for_renormalized_phi_bar_in_fungibility_proof} for the renormalized variable $U_{\leq k-10} \bar{\phi}_k$. But in view of the schematic identities \eqref{equ:schematic_form_nlw_for_renormalized_phi_bar_term2}--\eqref{equ:schematic_form_nlw_for_renormalized_phi_bar_term4} from Proposition~\ref{prop:nlw_for_renormalized_phi_bar_schematic_identities}, these are either of the same type as the first term on the right-hand side of~\eqref{equ:nlw_for_renormalized_phi_bar_term1}, which we have already dealt with above, or they are quintilinear and therefore easier to control.

 Thus, we can partition the time interval $I = \cup_j I_j$ into $N(K, E)$ consecutive intervals $I_j$ satisfying
 \begin{equation} \label{equ:divisibility_bound_gauged_phi_bar}
  \biggl( \sum_{k\in\bZ} \, \Bigl( \sum_{k_0 \in \bZ} 2^{\alpha |k-k_0|} \bigl\| P_{k_0} \bigl( U_{\leq k-10} \bar{\phi}_k \bigr) \bigr\|_{S_{k_0}[I_j]} \Bigr)^2 \biggr)^{\frac{1}{2}} \lesssim C(E)
 \end{equation}
 for any $0 \leq \alpha < 1$. It now remains to transfer these bounds to $\bar{\phi}$, i.e. to show that on each such interval $I_j$ we have that
 \begin{equation*}
  \biggl( \sum_{k \in \bZ} \| \bar{\phi}_k \|_{S_k[I_j]}^2 \biggr)^{\frac{1}{2}} \lesssim C(E).
 \end{equation*}
 To this end we use the exact orthogonality of the gauge transformations to write
 \[
  \bar{\phi}_k = U_{\leq k-10}^\dagger U_{\leq k-10} \bar{\phi}_k
 \]
 and then decompose into a Littlewood-Paley trichotomy in view of the localization of $\bar{\phi}_k$ to frequency $2^k$,
 \begin{equation} \label{equ:phi_bar_trichotomy}
  \begin{aligned}
   \bar{\phi}_k &= P_{\leq k-C} \bigl( U_{\leq k-10}^\dagger \bigr) \, P_{k+O(1)} \bigl( U_{\leq k-10} \bar{\phi}_k \bigr) \\
   &\quad + P_{k+O(1)} \bigl( U_{\leq k-10}^\dagger \bigr) \, P_{\leq k-C} \bigl( U_{\leq k-10} \bar{\phi}_k \bigr) \\
   &\quad + \sum_{k_1 > k + O(1)} \sum_{k_2 = k_1 + O(1)} P_k \bigl( P_{k_1} \bigl( U_{\leq k-10}^\dagger \bigr) \, P_{k_2} \bigl( U_{\leq k-10}^\dagger \bar{\phi}_k \bigr) \bigr) \\
   &\equiv (LH)_k + (HL)_k + (HH)_k.
  \end{aligned}
 \end{equation}
 We now estimate \eqref{equ:phi_bar_trichotomy} separately for each part of the $S_k$ norm. Here we observe that for those parts of the $S_k$ norm that do not involve a derivative the desired bound follows immediately from~\eqref{equ:divisibility_bound_gauged_phi_bar} and the boundedness of $P_{k_1} U_{\leq k_2}^\dagger$ for arbitrary $k_1, k_2 \in \bZ$ thanks to the exact orthogonality of $U_{\leq k_2}^\dagger$. Similarly, for all other parts of the $S_k$ norm when the derivative falls on the $U_{\leq k-10} \bar{\phi}_k$ term, the desired bound follows immediately.
 
 Let us therefore begin with the treatment of the Strichartz component of the $S_k$ norm when the derivative $\nabla_{t,x}$ falls onto the term $U_{\leq k-10}^\dagger$. Noting that the gauge transformation bounds from Lemma~\ref{lem:gauge_transform_bounds} also hold for the transpose $U_{\leq k-10}^\dagger$, we find for the $(LH)_k$ part that
 \begin{align*}
  &2^{(\frac{1}{q} + \frac{2}{p} -1) k} \sum_{k_1 \leq k-C} \bigl\| P_{k_1} \nabla_{t,x} U_{\leq k-10}^\dagger \bigr\|_{L^\infty_t L^\infty_x[I_j]} \bigl\| P_{k+O(1)} \bigl( U_{\leq k-10} \bar{\phi}_k \bigr) \bigr\|_{L^q_t L^p_x[I_j]} \\
  &\lesssim \sum_{k_1 \leq k+O(1)} 2^{-k} 2^{\frac{1}{2} k_1} \bigl\| P_{k_1} \nabla_{t,x} U_{\leq k-10}^\dagger \bigr\|_{L^\infty_t L^4_x[I_j]} \bigl\| P_k \bigl( U_{\leq k-10} \bar{\phi}_k \bigr) \bigr\|_{S_k[I_j]} \\
  &\lesssim \sum_{k_1 \leq k+O(1)} 2^{-k} 2^{\frac{1}{2} k_1} 2^{\frac{1}{2} k} C(E) \, \bigl\| P_k \bigl( U_{\leq k-10} \bar{\phi}_k \bigr) \bigr\|_{S_k[I_j]} \\
  &\lesssim C(E) \, \bigl\| P_k \bigl( U_{\leq k-10} \bar{\phi}_k \bigr) \bigr\|_{S_k[I_j]}.
 \end{align*}
 Square-summing over $k\in\bZ$ and invoking the previously established bound~\eqref{equ:divisibility_bound_gauged_phi_bar}, we obtain the desired estimate. We proceed analogously with the Strichartz component for the $(HL)_k$ and $(HH)_k$ parts. The weighted $L^2_t L^\infty_x$ component and the local energy decay component of the $S_k$ norm can also be treated similarly so that we now turn to the more delicate $Z_k^+$ norm, the $Z_k^-$ norm being handled in exactly the same manner. We consider in detail the high-high part $(HH)_k$ for the $Z_k^+$ norm, the $(LH)_k$ and $(HL)_k$ parts being similar.
 
 Here it turns out that we cannot avoid to pick up factors of $\| \phi \|_{S[I_j]}$ which would destroy our final goal to obtain a bound just in terms of the energy $E$ on the interval $I_j$. For this reason we have to slightly modify the definition of the gauge transformations to introduce another source of smallness to compensate factors of $\|\phi\|_{S[I_j]}$. At this point we recall that the non-perturbative term~\eqref{equ:WM_nonlinearity_bad0},
 \begin{equation} \label{equ:breakdown_criterion_phi_bar_WM_nonlinearity_bad0}
  2 \sum_{k_2 \leq k-10} \phi_{\leq k_2-10} \partial_\alpha \phi_{k_2}^\dagger \partial^\alpha \phi_k,
 \end{equation}
 in the decomposition \eqref{equ:WM_nonlinearity_bad0}--\eqref{equ:WM_nonlinearity_derivative_transferred_to_lowest_freq} of the wave maps nonlinearity at frequency $2^k$ necessitated the introduction of the gauge transformations $U_{\leq h}$. However, letting $m \gg 1$ be a sufficiently large integer depending only on the size of the $S[I]$ norm of $\phi$, we can in fact still ``peel off'' the good term
 \begin{equation*}
  2 \sum_{k-m < k_2 \leq k-10} \chi_{\{ r \leq 2^{-k+m} \}} \phi_{\leq k_2-10} \partial_\alpha \phi_{k_2}^\dagger \partial^\alpha \phi_k
 \end{equation*}
 from \eqref{equ:breakdown_criterion_phi_bar_WM_nonlinearity_bad0}. The latter term can be easily seen to have the divisibility property and can therefore be dealt with accordingly in the previous parts of this proof of Proposition~\ref{prop:partial_fungibility}. Hence, we only have to use a gauge transformation that suitably renormalizes the expression 
 \begin{equation*}
  2 \sum_{k_2 \leq k-m} \phi_{\leq k_2-10} \partial_\alpha \phi_{k_2}^\dagger \partial^\alpha \phi_k + 2 \sum_{k-m < k_2 \leq k-10} \chi_{\{ r > 2^{-k+m} \}} \phi_{\leq k_2-10} \partial_\alpha \phi_{k_2}^\dagger \partial^\alpha \phi_k.
 \end{equation*}
 Correspondingly, we modify the definition of the matrix $B_h$ in the construction of the gauge transformations $U_{\leq h}$ in~\eqref{equ:matrix_gauge_transform} to
 \begin{equation} \label{equ:modified_gauge_transform_definition}
  \begin{aligned}
   B_h &:= \frac{d}{dh} \sum_{\ell \in \bZ} \eta(h+10-m-\ell) (\phi_{\leq \ell-10} \phi_{\ell}^\dagger - \phi_{\ell} \phi_{\leq \ell-10}^\dagger) \\
   &\quad \quad + \frac{d}{dh} \sum_{\ell \in \bZ} \bigl( \eta(h-\ell) - \eta(h+10-m-\ell) \bigr) \chi_{\{ r > 2^{-h-10+m} \}} (\phi_{\leq \ell-10} \phi_{\ell}^\dagger - \phi_{\ell} \phi_{\leq \ell-10}^\dagger)  .
  \end{aligned}
 \end{equation}
 Then we may conclude as in the proof of Proposition~\ref{prop:nlw_for_renormalized_phi_bar_schematic_identities} that $(\partial_t + \partial_r) U_{\leq k-10}^\dagger$ is of the schematic form
 \begin{align*}
  (\partial_t + \partial_r) U_{\leq k-10}^\dagger &= \sum_{k_3 \leq k-m} U_{\leq k_3-1} (\partial_t + \partial_r) L(\phi_{\leq k_3-10}, \phi_{k_3}) \\
  &\quad + \sum_{k-m < k_3 \leq k-10} U_{\leq k_3-1} (\partial_t + \partial_r) \, \chi_{\{ r > 2^{-k_3+m}\}} L(\phi_{\leq k_3-10}, \phi_{k_3}).
 \end{align*}
 We now estimate in detail the $Z_k^+$ norm of the $(HH)_k$ part of~\eqref{equ:phi_bar_trichotomy} when $(\partial_t + \partial_r)$ falls onto $\phi_{k_3}$, the other case being similar. Thus, we end up having to estimate the $Z_k^+$ norm of the following schematic expression
 \begin{equation} \label{equ:fungibility_Z_k_norm_schematic} 
  \begin{aligned}
   &\sum_{k_1 > k+O(1)} \sum_{k_2 = k_1+O(1)} \sum_{k_3 \leq k-m} P_k \Bigl( P_{k_1} \Bigl( \sum_{k_3 \leq k-m} U_{\leq k_3-1} (\partial_t + \partial_r) \phi_{k_3} \Bigr) P_{k_2} \bigl( U_{\leq k-10} 
   \bar{\phi}_k \bigr) \Bigr) \\
   &\quad + \sum_{k_1 > k+O(1)} \sum_{k_2 = k_1+O(1)} \sum_{k-m < k_3 \leq k-10} P_k \Bigl( P_{k_1} \Bigl( U_{\leq k_3-1} \chi_{\{ r > 2^{-k+m} \}} (\partial_t + \partial_r) \phi_{k_3} \Bigr) P_{k_2} \bigl( U_{\leq k-10} \bar{\phi}_k \bigr) \Bigr)
  \end{aligned}
 \end{equation}
 and we begin with the first summand in ~\eqref{equ:fungibility_Z_k_norm_schematic}. We distinguish the cases $r \leq 2^{-k}$ and $r > 2^{-k}$. When $r \leq 2^{-k}$ we can place the whole output into the $(-)$ component of the $Z_k^+$ norm using just Bernstein and Strichartz estimates. Indeed, for the weighted $L^q_t L^p_x$ norm of the $(-)$ component we have 
 \begin{align*}
  &2^{(\frac{1}{q} + \frac{2}{p}-1) k} 2^{\lambda k} \sum_{\ell \leq -k} \sum_{k_1 > k+O(1)} \sum_{k_2 = k_1 + O(1)} \sum_{k_3 \leq k-m} \biggl\| \chi_{\{ r \sim 2^\ell\}} r^{\lambda} P_k \Bigl( P_{k_1} \Bigl( \sum_{k_3 \leq k-m} U_{\leq k_3-1} (\partial_t + \partial_r) \phi_{k_3} \Bigr) P_{k_2} \bigl( U_{\leq k-10} 
   \bar{\phi}_k \bigr) \Bigr) \biggr\|_{L^q_t L^p_x[I_j]} \\
  &\lesssim 2^{(\frac{1}{q} + \frac{2}{p}-1) k} \sum_{k_2 > k+O(1)} \sum_{k_3 \leq k-m} \bigl\| (\partial_t + \partial_r) \phi_{k_3} \bigr\|_{L^\infty_t L^p_x[I_j]}  \bigl\| P_{k_2} \bigl( U_{\leq k-10} \bar{\phi}_k \bigr) \bigr\|_{L^q_t L^\infty_x[I_j]} \\
  &\lesssim \biggl( \sum_{k_3 \leq k-m} 2^{(1-\frac{2}{p}) (k_3-k)} \|\phi_{k_3}\|_{S_{k_3}[I_j]} \biggr) \sum_{k_2 > k+O(1)} 2^{\frac{1}{q}(k-k_2)} \bigl\| P_{k_2} \bigl( U_{\leq k-10} \bar{\phi}_k \bigr) \bigr\|_{S_{k_2}[I_j]} \\
  &\lesssim 2^{-(1-\frac{2}{p}) m} \|\phi\|_{S[I]} \sum_{k_2 > k+O(1)} 2^{\frac{1}{q}(k-k_2)} \bigl\| P_{k_2} \bigl( U_{\leq k-10} \bar{\phi}_k \bigr) \bigr\|_{S_{k_2}[I_j]}.
 \end{align*}
 Choosing $m \gg 1$ sufficiently large depending on the size of the $S[I]$ norm of $\phi$ to compensate the factor of $\|\phi\|_{S[I]}$ on the right-hand side of the last line, we may then easily square-sum over $k\in\bZ$ and invoke the estimate~\eqref{equ:divisibility_bound_gauged_phi_bar} to obtain a final bound just in terms of the energy $E$. Similarly we bound the $Y_k$ norm of the $(-)$ component. 
 
 Let us now consider the case when $r > 2^{-k}$. Here we distinguish whether $(\partial_t + \partial_r) \phi_{k_3}$ is of $(+)$ type or of $(-)$ type. If it is of $(+)$ type, we place the whole expression into the $(+)$ component of the $Z_k^+$ norm. To this end we consider the $L^2_{t+r} L^\infty_{t-r}$ part of the $(+)$ component. For $0 < \lambda < \frac{1}{2}$ we have
 \begin{align*}
  &2^{-(\frac{1}{2}-\lambda)k} \biggl\| r^\lambda \chi_{\{ r > 2^{-k} \}} \sum_{k_1 > k+O(1)} \sum_{k_2 = k_1+O(1)} \sum_{k_3 \leq k-m} P_k \Bigl( P_{k_1} \Bigl( \sum_{k_3 \leq k-m} U_{\leq k_3-1} (\partial_t + \partial_r) \phi_{k_3}^{(+)} \Bigr) P_{k_2} \bigl( U_{\leq k-10} \bar{\phi}_k \bigr) \Bigr) \biggr\|_{L^2_{t+r} L^\infty_{t-r}[I_j]} \\
  &\lesssim 2^{-(\frac{1}{2}-\lambda)k} \sum_{k_2 > k+O(1)} \sum_{k_3 \leq k-m} \bigl\| r^\lambda (\partial_t + \partial_r) \phi_{k_3}^{(+)} \bigr\|_{L^2_{t+r} L^\infty_{t-r}[I_j]} \bigl\| P_{k_2} \bigl( U_{\leq k-10} \bar{\phi}_k \bigr) \bigr\|_{L^\infty_t L^\infty_x} \\
  &\lesssim \biggl( \sum_{k_3 \leq k-m} 2^{-(\frac{1}{2}-\lambda)(k-k_3)} \|\phi_{k_3}\|_{S_{k_3}[I]} \biggr) \sum_{k_2 \in \bZ} \bigl\| P_{k_2} \bigl( U_{\leq k-10} \bar{\phi}_k \bigr) \bigr\|_{S_{k_2}[I_j]} \\
  &\lesssim 2^{-(\frac{1}{2}-\lambda) m} \|\phi\|_{S[I]} \sum_{k_2 \in \bZ} \bigl\| P_{k_2} \bigl( U_{\leq k-10} \bar{\phi}_k \bigr) \bigr\|_{S_{k_2}[I_j]}
 \end{align*}
 and upon choosing $m \gg 1$ sufficiently large, we may easily square-sum in $k\in\bZ$ and obtain a final bound just in terms of the energy $E$ thanks to~\eqref{equ:divisibility_bound_gauged_phi_bar}. For $\lambda = \frac{1}{2}$ we unfortunately do not straightaway have summability in $k_3 \leq k-m$. In this case we split the weight $r^{\frac{1}{2}} = r^{\frac{1}{2}-\delta} r^{\delta}$ for some small $\delta > 0$ and absorb $r^\delta$ into the high-frequency factor $P_{k_2} \bigl( U_{\leq k-10} \bar{\phi}_k \bigr)$ via Strauss' improved Sobolev embedding in the radial case 
 \[
  \bigl\| \chi_{\{ r > 2^{-k}\}} r^\delta P_{k_2} \bigl( U_{\leq k-10} \bar{\phi}_k \bigr) \bigr\|_{L^\infty_t L^\infty_x[I_j]} \lesssim 2^{(\frac{1}{2}-\delta) k} 2^{-\frac{1}{2} k_2} \bigl\| P_{k_2} \bigl( U_{\leq k-10} \bar{\phi}_k \bigr\|_{L^\infty_t L^2_x[I_j]}.
 \]
 Then we can again sum over $k_3 \leq k-m$ and gain a smallness factor $2^{-\delta m}$. The $Y_k$ norm of the $(-)$ component of the $Z_k^+$ norm is easier to bound. It therefore remains to consider the case when $(\partial_t + \partial_r) \phi_{k_3}$ is of $(-)$ type. Here we try to place the whole expression into the $(-)$ component of the $Z_k^+$ norm. We start with the weighted $L^q_t L^p_x$ part
 \begin{align*}
  2^{(\frac{1}{q} + \frac{2}{p} - 1)k} 2^{\lambda k} \sum_{\ell > - k} \biggl\| \chi_{\{r \sim 2^\ell\}} r^\lambda \sum_{k_1 > k+O(1)} \sum_{k_2 = k+O(1)} \sum_{k_3 \leq k-m} P_k \Bigl( P_{k_1} \Bigl( \sum_{k_3 \leq k-m} U_{\leq k_3-1} (\partial_t + \partial_r) \phi_{k_3}^{(-)} \Bigr) P_{k_2} \bigl( U_{\leq k-10} \bar{\phi}_k \bigr) \Bigr) \biggr\|_{L^q_t L^p_x[I_j]},
 \end{align*}
 where we aim to estimate $(\partial_t + \partial_r) \phi_{k_3}^{(-)}$ in $L^q_t L^\infty_x$, while placing the high-frequency factor $P_{k_2} \bigl( U_{\leq k-10} \bar{\phi}_k \bigr)$ into $L^\infty_t L^p_x$. In order to ensure summability over the low frequencies $k_3 \leq k-m$, we exploit that we can absorb a weight $r^{\frac{1}{2} - \frac{1}{p}}$ into $P_{k_2} \bigl( U_{\leq k-10} \bar{\phi}_k \bigr)$ via the estimate 
 \[
  \bigl\| r^{\frac{1}{2}-\frac{1}{p}} P_{k_2} f \bigr\|_{L^\infty_t L^p_x} \lesssim 2^{-(\frac{1}{2} + \frac{1}{p}) k_2} \|\nabla_x P_{k_2} f\|_{L^2_x},
 \]
 which follows from interpolating Strauss' improved Sobolev embedding with the trivial energy estimate. This yields the following bound on the weighted $L^q_t L^p_x$ part of the $(-)$ component of the $Z_k^+$ norm
 \begin{align*}
  &\sum_{k_3 \leq k-10} 2^{(\frac{1}{q}-1+\lambda) k} 2^{(1-\lambda+\frac{1}{2}-\frac{1}{p}-\frac{1}{q}) k_3} \| \phi_{k_3} \|_{S_{k_3}[I_j]} \sum_{k_2 > k + O(1)} 2^{\frac{2}{p}k} 2^{-(\frac{1}{2} + \frac{1}{p}) k_2} \bigl\| P_{k_2} \bigl( U_{\leq k-10} \bar{\phi}_k \bigr) \bigr\|_{S_{k_2}[I_j]} \\
  &\lesssim 2^{-(1-\lambda+\frac{1}{2}-\frac{1}{p}-\frac{1}{q}) m} \sum_{k_2 > k + O(1)} 2^{-(\frac{1}{2} + \frac{1}{p}) (k_2-k)} \bigl\| P_{k_2} \bigl( U_{\leq k-10} \bar{\phi}_k \bigr) \bigr\|_{S_{k_2}[I_j]},
 \end{align*}
 which yields a smallness gain $2^{-\delta m}$ on account of the fact that $0 < \lambda < 1$ and $\frac{1}{q} + \frac{1}{p} < \frac{1}{2}$. Then we can easily square-sum this bound over $k \in \bZ$ and obtain a final estimate just in terms of $E$ by invoking~\eqref{equ:divisibility_bound_gauged_phi_bar}. The $Y_k$ norm of the $(-)$ component of the $Z_k^+$ norm can again be treated easily. 
 
 Finally, we have to estimate the $Z_k^+$ norm of the second summand in~\eqref{equ:fungibility_Z_k_norm_schematic}. Here we proceed similarly to the above treatment of the first summand only that we gain smallness using Strauss' improved Sobolev embedding and the restriction $r > 2^{-k+m}$. This finishes the proof of Proposition~\ref{prop:partial_fungibility}.
\end{proof}

Next, we give the proof of Proposition~\ref{prop:finite_S_gives_freq_envelope_bounds} using the weak divisibility of the $S$ norm.
\begin{proof}[Proof of Proposition~\ref{prop:finite_S_gives_freq_envelope_bounds}]
 We may assume that the interval $I$ is of the form $I = [0,T]$ for some $0 < T \leq \infty$. Then we first use Proposition~\ref{prop:partial_fungibility} to partition $I = \cup_{j=1}^N I_j$ into $N \equiv N(\|\phi\|_{S[I]}, E)$ consecutive intervals $I_j = [t_{j-1}, t_j]$ with $t_0 = 0$ and $t_N = T$ such that $\| \phi \|_{S[I_j]} \leq C(E)$ for $j = 1, \ldots, N$, where $C(E) > 0$ is a constant that depends only on the energy $E$ of the wave map $\phi$. For each interval $I_j = [t_{j-1}, t_j]$ we introduce a frequency envelope
 \[
  c_k^{(j)} := \sum_{\ell \in \bZ} 2^{-\sigma |k-\ell|} \big\| P_\ell \phi [t_{j-1}] \big\|_{\dot{H}^1_x \times L^2_x}, \quad k \in \bZ,
 \]
 for some small constant $\sigma > 0$, and show via an iterative bootstrap argument that
 \[
  \| P_k \phi \|_{S_k[I_j]} \lesssim c_k^{(j)},
 \]
 which then implies the assertion of the proposition. To this end we further partition each interval $I_j = \cup_i I_{ji}$ into finitely many consecutive intervals $I_{ji}$ which have suitable divisibility properties. On every interval $I_{ji}$ we now run a bootstrap argument. Starting with $I_{j1}$ we make the bootstrap assumption
 \[
  \| P_k \phi \|_{S_k[I_{j1}]} \leq C c_k^{(j)}
 \]
 for some sufficently large, absolute constant $C>0$ and now show that this implies the improved bound 
 \[
  \| P_k \phi \|_{S_k[I_{j1}]} \leq \frac{C}{2} c_k^{(j)}.
 \]
 Then we continue analogously on all remaining intervals $I_{ji}$. Here we again have to use the modified definition~\eqref{equ:modified_definition_phi_bar_A} of $\overline{\phi}_k$ and the modified definition~\eqref{equ:modified_gauge_transform_definition} of the gauge transformations $U_{\leq h}$ to have an additional soure of smallness at our disposal as in the preceding proof of Proposition~\ref{prop:partial_fungibility}. Then it is easy to show that for any $0 \leq \alpha < 1$,
 \[
  \sum_{k_0 \in \bZ} 2^{\alpha |k-k_0|} \bigl\| P_{k_0} \nabla_{t,x} \bigl( U_{\leq k-10} \overline{\phi}_k(t_j) \bigr) \bigr\|_{L^2_x} \ll C c_k^{(j)}.
 \]
 Moreover, picking the intervals $I_{ji}$ suitably, we may infer by divisibility arguments that for any $0 \leq \alpha < 1$,
 \[
  \sum_{k_0 \in \bZ} 2^{\alpha |k-k_0|} \bigl\| P_{k_0} \Box \big( U_{\leq k-10} \overline{\phi}_k \big) \bigr\|_{L^1_t L^2_x[I_{j1}]} \ll C c_k^{(j)}
 \]
 and then the energy estimate~\eqref{equ:energy_estimate} implies that 
 \[
  \sum_{k_0 \in \bZ} 2^{\alpha |k-k_0|} \bigl\| P_{k_0} \bigl( U_{\leq k-10} \overline{\phi}_k \bigr) \bigr\|_{S_{k_0}[I_{j1}]} \ll C c_k^{(j)}.
 \]
 Finally, we argue similarly as in the preceding proof of Proposition~\ref{prop:partial_fungibility} to pass from the last bound back to~$\phi_k$ and recover the improved bound
 \[
  \| P_k \phi \|_{S_k[I_{j1}]} \leq \frac{C}{2} c_{k}^{(j)}.
 \]
 This finishes the proof of Proposition~\ref{prop:finite_S_gives_freq_envelope_bounds}.
\end{proof}

We are now in a position to provide the proof of the main result of this section.

\begin{proof}[Proof of Proposition~\ref{prop:breakdown_criterion}]
 Suppose that $I \neq \bR$. Since we have $\| \phi \|_{S[I]} < \infty$ by assumption, Proposition~\ref{prop:finite_S_gives_freq_envelope_bounds} yields frequency envelope bounds for the evolution of $\phi$ on its maximal time interval of existence $I$,
 \[
  \| P_k \phi \|_{S_k[I]} \leq C c_k,
 \]
 where $\{ c_k \}_{k \in \bZ}$ is a frequency envelope covering the initial data $\phi[0]$. Then it is a standard argument to infer that a subcritical norm of $\phi[t]$ must stay finite on $I$, i.e. 
 \[
  \sup_{t \in I} \| \phi[t] \|_{H^s_x \times H^{s-1}_x} < \infty
 \]
 for some $s > 1$. But then the local well-posedness theory~\cite{KlMa1, KlMa2, KlMa3, Klainerman_Selberg1, Klainerman_Selberg2} implies that the evolution of $\phi$ extends smoothly beyond the time interval $I$, contradicting the maximality of $I$. Thus, we must have $I = \bR$ and it remains to prove the scattering assertion. To this end we first note that in the gauged wave equation
 \[
  \Box \big( U_{\leq k-10} \overline{\phi}_k \big) \equiv F_k, \quad k \in \bZ,
 \]
 we have that $\| F_k \|_{L^1_t L^2_x[\bR]} \lesssim c_k$. Now for $k \ll -1$ or $k \gg 1$ we already know from the a priori bounds that $\| \nabla_{t,x} \phi_k \|_{L^\infty_t L^2_x} \lesssim c_k$ has very small norm, so it suffices to consider $k = O(1)$. Picking a sufficently large time $T > 0$ such that 
 \[
  \| F_k \|_{L^1_t L^2_x[ [T, \infty) )} \ll 1,
 \]
 we then obtain that for $k = O(1)$,
 \[
  U_{\leq k-10} \overline{\phi}_k(t) = S(t-T)\big( U_{\leq k-10} \overline{\phi}_k[T] \big) + o_{L^\infty_t \dot{H}^1_x([T,\infty))}(1) \quad \text{for } \, t > T.
 \]
 In order to infer scattering, it then suffices to show that 
 \[
  \lim_{t\to\infty} \, \bigl\| \nabla_{t,x} \big( U_{\leq k-10} \overline{\phi}_k(t) - \phi_k(t) \big) \bigr\|_{L^2_x} = 0,
 \]
 for which in turn it suffices to prove that $\lim_{t\to\infty} \big\| U_{\leq k-10}(t) - Id \big\|_{L^\infty_x} = 0$ as well as $\lim_{t\to\infty} \| \phi_k(t) \|_{L^\infty_x} = 0$. On the one hand, it is easy to see that 
 \[
  \lim_{t\to\infty} \, \big\| S(t-T)\big( U_{\leq k-10} \overline{\phi}_k[T] \big) \big\|_{L^\infty_x} = 0,
 \]
 and the localization to frequency $k = O(1)$ also implies that the error $o_{L^\infty_t \dot{H}^1_x([T,\infty))}(1)$ goes to zero in $L^\infty_x$. This then implies that $\overline{\phi}_k(t)$ converges to zero in $L^\infty_x$. To pass from here to $\phi_k$ again requires modifying the definition of $\overline{\phi}_k$ as in the previous proof of Proposition~\ref{prop:partial_fungibility} so that the difference $\overline{\phi}_k - \phi_k$ becomes arbitrarily small. Having uniform smallness for $\phi_k$ in $L^\infty_x$, we then obtain that
 \[
  \lim_{t\to\infty} \big\| U_{\leq k-10}(t) - Id \big\|_{L^\infty_x} = 0,
 \]
 which implies 
 \[
  \overline{\phi}_k(t) = S(t-T)\big( U_{\leq k-10} \overline{\phi}_k[T] \big) + o_{L^\infty_t \dot{H}^1_x[[T,\infty))}(1) \quad \text{for } \, t > T.
 \]
 Then the modified definition of $\bar{\phi}_k$ yields that 
 \[
  \phi_k(t) = S(t-T)\big( U_{\leq k-10} \overline{\phi}_k[T] \big) + o_{L^\infty_t \dot{H}^1_x[[T,\infty))}(1) \quad \text{for } \, t > T,
 \]
 which gives scattering.
\end{proof}

\section{Concentration compactness step} \label{sec:Concentration_Compactness_Step}

In this section we begin with the actual proof of Theorem~\ref{thm:main_theorem}. We recall that our goal is to show that there exists a non-decreasing function $K \colon [0,\infty) \to [0, \infty)$ with the following property: Let $(\phi_0, \phi_1) \colon \bR^2 \to T\bS^m$ be radially symmetric, classical initial data of energy $E$. Then there exists a unique, classical, global wave map $\phi \colon \bR^{1+2} \to \bS^{m}$ with initial data $\phi[0] = (\phi_0, \phi_1)$ satisfying the a priori bound 
\[
 \| \phi \|_{S} \leq K(E).
\]
Once we have established this a priori bound, the scattering assertion of Theorem~\ref{thm:main_theorem} is an immediate consequence of Proposition~\ref{prop:breakdown_criterion}.

\medskip

We argue by contradiction and assume that Theorem~\ref{thm:main_theorem} fails. Then the existence of the function $K(\cdot)$ yielding the a priori bounds must fail at some finite energy level. Correspondingly, the following set of energies must be non-empty
\[
 {\mathcal E} := \biggl\{ E \, : \, \sup_{\{ \phi \, : \, E[\phi] \leq E \}} \|\phi\|_{S[I]} = +\infty \biggr\},
\]
where the supremum is taken over all radial, classical wave maps $\phi \colon I \times \bR^2 \to \bS^m$ defined on some time interval $I$ and with energy $E[\phi] \leq E$. By the small energy global regularity result from Theorem~\ref{thm:small_energy_global}, the infimum of the set ${\mathcal E}$ has to be strictly positive
\[
 E_{crit} := \inf {\mathcal E} > 0.
\]
Thus, we may pick a sequence of radially symmetric, classical wave maps $\phi^n \colon I^n \times \bR^2 \to \bS^{m}$, $n \geq 1$, with maximal intervals of existence $I^n$ such that
\[
 \lim_{n\to\infty} E[\phi^n] = E_{crit}, \quad \lim_{n\to\infty} \| \phi^n \|_{S[I^n]} = + \infty.
\]
In the following we call such a sequence of wave maps \emph{essentially singular}. Moreover, we denote the associated essentially singular sequence of initial data $\{ \phi^n[0] \}_{n \geq 1}$ by $\{ (\phi_0^n, \phi_1^n) \}_{n \geq 1}$. Here, the subscripts in $(\phi_0^n, \phi_1^n)$ should not to be confused with frequency localizations, but this will always be clear from the context.

\medskip 

Our goal is now to rule out the existence of such an essentially singular sequence of wave maps $\{ \phi^n \}_{n \geq 1}$, hence proving Theorem~\ref{thm:main_theorem}. To this end we follow the general philosophy of the concentration compactness/rigidity method introduced by Kenig-Merle~\cite{Kenig_Merle1, Kenig_Merle2}, but more precisely we shall follow the implementation of this strategy for energy critical wave maps into the hyperbolic plane as in~\cite{KS} as well as for the energy critical Maxwell-Klein-Gordon equation as in~\cite{KL}. In this section we carry out a ``twisted'' Bahouri-G\'erard type profile decomposition that takes into account the strong low-high interactions in the wave maps nonlinearity. It enables us to extract from the essentially singular sequence a non-trivial minimal blowup solution to (WM) whose orbit satisfies a strong compactness property. Then we exclude the existence of such a minimal blowup solution in the rigidity argument of the next, and final, section of this paper.

\medskip 

The first step consists in decomposing the essentially singular sequence of data $\{ \phi^n[0] \}_{n \geq 1}$ into frequency atoms using the M\'etivier-Schochet procedure~\cite{Metivier-Schochet} as in Bahouri-G\'erard~\cite{Bahouri_Gerard}. Roughly speaking, the basic idea then goes as follows. Ultimately, we would like to conclude that upon passing to a subsequence, if necessary, the essentially singular sequence of data $\{ \phi^n[0] \}_{n \geq 1}$ consists of exactly one frequency atom wich in turn consists of exactly one concentration profile (to be defined precisely in Subsection~\ref{subsec:adding_in_first_atom}) of asymptotic energy $E_{crit}$. In this scenario, the sequence $\{ \phi^n[0] \}_{n \geq 1}$ has sufficient compactness properties that allow us to pass to a certain limit whose wave maps evolution will be the desired minimal blowup solution to (WM) as detailed in Subsection~\ref{subsec:conclusion}. In order to rule out all other possible scenarios, we seek to prove uniform in~$n$, finite, global $S$ norm bounds on the sequence of wave maps evolutions $\{ \phi^n \}_{n \geq 1}$, which would contradict that the sequence is essentially singular. To this end we first achieve control over the wave maps evolutions of certain low frequency truncations of the essentially singular sequence of data $\{ \phi^n[0] \}_{n \geq 1}$. Using a finite inductive procedure over the increasing size of the frequency supports of these low frequency truncations, we then conclude uniform in $n$, finite, global $S$ norm bounds on the actual essentially singular sequence $\{ \phi^n \}_{n \geq 1}$. This inductive procedure over the increasing size of the frequency supports also enables us to disentangle the strong low-high frequency interactions in the wave maps nonlinearity.

\subsection{Decomposition into frequency atoms}

We now turn to the details of the decomposition of the essentially singular sequence of data $\{ \phi^n[0] \}_{n \geq 1}$ into frequency atoms. Here we follow relatively closely Section~9.1 and Section~9.2 in~\cite{KS} as well as Section~7.2 in~\cite{KL}, which in turn partially mimic Section~III.1 in Bahouri-G\'erard~\cite{Bahouri_Gerard}. First, we need to introduce some terminology from~\cite{Bahouri_Gerard}.

\medskip 

We call a sequence of positive numbers $\{ \lambda_n \}_{n \geq 1}$ a {\it scale}. Two scales $\{ \lambda_n^a \}_{n \geq 1}$ and $\{ \lambda_n^b \}_{n \geq 1}$ are {\it orthogonal} if
\[
 \lim_{n\to\infty} \, \frac{\lambda_n^a}{\lambda_n^b} + \frac{\lambda_n^b}{\lambda_n^a} = +\infty.
\]
Let $\{ (f^n, g^n) \}_{n \geq 1}$ be a bounded sequence of functions in $\dot{H}^1_x(\bR^2) \times L^2_x(\bR^2)$ and let $\{ \lambda_n \}_{n \geq 1}$ be a scale. Then the sequence $\{ (f^n, g^n) \}_{n \geq 1}$ is called {\it $\lambda_n$-oscillatory} if 
\[
 \lim_{R \to \infty} \limsup_{n \to \infty} \, \bigg( \int_{ \{ \lambda_n |\xi| \leq \frac{1}{R} \} } | \widehat{\nabla_x f^n}(\xi) |^2 + | \widehat{g^n}(\xi) |^2 \, d\xi + \int_{ \{ \lambda^n |\xi| \geq R \} } | \widehat{\nabla_x f^n}(\xi)|^2 + |\widehat{g^n}(\xi)|^2 \, d\xi \bigg) = 0
\]
and we say that the sequence $\{ (f^n, g^n) \}_{n \geq 1}$ is {\it $\lambda_n$-singular} if for all $0 < a < b$,
\[
 \lim_{n\to\infty} \int_{ \{ a \leq \lambda_n |\xi| \leq b \} } | \widehat{ \nabla_x f^n}(\xi) |^2 + |\widehat{g^n}(\xi)|^2 \, d\xi = 0.
\]

\medskip 

Using the M\'etivier-Schochet procedure~\cite{Metivier-Schochet} as implemented in Bahouri-G\'erard~\cite{Bahouri_Gerard}, we start off with a decomposition
\[
 \phi^n[0] = \sum_{a=1}^\Lambda \tilde{\phi}^{na}[0] + \tilde{\phi}^{n\Lambda}[0]
\]
of the essentially singular sequence of data $\{ \phi^n[0] \}_{n \geq 1}$ into $\lambda_n^a$-oscillatory \emph{frequency atoms} $\tilde{\phi}^{na}[0]$, $1 \leq a \leq \Lambda$, for pairwise orthogonal frequency scales $\{ \lambda_n^a\}_{n \geq 1}$ and into a remainder term $\tilde{\phi}^{n \Lambda}[0]$, which is $\lambda_n^a$-singular for $1 \leq a \leq \Lambda$ and has the smallness property 
\[
 \lim_{\Lambda\to\infty} \limsup_{n\to\infty} \| \tilde{\phi}^{n\Lambda}[0] \|_{\dot{B}^1_{2,\infty} \times \dot{B}^0_{2,\infty}} = 0.
\]

Next we achieve a sharp frequency localization of the atoms by picking a sequence $R_n \to \infty$ growing sufficently slowly and by then setting
\[
 \phi^{na}[0] := P_{[\mu_n^a - \log R_n, \mu_n^a + \log R_n]} \phi^n [0] \quad \text{for } a = 1, \ldots, \Lambda
\]
and 
\[
 \phi^{n \Lambda}[0] := P_{\cap_{a=1}^\Lambda [\mu_n^a - \log R_n, \mu_n^a + \log R_n]^c} \phi^n [0],
\]
where we use the notation $\mu_n^a = - \log ( \lambda_n^a )$. Then we have the new frequency atom decomposition
\begin{equation} \label{equ:frequency_atom_decomposition}
 \phi^n[0] = \sum_{a=1}^\Lambda \phi^{na}[0] + \phi^{n\Lambda}[0]
\end{equation}
with the same properties as above, but now with the additional sharp frequency localization property of the atoms. Moreover, we may assume that the atoms are ordered in terms of the increasing size of their frequency support scales $(\lambda_n^a)^{-1}$. 

\medskip 

As described above, ultimately we would like to conclude that there is exactly one atom in the decomposition~\eqref{equ:frequency_atom_decomposition}, i.e. $\Lambda = 1$, which is of asymptotic energy $E_{crit}$. If this is the case, we proceed directly as in Subsection~\ref{subsec:adding_in_first_atom}. Otherwise, we now conclude via a finite inductive procedure that the sequence $\{ \phi^n \}_{n \geq 1}$ cannot be essentially singular. To this end we fix an integer $\Lambda_0$ sufficiently large such that upon passing to a subsequence, if necessary,
\[
 \sum_{a > \Lambda_0} \limsup_{n\to\infty} \| \phi^{na}[0] \|_{\dot{H}^1_x \times L^2_x}^2 \leq \varepsilon_0,
\]
where $\varepsilon_0 > 0$ is a sufficently small constant that plays the role of a perturbative threshold in the key bootstrap argument in Proposition~\ref{prop:bootstrap_bounds_to_next_level_in_evolving_lowest_frequency_nonatomic_part} in Subsection~\ref{subsec:lowest_freq_non_atomic}. In particular, $\varepsilon_0$ will be chosen sufficently small depending only on the size of $E_{crit}$ and it will be chosen to be less than the small energy global regularity threshold from Theorem~\ref{thm:small_energy_global}. Then we observe that due to the sharp frequency localizations of the atoms $\phi^{na}[0]$, $1 \leq a \leq \Lambda_0$, the remainder term $\phi^{n \Lambda_0}[0]$ gets split into $\Lambda_0 + 1$ ``frequency shells''
\begin{equation} \label{equ:decomposition_remainder_term_into_frequency_shells}
 \phi^{n \Lambda_0}[0] = \phi^{n \Lambda_0^{(0)}}[0] + \phi^{n \Lambda_0^{(1)}}[0] + \ldots + \phi^{n \Lambda_0^{(\Lambda_0)}}[0],
\end{equation}
where $\phi^{n \Lambda_0^{(0)}}[0]$ shall denote the lowest frequency component.

\medskip

Our first step now consists in showing that the lowest frequency ``non-atomic'' component $\phi^{n \Lambda_0^{(0)}}[0]$ can be globally evolved and satisfies finite $S$ norm bounds just in terms of $E_{crit}$ uniformly for all sufficiently large $n$. Since the component $\phi^{n \Lambda_0^{(0)}}[0]$ may still have large energy, in order to be able to infer these $S$ norm bounds by bootstrap, we split $\phi^{n \Lambda_0^{(0)}}[0]$ into finitely many pieces
\[
 \phi^{n \Lambda_0^{(0)}}[0] = \sum_{\ell = 1}^{C_1} P_{J_\ell} \phi^{n \Lambda_0^{(0)}}[0]
\]
by means of frequency localization to consecutive frequency intervals $J_\ell = [a_\ell, b_\ell]$ such that
\[
 (-\infty, \mu_n^1 - \log R_n ] = \cup_{\ell=1}^{C_1} J_\ell 
\]
and such that for $\ell = 1, \ldots, C_1$,
\[
 \| P_{J_\ell} \phi^{n \Lambda_0^{(0)}}[0] \|^2_{\dot{H}^1_x \times L^2_x} \lesssim \varepsilon_0.
\]
Here we recall that $(-\infty, \mu_n^1 - \log R_n]$ is the (dyadic) frequency support of $\phi^{n \Lambda_0^{(0)}}[0]$. The number $C_1$ of such frequency intervals is of the order $O( \frac{E_{crit}}{\varepsilon_0} )$ since $\| \phi^{n \Lambda_0^{(0)}}[0] \|_{\dot{H}^1_x \times L^2_x}^2 \lesssim E_{crit}$.

We now intend to inductively show that if
\[
 \sum_{\ell \leq L} P_{J_\ell} \phi^{n \Lambda_0^{(0)}}[0] = P_{\leq b_{L}} \phi^{n \Lambda_0^{(0)}}[0]
\]
can be globally evolved with uniform finite $S$ norm bounds in terms of $E_{crit}$ for some $L \geq 1$ (for all sufficently large~$n$), then we may also conclude this for 
\[
 \sum_{\ell \leq L+1} P_{J_\ell} \phi^{n \Lambda_0^{(0)}}[0] = P_{\leq b_{L+1}} \phi^{n \Lambda_0^{(0)}}[0],
\]
and the induction start for the first piece $P_{J_1} \phi^{n \Lambda_0^{(0)}}[0]$ is provided by the small energy global regularity result from Theorem~\ref{thm:small_energy_global}.

\medskip 

However, up to this point we have totally ignored that the frequency truncations $P_{\leq b_L} \phi^{n \Lambda_0^{(0)}}[0] \equiv P_{\leq b_L} \phi^n[0]$ are not ``geometric'' in the sense that they are not actual maps $\bR^2 \to T\bS^m$ into the unit sphere and therefore do not constitute suitable initial data for the wave maps equation (WM). To overcome this issue we just project the frequency truncations back to the sphere, using the normal projection operator. As we shall see next, this operation is well-defined and the frequency localization properties are approximately preserved up to exponential tails if around the frequency cut-offs a certain Besov norm smallness condition is satisfied.

\medskip 

We begin with a general lemma that shows that the frequency truncations $P_{\leq \alpha_n} \phi_0^n$ of maps $\phi_0^n \colon \bR^2 \to \bS^m$, $n \geq 1$, remain close to the sphere if around the frequency cut-offs $\{ \alpha_n \}_{n \geq 1}$ a certain Besov norm smallness holds.

\begin{lem} \label{lem:frequency_truncation_stays_close_to_sphere}
 Let $\{ \alpha_n \}_{n\geq1}$ be a sequence of real numbers. Given $\delta_1 > 0$ one can choose $\delta_0 \ll \delta_1$ sufficiently small (depending on $E_{crit}$) so that if we have 
 \begin{equation} \label{equ:geometric_truncation_besov_smallness}
  \limsup_{n\to\infty} \, \bigl\| P_{[\alpha_n, \alpha_n + M]} \phi_0^n \bigr\|_{\dot{B}^1_{2,\infty}} \leq \delta_0
 \end{equation}
 for $M \sim \log ( (1+E_{crit}) \delta_1^{-1} )$, then it holds that 
 \[
  \limsup_{n\to\infty} \, \bigl\| 1 - (P_{\leq \alpha_n} \phi_0^n)^\dagger (P_{\leq \alpha_n} \phi_0^n) \bigr\|_{L^\infty_x} \lesssim \delta_1.
 \]
\end{lem}
\begin{proof}
 By definition we have 
 \[
  ( P_{\leq \alpha_n} \phi_0^n + P_{> \alpha_n} \phi_0^n )^\dagger ( P_{\leq \alpha_n} \phi_0^n + P_{> \alpha_n} \phi_0^n ) = 1
 \]
 and therefore
 \[
  1 - (P_{\leq \alpha_n} \phi_0^n)^\dagger (P_{\leq \alpha_n} \phi_0^n) = 2 P_{\leq \alpha_n + 10} \bigl( (P_{\leq \alpha_n} \phi_0^n)^\dagger (P_{> \alpha_n} \phi_0^n) \bigr) + P_{\leq \alpha_n + 10} \bigl( (P_{>\alpha_n} \phi_0^n)^\dagger (P_{>\alpha_n} \phi_0^n) \bigr).
 \]
 Then we use Bernstein's inequality and \eqref{equ:geometric_truncation_besov_smallness} to bound  
 \begin{align*}
  \limsup_{n\to\infty} \, \bigl\| P_{\leq \alpha_n + 10} \bigl( (P_{\leq \alpha_n} \phi_0^n)^\dagger (P_{> \alpha_n} \phi_0^n) \bigr) \bigr\|_{L^\infty_x} &\lesssim \limsup_{n\to\infty} \| P_{\leq \alpha_n} \phi_0^n \|_{L^\infty_x} \sum_{k=\alpha_n}^{\alpha_n+15} 2^k \| P_k \phi_0^n \|_{L^2_x} \\
  &\lesssim \limsup_{n\to\infty} \sum_{k=\alpha_n}^{\alpha_n+15} \| \nabla_x P_k \phi_0^n \|_{L^2_x} \\
  &\lesssim \delta_0 \lesssim \delta_1.
 \end{align*}
 Similarly, we obtain
 \begin{align*}
  \limsup_{n\to\infty} \, \bigl\| P_{\leq \alpha_n + 10} \bigl( (P_{>\alpha_n} \phi_0^n)^\dagger (P_{>\alpha_n} \phi_0^n) \bigr) \bigr\|_{L^\infty_x} &\lesssim \limsup_{n\to\infty} \sum_{k \geq \alpha_n} 2^{2 \alpha_n} \| P_k \phi_0^n \|_{L^2_x}^2 \\
  &\lesssim \limsup_{n\to\infty} \sum_{k=\alpha_n}^{\alpha_n+M} 2^{2 \alpha_n - 2 k} \|\nabla_x P_k \phi_0^n\|_{L^2_x}^2 + \sum_{k > \alpha_n+M} 2^{-2M} \|\nabla_x P_k \phi_0^n\|_{L^2_x}^2 \\
  &\lesssim M \delta_0 + 2^{-2M} E_{crit} \\
  &\lesssim \delta_1,
 \end{align*}
 where in the last step we use that $M \sim \log( (1+E_{crit}) \delta_1^{-1} )$ and choose $\delta_0 \lesssim \delta_1 \log^{-1} ((1+E_{crit}) \delta_1^{-1})$.
\end{proof}

For suitably chosen frequency cut-offs $\alpha_n$, the frequency truncations $P_{\leq \alpha_n} \phi_0^n$ therefore stay close to the sphere. In order to recover exact maps into the sphere, we then simply project the frequency truncations $P_{\leq \alpha_n} \phi_0^n$ back to the sphere, using the normal projection operator $\Pi$. Hence, we set 
\begin{equation} \label{equ:project_to_sphere}
 \Pi_{\leq \alpha_n} \phi_0^n := \Pi \bigl( P_{\leq \alpha_n} \phi_0^n \bigr) := \frac{P_{\leq \alpha_n} \phi_0^n}{|P_{\leq \alpha_n} \phi_0^n|}.
\end{equation}
For the associated frequency truncations $P_{\leq \alpha_n} \phi_1^n$ of the time derivatives $\phi_1^n$ in our essentially singular sequence of data $\phi^n[0] = (\phi_0^n, \phi_1^n)$, we just use linear orthogonal projection onto the respective fiber of the tangent space and map $P_{\leq \alpha_n} \phi_1^n$ to
\begin{equation} \label{equ:project_to_sphere_for_time_derivative}
 \Pi_{\leq \alpha_n} \phi_1^n := P_{\leq \alpha_n} \phi_1^n - \langle \Pi_{\leq \alpha_n} \phi_0^n, P_{\leq \alpha_n} \phi_1^n \rangle \Pi_{\leq \alpha_n} \phi_0^n \in T_{\Pi_{\leq \alpha_n} \phi_0^n} \bS^{m}.
\end{equation}
Thus, after frequency truncation of the data $\phi^n[0] = (\phi_0^n, \phi_1^n) \colon \bR^2 \to T \bS^{m}$ to dyadic frequencies less than $\alpha_n$, we may recover the data pair
\[
 \Pi_{\leq \alpha_n} \phi^n[0] := (\Pi_{\leq \alpha_n} \phi_0^n, \Pi_{\leq \alpha_n} \phi_1^n) \colon \bR^2 \to T \bS^{m}.
\]
In the next proposition we prove that this operation approximately preserves the frequency localization if the data satisfy a Besov norm smallness condition around the frequency cut-offs.
\begin{prop} \label{prop:data_approximate_preservation_of_frequency_localization}
 Let $\{ \alpha_n \}_{n \geq 1}$ be a sequence of real numbers. Given $\delta_1 > 0$ one can choose $\delta_0 \ll \delta_1$ sufficently small depending on the size of $E_{crit}$ so that if we have 
 \begin{equation} \label{equ:small_besov_error_around_frequency_cutoff}
  \limsup_{n\to\infty} \, \bigl\| P_{[\alpha_n-M, \alpha_n+M]} \phi^n[0] \bigr\|_{\dot{B}^1_{2,\infty} \times \dot{B}^0_{2,\infty}} \leq \delta_0
 \end{equation}
 for $M \sim \log ( (1+E_{crit}) \delta_1^{-1} )$, then for all sufficently large $n$ it holds that 
 \begin{equation} \label{equ:data_approximate_preservation_of_frequency_localization}
  \bigl\| P_k \bigl( P_{\leq \alpha_n} \phi^n[0] - \Pi_{\leq \alpha_n} \phi^n[0] \bigr) \bigr\|_{\dot{H}^1_x \times L^2_x} \lesssim \delta_1 2^{-|k-\alpha_n|} \quad \text{for } k \in \bZ.
 \end{equation}
\end{prop}
\begin{proof}
 We begin with the proof of \eqref{equ:data_approximate_preservation_of_frequency_localization} for the component $\Pi_{\leq \alpha_n} \phi_0^n$ for the low frequencies $k \leq \alpha_n$. Using that $\phi_0^n = \Pi(\phi_0^n)$, we may write
 \[
  \phi_0^n - \Pi ( P_{\leq \alpha_n} \phi_0^n ) = \int_{\alpha_n}^\infty \frac{d}{dh} \bigl( \Pi (P_{\leq h} \phi_0^n) \bigr) \, dh = \int_{\alpha_n}^\infty (P_h \phi_0^n) \Pi'( P_{\leq h} \phi_0^n ) \, dh.
 \]
 Here a somewhat delicate point arises, namely that the expression $ \Pi (P_{\leq h} \phi_0^n)$ is not necessarily defined for all $h\geq \alpha_n$. In order to deal with this, we can extend the operator $\Pi$ smoothly beyond a tubular neighborhood of $\bS^m$, but no longer necessarily taking values in $\bS^m$, but merely in $\R^{m+1}$. 
 Thus, we obtain from Bernstein's inequality that
 \begin{align*}
  &\bigl\| P_k \nabla_x \bigl( P_{\leq \alpha_n} \phi_0^n - \Pi_{\leq \alpha_n} \phi_0^n \bigr) \bigr\|_{L^2_x} \\
  &\lesssim \Bigl\| P_k \int_{\alpha_n}^\infty (\nabla_x P_h \phi_0^n) \Pi'(P_{\leq h} \phi_0^n) \, dh \Bigr\|_{L^2_x} + \Bigl\| P_k \int_{\alpha_n}^\infty (P_h \phi_0^n) (\nabla_x P_{\leq h} \phi_0^n) \Pi''(P_{\leq h} \phi_0^n) \, dh \Bigr\|_{L^2_x} \\
  &\lesssim 2^k \int_{\alpha_n}^\infty \| \nabla_x P_h \phi_0^n \|_{L^2_x} 2^{-h} \| \nabla_x P_{>h-10} \Pi'( P_{\leq h} \phi_0^n ) \|_{L^2_x} \, dh \\
  &\quad + 2^k \int_{\alpha_n}^\infty 2^{-h} \| \nabla_x P_h \phi_0^n \|_{L^2_x} \|\nabla_x P_{\leq h} \phi_0^n\|_{L^2_x} \| \Pi''(P_{\leq h} \phi_0^n) \|_{L^\infty_x} \, dh \\
  &\lesssim 2^k \int_{\alpha_n}^\infty 2^{-h} \| \nabla_x P_h \phi_0^n \|_{L^2_x} \| \nabla_x P_{\leq h} \phi_0^n \|_{L^2_x} \| \Pi''(P_{\leq h} \phi_0^n) \|_{L^\infty_x} \, dh.
 \end{align*}
 Using \eqref{equ:small_besov_error_around_frequency_cutoff}, we conclude that for all sufficently large $n$, the previous line is bounded by
 \begin{align*}
  &2^k E_{crit}^{\frac{1}{2}}  \int_{\alpha_n}^{\alpha_n + M} 2^{-h} \| \nabla_x P_h \phi_0^n \|_{L^2_x} \, dh + 2^k E_{crit}^{\frac{1}{2}} \int_{\alpha_n + M}^\infty 2^{-h} \|\nabla_x P_h \phi_0^n \|_{L^2_x} \, dh \\
  &\lesssim 2^{k-\alpha_n} E_{crit}^{\frac{1}{2}} M \delta_0 + 2^{k-\alpha_n} E_{crit} 2^{-M}.
 \end{align*}
 For the high frequencies $k > \alpha_n$, we use that 
 \[
  \nabla_x P_k  \Pi (P_{\leq \alpha_n} \phi_0^n) = \nabla_x P_k \int_{-\infty}^{\alpha_n} (P_h \phi_0^n) \Pi'(P_{\leq h} \phi_0^n) \, dh.
 \]
 Then we have 
 \begin{equation} \label{equ:data_approximate_preservation_of_frequency_localization_high_frequencies}
  \begin{aligned}
  \bigl\| \nabla_x P_k \Pi \bigl( P_{\leq \alpha_n} \phi_0^n \bigr) \bigr\|_{L^2_x} &\lesssim 2^{-k} \bigl\| \nabla_x^2 P_k \Pi ( P_{\leq \alpha_n} \phi_0^n ) \bigr\|_{L^2_x} \\
  &\lesssim 2^{-k} \Bigl\| P_k \int_{-\infty}^{\alpha_n} (\nabla_x^2 P_h \phi_0^n) \Pi'( P_{\leq h} \phi_0^n ) \, dh \Bigr\|_{L^2_x} \\
  &\quad + 2^{-k} \Bigl\| P_k \int_{-\infty}^{\alpha_n} (\nabla_x P_h \phi_0^n) (\nabla_x P_{\leq h} \phi_0^n) \Pi''(P_{\leq h} \phi_0^n) \, dh \Bigr\|_{L^2_x}.
  \end{aligned}
 \end{equation}
 We bound the first term on the right-hand side of \eqref{equ:data_approximate_preservation_of_frequency_localization_high_frequencies} by
 \begin{align*}
  2^{-k} \int_{-\infty}^{\alpha_n} 2^h \| \nabla_x P_h \phi_0^n \|_{L^2_x} \, dh &\lesssim 2^{-k} \int_{\alpha_n - M}^{\alpha_n} 2^h \|\nabla_x P_h \phi_0^n \|_{L^2_x} \, dh + 2^{-k} \int_{-\infty}^{\alpha_n-M} 2^h \|\nabla_x P_h \phi_0^n\|_{L^2_x} \, dh \\
  &\lesssim 2^{-(k-\alpha_n)} ( M \delta_0 + E_{crit}^{\frac{1}{2}} 2^{-M} ),
 \end{align*}
 while the second term on the right-hand side of \eqref{equ:data_approximate_preservation_of_frequency_localization_high_frequencies} can be estimated by
 \begin{align*}
  2^{-k} \int_{-\infty}^{\alpha_n} \| \nabla_x P_h \phi_0^n \|_{L^2_x} \| \nabla_x P_{\leq h} \phi_0^n \|_{L^\infty_x} \, dh &\lesssim 2^{-k} \int_{\alpha_n - M}^{\alpha_n} \|\nabla_x P_h \phi_0^n \|_{L^2_x} 2^h \|\nabla_x P_{\leq h} \phi_0^n\|_{L^2_x} \, dh \\
  &\quad + 2^{-k} \int_{-\infty}^{\alpha_n - M} \|\nabla_x P_h \phi_0^n \|_{L^2_x} 2^h \|\nabla_x P_{\leq h} \phi_0^n\|_{L^2_x} \, dh \\
  &\lesssim 2^{-(k-\alpha_n)} (1+E_{crit}) (M \delta_0 + 2^{-M}).
 \end{align*}
 Thus, we have inferred that 
 \[
  \bigl\| \nabla_x P_k \bigl( P_{\leq \alpha_n} \phi_0^n - \Pi_{\leq \alpha_n} \phi_0^n \bigr) \bigr\|_{L^2_x} \lesssim (1+E_{crit}) (M \delta_0 + 2^{-M}) 2^{-|k-\alpha_n|}.
 \]
 Since $M \sim \log( (1+E_{crit}) \delta_1^{-1})$, we obtain \eqref{equ:data_approximate_preservation_of_frequency_localization} for the component $\Pi_{\leq \alpha_n} \phi_0^n$ upon choosing $\delta_0 \ll \delta_1$ sufficently small depending on the size of $E_{crit}$.
 
 Finally, we turn to the proof of \eqref{equ:data_approximate_preservation_of_frequency_localization} for the time derivative component $\Pi_{\leq \alpha_n} \phi_1^n$. For the low frequencies $k \leq \alpha_n$ we exploit the fact that $(\phi_0^n)^\dagger \phi_1^n = 0$ to write 
 \begin{equation} \label{equ:data_approximate_preservation_of_frequency_localization_time_derivative_low_frequencies}
  \begin{aligned}
   P_{\leq \alpha_n} \phi_1^n - \Pi_{\leq \alpha_n} \phi_1^n &= - \langle \Pi_{\leq \alpha_n} \phi_0^n, P_{\leq \alpha_n} \phi_1^n \rangle \Pi_{\leq \alpha_n} \phi_0^n \\
   &= \int_{\alpha_n}^\infty \frac{d}{dh} \bigl( \langle \Pi_{\leq h} \phi_0^n, P_{\leq h} \phi_1^n \rangle \Pi_{\leq h} \phi_0^n \bigr) \, dh \\
   &= \int_{\alpha_n}^\infty \langle \Pi_{\leq h} \phi_0^n, P_{\leq h} \phi_1^n \rangle (P_h \phi_0^n) \Pi'(P_{\leq h} \phi_0^n) \, dh \\
   &\quad + \int_{\alpha_n}^\infty \langle (P_h \phi_0^n) \Pi'(P_{\leq h} \phi_0^n), P_{\leq h} \phi_1^n \rangle \Pi(P_{\leq h} \phi_0^n) \, dh \\
   &\quad + \int_{\alpha_n}^\infty \langle \Pi_{\leq h} \phi_0^n, P_h \phi_1^n \rangle \Pi_{\leq h} \phi_0^n \, dh.
  \end{aligned}
 \end{equation}
 Then we may bound the first term on the right-hand side of \eqref{equ:data_approximate_preservation_of_frequency_localization_time_derivative_low_frequencies} by
 \begin{align*}
  &\Bigl\| P_k \int_{\alpha_n}^\infty \langle \Pi_{\leq h} \phi_0^n, P_{\leq h} \phi_1^n \rangle (P_h \phi_0^n) \Pi'(P_{\leq h} \phi_0^n) \, dh \Bigr\|_{L^2_x} \\
  &\quad \lesssim 2^k \int_{\alpha_n}^\infty \| \Pi_{\leq h} \phi_0^n \|_{L^\infty_x} \| P_{\leq h} \phi_1^n \|_{L^2_x} \| P_h \phi_0^n \|_{L^2_x} \| \Pi'(P_{\leq h} \phi_0^n) \|_{L^\infty_x} \, dh \\
  &\quad \lesssim 2^k E_{crit}^{\frac{1}{2}} \int_{\alpha_n}^\infty 2^{-h} \|\nabla_x P_h \phi_0^n\|_{L^2_x} \, dh \\
  &\quad \lesssim 2^{k-\alpha_n} (1+E_{crit}) (M \delta_0 + 2^{-M}),
 \end{align*}
 which is of the desired form, while the other two terms on the right-hand side of \eqref{equ:data_approximate_preservation_of_frequency_localization_time_derivative_low_frequencies} can be dealt with similarly. Proving \eqref{equ:data_approximate_preservation_of_frequency_localization} for $\Pi_{\leq \alpha_n} \phi_1^n$ for the high frequencies $k > \alpha_n$ is a variant of the previous estimates.
\end{proof}

\begin{rem}
 We note that Proposition~\ref{prop:data_approximate_preservation_of_frequency_localization} and its proof are reminiscent of Proposition~11.1 in Sterbenz-Tataru~\cite{Sterbenz_Tataru1} where it is shown that for initial data sets with small ``energy dispersion'' frequency truncation followed by normal projection approximately preserves the frequency localization properties up to exponentially decaying tails.
\end{rem}

Thus, for the ensuing induction on frequency process, we shall use $\Pi_{\leq b_L} \phi^{n \Lambda_0^{(0)}}[0]$ as the actual data for the wave maps evolution. But in order for this data $\Pi_{\leq b_L} \phi^{n \Lambda_0^{(0)}}[0]$ to approximately have the same frequency localization properties up to exponentially decaying tails as the frequency truncations $P_{\leq b_L} \phi^{n \Lambda_0^{(0)}}[0]$, Proposition~\ref{prop:data_approximate_preservation_of_frequency_localization} requires the Besov norm smallness condition~\eqref{equ:small_besov_error_around_frequency_cutoff}. This, in particular, forces us to implement the following delicate selection procedure for the endpoints of the intervals $J_\ell = [a_\ell, b_\ell]$.

\medskip 

We first use the M\'etivier-Schochet procedure~\cite{Metivier-Schochet} to carry out a further refined decomposition of the lowest frequency ``non-atomic'' component $\phi^{n \Lambda_0^{(0)}}[0]$ into finitely many ``smaller'' frequency atoms. More precisely, given any $\delta_0 > 0$ we decompose into
\[
 \phi^{n \Lambda_0^{(0)}}[0] = \sum_{j = 1}^{J(\Lambda_0^{(0)})} \phi^{n a_j^{(0)}}[0] + \phi^{n J(\Lambda_0^{(0)})}[0]
\]
such that the ``small'' atoms $\phi^{n a_j^{(0)}}[0]$ are frequency localized to
\[
 |\xi| \in [ ( \lambda_n^{a_j^{(0)}})^{-1} (R_n^{(0)})^{-1}, (\lambda_n^{a_j^{(0)}})^{-1} R_n^{(0)} ]
\]
for a sequence $R_n^{(0)} \to \infty$ growing sufficently slowly as $n \to \infty$ and such that
\begin{equation} \label{equ:besov_error_non_atomic_component}
 \limsup_{n\to\infty} \| \phi^{n J(\Lambda_0^{(0)})}[0] \|_{\dot{B}^1_{2,\infty} \times \dot{B}^0_{2,\infty}} \leq \delta_0.
\end{equation}
By the orthogonality of the frequency scales, we may assume that for sufficently large $n$ the frequency intervals
\[
 \bigl[  (\lambda_n^{a_j^{(0)}})^{-1} (R_n^{(0)})^{-1}, (\lambda_n^{a_j^{(0)}})^{-1} R_n^{(0)} \bigr]
\]
are disjoint for different $a_j^{(0)}$.

Then we are in the position to describe how the frequency intervals $J_\ell$ can be chosen inductively. Assume that the intervals
\[
 J_1 = [a_1, b_1], \ldots, J_{\ell-1} = [a_{\ell-1}, b_{\ell-1}]
\]
have already been specified. Then we pick the maximal interval $[a_\ell, \tilde{b}_\ell]$ with the property that $a_\ell = b_{\ell-1}$ and such that
\[
 \bigl\| P_{[a_\ell, \tilde{b}_\ell]} \phi^{n \Lambda_0^{(0)}}[0] \bigr\|_{\dot{H}^1_x \times L^2_x}^2 = \varepsilon_0.
\]
If it happens that 
\[
 \tilde{b}_\ell \in [-\log(\lambda_n^{a_j^{(0)}}) - 2 \log(R_n^{(0)}), -\log(\lambda_n^{a_j^{(0)}}) + 2 \log(R_n^{(0)})],
\]
then we shift the interval endpoint upwards and take 
\[
 b_\ell = - \log(\lambda_n^{a_j^{(0)}}) + 2 \log(R_n^{(0)}),
\]
otherwise we set $b_\ell = \tilde{b}_\ell$.

The point of this construction is that for $\ell = 1, \ldots, C_1$, the frequency intervals $J_\ell = [a_\ell, b_\ell]$ now satisfy 
\[
 \bigl\| P_{J_\ell} \phi^{n \Lambda_0^{(0)}}[0] \bigr\|_{\dot{H}^1_x \times L^2_x}^2 \lesssim \varepsilon_0
\]
and have the desired \emph{key property} that for any fixed $M > 0$, it holds that 
\begin{equation} \label{equ:besov_smallness_edge_points}
 \limsup_{n\to\infty} \, \bigl\| P_{[b_\ell - M, b_\ell + M]} \phi^{n \Lambda_0^{(0)}}[0] \bigr\|_{\dot{B}^1_{2,\infty} \times \dot{B}^0_{2,\infty}} \leq \delta_0
\end{equation}
for all sufficently large $n$.

\subsection{Interlude: A concept of energy class radial wave maps} \label{subsec:energy_class_evolution}

At the end of the entire concentration compactness step in this Section~\ref{sec:Concentration_Compactness_Step} we want to extract a minimal blowup solution to the wave maps equation (WM) that is merely of energy class. However, since the local well-posedness theory \cite{KlMa1, KlMa2, KlMa3, Klainerman_Selberg1, Klainerman_Selberg2} only pertains to data of regularity $H^{1+}_x(\R^2) \times H^{0+}_x(\R^2)$, we first of all have to introduce a notion of the wave maps evolution of radially symmetric energy class data. We shall achieve this analogously to the procedures in \cite{KS} and \cite{KL} by regularization and reduction to the small energy case via finite speed of propagation. We begin with the following ``high-frequency perturbation'' lemma.
\begin{lem} \label{lem:localperturbative} 
Let $\phi: [0,T] \times \R^2 \rightarrow \mathbb{S}^m$ be a wave map which is supported in the frequency interval $(-\infty, K]$ for some $K \in \R$ up to an exponentially decaying tail, more precisely such that 
\[
 \| \phi_k \|_{S_k([0,T])} \leq \delta_1 2^{-\sigma (k-K)}, \quad k \geq K,
\]
for some positive constants $0 < \delta_1 \ll 1$ and $\sigma > 0$. Suppose that
\[
 \| \phi \|_{S([0,T])} \leq C_1
\]
for some $C_1>0$. Then there exists $\delta \equiv \delta(C_1, \sigma)$ such that if $\epsilon[0] \in H^{1+}_x \times H^{0+}_x$ is a data pair such that $(\phi+\epsilon)[0]$ constitutes admissible initial data mapping into $T \mathbb{S}^m$ and such that 
\[
 \big\| \epsilon[0] \big\|_{\dot{H}^1_x \times L^2_x} \leq \delta \quad \text{ with } \delta_1 \ll \delta,
\]
and moreover if $\epsilon[0]$ is supported up to exponentially decaying tails at frequencies $[K, \infty)$ in the sense that 
\[
 \big\| P_k \epsilon[0] \big\|_{\dot{H}^1_x \times L^2_x} \leq \delta 2^{-\sigma(K-k)}, \quad k<K, 
\]
then the data $(\phi+\epsilon)[0]$ can be evolved on $[0,T] \times \R^2$ as a wave map $(\phi+\epsilon)(t, x)$ of class $H^{1+}_x \times H^{0+}_x$ and we have that
\[
 \| \epsilon \|_{S([0,T])} \leq C_2(C_1,\sigma) \, \delta.
\]
\end{lem}
\begin{proof} 
From the basic local existence theory~\cite{KlMa1, KlMa2, KlMa3, Klainerman_Selberg1, Klainerman_Selberg2} we know that the data $(\phi+\epsilon)[0]$ can be evolved locally in time as a wave map of class $H^{1+}_x \times H^{0+}_x$. It remains to show that this evolution extends all the way up to time $t = T$. This we do by a bootstrap argument which is a simplified version of the one that will be used for controlling the lowest frequency ``non-atomic'' component in Subsection~\ref{subsec:lowest_freq_non_atomic}. In fact, we make the bootstrap assumption 
\[
 \|\epsilon_k\|_{S_k([0,T])} \leq C (c_k + d_k)
\]
for some sufficiently large $C > 0$, where $\{ c_k \}_{k\in\Z}$ is such that $c_k = 0$ for $k<K$ and $\sum_{k\in\Z} c_k^2 \leq \delta^2$, while $\{ d_k \}_{k\in\bZ}$ is such that $d_k = \delta 2^{-\sigma |k-K|}$ for $k < K$ and $d_k = 0$ for $k \geq K$. Then we argue as in the bootstrap argument in the proof of Proposition~\ref{prop:bootstrap_bounds_to_next_level_in_evolving_lowest_frequency_nonatomic_part} in Subsection~\ref{subsec:lowest_freq_non_atomic}.
\end{proof}

Now assume that we are given a radially symmetric data pair $\phi[0] \colon \R^{2} \rightarrow T\mathbb{S}^m$ of energy class $\phi[0] \in (\dot{H}^1_x \cap L^\infty_x) \times L^2_x$. Then we can pass to the frequency truncated data $\Pi_{<K} \phi[0]$ for any $K \in \R$. As each of these data sets are of class $H^{1+}_x \times H^{0+}_x$ we have a canonical evolution for them on some local time intervals, which may however a priori depend on $K$. But then we observe  
\begin{lem} \label{lem:datalocalization} 
Let $\phi[0] \colon \R^{2} \rightarrow T\mathbb{S}^m$ be radially symmetric data of energy class $\phi[0] \in (\dot{H}^1_x \cap L^\infty_x) \times L^2_x$. Then there exists a time $T_*>0$, $T_* \equiv T_*(\phi[0])$, and some sufficiently large $K_0 \in \R$ such that for any $K \geq K_0$, the wave maps evolution $\phi^{(K)}(t, x)$ of the data $\Pi_{<K} \phi[0]$ exists on $[0,T_*]\times \R^2$, and moreover, we have a uniform bound 
\[
 \big\|\phi^{(K)}\big\|_{S([0, T_*])} \leq C\big( \phi[0] \big), \quad K \geq K_0. 
\]
\end{lem}
\begin{proof} 
The claim will follow from the small energy global regularity result from Theorem~\ref{thm:small_energy_global} via Huygen's principle and a simple partition of unity argument to patch together the global-in-space solution from spatially localized ones. To this end we define the annuli
\[
 C_{R_1, R_2} := \{ x \in \R^2 \colon R_1 \leq |x| \leq R_2 \}, \quad 0 \leq R_1 < R_2 \leq +\infty
\]
and denote by $\eta > 0$ the cutoff for the small energy global regularity theory. Given radially symmetric data $\phi[0] \colon \R^2 \rightarrow T\mathbb{S}^m$, there exists a covering of $\R^2$ by finitely many annuli $C_{R_{j-1}, R_j}$, $j = 1,2, \ldots, J,$ with $R_0 = 0, R_J = +\infty$, $R_{j} \leq 2 R_{j-1}$ for $j \in \{2, \ldots, J-1 \}$, and such that $\phi[0] \big|_{C_{R_{j-1}, R_j}}$ coincides with the restriction to $C_{R_{j-1}, R_j}$ of some radially symmetric data $\phi^{(j)}[0] \colon \R^2 \rightarrow T\mathbb{S}^m$ of energy less than $\eta$. To see this, pick finitely many such annuli $C_{R_{j-1}, R_j}$ with the property that 
\[
 \big\| \phi[0] \big|_{C_{ \frac{1}{2} R_{j-1}, 2R_j}} \big\|_{\dot{H}^1_x \times L^2_x} \ll \eta. 
\]
Then observe that for any $x, y \in C_{\frac{1}{2} R_{j-1}, 2 R_j}$, $j \in \{2, \ldots, J-1\}$, with $r = |x|$ and $\tilde{r} = |y|$, we have that
\[
 \big| \phi(x) - \phi(y) \big| = \bigg| \int_{\tilde{r}}^r \nabla_r \phi(r) \, dr \bigg| \lesssim \Big| \log \Big( \frac{r}{\tilde{r}} \Big) \Big|^{\frac{1}{2}} \big\|\nabla_x\phi\big\|_{C_{\frac{1}{2} R_{j-1}, 2R_j}} \ll \eta.
\]
Then picking $x_0 \in C_{R_{j-1}, R_j}$ arbitrarily and introducing 
\[
 \phi^{(j)} := \Pi \Big( \chi_{C_{\frac{1}{2} R_{j-1}, 2R_j}} \phi + \big( 1- \chi_{C_{\frac{1}{2} R_{j-1}, 2R_j}} \big) \phi(x_0) \Big), \quad j \in \{2, \ldots, J-1\}, 
\]
where $\bigl\{ \chi_{C_{\frac{1}{2} R_{j-1}, 2R_j}} \bigr\}_{j = 1, \ldots, J}$ is a smooth partition of unity of $\R^2$ with $\chi_{C_{\frac{1}{2} R_{j-1}, 2R_j}}\big|_{C_{R_{j-1}, R_j}} \equiv 1$, we obtain 
\[
 \big\| \nabla_x \phi^{(j)} \big\|_{L_x^2(\R^2)} \ll \eta
\]
as well as 
\[
 \phi^{(j)}\big|_{C_{R_{j-1}, R_j}} \equiv \phi.
\]
Further, we set 
\[
 \partial_t \phi^{(j)} := \chi_{C_{\frac{1}{2} R_{j-1}, 2 R_j}} \partial_t \phi - \chi_{C_{\frac{1}{2} R_{j-1}, 2 R_j}} \bigl( (\partial_t \phi)^\dagger \phi^{(j)} \bigr) \, \phi^{(j)}.
\]
Then we have $\partial_t\phi^{(j)}\big|_{C_{R_{j-1}, R_j}} = \partial_t\phi$ and it is easily seen that 
\[
 \big\| \partial_t \phi^{(j)} \big\|_{L_x^2(\R^2)} \ll \eta. 
\]
Importantly, the constructed data pair $\phi^{(j)}[0] = (\phi^{(j)}, \partial_t \phi^{(j)}) \colon \R^2 \to T\mathbb{S}^m$ is again radially symmetric. It is also straightforward to modify this construction on $C_{\frac{1}{2} R_{j-1}, 2R_j}$ for $j=1$ and for $j = J$. Replacing $\phi[0]$ by $\phi^{(K)}[0] = \Pi_{<K}\phi[0]$ results in the data $\phi^{(j, K)}[0]$, which satisfy the same bounds (with a slightly different implied constant) for large enough $K$. It then follows from the small energy global regularity theory that we can evolve these data $\phi^{(j,K)}[0]$ to global wave maps and a simple application of Huygen's principle implies that the evolutions $\phi^{(K)}$ exist on a joint time interval $[0,T_*]$ for some $0 < T_* \ll 1$ for all large enough $K \geq K_0$. Moreover, we may infer uniform bounds 
\[
 \big\|\phi^{(K)}\big\|_{S([0,T_*])} \leq C_*(\phi[0]), \quad K \geq K_0. 
\]
\end{proof}
Combining the two preceding lemmas, we can now deduce the desired concept of energy class evolution. Given radially symmetric energy class data $\phi[0] \colon \R^{2} \rightarrow T\mathbb{S}^m$ with $\phi[0] \in (\dot{H}^1_x \cap L^\infty_x) \times L^2_x$, we see that by Lemma~\ref{lem:datalocalization} there exists $T_*(\phi[0])>0$ such that for all sufficiently large $K \geq K_0$ the wave maps evolutions $\phi^{(K)}(t, x)$ exist on a joint time interval $[0, T_*]$ and moreover, using Lemma~\ref{lem:localperturbative}, we see that the sequence $\{\phi^{(K)}\}_{K \geq K_0}$ converges in the sense of $\| \cdot \|_{S([0, T_*])}$. It also follows that the limit is canonical, i.e. it does not depend on the precise choice of regularization. Correspondingly, we introduce
\begin{defn} \label{defn:energy_class_evolution}
 Let $\phi[0] \colon \R^{2} \rightarrow T\mathbb{S}^m$ be a radially symmetric energy class data pair $\phi[0] \in (\dot{H}^1_x \cap L^\infty_x) \times L^2_x$ and let $\{ \phi^{(K)}[0] \}_{K}$ be a sequence of frequency truncated data $\phi^{(K)}[0] = \Pi_{<K} \phi[0]$ such that $\phi^{(K)}[0] \to \phi[0]$ as $K \to \infty$ in the sense of $\dot{H}^1_x \times L^2_x$. We denote by $\phi^{(K)}$ the smooth local wave maps evolutions of the data $\phi^{(K)}[0]$ and define $I = (-T_0, T_1) = \cup \tilde{I}$ to be the union of all open time intervals $\tilde{I} \ni 0$ such that
 \[
  \sup_{J \subset \tilde{I}, J \text{ closed}} \liminf_{K\to\infty} \big\| \phi^{(K)} \big\|_{S[J]} < \infty.
 \]
 Then we define the wave maps evolution of $\phi[0]$ on $I \times \R^2$ to be 
 \[
  \phi[t] := \lim_{K\to\infty} \phi^{(K)}[t], \quad t \in I,
 \]
 where the limit is taken in the energy topology. We refer to $I$ as the maximal lifespan of $\phi$. For any closed time interval $J \subset I$, we set 
 \[
  \| \phi \|_{S[J]} := \lim_{K\to\infty} \big\| \phi^{(K)} \big\|_{S[J]}.
 \]
\end{defn}

Our above considerations also imply the following characterization of the maximal lifespan $I$ of an energy class wave maps evolution.
\begin{lem}
 Let $\phi$, $\phi^{(K)}$ and $I$ be as in the preceding Definition~\ref{defn:energy_class_evolution}. Suppose that $I \neq \R$. Then it must hold that
 \[
  \sup_{J \subset I, J \text{ closed}} \liminf_{K\to\infty} \big\| \phi^{(K)} \big\|_{S[J]} = \infty.
 \]
\end{lem}

\subsection{Evolving the lowest frequency non-atomic part} \label{subsec:lowest_freq_non_atomic}

In this subsection we prove that the lowest frequency non-atomic part $\phi^{n \Lambda_0^{(0)}}[0]$, more precisely the associated wave map data $\Pi_{\leq \mu_n^1 - \log(R_n)} \phi^n[0]$, can be globally evolved for all sufficiently large $n$ with uniform in $n$ bounds on the $S$ norms of the global evolutions just in terms of $E_{crit}$. The end result is summarized in the following proposition.

\begin{prop} \label{prop:evolving_lowest_frequency_nonatomic_part}
 Let $\phi^{n \Lambda_0^{(0)}}[0]$ be defined as in \eqref{equ:decomposition_remainder_term_into_frequency_shells}. Then provided $\delta_0 \ll \delta_1 \ll \varepsilon_0$ are chosen sufficently small depending on the size of $E_{crit}$, the associated wave map initial data $\Pi_{\leq \mu_n^1 - \log(R_n)} \phi^n[0]$ can be evolved globally in time for all sufficently large $n$ and their evolutions denoted by $\Pi_{\leq \mu_n^1 - \log(R_n)} \phi^n$ obey 
 \begin{equation*}
  \bigl\| \Pi_{\leq \mu_n^1 - \log(R_n)} \phi^n \bigr\|_{S[\bR]} \leq C(E_{crit})
 \end{equation*}
 uniformly for all sufficently large $n$ for some constant $C(E_{crit}) > 0$ that depends only on $E_{crit}$. 
\end{prop}

We shall prove Proposition~\ref{prop:evolving_lowest_frequency_nonatomic_part} via a finite inductive procedure over the increasing size of the frequencies, more precisely we shall inductively conclude that for $L = 1, \ldots, C_1$, the data $\Pi_{\leq b_L} \phi^n[0]$ can be globally evolved for all sufficently large $n$ with uniform in $n$ bounds on the $S$ norms of the evolutions. The induction start is given by the small energy global regularity result from Theorem~\ref{thm:small_energy_global} and the induction step is provided by the following key proposition of this subsection.

\begin{prop} \label{prop:bootstrap_bounds_to_next_level_in_evolving_lowest_frequency_nonatomic_part}
 Assume that for some $1 \leq L < C_1$, the data $\Pi_{\leq b_L} \phi^{n \Lambda_0^{(0)}}[0]$ can be evolved globally in time for all large $n$ and that their evolutions denoted by $\Pi_{\leq b_L} \phi^{n \Lambda_0^{(0)}}$ satisfy 
 \[
  \bigl\| \Pi_{\leq b_L} \phi^{n \Lambda_0^{(0)}} \bigr\|_{S[\bR]} \leq C_2.
 \]
 Then provided $\delta_0 \ll \delta_1$ and $\delta_1 \equiv \delta_1(C_2, E_{crit})$ are chosen sufficiently small, there exists $C_3 \equiv C_3(C_2) > 0$ such that for all sufficently large $n$, $\Pi_{\leq b_{L+1}} \phi^{n \Lambda_0^{(0)}}[0]$ can be globally evolved and their evolutions denoted by $\Pi_{\leq b_{L+1}} \phi^{n \Lambda_0^{(0)}}$ satisfy
 \[
  \bigl\| \Pi_{\leq b_{L+1}} \phi^{n \Lambda_0^{(0)}} \bigr\|_{S[\bR]} \leq C_3.
 \]
\end{prop}
\begin{proof} 
To simplify the notation in the following proof we shall write
\[ 
 \phi \equiv \Pi_{\leq b_L} \phi^{n \Lambda_0^{(0)}}, \quad \epsilon = \Pi_{\leq b_{L+1}} \phi^{n \Lambda_0^{(0)}} - \Pi_{\leq b_L} \phi^{n \Lambda_0^{(0)}},
\]
keeping in mind that $\phi$ and $\epsilon$ depend on $n$. Since by assumption $\phi$ exists globally in time with finite $S$ norm, we have frequency envelope bounds for $\phi$ thanks to Proposition~\ref{prop:finite_S_gives_freq_envelope_bounds}. In particular, due to the Besov smallness~\eqref{equ:besov_smallness_edge_points} around the frequency interval endpoint $b_L$, this implies that for all sufficently large~$n$,
\begin{equation} \label{equ:bootstrap_non_atomic_exponential_decay_phi}
 \| P_k \phi \|_{S_k[\bR]} \lesssim \delta_1 2^{-\sigma (k-b_L)} \text{ for } k > b_L.
\end{equation}

While $\phi$ exists globally in time, $\epsilon$ only exists locally in time and for now, any statement we make about $\epsilon$ is meant locally in time on some interval $I_0$ around $t=0$. In order to prove global existence and finite $S$ norm bounds for $\epsilon$, we use Proposition~\ref{prop:partial_fungibility} to partition the time axis $\bR = \cup_{j=1}^N I_j$ into $N \equiv N(C_2, E_{crit})$ many time intervals $I_j$ with the property that
\[
 \|\phi\|_{S[I_j]} \lesssim C(E_{crit}) \quad \text{for } j = 1, \ldots, N.
\]
We tacitly assume that these intervals are intersected with $I_0$ and that the interval $I_1$ contains time $t=0$. Our strategy is to iteratively prove $S$ norm bounds for $\epsilon$ by bootstrap on each interval $I_j$. Here we encounter the danger that the energy of $\epsilon$ could keep growing as we move to later time intervals and could thereby leave the perturbative regime before we would have concluded $S$ norm bounds on $\epsilon$ on all time intervals $I_1, \ldots, I_N$. However, we will see that the energy transfer between $\epsilon$ and $\phi$ is controlled by the underlying Besov error~$\delta_0$, which therefore implies approximate energy conservation for $\epsilon$. We now fix the interval $I_1$ with the understanding that all the arguments in this step can be carried out for the later time intervals $I_2, \ldots, I_N$. 

For the ensuing bootstrap argument we introduce another smallness parameter $\delta_2 > 0$ satisfying 
\[
 0 < \delta_0 \ll \delta_1 \ll \delta_2 \ll \varepsilon_0 \ll 1.
\]
Here we recall that $\delta_0$ is the underlying Besov error~\eqref{equ:besov_error_non_atomic_component} of the non-atomic component in the frequency atom decomposition and that $\delta_1$ controls the approximate frequency localization~\eqref{equ:data_approximate_preservation_of_frequency_localization} of the normal projection operator. Then we make the following bootstrap assumptions.

\medskip

\noindent {\bf Bootstrap assumptions:} Suppose that
\begin{equation} \label{equ:bootstrap_assumption_non_atomic}
 \| P_k \epsilon \|_{S_k[I_1]} \leq C ( c_k + d_k ) \text{ for } k \in \bZ,
\end{equation}
where $C \equiv C(E_{crit}) \gg 1$ is a sufficently large constant and where $\{ c_k \}_{k \in \bZ}$ satisfies
\begin{equation} \label{equ:bootstrap_non_atomic_envelope_bound1}
 c_k = 0 \text{ for } k < b_L, \quad c_k = 0 \text{ for } k > b_{L+1}, \quad \sum_{k\in\bZ} c_k^2 \lesssim \varepsilon_0
\end{equation}
as well as 
\begin{equation} \label{equ:equ:bootstrap_non_atomic_envelope_bound2}
 \sum_{b_L \leq k \leq b_L + M} c_k^2 \leq \delta_2^2 \text{ for some } M \gg 1 \text{ with } 2^{-M} \leq \delta_2^{3}.
\end{equation}
Moreover, $\{ d_k \}_{k\in\bZ}$ satisfies
\begin{equation} \label{equ:bootstrap_non_atomic_envelope_bound3}
 d_k = \delta_2 2^{-\sigma (b_L - k)} \text{ for } k < b_L, \quad d_k = 0 \text{ for } b_L \leq k \leq b_{L+1}, \quad d_k = \delta_2 2^{-\sigma (k-b_{L+1})} \text{ for } k > b_{L+1}.
\end{equation}

The idea here is that $\{ d_k \}_{k\in\bZ}$ incorporates the frequency leakage to $\epsilon$ coming from $\phi$ in the difference equation for $\epsilon$ as well as the approximate frequency localization of the data $\epsilon[0]$ up to exponential tails coming from application of the normal projection operator $\Pi$. Upon writing
\begin{equation}
 \epsilon[0] = P_{[b_L, b_{L+1}]} \phi^{n \Lambda_0^{(0)}}[0] + \bigl( \Pi_{\leq b_{L+1}} \phi^{n \Lambda_0^{(0)}} - P_{\leq b_{L+1}} \phi^{n \Lambda_0^{(0)}} \bigr)[0] - \bigl( \Pi_{\leq b_L} \phi^{n \Lambda_0^{(0)}} - P_{\leq b_L} \phi^{n \Lambda_0^{(0)}} \bigr)[0],
\end{equation}
it becomes clear that by the choice of the frequency intervals $[b_L, b_{L+1}]$, by the Besov smallness~\eqref{equ:besov_smallness_edge_points} around the endpoints of these intervals and by the approximate preservation of the frequency localization by the normal projection operator $\Pi$ as established in Proposition~\ref{prop:data_approximate_preservation_of_frequency_localization}, we have for all sufficiently large $n$ that
\begin{equation*}
 \bigl\| P_k \epsilon [0] \bigr\|_{\dot{H}^1_x \times L^2_x} \lesssim c_k + d_k.
\end{equation*}

In a first step we use a direct energy conservation argument to recover the envelope for the low frequencies $k < b_L$ with respect to the kinetic energy norm. Then we will crucially use this bound and the ``renormalized difference equation'' for $\epsilon$ to recover the full bound.
\begin{lem} \label{lem:bootstrap_non_atomic_kinetic}
 If $\epsilon$ satisfies the assumption~\eqref{equ:bootstrap_assumption_non_atomic}, then for some $1 \ll D_0 \ll C$ the following improved bound for the kinetic energy component for the low frequencies $k < b_L$ holds
 \begin{equation} \label{equ:bootstrap_non_atomic_kinetic}
  \| \nabla_{t,x} P_k \epsilon \|_{L^\infty_t L^2_x[I_1]} \leq D_0 ( c_k + d_k ).
 \end{equation}
\end{lem}
\begin{proof}
 We may assume that $k=0$, so in particular we have $b_L > 0$ in the following. Then we observe that $\epsilon_0$ satisfies the wave equation
 \[
  \Box \epsilon_0 = - P_0 \bigl( (\phi + \epsilon) \partial_\alpha(\phi + \epsilon)^\dagger \partial^\alpha (\phi+\epsilon) \bigr) + P_0 \bigl( \phi \partial_\alpha \phi^\dagger \partial^\alpha \phi \bigr) \equiv P_0 X(\phi, \epsilon).
 \]
 Passing to the corresponding energy identity, we see that we need to show 
 \[
  \biggl| \int_{I_1} \int_{\bR^2} P_0 X(\phi, \epsilon)^\dagger \partial_t \epsilon_0 \, dx \, dt \biggr| \ll C^2 \delta_2^2 2^{-2 \sigma b_L}.
 \]
 We now prove this bound separately for each term in
 \begin{equation} \label{equ:bootstrap_non_atomic_kinetic_terms}
  \begin{aligned}
   -X(\phi, \epsilon) &= \epsilon \partial_\alpha \phi^\dagger \partial^\alpha \phi + 2 \phi \partial_\alpha \epsilon^\dagger \partial^\alpha \phi + \epsilon \partial_\alpha \epsilon^\dagger \partial^\alpha \epsilon + 2 \epsilon \partial_\alpha \epsilon^\dagger \partial^\alpha \phi + \phi \partial_\alpha \epsilon^\dagger \partial^\alpha \epsilon.
  \end{aligned}
 \end{equation}

 \noindent {\it (i) Contribution of $P_0 \bigl( \epsilon \partial_\alpha \phi^\dagger \partial^\alpha \phi \bigr)$.} We split into
 \begin{equation} \label{equ:bootstrap_non_atomic_kinetic_epp}
  P_0 \bigl( \epsilon \partial_\alpha \phi^\dagger \partial^\alpha \phi \bigr) = P_0 \bigl( P_{> b_L} \epsilon \partial_\alpha \phi^\dagger \partial^\alpha \phi \bigr) + P_0 \bigl( P_{[-10, b_L]} \epsilon \partial_\alpha \phi^\dagger \partial^\alpha \phi \bigr) + P_0 \bigl( P_{\leq -10} \epsilon \partial_\alpha \phi^\dagger \partial^\alpha \phi \bigr).
 \end{equation}
 Then we have for the first term
 \begin{align*}
  \biggl| \int_{I_1} \int_{\bR^2} P_0 \bigl( P_{> b_L} \epsilon \partial_\alpha \phi^\dagger \partial^\alpha \phi \bigr)^\dagger \partial_t \epsilon_0 \, dx \, dt \biggr| \lesssim \sum_{k_1 > b_L} \sum_{k_2 = k_1+O(1)} \biggl| \int_{I_1} \int_{\bR^2} P_0 \bigl( P_{k_1} \epsilon P_{k_2} \bigl( \partial_\alpha \phi^\dagger \partial^\alpha \phi \bigr) \bigr)^\dagger \partial_t \epsilon_0 \, dx \, dt \biggr|.
 \end{align*}
 In case of high-high interactions, we place $P_{k_2}\bigl( \partial_\alpha \phi^\dagger \partial^\alpha \phi \bigr)$ into $L^{\frac{3}{2}}_t L^{\frac{3}{2}}_x$, while in case of low-high and high-low interactions, we put $P_{k_2} \bigl( \partial_\alpha \phi^\dagger \partial^\alpha \phi \bigr)$ into $L^{\frac{5}{3}}_t L^{\frac{5}{3}}_x$. Thus, in the former case we obtain from the null form estimate~\eqref{equ:null_form_estimate}, the bootstrap assumptions~\eqref{equ:bootstrap_assumption_non_atomic} and the exponential decay~\eqref{equ:bootstrap_non_atomic_exponential_decay_phi} of the $S_k$ norms of $\phi$ for $k > b_L$ the bound
 \begin{align*}
  &\lesssim \sum_{k_1 > b_L} \sum_{k_3 > k_1} \bigl\| P_{k_1} \epsilon \bigr\|_{L^3_t L^\infty_x[I_1]} \bigl\| P_{k_3} \phi \bigr\|_{S_{k_3}[I_1]}^2 \bigl\| \partial_t \epsilon_0 \bigr\|_{L^\infty_t L^3_x[I_1]} \\
  &\lesssim \sum_{k_1 > b_L} \sum_{k_3 > k_1} 2^{-\frac{1}{3} k_1} \varepsilon_0 \delta_1^2 2^{-2\sigma(k_3-b_L)} C \delta_2 2^{-\sigma b_L} \\
  &\ll C^2 \delta_2^2 2^{- 2 \sigma b_L}
 \end{align*}
 for $\delta_1 > 0$ sufficiently small. In the latter case, we similarly infer the estimate
 \begin{align*}
  &\lesssim \sum_{k_1 > b_L} \sum_{k_3 \leq k_1 + O(1)} \bigl\| P_{k_1} \epsilon \bigr\|_{L^{\frac{5}{2}}_t L^\infty_x[I_1]} 2^{\frac{1}{10} k_1} 2^{\frac{1}{10} k_3} \bigl\| P_{k_1} \phi \bigr\|_{S_{k_1}[I_1]} \bigl\| P_{k_3} \phi \bigr\|_{S_{k_3}[I_1]} \bigl\| \partial_t \epsilon_0 \bigr\|_{L^\infty_t L^{\frac{5}{2}}_x[I_1]} \\
  &\lesssim \sum_{k_1 > b_L} \sum_{k_3 \leq k_1 + O(1)} 2^{-\frac{1}{5} k_1} \varepsilon_0 2^{\frac{1}{10} k_1} 2^{\frac{1}{10} k_3} \delta_1 2^{-\sigma (k_1 - b_L)} \| P_{k_3} \phi \|_{S_{k_3}[I_1]} C \delta_2 2^{-\sigma b_L} \\
  &\lesssim \sum_{k_1 > b_L} 2^{-\frac{1}{5} k_1} \varepsilon_0 \delta_1 2^{-\sigma (k_1 - b_L)} C(E_{crit}) C \delta_2 2^{-\sigma b_L} \\
  &\ll C^2 \delta_2^2 2^{-2 \sigma b_L}. 
 \end{align*}
 Next, for the contribution of the second term in~\eqref{equ:bootstrap_non_atomic_kinetic_epp}
 \[
  \biggl| \int_{I_1} \int_{\bR^2} P_0 \bigl( P_{[-10,b_L]} \epsilon \partial_\alpha \phi^\dagger \partial^\alpha \phi \bigr)^\dagger \partial_t \epsilon_0 \, dx \, dt \biggr|,
 \]
 we argue similarly and use that
 \[
  \bigl\| P_{k_1} \epsilon \bigr\|_{L^q_t L^\infty_x} \lesssim C \delta_2 2^{-\frac{1}{q} k_1 - \sigma (b_L - k_1)} \text{ for } k_1 < b_L
 \]
 and a divisibility argument (from $\partial_\alpha \phi^\dagger \partial^\alpha \phi$) to force smallness. Finally, in order to treat the third term in~\eqref{equ:bootstrap_non_atomic_kinetic_epp}, for each frequency $k \leq -10$ we consider the expression
 \[
  \biggl| \int_{I_1} \int_{\bR^2} P_0 \bigl( P_k \epsilon P_0 ( \partial_\alpha \phi^\dagger \partial^\alpha \phi ) \bigr)^\dagger \partial_t \epsilon_0 \, dx \, dt \biggr|
 \]
 and now further distinguish for the radial variable the cases $r \leq 1$ and $r > 1$. 
 
 When $r \leq 1$ in case of low-high and high-low interactions in $P_0 \bigl( \partial_\alpha \phi^\dagger \partial^\alpha \phi \bigr)$ we easily estimate
 \begin{align*}
  &\biggl| \int_{I_1} \int_{\bR^2} \chi_{\{ r \leq 1\}} P_0 \bigl( P_k \epsilon P_0 ( \partial_\alpha \phi^\dagger \partial^\alpha \phi ) \bigr)^\dagger \partial_t \epsilon_0 \, dx \, dt \biggr| \\
  &\lesssim \sum_{\ell < 0} \, \bigl\| \chi_{\{ r \sim 2^\ell \}} r^{+\frac{1}{2}} \bigr\|_{L^6_x} \| P_k \epsilon \|_{L^\infty_t L^\infty_x[I_1]} \bigl\| P_0 \bigl( \partial_\alpha \phi \partial^\alpha \phi \bigr) \bigr\|_{L^2_t L^3_x[I_1]} \sup_{\ell \in \bZ} \, \bigl\| \chi_{\{ r \sim 2^\ell \}} r^{-\frac{1}{2}} \partial_t \epsilon_0 \bigr\|_{L^2_t L^2_x[I_1]} \\
  &\lesssim \| P_k \epsilon \|_{S_k[I_1]} \bigl\| P_0 \bigl( \partial_\alpha \phi \partial^\alpha \phi \bigr) \bigr\|_{L^2_t L^3_x[I_1]} \| \epsilon_0 \|_{S_0[I_1]} \\
  &\lesssim 2^{\sigma k} C^2 \delta_2^2 2^{-2 \sigma b_L} \bigl\| P_0 \bigl( \partial_\alpha \phi \partial^\alpha \phi \bigr) \bigr\|_{L^2_t L^3_x[I_1]}.
 \end{align*}
 Then one can sum over $k \leq -10$ and smallness is obtained from divisibility of the $\bigl\| P_0 \bigl( \partial_\alpha \phi \partial^\alpha \phi \bigr) \bigr\|_{L^2_t L^3_x[I_1]}$ norm. In case of high-high interactions in $P_0 \bigl( \partial_\alpha \phi^\dagger \partial^\alpha \phi \bigr)$, we use H\"older's inequality to place $P_k \epsilon$ into $L^\infty_t L^\infty_x[I_1]$, the null form $P_0 \bigl( \partial_\alpha \phi^\dagger \partial^\alpha \phi \bigr)$ into $L^{\frac{3}{2}}_t L^{\frac{3}{2}}_x[I_1]$ and $\partial_t \epsilon_0$ into $L^3_t L^\infty_x[I_1]$, gaining smallness from divisibility of the $L^{\frac{3}{2}}_t L^{\frac{3}{2}}_x[I_1]$ norm in this case. 
 
 We can therefore now turn to the more subtle case when $r > 1$. We first treat the high-high interactions in $P_0 \bigl( \partial_\alpha \phi^\dagger \partial^\alpha \phi \bigr)$. Here we use Strauss' improved Sobolev embedding to estimate 
 \begin{align*}
  &\biggl| \int_{I_1} \int_{\bR^2} \chi_{\{ r > 1\}} P_0 \bigl( P_k \epsilon \sum_{k_1 > 0} P_0 ( \partial_\alpha \phi_{k_1}^\dagger \partial^\alpha \phi_{k_1+O(1)} ) \bigr)^\dagger \partial_t \epsilon_0 \, dx \, dt \biggr| \\ 
  &\lesssim \sum_{\ell > 0} \, \bigl\| \chi_{\{ r \sim 2^\ell\}} P_k \epsilon \bigr\|_{L^\infty_t L^\infty_x} \Bigl\| \sum_{k_1 > 0} r^{+\frac{1}{2}} ( \partial_\alpha \phi_{k_1}^\dagger \partial^\alpha \phi_{k_1+O(1)} ) \Bigr\|_{L^2_t L^2_x[I_1]} \sup_{\ell \in \bZ} \, \bigl\| \chi_{\{ r \sim 2^\ell \}} r^{-\frac{1}{2}} \partial_t \epsilon_0 \bigr\|_{L^2_t L^2_x[I_1]} \\
  &\lesssim \sum_{\ell > 0} \min \{ 2^{-\frac{1}{2}(\ell+k)}, 1 \} \|P_k \epsilon \|_{S_k[I_1]} \| \epsilon_0 \|_{S_0[I_1]} \Bigl\| \sum_{k_1 > 0} r^{+\frac{1}{2}} ( \partial_\alpha \phi_{k_1}^\dagger \partial^\alpha \phi_{k_1+O(1)} ) \Bigr\|_{L^2_t L^2_x[I_1]} \\
  &\lesssim (1 + |k|) 2^{\sigma k} C^2 \delta_2^2 2^{-2 \sigma b_L} \Bigl\| \sum_{k_1 > 0} r^{+\frac{1}{2}} ( \partial_\alpha \phi_{k_1}^\dagger \partial^\alpha \phi_{k_1+O(1)} ) \Bigr\|_{L^2_t L^2_x[I_1]}.
 \end{align*}
 This bound can be summed over $k \leq -10$ and we obtain smallness since the $L^2_t L^2_x[I_1]$ norm here has the divisibility property thanks to the weighted null form estimate~\eqref{equ:weighted_null_form_estimate}. In case of low-high interactions in $P_0 \bigl( \partial_\alpha \phi^\dagger \partial^\alpha \phi \bigr)$ (and then analogously for high-low interactions), we write 
 \begin{align*}
  &\int_{I_1} \int_{\bR^2} \chi_{\{ r > 1\}} P_0 \bigl( P_k \epsilon \sum_{k_1 < O(1)} P_0 ( \partial_\alpha \phi_{k_1}^\dagger \partial^\alpha \phi_{0} ) \bigr)^\dagger \partial_t \epsilon_0 \, dx \, dt \\
  &= \int_{I_1} \int_{\bR^2} \chi_{\{ r > 1\}} P_0 \bigl( P_k \epsilon \sum_{k <k_1 < O(1)} P_0 ( \partial_\alpha \phi_{k_1}^\dagger \partial^\alpha \phi_{0} ) \bigr)^\dagger \partial_t \epsilon_0 \, dx \, dt \\
  &\quad + \sum_{k_1 \leq k} \int_{I_1} \int_{\bR^2} \chi_{\{ r > 1 \}} P_0 \bigl( P_k \epsilon P_0 ( \partial_\alpha \phi_{k_1}^\dagger \partial^\alpha \phi_{0} ) \bigr)^\dagger \partial_t \epsilon_0 \, dx \, dt.
 \end{align*}
 Then we can bound the first term on the right-hand side exactly as in the high-high case only that we pay a price of $|k|^2$, which can still be absorbed by the gain $2^{\sigma k}$ and then summed over $k \leq -10$. For the second term on the right-hand side we distribute the weight differently, more precisely for some $0 < \lambda < \frac{1}{2}$ we again use Strauss' improved Sobolev embedding and estimate by
 \begin{align*}
  &\sum_{\ell > 0} \sum_{k_1 \leq k} \bigl\| \chi_{\{ r \sim 2^\ell \}} r^{\frac{1}{2} - \lambda} P_k \epsilon \bigr\|_{L^\infty_t L^\infty_x[I_1]} \bigl\| r^{+\lambda} \partial_\alpha \phi_{k_1}^\dagger \partial^\alpha \phi_0 \bigr\|_{L^2_t L^2_x[I_1]} \sup_{\ell \in \bZ} \, \bigl\| \chi_{\{ r \sim 2^\ell \}} r^{-\frac{1}{2}} \partial_t \epsilon_0 \bigr\|_{L^2_t L^2_x} \\
  &\lesssim \sum_{k_1 \leq k} \sum_{\ell > 0} \min \{ 2^{-\lambda (\ell + k)}, 1 \} 2^{-(\frac{1}{2}-\lambda)k} \| P_k \epsilon \|_{S_k[I_1]} \bigl\| r^{+\lambda} \partial_\alpha \phi_{k_1}^\dagger \partial^\alpha \phi_0 \bigr\|_{L^2_t L^2_x[I_1]} \| \epsilon_0 \|_{S_0[I_1]} \\
  &\lesssim \sum_{k_1 \leq k} (1+|k|) 2^{-(\frac{1}{2}-\lambda) (k-k_1)} \|P_k \epsilon\|_{S_k[I_1]} 2^{-(\frac{1}{2}-\lambda) k_1} \bigl\| r^{+\lambda} \partial_\alpha \phi_{k_1}^\dagger \partial^\alpha \phi_0 \bigr\|_{L^2_t L^2_x[I_1]} \| \epsilon_0 \|_{S_0[I_1]} \\
  &\lesssim (1+|k|) 2^{\sigma k} C^2 \delta_2^2 2^{-2\sigma b_L} \biggl( \sum_{k_1 \leq 0} 2^{-2(\frac{1}{2}-\lambda)k_1} \bigl\| r^{+\lambda} \partial_\alpha \phi_{k_1}^\dagger \partial^\alpha \phi_0 \bigr\|_{L^2_t L^2_x[I_1]}^2 \biggr)^{\frac{1}{2}}.
 \end{align*}
 The last line can then be summed over $k \leq -10$ and by the weighted null form estimate~\eqref{equ:weighted_null_form_estimate} the last factor here has the divisibility property yielding smallness.
 
 \medskip
 
 \noindent {\it (ii) Contribution of $P_0 \bigl( \phi \partial_\alpha \epsilon^\dagger \partial^\alpha \phi \bigr)$.} Here we split into
 \begin{align*}
  \int_{I_1} \int_{\bR^2} P_0 \bigl( \phi \partial_\alpha \epsilon^\dagger \partial^\alpha \phi \bigr)^\dagger \partial_t \epsilon_0 \, dx \,dt &= \int_{I_1} \int_{\bR^2} P_0 \bigl( \phi_{\leq -10} \partial_\alpha \epsilon^\dagger \partial^\alpha \phi \bigr)^\dagger \partial_t \epsilon_0 \, dx \,dt \\
  &\quad + \int_{I_1} \int_{\bR^2} P_0 \bigl( \phi_{> -10} \partial_\alpha \epsilon^\dagger \partial^\alpha \phi \bigr)^\dagger \partial_t \epsilon_0 \, dx \,dt.
 \end{align*}
 Since the second term is truly quadrilinear and thus easier to treat, we only turn to the details of the estimate of the first term, which we may schematically write as
 \begin{equation} \label{equ:bootstrap_non_atomic_kinetic_pep_main}
  \int_{I_1} \int_{\bR^2} P_0 \bigl( \partial_\alpha \epsilon^\dagger \partial^\alpha \phi \bigr) \phi_{\leq -10}^\dagger \partial_t \epsilon_0 \, dx \, dt.
 \end{equation}
 Since $\phi + \epsilon$ and $\phi$ are wave maps into the sphere, we have the following geometric identity at our disposal
 \[
  0 = (\phi + \epsilon)^\dagger \partial_t (\phi + \epsilon) = \phi^\dagger \partial_t \epsilon + \epsilon^\dagger \partial_t \phi + \epsilon^\dagger \partial_t \epsilon.
 \]
 Upon applying the projection $P_0$ and decomposing $P_0 ( \phi^\dagger \partial_t \epsilon )$ into a Littlewood-Paley trichotomy, we find that we can schematically rewrite the term~\eqref{equ:bootstrap_non_atomic_kinetic_pep_main} as 
 \begin{align*}
  \int_{I_1} \int_{\bR^2} P_0 \bigl( \partial_\alpha \epsilon^\dagger \partial^\alpha \phi \bigr) \Bigl(- \phi_0^\dagger \partial_t \epsilon_{\leq -10} - \sum_{k_1 > 0} P_0 \bigl( \phi_{k_1}^\dagger \partial_t \epsilon_{k_1} \bigr) - P_0 ( \epsilon^\dagger \partial_t \phi ) - P_0 ( \epsilon^\dagger \partial_t \epsilon ) \Bigr) \, dx \, dt.
 \end{align*}
 This leads to four contributions, which we now estimate separately.
 
 \noindent {\it (ii.1)} We begin with the term
 \[
  \int_{I_1} \int_{\bR^2} P_0 \bigl( \partial_\alpha \epsilon^\dagger \partial^\alpha \phi \bigr) \phi_0^\dagger \partial_t \epsilon_{\leq -10} \, dx \, dt.
 \]
 In case of high-low or low-high interactions in $P_0 \bigl( \partial_\alpha \epsilon^\dagger \partial^\alpha \phi \bigr)$, we use the null form estimate~\eqref{equ:null_form_estimate} to place $P_0 \bigl( \partial_\alpha \epsilon^\dagger \partial^\alpha \phi \bigr)$ into $L^{\frac{5}{3}}_t L^{\frac{5}{3}}_x$ and put $\phi_0$ as well as $\partial_t \epsilon_{\leq -10}$ into $L^5_t L^5_x$, gaining smallness from divisibility of $\| \phi_0 \|_{L^5_t L^5_x}$. In case of high-high interactions in $P_0 \bigl( \partial_\alpha \epsilon^\dagger \partial^\alpha \phi \bigr)$ we integrate by parts to move a derivative $\partial_\alpha$ away from $\epsilon$, which leads to the terms
 \[
  \int_{I_1} \int_{\bR^2} P_0 \bigl( \epsilon^\dagger \Box \phi \bigr) \phi_0^\dagger \partial_t \epsilon_{\leq -10} \, dx \, dt + \int_{I_1} \int_{\bR^2} P_0 \bigl( \epsilon^\dagger \partial^\alpha \phi \bigr) \partial_\alpha \phi_0^\dagger \partial_t \epsilon_{\leq -10} \, dx \, dt + \int_{I_1} \int_{\bR^2} P_0 \bigl( \epsilon^\dagger \partial^\alpha \phi \bigr) \phi_0^\dagger \partial_\alpha \partial_t \epsilon_{\leq -10} \, dx \, dt,
 \]
 where in the first term we have to reinsert the wave maps equation for $\phi$. Then one can essentially argue as in {\it (i)} to bound these terms.
 
 \noindent {\it (ii.2)} Next, we consider the term 
 \[
  \int_{I_1} \int_{\bR^2} P_0 \bigl( \partial_\alpha \epsilon^\dagger \partial^\alpha \phi \bigr) \sum_{k_1 > 0} P_0 \bigl( \phi_{k_1}^\dagger \partial_t \epsilon_{k_1} \bigr) \, dx \, dt.
 \]
 In case of high-low interactions in $P_0 \bigl( \partial_\alpha \epsilon^\dagger \partial^\alpha \phi \bigr)$ (and then analogously for low-high interactions), we split into 
 \begin{equation*}
  \sum_{0 < k_1 \leq b_L} \int_{I_1} \int_{\bR^2} P_0 \bigl( \partial_\alpha \epsilon_0^\dagger \partial^\alpha \phi_{\leq -10} \bigr) P_0 \bigl( \phi_{k_1}^\dagger \partial_t \epsilon_{k_1} \bigr) \, dx \, dt + \sum_{k_1 > b_L} \int_{I_1} \int_{\bR^2} P_0 \bigl( \partial_\alpha \epsilon_0^\dagger \partial^\alpha \phi_{\leq -10} \bigr) P_0 \bigl( \phi_{k_1}^\dagger \partial_t \epsilon_{k_1} \bigr) \, dx \, dt.
 \end{equation*}
 Then it is easy to see that in the first summand we can restrict to $r > 1$, which allows us to estimate via the weighted null form estimate~\eqref{equ:weighted_null_form_estimate} for some $0 < \lambda < \frac{1}{2}$,
 \begin{align*}
  &\biggl| \sum_{0 < k_1 \leq b_L} \int_{I_1} \int_{\bR^2} P_0 \bigl( \partial_\alpha \epsilon_0^\dagger \partial^\alpha \phi_{\leq -10} \bigr) \chi_{\{ r > 1\}} P_0 \bigl( \phi_{k_1}^\dagger \partial_t \epsilon_{k_1} \bigr) \, dx \, dt \biggr| \\
  &\lesssim \bigl\| r^{+\lambda} P_0 \bigl( \partial_\alpha \epsilon_0^\dagger \partial^\alpha \phi_{\leq -10} \bigr) \bigr\|_{L^2_t L^2_x[I_1]} \sum_{0 \leq k_1 < b_L} \bigl\| \chi_{\{ r > 2^{-k_1}\}} r^{-\lambda} \phi_{k_1} \bigr\|_{L^2_t L^\infty_x[I_1]} \| \partial_t \epsilon_{k_1} \|_{L^\infty_t L^2_x[I_1]} \\
  &\lesssim \| \epsilon_0 \|_{S_0[I_1]} \| \phi \|_{S[I_1]} \biggl( \sum_{0 < k_1 \leq b_L} 2^{2(\frac{1}{2}-\lambda) k_1} \bigl\| \chi_{\{ r > 2^{-k_1}\}} r^{-\lambda} \phi_{k_1} \bigr\|_{L^2_t L^\infty_x[I_1]}^2 \biggr)^{\frac{1}{2}} \biggl( \sum_{0 < k_1 \leq b_L} 2^{-2 (\frac{1}{2} - \lambda) k_1} C^2 \delta_2^2 2^{-2\sigma (b_L -k_1)} \biggr)^{\frac{1}{2}} \\
  &\lesssim C^2 \delta_2^2 2^{-2\sigma b_L} \biggl( \sum_{0 < k_1 \leq b_L} 2^{2(\frac{1}{2}-\lambda) k_1} \bigl\| \chi_{\{ r > 2^{-k_1}\}} r^{-\lambda} \phi_{k_1} \bigr\|_{L^2_t L^\infty_x[I_1]}^2 \biggr)^{\frac{1}{2}}
 \end{align*}
 and here the last factor has the divisibility property yielding smallness. For the second summand we use Bernstein and the null form estimate~\eqref{equ:null_form_estimate} to bound by
 \begin{align*}
  &\biggl| \sum_{k_1 > b_L} \int_{I_1} \int_{\bR^2} P_0 \bigl( \partial_\alpha \epsilon_0^\dagger \partial^\alpha \phi_{\leq -10} \bigr) P_0 \bigl( \phi_{k_1}^\dagger \partial_t \epsilon_{k_1} \bigr) \, dx \, dt \biggr| \\
  &\lesssim \bigl\| \partial_\alpha \epsilon_0^\dagger \partial^\alpha \phi_{\leq -10} \bigr\|_{L^{\frac{5}{3}}_t L^{\frac{5}{3}}_x[I_1]} \sum_{k_1 > b_L} \| \phi_{k_1} \|_{L^{\frac{5}{2}}_t L^\infty_x[I_1]} \| \partial_t \epsilon_{k_1} \|_{L^\infty_t L^2_x[I_1]} \\
  &\lesssim \| \epsilon_0 \|_{S_0[I_1]} \|\phi\|_{S[I_1]} \sum_{k_1 > b_L} 2^{-\frac{5}{2} k_1} \delta_1 2^{-\sigma (k_1-b_L)} \|\epsilon_{k_1}\|_{S_{k_1}[I_1]} \\
  &\lesssim C \delta_2 2^{-\sigma b_L} C(E_{crit}) \delta_1 2^{-\frac{5}{2} b_L} \varepsilon_0^{\frac{1}{2}} \\
  &\ll C^2 \delta_2^2 2^{-2\sigma b_L}.
 \end{align*}
 In case of high-high interactions in $P_0 \bigl( \partial_\alpha \epsilon^\dagger \partial^\alpha \phi \bigr)$ we have to integrate by parts to move a derivative $\partial_\alpha$ away from $\epsilon$  and then one can argue essentially as in $(i)$.

 \noindent {\it (ii.3)} Here we further decompose the term $P_0 \bigl( \epsilon^\dagger \partial_t \phi \bigr)$ into a Littlewood-Paley trichotomy so that we schematically have 
 \begin{equation} \label{equ:bootstrap_non_atomic_kinetic_pep_c_splitting}
  \begin{aligned}
    \int_{I_1} \int_{\bR^2} P_0 \bigl( \partial_\alpha \epsilon^\dagger \partial^\alpha \phi \bigr) P_0 \bigl( \epsilon^\dagger \partial_t \phi \bigr) \, dx \, dt &= \int_{I_1} \int_{\bR^2} P_0 \bigl( \partial_\alpha \epsilon^\dagger \partial^\alpha \phi \bigr) P_0 \bigl( \epsilon_{\leq -10}^\dagger \partial_t \phi_0 \bigr) \, dx \, dt \\
    &\quad + \int_{I_1} \int_{\bR^2} P_0 \bigl( \partial_\alpha \epsilon^\dagger \partial^\alpha \phi \bigr) P_0 \bigl( \epsilon_0^\dagger \partial_t \phi_{\leq -10} \bigr) \, dx \, dt \\
    &\quad + \sum_{k_1 > 0} \int_{I_1} \int_{\bR^2} P_0 \bigl( \partial_\alpha \epsilon^\dagger \partial^\alpha \phi \bigr) P_0 \bigl( \epsilon_{k_1}^\dagger \partial_t \phi_{k_1} \bigr) \, dx \, dt.
  \end{aligned}
 \end{equation}
 
 For the first term on the right-hand side of~\eqref{equ:bootstrap_non_atomic_kinetic_pep_c_splitting} for high-low interactions in $P_0 \bigl( \partial_\alpha \epsilon^\dagger \partial^\alpha \phi \bigr)$ (and similarly for low-high interactions) we can proceed as we did for the third term in~\eqref{equ:bootstrap_non_atomic_kinetic_epp}, where now $\partial_\alpha \epsilon^\dagger \partial^\alpha \phi$ replaces $\partial_\alpha \phi^\dagger \partial^\alpha \phi$. Correspondingly, the divisibility part to achieve smallness has to be modified a bit, and in fact has to come from the last factor $\partial_t \phi_0$ which we place in the local energy decay space. Specifically, we arrive at the expression
 \[
  \biggl( \sum_{\ell \in \bZ} (1 + \ell^2)^{-1} \bigl\| \chi_{\{ r \sim 2^\ell \}} r^{-\frac{1}{2}} \partial_t \phi_0 \bigr\|_{L^2_t L^2_x[I_1]}^2 \biggr)^{\frac{1}{2}},
 \]
 which has the divisibility property here. For high-high interactions in $P_0 \bigl( \partial_\alpha \epsilon^\dagger \partial^\alpha \phi \bigr)$ one again has to integrate by parts to move a derivative $\partial_\alpha$ away from $\epsilon$.
 
 The second term on the right-hand side of~\eqref{equ:bootstrap_non_atomic_kinetic_pep_c_splitting} can be treated easily using the null form estimate~\eqref{equ:null_form_estimate} and Strichartz estimates in case of high-low and low-high interactions in $P_0 \bigl( \partial_\alpha \epsilon^\dagger \partial^\alpha \phi \bigr)$, while for high-high interactions one has to integrate by parts again.
 
 Finally, the third term on the right-hand side of~\eqref{equ:bootstrap_non_atomic_kinetic_pep_c_splitting} can be dealt with analogously to {\it (ii.2)}.
 
 \noindent {\it (ii.4)} In this case we can essentially proceed as in {\it (ii.3)} only that it is easier to achieve smallness thanks to the extra factor $\epsilon$.
 
 \medskip
 
 \noindent {\it (iii) Contribution of $P_0 \bigl( \epsilon \partial_\alpha \epsilon^\dagger \partial^\alpha \epsilon \bigr)$.} As usual we split into 
 \begin{align*}
  \int_{I_1} \int_{\bR^2} P_0 \bigl( \epsilon \partial_\alpha \epsilon^\dagger \partial^\alpha \epsilon \bigr)^\dagger \partial_t \epsilon_0 \, dx \,dt &= \int_{I_1} \int_{\bR^2} P_0 \bigl( \epsilon _{\leq -10} \partial_\alpha \epsilon^\dagger \partial^\alpha \epsilon \bigr)^\dagger \partial_t \epsilon_0 \, dx \,dt \\
  &\quad + \sum_{k_1 > -10} \sum_{k_2 = k_1 + O(1)} \int_{I_1} \int_{\bR^2} P_0 \bigl( P_{k_1} \epsilon  P_{k_2} ( \partial_\alpha \epsilon^\dagger \partial^\alpha \epsilon ) \bigr)^\dagger \partial_t \epsilon_0 \, dx \,dt.
 \end{align*}
 For the first term on the right-hand side we may proceed as at the end of {\it (i)} to infer the desired bound. For the second term we distinguish high-high interactions in $P_{k_2} ( \partial_\alpha \epsilon^\dagger \partial^\alpha \epsilon )$ from high-low and low-high interactions. In the former case we place $P_{k_2} (\partial_\alpha \epsilon^\dagger \partial^\alpha \epsilon)$ into $L^{\frac{3}{2}}_t L^{\frac{3}{2}}_x$ and use the null form estimate~\eqref{equ:null_form_estimate} to infer
 \begin{align*}
  &\biggl| \sum_{k_1 > -10} \sum_{k_2 = k_1 + O(1)} \int_{I_1} \int_{\bR^2} P_0 \bigl( P_{k_1} \epsilon  P_{k_2} ( \partial_\alpha \epsilon^\dagger \partial^\alpha \epsilon ) \bigr)^\dagger \partial_t \epsilon_0 \, dx \,dt \biggr| \\
  &\lesssim \sum_{k_3 > -10} \| P_{k_3} \epsilon \|_{S_{k_1}[I_1]}^2 \sum_{k_1 > -10} \| P_{k_1} \epsilon \|_{L^3_t L^\infty_x[I_1]} \| \partial_t \epsilon_0 \|_{L^\infty_t L^3_x[I_1]} \\
  &\lesssim \varepsilon_0 \sum_{k_1 > -10} 2^{-\frac{1}{3} k_1} \| P_{k_1} \epsilon \|_{S_{k_1}[I_1]} \| \partial_t \epsilon_0 \|_{S_0[I_1]}. 
 \end{align*}
 By further distinguishing the frequency regimes and exploiting the smallness and decay properties \eqref{equ:bootstrap_non_atomic_envelope_bound1}--\eqref{equ:bootstrap_non_atomic_envelope_bound3} of $\{ c_k \}_{k\in\bZ}$ and $\{ d_k \}_{k\in\bZ}$, we find that
 \begin{align*}
  &\sum_{k_1 > -10} 2^{-\frac{1}{3} k_1} \| P_{k_1} \epsilon \|_{S_{k_1}[I_1]} \| \partial_t \epsilon_0 \|_{S_0[I_1]} \\
  &\lesssim \sum_{-10 < k_1 \leq b_L} 2^{-(\frac{1}{3} - \sigma) k_1} C^2 \delta_2^2 2^{-2\sigma b_L} + \sum_{b_L < k_1 \leq b_L + M} 2^{-\frac{1}{3} k_1} c_{k_1} C \delta_2 2^{-\sigma b_L} + \sum_{k_1 > b_L + M} 2^{-\frac{1}{3} k_1} c_{k_1} C \delta_2 2^{-\sigma b_L} \\
  &\lesssim C^2 \delta_2^2 2^{-2 \sigma b_L} + \biggl( \sum_{b_L < k_1 \leq b_L+M} c_{k_1}^2 \biggr)^{\frac{1}{2}} \delta_2 2^{-(\sigma + \frac{1}{3}) b_L} + 2^{-\frac{1}{3} M} \varepsilon_0 C \delta_2 2^{-(\sigma + \frac{1}{3}) b_L} \\
  &\lesssim C^2 \delta_2^2 2^{-2 \sigma b_L}. 
 \end{align*}
 In total, for sufficiently small $\varepsilon_0 > 0$ we obtain the desired bound
 \[
  \biggl| \sum_{k_1 > -10} \sum_{k_2 = k_1 + O(1)} \int_{I_1} \int_{\bR^2} P_0 \bigl( P_{k_1} \epsilon  P_{k_2} ( \partial_\alpha \epsilon^\dagger \partial^\alpha \epsilon ) \bigr)^\dagger \partial_t \epsilon_0 \, dx \,dt \biggr| \lesssim \varepsilon_0 C^2 \delta_2^2 2^{-2 \sigma b_L} \ll C^2 \delta_2^2 2^{-2 \sigma b_L}.
 \]
 In the case of high-low or low-high interactions, we place $P_{k_2} ( \partial_\alpha \epsilon^\dagger \partial^\alpha \epsilon )$ into $L^{\frac{3}{2}+}_t L^{\frac{3}{2}+}_x$ and then proceed similarly as above.
 
 \medskip
 
 \noindent {\it (iv) Contributions of $P_0 \bigl( \epsilon \partial_\alpha \epsilon^\dagger \partial^\alpha \phi \bigr)$ and $P_0 \bigl( \phi \partial_\alpha \epsilon^\dagger \partial^\alpha \epsilon \bigr)$}: These terms do not offer anything new and can be treated similarly as above, which finishes the proof of Lemma~\ref{lem:bootstrap_non_atomic_kinetic}.
\end{proof}

Next, we recover the full envelope bound. We may assume that $k=0$ and now want to infer the improved bound $\| \epsilon_0 \|_{S_0[I_1]} \leq \frac{C}{2} (c_0 + d_0)$. To this end we will use the ``renormalized difference equation'' for $\epsilon_0$ given by
\begin{equation*}
 \Box \bigl( U_{\leq -10}^{(\phi + \epsilon)} \bar{(\phi + \epsilon)}_0 \bigr) - \Box \bigl( U_{\leq -10}^{(\phi)} \bar{\phi}_0 \bigr) \equiv F_0^{(\phi+\epsilon)} - F_0^{(\epsilon)},
\end{equation*}
where $U_{\leq h}^{(\phi + \epsilon)}$ and $U_{\leq h}^{(\phi)}$ denote the gauge transformations defined by $\phi + \epsilon$, respectively $\phi$. We will proceed in several steps which we briefly sketch before we turn to the details.
\begin{itemize}
 \item[(a)] First we show that for some $D_1 \gg 1$ with $1 \ll D_0 \ll D_1 \ll C$, we have for any $0 \leq \alpha < 1$ that
  \begin{equation} \label{equ:bootstrap_non_atomic_L1L2}
  \sum_{k_0 \in \bZ} 2^{\alpha |k_0|} \, \bigl\| P_{k_0} \Box \bigl( U_{\leq -10}^{(\phi+\epsilon)} \bar{(\phi + \epsilon)}_0 - U_{\leq -10}^{(\phi)} \bar{\phi}_0 \bigr) \bigr\|_{L^1_t L^2_x[I_1]} \leq D_1 (c_0 + d_0).
 \end{equation}
 Note that here it is again necessary to control the above weighted sum over all frequency outputs since the frequency localizations of $\bar{(\phi + \epsilon)}_0$ and $\bar{\phi}_0$ to frequency $\sim 1$ are only approximately preserved by the gauge transformations up to exponentially decaying tails. By the energy estimate~\eqref{equ:energy_estimate} we can then conclude that
 \begin{equation} \label{equ:bootstrap_non_atomic_gauged_difference_S_norm}
  \sum_{k_0 \in \bZ} 2^{\alpha |k_0|} \, \bigl\| P_{k_0} \bigl( U_{\leq -10}^{(\phi + \epsilon)} \bar{(\phi+\epsilon)}_0 - U_{\leq -10}^{(\phi)} \bar{\phi}_0 \bigr) \bigr\|_{S_{k_0}[I_1]} \lesssim D_1 ( c_0 + d_0 ).
 \end{equation}
 \item[(b)] Next we show that we have
 \begin{equation} \label{equ:bootstrap_non_atomic_difference_Us_S_norm}
  \sum_{k_0 \in \bZ} 2^{\alpha |k_0|} \, \bigl\| P_{k_0} \bigl( U_{\leq -10}^{(\phi+\epsilon)} - U_{\leq -10}^{(\phi)} \bigr) \bar{\phi}_0 \bigr\|_{S_{k_0}[I_1]} \lesssim D_1 ( c_0 + d_0 ).
 \end{equation}
 This step crucially relies on the improved bounds on the kinetic energy of $\epsilon_k$ for the low frequencies $k < b_L$, which we have already established in Lemma~\ref{lem:bootstrap_non_atomic_kinetic}.
  \item[(c)] Combining \eqref{equ:bootstrap_non_atomic_gauged_difference_S_norm} and \eqref{equ:bootstrap_non_atomic_difference_Us_S_norm} we immediately obtain
  \begin{equation} \label{equ:bootstrap_non_atomic_difference_phi_S_norm}
   \sum_{k_0 \in \bZ} 2^{\alpha |k_0|} \, \bigl\| P_{k_0} U_{\leq -10}^{(\phi+\epsilon)} \bigl( \bar{(\phi + \epsilon)}_0 - \bar{\phi}_0 \bigr) \bigr\|_{S_{k_0}[I_1]} \lesssim D_1 ( c_0 + d_0 ).
  \end{equation}
  \item[(d)] Then we write
  \[
   \bar{(\phi + \epsilon)}_0 - \bar{\phi}_0 = \bigl( U_{\leq -10}^{(\phi+\epsilon)} \bigr)^\dagger U_{\leq -10}^{(\phi+\epsilon)} \bigl( \bar{(\phi + \epsilon)}_0 - \bar{\phi}_0 \bigr) 
  \]
  and decompose into a Littlewood-Paley trichotomy. Using the bound~\eqref{equ:bootstrap_non_atomic_difference_phi_S_norm} we may then proceed as in the proof of Proposition~\ref{prop:partial_fungibility} to infer that for some $D_2 \gg 1$ with $1 \ll D_0 \ll D_1 \ll D_2~\ll~C$,
  \begin{equation} \label{equ:bootstrap_non_atomic_difference_phi_bar_S_norm}
   \bigl\| \bar{(\phi + \epsilon)}_0 - \bar{\phi}_0 \bigr\|_{S_0[I_1]} \leq D_2 ( c_0 + d_0 ).
  \end{equation}
  Here we do not have to work with the modified version of the gauge transformations as introduced in Proposition~\ref{prop:partial_fungibility} because we already have that $\| \phi \|_{S[I_1]} \leq C(E_{crit})$. We will correspondingly not provide further details on this step.
  \item[(e)] Finally, we invoke the bound \eqref{equ:bootstrap_non_atomic_difference_phi_bar_S_norm} and exploit the inherently multilinear structure of 
  \[
   \bar{(\phi + \epsilon)}_0 - \bar{\phi}_0 - \epsilon_0
  \]
  to conclude the desired improved bound
  \[
   \| \epsilon_0 \|_{S_0[I_1]} \leq \frac{C}{2} ( c_0 + d_0 ).
  \]
\end{itemize}
We now provide the details of the key steps (a), (b) and (e) in separate lemmas, starting with (a).

\begin{lem}
 If $\epsilon$ satisfies the bootstrap assumption~\eqref{equ:bootstrap_assumption_non_atomic} for some $C \equiv C(E_{crit}) \gg 1$ sufficiently large, then we have for some $1 \ll D_1 \ll C$ that for all $0 \leq \alpha < 1$,
 \begin{equation} \label{equ:bootstrap_non_atomic_L1L2_lemma}
  \sum_{k_0 \in \bZ} 2^{\alpha |k_0|} \, \bigl\| P_{k_0} \Box \bigl( U_{\leq -10}^{(\phi+\epsilon)} \bar{(\phi + \epsilon)}_0 - U_{\leq -10}^{(\phi)} \bar{\phi}_0 \bigr) \bigr\|_{L^1_t L^2_x[I_1]} \leq D_1 (c_0 + d_0).
 \end{equation}
\end{lem}
\begin{proof}
 We begin by observing that it suffices to prove 
 \begin{equation} \label{equ:bootstrap_non_atomic_L1L2_freq0}
  \bigl\| \Box \bigl( U_{\leq -10}^{(\phi+\epsilon)} \bar{(\phi + \epsilon)}_0 - U_{\leq -10}^{(\phi)} \bar{\phi}_0 \bigr) \bigr\|_{L^1_t L^2_x[I_1]} \leq D_1 (c_0 + d_0).
 \end{equation}
 Then the bound on the weighted sum \eqref{equ:bootstrap_non_atomic_L1L2_lemma} over all output frequencies can be established as in the proof of Proposition~\ref{prop:partial_fungibility}. 
 
 For each characteristic component of the renormalized wave maps nonlinearity as summarized in Proposition~\ref{prop:nlw_for_renormalized_phi_bar_schematic_identities}, we now have to establish the estimate~\eqref{equ:bootstrap_non_atomic_L1L2_freq0} for the difference of the corresponding expression for $\phi + \epsilon$ and $\phi$. Here we show this in detail for the first term on the right-hand side of~\eqref{equ:nlw_for_renormalized_phi_bar_term1}, namely when a derivative falls on the lowest frequency; for all other terms one has to proceed similarly. We are therefore now facing a difference term of the schematic form
 \begin{equation*}
  U_{\leq -10}^{(\phi+\epsilon)} \sum_{k_2 \leq -10} (\phi + \epsilon)_{k_2-10 < \cdot \leq -10} \partial_\alpha (\phi + \epsilon)_{k_2}^\dagger \partial^\alpha (\phi+\epsilon)_0 - U_{\leq -10}^{(\phi)} \sum_{k_2 \leq -10} \phi_{k_2-10 < \cdot \leq -10} \partial_\alpha \phi_{k_2}^\dagger \partial^\alpha \phi_0.
 \end{equation*}
 This leads to a number of expressions which we are now estimating separately.
 
 \medskip 
 
 \noindent {\it (i) Contribution of $\bigl( U_{\leq -10}^{(\phi + \epsilon)} - U_{\leq -10}^{(\phi)} \bigr) \sum_{k_2 \leq -10} \phi_{k_2-10 < \cdot \leq -10} \partial_\alpha \phi_{k_2}^\dagger \partial^\alpha \phi_0$.} We distinguish further between the relation of the frequency scale $0$ to $b_L$, which is the upper endpoint of the ``essential'' frequency support of~$\phi$.
 
 \medskip 
 
 \noindent {\it (i.1) $0 \geq b_L$.} This case is straightforward on account of the smallness $\|P_k \phi\|_{S_k} \lesssim \delta_1 2^{-\sigma (k-b_L)}$ for $k > b_L$. We simply use the exact orthogonality of the gauge transformations and the trilinear estimate~\eqref{equ:trilinear_estimate} to bound
 \begin{align*}
  \Bigl\| \bigl( U_{\leq -10}^{(\phi + \epsilon)} - U_{\leq -10}^{(\phi)} \bigr) \sum_{k_2 \leq -10} \phi_{k_2-10 < \cdot \leq -10} \partial_\alpha \phi_{k_2}^\dagger \partial^\alpha \phi_0 \Bigr\|_{L^1_t L^2_x[I_1]} &\lesssim \sum_{k_2 \leq -10} \| \phi_{k_2} \|_{S_{k_2}[I_1]}^2 \|\phi_0\|_{S_0[I_1]} \\
  &\lesssim C(E_{crit}) \delta_1 2^{+\sigma b_L} \\
  &\ll \delta_2 2^{+\sigma b_L}.
 \end{align*}
 
 \medskip 
 
 \noindent {\it (i.2) $0 < b_L$.} Here we have to argue more carefully, since we now have to gain smallness from the difference of the gauge transformations $U_{\leq -10}^{(\phi+\epsilon)} - U_{\leq -10}^{(\phi)}$. By the defining ODE \eqref{equ:ode_gauge_transform} for the gauge transformations we have that 
 \[
  U_{\leq -10}^{(\phi+\epsilon)} - U_{\leq -10}^{(\phi)} = \int_{-\infty}^{-10} U_{\leq h}^{(\phi+\epsilon)} \bigl( B_h^{(\phi+\epsilon)} - B_h^{(\phi)} \bigr) \, dh + \int_{-\infty}^{-10} \bigl( U_{\leq h}^{(\phi+\epsilon)} - U_{\leq h}^{(\phi)} \bigr) B_h^{(\phi)} \, dh.
 \]
 By further expanding the second term on the right-hand side, we obtain an expansion containing only terms involving differences $B_h^{(\phi+\epsilon)} - B_h^{(\phi)}$, namely
 \begin{equation} \label{equ:bootstrap_non_atomic_difference_gauge_transformations}
  U_{\leq -10}^{(\phi+\epsilon)} - U_{\leq -10}^{(\phi)} = \sum_{n=1}^\infty \int_{-\infty}^{-10} \int_{-\infty}^{h_1} \cdots \int_{-\infty}^{h_{n-1}} U_{\leq h_n}^{(\phi+\epsilon)} \bigl( B_{h_n}^{(\phi+\epsilon)} - B_{h_n}^{(\phi)} \bigr) B_{h_{n-1}}^{(\phi)} \cdots B_{h_1}^{(\phi)} \, dh_n \ldots dh_1.
 \end{equation}
 Since the repeated expansion leads to difference terms of lower and lower frequencies, this infinite sum is convergent due to the assumed exponential decay of $d_k$ for $k < b_L$. In view of the definition~\eqref{equ:matrix_gauge_transform} of $B_h^{(\phi+\epsilon)}$ and $B_h^{(\phi)}$, we may therefore replace \eqref{equ:bootstrap_non_atomic_difference_gauge_transformations} by the schematic expression 
 \[
  \sum_{k_1 \leq -10}  L(\phi_{\leq k_1-10}, \epsilon_{k_1}) + L(\epsilon_{\leq k_1-10}, \phi_{k_1}) + L(\epsilon_{\leq k_1-10}, \epsilon_{k_1}).
 \]
 Noting that by a simple divisibility argument we may assume that
 \[
  \Bigl\| \sum_{k_2 \leq -10} \phi_{k_2-10 < \cdot \leq -10} \partial_\alpha \phi_{k_2}^\dagger \partial^\alpha \phi_0 \Bigr\|_{L^1_t L^2_x[I_1]} \ll 1,
 \]
 we find that
 \begin{align*}
  &\Bigl\| \bigl( U_{\leq -10}^{(\phi + \epsilon)} - U_{\leq -10}^{(\phi)} \bigr) \sum_{k_2 \leq -10} \phi_{k_2-10 < \cdot \leq -10} \partial_\alpha \phi_{k_2}^\dagger \partial^\alpha \phi_0 \Bigr\|_{L^1_t L^2_x[I_1]} \\
  &\lesssim \Bigl\| \Bigl( \sum_{k_1 \leq -10}  L(\phi_{\leq k_1-10}, \epsilon_{k_1}) + L(\epsilon_{\leq k_1-10}, \phi_{k_1}) + L(\epsilon_{\leq k_1-10}, \epsilon_{k_1}) \Bigr) \sum_{k_2 \leq -10} \phi_{k_2-10 < \cdot \leq -10} \partial_\alpha \phi_{k_2}^\dagger \partial^\alpha \phi_0 \Bigr\|_{L^1_t L^2_x[I_1]} \\
  &\lesssim \sum_{k_1 \leq -10} \| \epsilon_{k_1} \|_{L^\infty_t L^\infty_x[I_1]} \, \Bigl\| \sum_{k_2 \leq -10} \phi_{k_2-10 < \cdot \leq -10} \partial_\alpha \phi_{k_2}^\dagger \partial^\alpha \phi_0 \Bigr\|_{L^1_t L^2_x[I_1]} \\
  &\lesssim \sum_{k_1 \leq -10} \delta_2 2^{-\sigma (b_L - k_1)} \, \Bigl\| \sum_{k_2 \leq -10} \phi_{k_2-10 < \cdot \leq -10} \partial_\alpha \phi_{k_2}^\dagger \partial^\alpha \phi_0 \Bigr\|_{L^1_t L^2_x[I_1]} \\
  &\ll \delta_2 2^{-\sigma b_L}.
 \end{align*}

 \medskip
 
 \noindent {\it (ii) Contribution of $U_{\leq -10}^{(\phi+\epsilon)} \sum_{k_2 \leq -10} \epsilon_{k_2-10 < \cdot \leq -10} \partial_\alpha \phi_{k_2}^\dagger \partial^\alpha \phi_0$.} Again we further distinguish between the relation of the frequency scale $0$ and $b_L$.
 
 \medskip 
 
 \noindent {\it (ii.1) $0 \geq b_L$.} Here we immediately obtain from the smallness $\|\phi_k\|_{S_k} \lesssim \delta_1 2^{-\sigma (k-b_L)}$ for $k > b_L$ and the trilinear estimate~\eqref{equ:trilinear_estimate} the desired bound 
 \begin{align*}
  \Bigl\| U_{\leq -10}^{(\phi+\epsilon)} \sum_{k_2 \leq -10} \epsilon_{k_2-10 < \cdot \leq -10} \partial_\alpha \phi_{k_2}^\dagger \partial^\alpha \phi_0 \Bigr\|_{L^1_t L^2_x[I_1]} &\lesssim \Bigl( \sum_{k_1} \|\epsilon_{k_1}\|_{S_{k_1}[I_1]}^2 \Bigr)^{\frac{1}{2}} \Bigl( \sum_{k_2} \|\phi_{k_2}\|_{S_{k_2}[I_1]}^2 \Bigr)^{\frac{1}{2}} \|\phi_0\|_{S_0[I_1]} \\
  &\lesssim \varepsilon_0^{\frac{1}{2}} C(E_{crit}) \delta_1 2^{+\sigma b_L} \\
  &\ll \delta_2 2^{+\sigma b_L}.
 \end{align*}
 
 \medskip 
 
 \noindent {\it (ii.2) $0 < b_L$.} By the trilinear estimate~\eqref{equ:trilinear_estimate} here we have 
 \begin{align*}
  \Bigl\| U_{\leq -10}^{(\phi+\epsilon)} \sum_{k_2 \leq -10} \epsilon_{k_2-10 < \cdot \leq -10} \partial_\alpha \phi_{k_2}^\dagger \partial^\alpha \phi_0 \Bigr\|_{L^1_t L^2_x[I_1]} &\lesssim \sum_{k_2 \leq -10} \sum_{k_2-10 < k_1 \leq -10} 2^{-\frac{1}{4} (k_1-k_2)} \| \epsilon_{k_1} \|_{S_{k_1}[I_1]} \|\phi_{k_2}\|_{S_{k_2}[I_1]} \|\phi_0\|_{S_0[I_1]} \\
  &\lesssim \delta_2 2^{-\sigma b_L} C(E_{crit})
 \end{align*}
 and then one forces smallness via a divisibility argument. 
 
 \medskip
 
 \noindent {\it (iii) Contribution of $U_{\leq -10}^{(\phi + \epsilon)} \sum_{k_2 \leq -10} \epsilon_{k_2-10 < \cdot \leq -10} \partial_\alpha \epsilon_{k_2}^\dagger \partial^\alpha \epsilon_0$.} In this case the smallness comes from the bound $\sum_{k \in \bZ} c_k^2 \lesssim \varepsilon_0$. Indeed, by the trilinear estimate we obtain 
 \begin{align*}
  \Bigl\| U_{\leq -10}^{(\phi + \epsilon)} \sum_{k_2 \leq -10} \epsilon_{k_2-10 < \cdot \leq -10} \partial_\alpha \epsilon_{k_2}^\dagger \partial^\alpha \epsilon_0 \Bigr\|_{L^1_t L^2_x[I_1]} &\lesssim \Bigl( \sum_{k_2 \leq -10} \| \epsilon_{k_2} \|_{S_{k_2}[I_1]}^2 \Bigr) \| \epsilon_0 \|_{S_0[I_1]} \\
  &\lesssim \varepsilon_0 C (c_0 + d_0) \\
  &\leq D_1 (c_0 + d_0).
 \end{align*}

 \medskip
 
 \noindent {\it (iv) All other contributions can be estimated similarly to the above cases.}
\end{proof}

We proceed with step (b).

\begin{lem}
 If $\epsilon$ satisfies the bootstrap assumption~\eqref{equ:bootstrap_assumption_non_atomic} for some $C \equiv C(E_{crit}) \gg 1$ sufficently large, then we have for some $1 \ll D_0 \ll D_1 \ll C$ that for all $0 \leq \alpha < 1$ 
 \begin{equation} \label{equ:bootstrap_non_atomic_step_b}
  \sum_{k_0 \in \bZ} 2^{\alpha |k_0|} \bigl\| P_{k_0} \bigl( \bigl( U_{\leq -10}^{(\phi+\epsilon)} - U_{\leq -10}^{(\phi)} \bigr) \bar{\phi}_0 \bigr) \bigr\|_{S_{k_0}[I_1]} \lesssim D_1 (c_0 + d_0).
 \end{equation}
\end{lem}
\begin{proof}
 We have to prove the bound \eqref{equ:bootstrap_non_atomic_step_b} separately for each component of our $S$ norm. To this end we distinguish again between the relationship of the frequency scale $0$ to $b_L$. In case that $b_L < 0$, we exploit the smallness $\| \phi_k \|_{S_k} \lesssim \delta_1 2^{-\sigma (k-b_L)}$ for $k > b_L$. Instead when $b_L \geq 0$, we have to invoke the already improved estimate~\eqref{equ:bootstrap_non_atomic_kinetic} on the kinetic energy of $\epsilon_k$ for the low frequencies $k < b_L$. Moreover, we achieve control over the weighted sum over all frequencies $k_0 \in \bZ$ in \eqref{equ:bootstrap_non_atomic_step_b} by essentially playing out the frequencies using Bernstein's inequality.
\end{proof}

Finally, we turn to step (e). Having established that $\bigl\| \bar{(\phi + \epsilon)}_0 - \bar{\phi}_0 \bigr\|_{S_0[I_1]} \leq D_2 ( c_0 + d_0 )$ for some $D_2 \ll C$, we may now quickly infer the desired improved bound 
\[
 \| \epsilon_0 \|_{S_0[I_1]} \leq \frac{C}{2} ( c_0 + d_0 )
\]
and thus close our bootstrap argument by exploiting the inherently multilinear structure of $\bar{(\phi + \epsilon)}_0 - \bar{\phi}_0 - \epsilon_0$. Indeed, recalling the definition~\eqref{equ:definition_phi_bar} of $\bar{\phi}$, we see that each multilinear expression in $\bar{(\phi + \epsilon)}_0 - \bar{\phi}_0 - \epsilon_0$ contains at least one factor of $\epsilon$. Then we again distinguish the relationship between the frequency scale $0$ and $b_L$ and either exploit the smallness $\|\phi_k\|_{S_k} \lesssim \delta_1 2^{-\sigma (k-b_L)}$ for $k \geq b_L$ or the already improved bound on the kinetic energy of $\epsilon_k$ for $k < b_L$, which completes the final step (e).  

\medskip 

At this point we have thus achieved that for $C \equiv C(E_{crit}) \gg 1$ sufficently large, we may infer by bootstrap on the time interval $I_1$ the improved bound
\begin{equation} \label{equ:non_atomic_bootstrap_result}
 \|P_k \epsilon\|_{S_k[I_1]} \leq \frac{C}{2} (c_k + d_k) \quad \text{for } k \in \bZ.
\end{equation}
We would now like to repeat the above bootstrap argument on the remaining time intervals $I_2, \ldots, I_N$ in order to conclude that $\epsilon$ exists globally in time and satisfies uniform in $n$ global-in-time $S$ norm bounds. As alluded to before, here we encounter the danger that the energy of $\epsilon$ could keep growing as we move to later time intervals. However, by \eqref{equ:non_atomic_bootstrap_result} the frequency profile of $\epsilon$ is essentially preserved along the evolution on the time interval $I_1$. For this reason there cannot actually be much energy transfer between $\epsilon$ and $\phi$, since $\phi$ is exponentially decaying for frequencies $k > b_L$. The next lemma indeed shows that the energy of $\epsilon$ is approximately preserved along the evolution on the interval $I_1$. More precisely, the energy transfer is controlled by the smallness parameter $\delta_2$, which can be chosen arbitrarily small.

\begin{lem}
 Assuming the bounds \eqref{equ:non_atomic_bootstrap_result} on the evolution of $\epsilon$ on $I_1$, it holds that
 \begin{equation}
  \| \nabla_{t,x} \epsilon(t) \|_{L^2_x}^2 \leq \| \nabla_{t,x} \epsilon(0) \|_{L^2_x}^2 + C(E_{crit}) \, \delta_2 \quad \text{for } t \in I_1.
 \end{equation}
\end{lem}
\begin{proof}
By energy conservation for the wave maps evolutions $\phi + \epsilon$ and $\phi$ on $I_1$, we have that 
\[
 \sum_{\alpha = 0}^2 \| \partial_\alpha (\phi + \epsilon)(t) \|_{L^2_x}^2 = \text{const.} \quad \text{for } t \in I_1
\]
and 
\[
 \sum_{\alpha = 0}^2 \| \partial_\alpha \phi(t) \|_{L^2_x}^2 = \text{const.} \quad \text{for } t \in I_1.
\]
Moreover, since it holds that
\[
 \sum_{\alpha=0}^2 \| \partial_\alpha \epsilon(t) \|_{L^2_x}^2 = \sum_{\alpha = 0}^2 \| \partial_\alpha (\phi + \epsilon)(t) \|_{L^2_x}^2 - \sum_{\alpha = 0}^2 \| \partial_\alpha \phi(t) \|_{L^2_x}^2 - 2 \sum_{\alpha = 0}^2 \int_{\bR^2} ( \partial_\alpha \epsilon^\dagger \partial_\alpha \phi )(t,x) \, dx,
\]
it suffices to estimate on the time interval $I_1$,
\[
 \sum_{\alpha = 0}^2 \int_{\bR^2} \partial_\alpha \epsilon^\dagger \partial_\alpha \phi \, dx = \sum_{k\in\bZ} \sum_{\alpha = 0}^2 \int_{\bR^2} P_k \partial_\alpha \epsilon^\dagger \, \tilde{P}_k \partial_\alpha \phi \, dx,
\]
where $\tilde{P}_k$ is a slightly fattened Littlewood-Paley projection such that $\tilde{P}_k P_k = P_k$. Using the bounds~\eqref{equ:non_atomic_bootstrap_result} on $\epsilon$ as well as the exponential decay~\eqref{equ:bootstrap_non_atomic_exponential_decay_phi} of $\phi$ for frequencies $k \geq b_L$, we therefore obtain on $I_1$ that
\begin{align*}
 \bigg| \sum_{\alpha = 0}^2 \int_{\bR^2} \partial_\alpha \epsilon^\dagger \partial_\alpha \phi \, dx \bigg| &\lesssim \sum_{k\in\bZ} \| P_k \nabla_{t,x} \epsilon \|_{L^\infty_t L^2_x[I_1]} \| \tilde{P}_k \nabla_{t,x} \phi \|_{L^\infty_t L^2_x[I_1]} \\
 &\lesssim \sum_{k \in (-\infty, b_L) \cap \bZ} \delta_2 2^{-\sigma (b_L - k)} \| \tilde{P}_k \nabla_{t,x} \phi\|_{L^\infty_t L^2_x[I_1]} + \sum_{k \in [b_L, +\infty) \cap \bZ} (c_k + d_k) \delta_1 2^{-\sigma(k-b_L)} \\
 &\lesssim_{E_{crit}} \delta_2 + \delta_1 \\
 &\lesssim_{E_{crit}} \delta_2.
\end{align*}
\end{proof}

Since the number of time intervals $N \equiv N(C_2, E_{crit})$ is controlled by the size of $C_2$ and $E_{crit}$, we can a priori ensure to choose the underlying Besov error $\delta_0$ so small that the energy of $\epsilon$ never leaves the perturbative regime, i.e. it stays less than, say, $2 \varepsilon_0$. Hence, we can carry out the above bootstrap argument to infer $S$ norm bounds on $\epsilon$ on all remaining time intervals $I_2, \ldots, I_N$. This finishes the proof of Proposition~\ref{prop:bootstrap_bounds_to_next_level_in_evolving_lowest_frequency_nonatomic_part}.
\end{proof}

\subsection{Selecting concentration profiles and adding the first large frequency atom} \label{subsec:adding_in_first_atom}

\subsubsection{Setting up a profile decomposition}

Having established control over the global evolution of the lowest frequency non-atomic part $\Pi_{\leq \mu_n^1 - \log(R_n)} \phi^n[0]$ in the previous subsection, we now ``add in'' the first large frequency atom $\phi^{n1}[0] = P_{[\mu_n^1 - \log(R_n), \mu_n^1 + \log(R_n)]} \phi^n [0]$ in the sense that we now attempt to evolve the data
\[
 \Pi_{\leq \mu_n^1 + \log(R_n)} \phi^n[0].
\]
In this subsection we shall slightly abuse notation and write for simplicity
\[
 \phi^n[0] \equiv \Pi_{\leq \mu_n^1 + \log(R_n)} \phi^n[0].
\]
Moreover, we denote by 
\[
 u^n[0] \equiv \Pi_{\leq \mu_n^1 - \log(R_n)} \phi^n[0]
\]
the data of the lowest frequency non-atomic part and by
\[
  \mu_n \equiv \mu_n^1 - \log(R_n)
\]
the upper frequency cut-off delimiting the lowest frequency non-atomic part. Finally, we shall write 
\[
 \epsilon^n[0] \equiv \Pi_{\leq \mu_n^1 + \log(R_n)} \phi^n[0] - \Pi_{\leq \mu_n^1 - \log(R_n)} \phi^n[0].
\]
By rescaling we may assume that $\mu_n^1 \equiv 0$.

\medskip 

In fact, we shall first attempt to evolve a sequence of slightly modified data 
\[
 \phi^{n,<K_*}[0] \equiv \Pi_{<K_*} \phi^n[0],
\]
which have better frequency localization properties, and thence infer bounds on the evolutions of the original data $\Pi_{\leq \mu_n^1 + \log(R_n)} \phi^n[0]$. Correspondingly, we also introduce the notation
\[
 \epsilon^{n, <K_*}[0] \equiv \phi^{n, < K_*}[0] - u^n[0].
\]
In order to obtain these bounds, we shall have to take advantage of a suitable profile decomposition, which however is quite delicate to achieve on account of the poor perturbative properties of the wave maps equation expressed in terms of the coordinate functions. 

\medskip 

The very first step to pick the concentration profiles is to correctly identify their asymptotic behavior. To guess the correct ansatz, we shall throughout assume that the evolutions $\phi^n$ of the data $\phi^n[0]$ exist on some large time interval $I$. In particular, this means that the gauge transformations $U_{<k}^{(\phi^n)}$ associated with the evolutions $\phi^n$ are a priori well-defined. At this point it is important to observe that in the trilinear estimate~\eqref{equ:trilinear_estimate} we do not gain exponentially in the difference of the largest to the smallest frequency present. For this reason the low frequencies will exert a non-negligible influence on the high frequencies for large times, which we have to take into account by the correct choice of the concentration profiles. At frequency $k = 0$ we arrive at the equation 
\begin{align*}
 \Box \bigl( U_{<0}^{(\phi^n)} \overline{\phi^n_0} \bigr) = \Box \big( U_{<0}^{(\phi^n)} \overline{(u^{n} + \epsilon^{n})}_0 \big) =: 2 U_{<0}^{(\phi^n)} \cA_{\alpha, low}(u^{n}) \partial^{\alpha} \epsilon_0^{n} + U_{<0}^{(\phi^n)} F_0(u^{n}, \epsilon^{n}),
\end{align*}
where $\cA_{\alpha, low}(u^n)$ incorporates all (perturbative) low frequency terms stemming purely from $u^{n}$ for which there are no exponential frequency gains in the corresponding interactions in the wave maps nonlinearity. In view of the structure of the wave maps nonlinearity as detailed in Proposition~\ref{prop:nlw_for_renormalized_phi_bar_schematic_identities}, we arrive at the expression
\[
 \cA_{\alpha, low}(u^n) = - \sum_{k_2 < \mu_n} \sum_{k_2 - 10 < k_1 < \mu_n} ( u^n_{k_1} ) (\partial_\alpha u_{k_2}^n )^\dagger - (\partial_\alpha u^n_{k_2}) (u_{k_1}^n)^\dagger + \bigl( U_{< \mu_n}^{(u^n)} \bigr)^\dagger \partial_\alpha U_{< \mu_n}^{(u^n)} + A_{\alpha; < \mu_n}^{(u^n)},
\]
where $U^{(u^n)}_{< \mu_n}$ is the gauge transformation defined by $u^n$ and $A_{\alpha; < \mu_n}^{(u^n)}$ is the connection form \eqref{equ:definition_connection_form_A} defined by~$u^n$. Then we may write
\[
 U_{<0}^{(\phi^n)} \cA_{\alpha, low}(u^{n}) \partial^{\alpha} \epsilon_0^{n}  = U_{<0}^{(\phi^n)} \cA_{\alpha, low}(u^{n}) (U_{<0}^{(\phi^n)})^\dagger \partial^{\alpha} (U_{<0}^{(\phi^n)} \epsilon_0^{n}) + error,
\]
where we expect $error$ to be a better term with a derivative on a low frequency term. Furthermore, we expect the high-frequency contribution to $U_{<0}^{(\phi^n)}$ coming from $\epsilon^{n}$ to be negligible, and so we replace the first term on the right by $U_{<0}^{(u^{n})} \cA_{\alpha, low}(u^{n}) (U_{<0}^{(u^{n})})^\dagger \partial^{\alpha} (U_{<0}^{(\phi^n)} \epsilon_0^{n})$. In fact, due to the rapid decay of $u^{n}$ at large frequencies $k > \mu_n$, where we recall that $\mu_n$ denotes the frequency cutoff delimiting the frequency atom $\epsilon^{n}$ from below, we shall replace this by the even more natural $U_{<\mu_n}^{(u^{n})} \cA_{\alpha, low}(u^{n}) (U_{<\mu_n}^{(u^{n})})^\dagger \partial^{\alpha}(U_{<0}^{(\phi^n)}\epsilon_0^{n})$. Importantly, note that this is now no longer dependent on the choice of a specific frequency $k = 0$. Moreover, we observe that to formulate the right equation we have to pass to the variable $\tilde{\epsilon}^{n}_0: = U_{<0}^{(\phi^n)} \epsilon_0^{n}$, and more generally $\tilde{\epsilon}^{n}_k: = U_{<k}^{(\phi^n)} \epsilon_k^{n}$.

\medskip 

Then we shall use the following very natural equation to select our concentration profiles (which in effect is the same for all frequencies and so we shall not indicate a localization here)
\begin{equation} \label{eq:therightone}
 \boxed{ \Box \tilde{\epsilon} = 2 U_{<\mu_n}^{(u^{n})} \cA_{\alpha, low}(u^{n}) \bigl( U_{<\mu_n}^{(u^{n})} \bigr)^\dagger \partial^{\alpha} \tilde{\epsilon}. }
\end{equation}
Observe that the anti-symmetric matrix $U_{<\mu_n}^{(u^{n})} \cA_{\alpha, low}(u^{n}) \bigl( U_{<\mu_n}^{(u^{n})} \bigr)^\dagger$ is defined purely in terms of the low frequency constituent~$u^n$, over which we already have control. 

\medskip 

After these heuristics, we now proceed with the actual selection of the concentration profiles. Here we shall proceed naturally in two steps, first picking the profile at time $t = 0$ for which of course the flow associated with \eqref{eq:therightone} is irrelevant, and then picking the temporally unbounded profiles, using the terminology of \cite{KL} and \cite{KS}. A technical difficulty here consists in identifying a profile which actually maps into the target sphere $\bS^m$. Also, since functions in $\dot{H}^1_x(\R^2)$ are not even distributions, we start by carefully implementing a frequency cutoff away from $-\infty$, and more precisely to the frequency interval $[-K, \infty)$ for some $K \gg 1$. The number $K$ will later on play the role of a parameter which we need to fine tune. Furthermore, we also need to restrict the frequencies from above, below some threshold $K_*$. This we do by passing to the truncated data $\phi^{n,<K_*}[0]$. Once we obtain bounds for the $S$ norms of the evolution that are uniform in $K_*$, we shall be able to invoke a simpler perturbative argument just as in the control of the lowest frequency non-atomic part to infer the desired bound. 

The following lemma provides a first version of a decomposition into a bounded and an unbounded profile.
\begin{lem} \label{lem:prelim profile} 
Writing $\phi^{n,<K_*} = u^{n} + \epsilon^{n,<K_*}$ and passing to a suitable subsequence with respect to $n$, there exists $p_*\in \bS^m$ and for each $K\geq K_0$ and $n$ sufficiently large, a decomposition (all at time $t = 0$)
\begin{align*}
 \epsilon^{n,<K_*} + p_* &= \epsilon_{bounded}^{Kn,<K_*} + \epsilon_{unbounded}^{Kn,<K_*} + \eta^{Kn,<K_*}, \\
 \partial_t \epsilon^{n,<K_*} &= \gamma_{bounded}^{Kn,<K_*} +  \gamma_{unbounded}^{Kn,<K_*} + \zeta^{Kn,<K_*},
\end{align*}
where we have 
\begin{align*}
 \limsup_{K \rightarrow\infty} \, \limsup_{n\rightarrow\infty} \, \Bigl( \big\|\nabla_x\epsilon_{bounded}^{Kn,<K_*}\big\|_{L_x^2}^2 + \big\|\gamma_{bounded}^{Kn,<K_*}\big\|_{L_x^2}^2 \Bigr) &\leq \limsup_{n\rightarrow\infty}\big\|\nabla_{t,x}\epsilon^{n,<K_*}\big\|_{L_x^2}, \\
 \lim_{K\rightarrow\infty} \limsup_{n\rightarrow\infty} \big\|\nabla_x\eta^{Kn,<K_*}\big\|_{L_x^2} &= 0, \\
 \lim_{K\rightarrow\infty} \limsup_{n\rightarrow\infty} \big\|\zeta^{Kn,<K_*}\big\|_{L_x^2} &= 0,
\end{align*}
as well as the asymptotic orthogonality relations 
\begin{align*}
 \lim_{n\rightarrow\infty} \int_{\R^2} \bigl( \nabla_x\epsilon_{bounded}^{Kn,<K_*} \bigr)^\dagger \cdot \nabla_x\epsilon_{unbounded}^{Kn,<K_*}(x)\,dx &= 0, \\
 \lim_{n\rightarrow\infty} \int_{\R^2} \bigl( \gamma_{bounded}^{Kn,<K_*} \bigr)^\dagger \cdot \gamma_{unbounded}^{Kn,<K_*}(x)\,dx &= 0.
\end{align*}
The components $\epsilon_{bounded}^{Kn,<K_*}$ map into $\bS^m$ and the data pair $(\epsilon_{bounded}^{Kn,<K_*}, \gamma_{bounded}^{Kn,<K_*})$ is admissible in the sense that 
\[
 \bigl( \epsilon_{bounded}^{Kn,<K_*} \bigr)^\dagger \cdot  \gamma_{bounded}^{Kn,<K_*} = 0
\]
poinwise. Moreover, the limits $\lim_{K \rightarrow \infty}\nabla_x\epsilon^{Kn,<K_*}_{bounded} =: \nabla_x\epsilon_{bounded}^{<K_*}$ and $\lim_{K\rightarrow\infty}\gamma^{Kn,<K_*}_{bounded} =: \gamma_{bounded}^{<K_*}$ exist in $L^2_x$ independently of $n$ with 
\[
\epsilon_{bounded}^{<K_*} \in L^\infty_x \cap \dot{H}^1_x, \quad \gamma_{bounded}^{<K_*} \in L^2_x.
\]
We also have the compatibility relation $ \bigl( \epsilon_{bounded}^{<K_*} \bigr)^\dagger \cdot \gamma_{bounded}^{<K_*} = 0$ pointwise. Furthermore, it holds that $\epsilon_{unbounded}^{Kn,<K_*} \to 0$ in $L^\infty_{loc}$ as $n\rightarrow\infty$. Finally, we have that $u^{n} - p_* \rightarrow 0$ in $L^\infty_{loc}$ as $n \rightarrow \infty$ and for any $R_0>0$, we have 
\begin{align*}
 \lim_{K\rightarrow\infty}\limsup_{n\rightarrow\infty}\big\|\chi_{B_{R_0}}\eta^{Kn,<K_*}\big\|_{L_x^\infty\cap \dot{H}^1_x} &= 0, \\
 \lim_{K\rightarrow\infty}\limsup_{n\rightarrow\infty}\big\|\chi_{B_{R_0}}\zeta^{Kn,<K_*}\big\|_{L_x^2} &= 0.
\end{align*}
\end{lem} 
\begin{proof} 
The extra parameter $K$ here plays the role of an additional frequency cutoff. Write 
\[
 \epsilon^{n,<K_*} = P_{[-K, \infty)} \epsilon^{n,<K_*} + P_{(-\infty, -K)} \epsilon^{n,<K_*}. 
\]
The sequence $\bigl\{ P_{[-K, \infty)} \epsilon^{n} \bigr\}_n$ being bounded in $H^1_x(\R^2)$, passing to a subsequence we may pick a weak limit $\epsilon^{K,<K_*}_{bounded}$, which is either zero or non-zero. Then upon passing to the subsequence,
\[
 P_{[-K, \infty)} \epsilon^{n,<K_*} - \epsilon_{bounded}^{K,<K_*}
\]
converges weakly toward zero as $n \to \infty$, and in particular, it converges pointwise toward zero (due to the essentially sharp frequency localization). By letting $K$ run through all positive integers and implementing a Cantor diagonal argument to successive subsequences, we can arrange that $P_{[-K,\infty)} \epsilon^{\tilde{K},<K_*}_{bounded} = \epsilon^{K,<K_*}_{bounded}$ for $\tilde{K}>K$, and thence that $\nabla_x \epsilon^{K,<K_*}_{bounded}$ converges in $L^2_x$ as $K \rightarrow \infty$ to, say, $\nabla_x \epsilon^{<K_*}_{bounded}$ with $\epsilon^{<K_*}_{bounded} \in \dot{H}^1_x \cap L^\infty_x$. Similarly, we may assume that $\lim_{K_*\rightarrow\infty}\nabla_x\epsilon^{<K_*}_{bounded}$ exists in $L^2_x$. Now write
\[
 \phi^{n,<K_*} = u^{n} + \epsilon^{K, <K_*}_{bounded} + P_{(-\infty, -K]} \epsilon^{n,<K_*} + \bigl( P_{[-K, \infty)} \epsilon^{n,<K_*} - \epsilon_{bounded}^{K,<K_*} \bigr).
\]
Localizing to a large ball $B_R$ around the origin, notice that (with the error vanishing as $n\rightarrow\infty$)
\begin{align*}
 u^{n} \big|_{B_R} &= p_n + o_{L^\infty_x}(1), \\ 
 \chi_{B_R} \bigl( P_{[-K, \infty)} \epsilon^{n,<K_*} - \epsilon_{bounded}^{K,<K_*} \bigr) &= o_{L^\infty_x}(1). 
\end{align*}
It follows that $\big( \epsilon^{K,<K_*}_{bounded} + P_{(-\infty, -K]} \epsilon^{n,<K_*} \big) \big|_{B_R} + p_n$ is within a $o_{L^\infty_x}(1)$ neighborhood of $\bS^m$. Passing to a subsequence, we may suppose $p_n \rightarrow p_*$ for some $p_* \in \bS^m$, whence $\big( \epsilon^{K,<K_*}_{bounded} + P_{(-\infty, -K]} \epsilon^{n,<K_*} \big) \big|_{B_R} + p_*$ is arbitrarily close to $\bS^m$ provided $n$ is sufficiently large.
Further, observe that 
\[
 \big\| \epsilon^{K,<K_*}_{bounded} \big\|_{L^\infty_x( \{ R \sim R_* \})} \lesssim (2^K R_*)^{-1} + \big\| \nabla \epsilon^{K,<K_*}_{bounded} \big\|_{L^2_x(\{R\sim R_*\})}.
\]
Now for a suitable quantity $\delta(K) \rightarrow 0$ as $K \rightarrow \infty$, put $R_* \geq 2^{-K} \delta(K)^{-\frac12}$. By convergence of $\nabla\epsilon^{K,<K_*}_{bounded}$ in $L^2_x$ as $K\rightarrow\infty$, we get from this 
\[
 \big\| \epsilon^{K,<K_*}_{bounded} \big\|_{L^\infty_x(\{ R>2^{-K}\delta(K)^{-\frac12} \})} \rightarrow 0
\]
as $K\rightarrow\infty$. Also, putting 
\[
 \delta(K,K_*) := \limsup_{n\rightarrow\infty}\big\|\nabla P_{(-\infty, -K]} \epsilon^{n,<K_*}\big\|_{L_x^2},
\]
then indeed we have $\lim_{K\rightarrow\infty}\delta(K,K_*) = 0$, and passing to a subsequence as $K\rightarrow\infty$ we can find $p_1 \in \R^{m+1}$ such that 
\[
 P_{(-\infty, -K]} \epsilon^{n,<K_*} \big|_{B_{R_{*K}} \backslash B_{\frac12 R_{*K}}} = p_1 + o_{L^\infty_x}(1)
\]
as $n, K\rightarrow\infty$, where $R_{*K} := 2^{-K} \delta(K,K_*)^{-\frac12}$, since the variation of the function $P_{(-\infty, -K]} \epsilon^{n,<K_*}$ on $B_{R_{*K}} \backslash B_{\frac12 R_{*K}}$ is $\lesssim \delta(K,K_*)^{\frac12}$ as $n$ is very large. 

It follows that choosing $\chi_{B_{R_{*K}}}$ as before, the expression
\[
\chi_{B_{R_{*K}}} \epsilon^{K,<K_*}_{bounded} + \chi_{B_{R_{*K}}} P_{(-\infty, -K]} \epsilon^{n,<K_*} + (1-\chi_{B_{R_{*K}}}) \cdot p_1 + p_*
\]
converges toward $\bS^m$ as $n, K \rightarrow \infty$. Letting $\Pi$ be the normal projection onto $\bS^m$ and labelling 
\begin{align*}
 \tilde{\epsilon}^{Kn,<K_*}_{bounded} &:= \Pi \big( \chi_{B_{R_{*K}}} \epsilon^{K,<K_*}_{bounded} + \chi_{B_{R_{*K}}} P_{(-\infty, -K]} \epsilon^{n,<K_*} + (1-\chi_{B_{R_{*K}}}) \cdot p_1 + p_* \big), \\
 \epsilon^{Kn,<K_*}_{unbounded} &:= P_{[-K, \infty)}\epsilon^{n,<K_*} - \epsilon^{K,<K_*}_{bounded}, \\
 \eta^{nK,<K_*} &:= (1-\Pi) \big(\chi_{B_{R_{*K}}} \epsilon^{K,<K_*}_{bounded} + \chi_{B_{R_{*K}}} P_{(-\infty, -K]} \epsilon^{n,<K_*} + (1-\chi_{B_{R_{*K}}}) \cdot p_1 + p_*\big) \\
 &\quad \quad + (1- \chi_{B_{R_{*K}}}) (P_{(-\infty, -K]} \epsilon^{n1,<K_*} - p_1) + (1-\chi_{B_{R_{*K}}}) \epsilon^{K,<K_*}_{bounded},
\end{align*}
we infer a representation 
\[
 \epsilon^{n, <K_*} + p_* = \tilde{\epsilon}_{bounded}^{Kn,<K_*} + \epsilon_{unbounded}^{Kn,<K_*} + \eta^{Kn,<K_*}
\]
with all the properties of the lemma (but with $\tilde{\epsilon}_{bounded}^{Kn,<K_*}$ in place of $\epsilon_{bounded}^{Kn,<K_*}$). We carefully observe that 
\[
 (1- \chi_{B_{R_{*K}}}) \bigl( P_{(-\infty, -K]} \epsilon^{n,<K_*} - p_1)  + (1-\chi_{B_{R_{*K}}} \bigr) \epsilon^{K, <K_*}_{bounded} 
\]
is supported outside of $B_{R_0}$ for $K$ large enough, and that 
\[
 (1-\Pi) \bigl( \chi_{B_{R_{*K}}} \epsilon^{K,<K_*}_{bounded} + \chi_{B_{R_{*K}}} P_{(-\infty, -K]} \epsilon^{n,<K_*} + (1-\chi_{B_{R_{*K}}}) \cdot p_1 + p_* \bigr)
\]
converges toward zero in the $L^\infty_x$-norm as $n, K\rightarrow\infty$. It remains to show smallness of the term 
\[
 (1-\Pi) \big(\chi_{B_{R_{*K}}} \epsilon^{K,<K*}_{bounded} + \chi_{B_{R_{*K}}} P_{(-\infty, -K]} \epsilon^{n,<K_*} + (1-\chi_{B_{R_{*K}}}) \cdot p_1 + p_* \big).
\]
with respect to $\dot{H}^1_x$. This is clear by an argument as above, provided we include a cutoff $\chi_{\frac12 B_{R_{*K}}^c}$ in front. Thus, consider now the term 
\begin{align*}
 &\chi_{\frac12 B_{R_{*K}}}(1-\Pi)\big(\chi_{B_{R_{*K}}}\epsilon^{K,<K_*}_{bounded} + \chi_{B_{R_{*K}}} P_{(-\infty, -K]} \epsilon^{n,<K_*} + (1-\chi_{B_{R_{*K}}}) \cdot p_1 + p_* \big) \\
 &= \chi_{\frac12 B_{R_{*K}}}(1-\Pi)\big(\epsilon^{K,<K_*}_{bounded} + P_{(-\infty, -K]} \epsilon^{n,<K_*} + p_* \big) \\
 &= \chi_{\frac12 B_{R_{*K}}}(1-\Pi)\big(\phi^{n,<K_*} - (u^{n} - p_*) - ( P_{[-K,\infty)}\epsilon^{n1,<K_*} - \epsilon^{K,<K_*}_{bounded} ) \big).
\end{align*}
Extending $(1-\Pi)$ smoothly to all of $\R^{m+1}$ and using the same notation for the global operator, we get on account of 
\[
 (1-\Pi) \big( \phi^{n,<K_*} \big) = 0
\]
the following
\begin{align*}
&\Big\| \nabla_{t,x} \chi_{\frac12 B_{R_{*K}}} (1-\Pi) \big(\phi^{n,<K_*} - (u^{n} -p_*) - \big( P_{[-K,\infty)} \epsilon^{n,<K_*} - \epsilon^{K,<K_*}_{bounded} \big) \big) \Big\|_{L_x^2} \\
& = \Big\| \nabla_{t,x} \int_0^1 \chi_{\frac12 B_{R_{*K}}} \partial_s (1-\Pi) \big(\phi^{n,<K_*} - s ( u^{n} - p_* ) - s \big( P_{[-K,\infty)} \epsilon^{n,<K_*} - \epsilon^{K,<K_*}_{bounded} \big) \big) \, ds \Big\|_{L_x^2} \\
&\lesssim \big\| \nabla_{t,x} \big( \chi_{\frac12 B_{R_{*K}}} (u^n - p_*) \big) \big\|_{L_x^2} + \big\| \nabla_{t,x} \big( \chi_{\frac12 B_{R_{*K}}} \big( P_{[-K,\infty)} \epsilon^{n,<K_*} - \epsilon^{K,<K_*}_{bounded} ) \big) \big\|_{L_x^2} \\
&\quad + \big\| \big( \chi_{\frac12 B_{R_{*K}}}(u^{n} - p_*) \big) \big\|_{L_x^\infty} + \big\| \big( \chi_{\frac12 B_{R_{*K}}} ( P_{[-K,\infty)} \epsilon^{n,<K_*} - \epsilon^{K,<K_*}_{bounded} ) \big) \big\|_{L_x^\infty}.
\end{align*}
All terms at the end are easily seen to converge to $0$ as $n\rightarrow\infty$. To see this for the second term, use that 
\begin{align*}
 \big\| \nabla_x(\chi_{\frac12 B_{R_{*K}}}) (P_{[-K,\infty)}\epsilon^{n,<K_*} - \epsilon^{K,<K_*}_{bounded}) \big) \big\|_{L_x^2} \rightarrow 0 
\end{align*}
as $n\rightarrow\infty$ since $P_{[-K,\infty)} \epsilon^{n,<K_*} - \epsilon^{K,<K_*}_{bounded} \rightharpoonup 0$ weakly as $n\rightarrow\infty$. Also we have 
\begin{align*}
 \big\| \chi_{\frac12 B_{R_{*K}}} \nabla_x ( P_{[-K,\infty)}\epsilon^{n,<K_*} - \epsilon^{K,<K_*}_{bounded}) \big) \big\|_{L_x^2} &\leq \big\| \chi_{\frac12 B_{R_{*K}}} \nabla_x( P_{[-K,K]}\epsilon^{n,<K_*} - P_{(-\infty, K]} \epsilon^{K,<K_*}_{bounded} \big) \big\|_{L_x^2} \\ 
 &\quad + \big\|\chi_{\frac12 B_{R_{*K}}} \nabla_x \big( P_{(K, \infty)} \epsilon^{n,<K_*} - P_{(K, \infty)} \epsilon^{K,<K_*}_{bounded} \big) \big\|_{L_x^2}.
\end{align*}
Then the first term on the right-hand side converges to $0$ as $n\rightarrow\infty$ and the second converges to $0$ as $K\rightarrow\infty$ uniformly in $n$. For the term
\[
 \big\| \nabla_{t,x} \big(\chi_{\frac12 B_{R_{*K}}} (u^{n} - p_*) \big) \big\|_{L_x^2},  
\]
use the fact that $\lim_{n\rightarrow\infty} \chi_{\frac12 B_{R_{*K}}}(u^{n} - p_*) = 0$ as well as the evacuation to extremely low frequencies for $u^{n}$ as $n\rightarrow\infty$. We also observe here that since $\nabla_x \tilde{\epsilon}^{Kn,<K_*}_{bounded} \rightarrow \nabla_x \epsilon^{<K_*}_{bounded}$ for any large enough $n$ and $K\rightarrow\infty$, we may indeed pick $\epsilon^{<K_*}_{bounded}$ to map into $\bS^m$. 

\medskip 

As for decomposing the time derivative, let $\gamma^{K,<K_*}_{bounded}$ be a weak limit of $P_{[-K,\infty)} \partial_t \epsilon^{n,<K_*}$ as $n \to \infty$, and then let 
\begin{align*}
 \gamma^{Kn,<K_*}_{unbounded} &:= P_{[-K,\infty)} \partial_t \epsilon^{n,<K_*} - \gamma^{K,<K_*}_{bounded}, \\
 \zeta^{Kn,<K_*} &:= P_{(-\infty, -K]} \partial_t \epsilon^{n,<K_*}.
\end{align*}
We have 
\begin{align*}
0 = \bigl( u^{n} - p_* + \epsilon^{Kn,<K_*}_{bounded} +  \epsilon^{Kn,<K_*}_{unbounded} +  \eta^{Kn,<K_*} \bigr)^\dagger \cdot \bigl( \partial_t u^{n} + \gamma^{K,<K_*}_{bounded} + \gamma^{Kn,<K_*}_{unbounded} + \zeta^{Kn,<K_*} \bigr).
\end{align*}
But then on any bounded set $D$ we have in the pointwise sense that
\[
 \lim_{n\rightarrow\infty} \bigl( u^{n} - p_* +  \epsilon^{Kn,<K_*}_{unbounded} \bigr)^\dagger \cdot \bigl( \partial_t u^{n} + \gamma^{K,<K_*}_{bounded} + \gamma^{Kn,<K_*}_{unbounded} + \zeta^{Kn,<K_*} \bigr) = 0
\]
and similarly (on a bounded set $D$) we have the pointwise limit
\[
 \lim_{K\rightarrow\infty} \, \bigl( \eta^{Kn,<K_*} \bigr)^\dagger \cdot \bigl( \partial_t u^{n} + \gamma^{K,<K_*}_{bounded} + \gamma^{Kn,<K_*}_{unbounded} + \zeta^{Kn,<K_*} \bigr)  = 0.
\]
Also, again exploiting the frequency localization, we get 
\[
 \lim_{n\rightarrow\infty} \, \bigl( \epsilon^{Kn,<K_*}_{bounded} \bigr)^\dagger \bigl( \gamma^{Kn,<K_*}_{unbounded} + \zeta^{Kn,<K_*} \bigr) = 0.
\]
Finally, we conclude that if we put 
\[
 \tilde{\gamma}^{Kn,<K_*}_{bounded} := \gamma^{K,<K_*}_{bounded} - \epsilon^{Kn,<K_*}_{bounded} \, \bigl( \epsilon^{Kn,<K_*}_{bounded} \bigr)^\dagger \cdot  \gamma^{K,<K_*}_{bounded} 
\]
and then write 
\begin{align*}
 \partial_t \epsilon^{n,<K_*} = \tilde{\gamma}^{Kn,<K_*}_{bounded} + \gamma^{Kn,<K_*}_{unbounded}  + \tilde{\zeta}^{Kn,<K_*}
\end{align*}
with 
\[
 \tilde{\zeta}^{Kn,<K_*} = \zeta^{Kn,<K_*} + \epsilon^{Kn,<K_*}_{bounded} \, \bigl( \epsilon^{Kn,<K_*}_{bounded} \bigr)^\dagger \cdot \gamma^{K,<K_*}_{bounded}.
\]
Then we have found the desired representation with $\tilde{\gamma}^{Kn,<K_*}_{bounded}$ in place of $\gamma^{Kn,<K_*}_{bounded}$ and $\tilde{\zeta}^{Kn,<K_*}$ in place of $\zeta^{Kn,<K_*}$. 
\end{proof}

Note from the preceding proof that setting now 
\[
 \epsilon^{Kn,<K_*}_{bounded} = \Pi \Bigl( \chi_{B_{R_{*K}}}\epsilon^{K,<K_*}_{bounded} + \chi_{B_{R_{*K}}}P_{(-\infty, -K]} \epsilon^{n,<K_*} + (1-\chi_{B_{R_{*K}}})\cdot p_1 + p_* \Bigr),
\]
the implied frequency localization from above (up to exponential tails) allows us to conclude the more precise local convergence statement that on any bounded set $D$ we have 
\[
 \lim_{K\rightarrow\infty}\big\|\epsilon^{Kn,<K_*}_{bounded} - \epsilon^{<K_*}_{bounded}\big\|_{L^\infty_x \cap \dot{H}^1_x(D)} = 0
\]
uniformly for all sufficiently large $n$. Similarly, upon passing to a subsequence if necessary, we have that for any such $D$ 
\[
 \lim_{K\rightarrow\infty}\big\|\gamma^{Kn,<K_*}_{bounded} - \gamma^{<K_*}_{bounded}\big\|_{L^2_x(D)} = 0
\]
for a suitable limit function $\gamma^{<K_*}_{bounded}\in L^2_x$. 

\medskip 

In the preceding we have always kept the {\it{upper frequency bound}} $K_*$ fixed. However, we may now increase $K_*$ and only consider subsequences of the sequence of $n$ considered previously. Doing this we may achieve that actually 
\[
 \lim_{K_*\rightarrow\infty} \big( \nabla_x\epsilon^{<K_*}_{bounded}, \gamma^{<K_*}_{bounded} \big)
\]
exists in $L^2_x$. In what follows we shall use the notation $\epsilon^{<K_*}_{bounded}[0] := \big( \epsilon^{<K_*}_{bounded}, \gamma^{<K_*}_{bounded}\big)$.

\medskip 

It now remains to reveal the fine structure of the data pair $\big(\epsilon^{Kn,<K_*}_{unbounded}, \gamma^{Kn,<K_*}_{unbounded}\big)$, i.e. to resolve it into profiles. To this end we shall use the key equation \eqref{eq:therightone}. As we have to work with the gauged variables for each frequency block, we use for each dyadic frequency $k \in \bZ$ the data pair 
\begin{equation} \label{eq:tildeepsilonKn<K}
\tilde{\epsilon}^{Kn,<K_*}_{unbounded,k}[0]: = \Bigl( U_{<k}^{(\phi^{n,<K_*})} \epsilon^{Kn,<K_*}_{unbounded,k}, U_{<k}^{(\phi^{n,<K_*})}\gamma^{Kn,<K_*}_{unbounded} + \partial_t(U_{<k}^{(\phi^{n,<K_*})})\epsilon^{Kn,<K_*}_{unbounded,k} \Bigr).
\end{equation}
To state the decomposition into concentration profiles for the linear evolution, we first need a precise description of solutions to \eqref{eq:therightone}. 
\begin{lem} \label{lem:soloftherightoneasymptotics}
Given radial $\tilde{\epsilon}[0] \in \dot{H}^1_x \times L^2_x(\R^2; \R^{m+1})$, there exists a unique solution $\tilde{\epsilon}^{(n)}(t, \cdot) \in C^0(\R; \dot{H}^1_x \times L^2_x(\R^2; \R^{m+1}))$ to \eqref{eq:therightone} satisfying uniform bounds 
\[
 \|\tilde{\epsilon}^{(n)} \|_{S[\bR]} \leq C\bigl( \|u^{n}\|_{S[\bR]} \bigr) \big\| \tilde{\epsilon}[0] \big\|_{\dot{H}^1_x \times L^2_x(\R^2; \R^{m+1})}.
\]
The solution $\tilde{\epsilon}^{(n)}(t, \cdot)$ vanishes asymptotically in the sense that given $\gamma>0$, there exists a decomposition 
\[
 \tilde{\epsilon}^{(n)} = \tilde{\epsilon}^{(n)}_1 + \tilde{\epsilon}^{(n)}_2
\]
and a time $t_0(\gamma, \|u^{n}\|_{S}, \tilde{\epsilon}[0]) > 0$ such that for all $n$,
\[
\big\|\tilde{\epsilon}^{(n)}_1\big\|_{S[\bR]} < \gamma, \quad \big\|\tilde{\epsilon}^{(n)}_2\big\|_{L^\infty_t L^\infty_x[ (-\infty,-t_0] \cup [t_0, \infty) ]} < \gamma.
\]
Finally, we have asymptotic energy conservation in the sense that 
\[
 \lim_{n\rightarrow\infty} \sup_{t\in \R} \, \Big| \big\|\nabla_{t,x}\tilde{\epsilon}^{(n)}(t, \cdot)\big\|_{L_x^2} - \big\|\tilde{\epsilon}[0]\big\|_{\dot{H}^1_x \times L^2_x} \Big| = 0.
\]
\end{lem}
\begin{proof}
The a priori bounds follow from the divisibility argument used in the proof of Proposition~\ref{prop:partial_fungibility}. Moreover, we can infer that if $\{c_k\}_{k\in \bZ}$ is a sufficiently flat frequency envelope covering the data $\tilde{\epsilon}[0]$, then we obtain 
\[
 \| \tilde{\epsilon}^{(n)}_k \|_{S_k[\bR]} \leq D(\| u^{n} \|_{S[\bR]}) c_k. 
\]
The asymptotic decay follows as in the proof of Proposition~9.20 in~\cite{KS}. Correspondingly, we only turn to the details of the proof of the asymptotic energy conservation. Our reasoning will be closely related to but in certain aspects a much simplified version of the proof of Proposition~9.14 in~\cite{KS}. Thus, let $P_k$ be slightly modified Littlewood-Paley projections with the property that the corresponding cutoffs $\chi_k(\xi)$ on the frequency side satisfy for all $\xi \neq 0$ that
\[
 \sum_{k \in \bZ} \chi_k(\xi)^2 = 1.
\]
In particular, we then have (setting $\tilde{\epsilon}_k = P_k\tilde{\epsilon}$)
\[
\sum_{k \in \bZ} \big\| \nabla_{t,x} \tilde{\epsilon}_k(0, \cdot) \big\|_{L_x^2}^2 = \big\| \nabla_{t,x} \tilde{\epsilon}(0, \cdot) \big\|_{L_x^2}^2. 
\]
In particular, it suffices to prove asymptotic energy conservation for the frequency localized functions $\tilde{\epsilon}_k$. To simplify the notation within this proof, we shall write 
\[
 \cA_{\alpha} \equiv U_{<\mu_n}^{(u^{n})} \cA_{\alpha, low}(u^{n}) \bigl( U_{<\mu_n}^{(u^{n})} \bigr)^\dagger.
\]
Then on localizing \eqref{eq:therightone} to dyadic frequency $k \in \bZ$, we obtain 
\[
\Box \tilde{\epsilon}_k^{(n)} = 2 \cA_{\alpha, <k} \partial^{\alpha} \tilde{\epsilon}_k^{(n)} + F_k^{(n)}, 
\]
where we have $\lim_{n\rightarrow\infty}\sum_k \big\| F_k^{(n)} \big\|_{L_t^1 L_x^2}^2 = 0$. We shall now consider the quantity 
\[
E_k^{(n)}(t): = \frac12 \int_{\R^2} \Big( \big| \partial_t \tilde{\epsilon}_k^{(n)} - \cA_{0,<k} \tilde{\epsilon}_k^{(n)} \big|^2 + \sum_{j = 1,2} \big| \partial_j \tilde{\epsilon}_k^{(n)} - \cA_{j,<k} \tilde{\epsilon}_k^{(n)} \big|^2 \Big) \, dx
\]
and show that uniformly for all $t \in \bR$,
\begin{equation} \label{eq:asymptoticenergyrelation}
 \lim_{n\rightarrow\infty} E_k^{(n)}(t) = \big\| \nabla_{t,x} \tilde{\epsilon}_k(0, \cdot) \big\|_{L_x^2}^2.
\end{equation}
Then on account of the fact that uniformly for all $t \in \bR$,
\[
 \lim_{n\rightarrow\infty} \sum_{k\in\bZ} \big\| \cA_{\alpha,<k} \tilde{\epsilon}_k^{(n)} \big\|_{L_x^2}^2 = 0,
\]
the final conclusion of the lemma follows. In order to show \eqref{eq:asymptoticenergyrelation}, we differentiate $E_k^{(n)}(t)$ and find
\begin{align*}
(E_k^{(n)})'(t) &= \int_{\R^2} \big( \partial_{tt} \tilde{\epsilon}_k^{(n)} - \cA_{0,<k} \partial_t \tilde{\epsilon}_k^{(n)} - \partial_t \cA_{0,<k} \tilde{\epsilon}_k^{(n)} \big)^\dagger \cdot \big( \partial_t \tilde{\epsilon}_k^{(n)} - \cA_{0,<k} \tilde{\epsilon}_k^{(n)} \big) \, dx \\
&\quad + \sum_{j=1,2} \int_{\R^2} \big( \partial_{tj} \tilde{\epsilon}_k^{(n)} - \cA_{j,<k} \partial_t \tilde{\epsilon}_k^{(n)} - \partial_t \cA_{j,<k} \tilde{\epsilon}_k^{(n)} \big)^\dagger \cdot \big( \partial_j \tilde{\epsilon}_k^{(n)} - \cA_{j,<k}\tilde{\epsilon}_k^{(n)} \big) \, dx \\
&\equiv I + II.
\end{align*}
To simplify things below, we first observe the following schematic vanishing relations
\begin{align*}
\lim_{n\rightarrow\infty} \int_{\R^{1+2}} \partial^{\alpha} \cA_{\alpha, <k} \nabla_{t,x} \tilde{\epsilon}_k^{(n)} \tilde{\epsilon}_k^{(n)} \, dx \, dt &= 0, \\
\lim_{n\rightarrow\infty} \int_{\R^{1+2}} (\partial_t \cA_{j, <k} - \partial_j \cA_{0, <k}) \nabla_{t,x} \tilde{\epsilon}_k^{(n)} \tilde{\epsilon}_k^{(n)} \, dx \, dt &= 0, \\
\lim_{n\rightarrow\infty} \int_{\R^{1+2}} (\cA_{<k})^2 \nabla_{t,x} \tilde{\epsilon}_k^{(n)} \tilde{\epsilon}_k^{(n)} \, dx \, dt &= 0, \\
\lim_{n\rightarrow\infty} \int_{\R^{1+2}} (\nabla_{t,x} \cA_{<k}) \cA_{<k} (\tilde{\epsilon}_k^{(n)})^2 \, dx \, dt &= 0.
\end{align*}
Here one may replace $\R^{1+2}$ by $I \times \R^2$ for any time interval $I$, the vanishing relations being uniform in $I$. To see the first of these relations, we write schematically
\[
\partial^{\alpha} \cA_{\alpha, <k} = P_{<k} \Big( U_{<\mu_n}^{(u^{n})} \partial_{\alpha} u^{n} \partial^{\alpha} u^{n} \big(U_{<\mu_n}^{(u^{n})} \big)^\dagger \Big) + \text{cubic terms}, 
\]
where the cubic terms arise upon differentiating $U_{<\mu_n}^{(u^{n})}$ or re-expanding $\Box u^{n}$ using the wave maps equation. Then it is straightforward to place the cubic terms into $L_t^1 L_x^M$ using the Strichartz type norms in our $S$~space in conjunction with Bernstein's inequality. In fact, from the definition of $u^{n}$, the cubic terms live at frequency~$<~\mu_n$ up to asymptotically vanishing terms with respect to $L_t^1 L_x^M$. Then one closes by observing that 
\begin{align*}
 \big\| P_l \bigl( \nabla_{t,x} \tilde{\epsilon}_k^{(n)} \tilde{\epsilon}_k^{(n)} \bigr) \big\|_{L_t^\infty L_x^{1+}} &\lesssim 2^{-\sigma|l-k|} 2^{-(1-)k} \big\| \tilde{\epsilon}_k^{(n)} \big\|_{S_k}^2 \lesssim 2^{-\sigma|l-k|} 2^{-(1-)k} c_k,
\end{align*}
where $\{ c_k \}_{k\in\bZ}$ is a sufficiently flat frequency envelope covering the data $\tilde{\epsilon}[0]$. On the other hand, owing to the null structure, the principal term 
\[
 P_{<k} \Big( U_{<\mu_n}^{(u^{n})} \partial_{\alpha} u^{n} \partial^{\alpha}u^{n} \bigl( U_{<\mu_n}^{(u^{n})} \bigr)^\dagger \Big) 
\]
can be placed into $L_{t,x}^{\frac32+}$, again essentially reduced to frequencies $<\mu_n$, while one uses 
\[
\big\| P_l \bigl( \nabla_{t,x} \tilde{\epsilon}_k^{(n)} \tilde{\epsilon}_k^{(n)} \bigr) \big\|_{L_{t,x}^{3-}} \lesssim 2^{-\sigma|l-k|} 2^{-(0+)k} c_k. 
\]
The second null form 
\[
\int_{\R^2}(\partial_t \cA_{j, <k} - \partial_j \cA_{0, <k}) \nabla_{t,x}\tilde{\epsilon}_k^{(n)}\tilde{\epsilon}_k^{(n)}\,dx
\]
is handled similarly, and since we can similarly bound $\big\| ( \cA_{<k} )^2\big\|_{L_{t,x}^{\frac32+}}$ and $\big\| ( \nabla_{t,x} \cA_{<k} ) \cA_{<k} \big\|_{L_{t,x}^{\frac32+}}$, we also get the remaining vanishing relations. 

It now remains to reduce $(E_k^{(n)})'(t)$ to the expressions in the preceding vanishing relations. To this end we first observe that by the anti-symmetry of $\cA_\alpha$ it holds that
\[
 \bigl( \cA_{0, <k} \partial_t \tilde{\epsilon}_k^{(n)} \bigr)^\dagger \cdot \partial_t \tilde{\epsilon}_k^{(n)} = 0,
\]
and thus for the term $I$ we have 
\[
 I = \int_{\R^2} \bigl( \partial_{tt} \tilde{\epsilon}_k^{(n)} - 2 \cA_{0, <k} \partial_t \tilde{\epsilon}_k^{(n)} - \partial_t \cA_{0, <k} \tilde{\epsilon}_k^{(n)} \bigr)^\dagger \cdot \bigl( \partial_t \tilde{\epsilon}_k^{(n)} - \cA_{0, <k} \tilde{\epsilon}_k^{(n)} \bigr) \, dx + error,
\]
where $error$ refers to terms satisfying the above vanishing relations. Next, we integrate by parts in term $II$ to find that
\begin{align*}
II &= \sum_{j=1,2} \int_{\R^2} \big( \partial_{tj} \tilde{\epsilon}_k^{(n)} - \cA_{j,<k} \partial_t \tilde{\epsilon}_k^{(n)} - \partial_t \cA_{j,<k} \tilde{\epsilon}_k^{(n)} \big)^\dagger \cdot \big( \partial_j \tilde{\epsilon}_k^{(n)} - \cA_{j,<k} \tilde{\epsilon}_k^{(n)} \big) \,dx \\
&= - \sum_{j=1,2} \int_{\R^2} \bigl( \partial_{t} \tilde{\epsilon}_k^{(n)} \bigr)^\dagger \cdot \big( \partial_{jj} \tilde{\epsilon}_k^{(n)} - \cA_{j,<k} \partial_j \tilde{\epsilon}_k^{(n)} - \partial_j \cA_{j,<k} \tilde{\epsilon}_k^{(n)} \big) \, dx \\
&\quad - \sum_{j=1,2} \int_{\R^2} \big( \cA_{j,<k}\partial_t\tilde{\epsilon}_k^{(n)} + \partial_t \cA_{j,<k} \tilde{\epsilon}_k^{(n)} \big)^\dagger \cdot \big( \partial_j \tilde{\epsilon}_k^{(n)} - \cA_{j,<k} \tilde{\epsilon}_k^{(n)} \big) \, dx.
\end{align*}
Then by the anti-symmetry of $\cA_{\alpha}$ we have the relation $\bigl( \cA_{j,<k} \partial_t \tilde{\epsilon}_k^{(n)} \bigr)^\dagger \cdot \partial_j \tilde{\epsilon}_k^{(n)} = - \bigl( \partial_t \tilde{\epsilon}_k^{(n)} \bigr)^\dagger \cdot \cA_{j,<k} \partial_j \tilde{\epsilon}_k^{(n)}$, and may write the preceding further as 
\begin{align*}
 II &= - \sum_{j=1,2} \int_{\R^2} \bigl( \partial_{t} \tilde{\epsilon}_k^{(n)} \bigr)^\dagger \cdot \bigl( \partial_{jj} \tilde{\epsilon}_k^{(n)} - 2 \cA_{j,<k} \partial_j \tilde{\epsilon}_k^{(n)} - \partial_j \cA_{j,<k} \tilde{\epsilon}_k^{(n)} \big) \, dx \\
 &\quad - \sum_{j=1,2} \int_{\R^2} \bigl( \partial_t \cA_{j,<k} \tilde{\epsilon}_k^{(n)} \bigr)^\dagger \cdot\partial_j\tilde{\epsilon}_k^{(n)} \,dx + error.
\end{align*}
We further modify the first term on the right-hand side above to obtain 
\begin{align*}
II &= - \sum_{j=1,2}\int_{\R^2} \bigl( \partial_{t} \tilde{\epsilon}_k^{(n)} - \cA_{0,<k} \tilde{\epsilon}_k^{(n)} \bigr)^\dagger \cdot \bigl( \partial_{jj} \tilde{\epsilon}_k^{(n)} - 2 \cA_{j,<k} \partial_j \tilde{\epsilon}_k^{(n)} - \partial_j \cA_{j,<k} \tilde{\epsilon}_k^{(n)} \bigr) \, dx \\
  &\quad - \sum_{j=1,2} \int_{\R^2} \bigl( \cA_{0,<k} \tilde{\epsilon}_k^{(n)} \bigr)^\dagger \cdot \partial_{jj} \tilde{\epsilon}_k^{(n)} \, dx - \sum_{j=1,2} \int_{\R^2} \bigl( \partial_t \cA_{j,<k} \tilde{\epsilon}_k^{(n)} \bigr)^\dagger \cdot\partial_j\tilde{\epsilon}_k^{(n)} \,dx + error \\
  & =  - \sum_{j=1,2} \int_{\R^2} \big( \partial_{t} \tilde{\epsilon}_k^{(n)} - \cA_{0,<k} \tilde{\epsilon}_k^{(n)} \big)^\dagger \cdot \big( \partial_{jj} \tilde{\epsilon}_k^{(n)} - 2 \cA_{j,<k} \partial_j \tilde{\epsilon}_k^{(n)} - \partial_j \cA_{j,<k} \tilde{\epsilon}_k^{(n)} \big) \, dx \\
  &\quad + \sum_{j=1,2} \int_{\R^2} \big( \partial_j \cA_{0,<k} \tilde{\epsilon}_k^{(n)} \big)^\dagger \cdot \partial_{j} \tilde{\epsilon}_k^{(n)} \, dx - \sum_{j=1,2} \int_{\R^2} \bigl( \partial_t \cA_{j,<k} \tilde{\epsilon}_k^{(n)} \bigr)^\dagger \cdot\partial_j\tilde{\epsilon}_k^{(n)} \,dx + error,
\end{align*}
where we also used that $\big( \cA_{0,<k} \partial_j \tilde{\epsilon}_k^{(n)} \big)^\dagger \cdot \partial_{j} \tilde{\epsilon}_k^{(n)} = 0$. Putting things together, we find that 
\begin{align*}
(E_k^{(n)})'(t) &= \int_{\R^2} \big( - \Box \tilde{\epsilon}_k^{(n)} + 2 \cA_{\alpha,<k} \partial^{\alpha} \tilde{\epsilon}_k^{(n)} + \partial^\alpha \cA_{\alpha,<k} \tilde{\epsilon}_k^{(n)} \big)^\dagger \cdot \big( \partial_t \tilde{\epsilon}_k^{(n)} - \cA_{0,<k} \tilde{\epsilon}_k^{(n)} \big) \, dx \\
&\quad + \sum_{j=1,2} \int_{\R^2} \big( ( \partial_j \cA_{0,<k} - \partial_t \cA_{j, <k} ) \tilde{\epsilon}_k^{(n)} \big)^\dagger \cdot \partial_{j} \tilde{\epsilon}_k^{(n)} \, dx + error
\end{align*}
and the asymptotic energy conservation follows by using the equation satisfied by $\tilde{\epsilon}_k^{(n)}$ as well as the above vanishing relations.
\end{proof}

We now turn to the process of extracting linear concentration profiles with respect to the wave operator
\[
 \tilde{\Box}_{A^n} := \Box - 2 U_{< \mu_n}^{(u^{n})} \cA_{\alpha, low}(u^{n}) \big( U_{<\mu_n}^{(u^{n})} \big)^\dagger \partial^{\alpha}.
\]
In the following we use notation borrowed from \cite{Bahouri_Gerard}, \cite{KS} and \cite{KL}. We denote by $S_{A^n}$ the wave propagator associated with $\tilde{\Box}_{A^n}$ and shall say that a sequence of data pairs $\tilde{\epsilon}^n[0] \colon \R^2 \rightarrow \R^{m+1} \times \R^{m+1}$ is {\it essentially supported at frequency $\sim 1$}, provided it holds that
\[
 \lim_{R \rightarrow \infty} \limsup_{n\rightarrow\infty} \big\| P_{[-R, R]^c} \tilde{\epsilon}^n[0] \big\|_{\dot{H}^1_x \times L^2_x} = 0. 
\]
Given a sequence of data $\{ \tilde{\epsilon}^n[0] \}_{n \geq 1}$ essentially supported at frequency $\sim 1$ and uniformly bounded in $\dot{H}^1_x(\R^2; \R^{m+1}) \times L^2_x(\R^2; \R^{m+1})$, we introduce the set 
\begin{align*}
 &\mathcal{U}_{A^n}(\tilde{\epsilon}^n[0]) := \Big\{ V \in L_{t,loc}^2 H^1_x(\R^{1+2}; \R^{m+1}) \cap C^1_t L^2_x(\R^{1+2}; \R^{m+1}) \, : \, \exists \{(t_n, x_n)\}_{n\geq 1} \subset \R \times \R^2 \,\text{s.t.} \\ 
 &\hspace{10.5cm} S_{A^n} \big(\tilde{\epsilon}^n[0]\big)(t+t_n, x+x_n) \rightharpoonup V \Big\}.
\end{align*}
Here the weak limit is in the sense of $L^2_{t, loc} H^1_x$ and we observe that each such weak limit $V(t,x)$ solves $\Box V = 0$ in the sense of distributions. We define 
\[
 \eta_{A^n} \big( \tilde{\epsilon}^n[0] \big) := \sup \Big\{ E(V) = \int_{\R^2} \big| \nabla_{t,x} V \big|^2 \, dx \, : \, V \in \mathcal{U}_{A^n}(\tilde{\epsilon}^n[0]) \Big\}.
\]
Then the extraction process of the linear concentration profiles is summarized in the following proposition, which we formulate in the context of general data which have the weak frequency localization properties of~$\epsilon^{n}[0]$.
\begin{prop} \label{prop:linprofiles}
Let $\{ \tilde{\epsilon}^{n}[0] \}_{n \geq 1}$ be a sequence of radially symmetric data pairs $\R^2 \rightarrow \R^{m+1} \times \R^{m+1}$, which are essentially supported at frequency $\sim 1$ and satisfy a uniform energy bound 
\[
 \sup_{n \geq 1} \big\| \tilde{\epsilon}^{n}[0] \big\|_{\dot{H}^1_x \times L^2_x} \lesssim 1.
\]
Upon passing to a suitable subsequence, there exists a sequence of profiles $\tilde{\epsilon}_{b}[0]$, $b \geq 1$, as well as sequences of time shifts $\{t^n_b\}_{n \geq 1}$, $b \geq 1$, satisfying the divergence relation
\[
  \lim_{n\rightarrow\infty}\big|t^n_b - t^n_{b'}\big| = +\infty, \quad b \neq b', 
\]
so that we have for any $B \geq 1$,
\[
 S_{A^n}\big(\tilde{\epsilon}^{n}[0]\big)(t, x) = \sum_{b=1}^B S_{A^n(t+t^n_b, \cdot)}\big(\tilde{\epsilon}_b[0]\big)(t - t^n_b, x) + \tilde{\epsilon}^{n}_{B}(t, x)
\]
and such that 
\[
\lim_{B\rightarrow\infty}\limsup_{n\rightarrow\infty} \eta_{A^n}(\tilde{\epsilon}^{n}_{B}[0]) = 0.
\]
If the data $\tilde{\epsilon}^{Kn,<K_*}[0]$ also depend on two parameters $K$ and $K_*$ such that 
\[
 \lim_{K_{1,2}\rightarrow\infty} \limsup_{K_*\rightarrow\infty} \limsup_{n\rightarrow\infty} \big\|\tilde{\epsilon}^{K_1n,<K_*}[0] - \tilde{\epsilon}^{K_2n,<K_*}[0]\big\|_{\dot{H}^1_x \times L^2_x} = 0
\]
and similarly for $K_*$, then constructing the corresponding profiles $\tilde{\epsilon}_b^{K, <K_*}$ and the remainder term $\tilde{\epsilon}^{Kn,<K_*}_{B}$, we may assume that the limits
\[
\lim_{K\rightarrow\infty}\lim_{K_*\rightarrow\infty}\tilde{\epsilon}_b^{K, <K_*}[0]
\]
exist in $\dot{H}^1_x \times L^2_x$, and also that $\lim_{K\rightarrow\infty} \lim_{K_*\rightarrow\infty} \eta_{A^n}\big(\tilde{\epsilon}^{Kn,<K_*}_{B}[0]\big)$ exists. Finally, we have asymptotic energy conservation 
\[
 \lim_{n\rightarrow\infty}\big\|\nabla_{t,x}\big(S_{A^n(t+t^n_b, \cdot)}(\tilde{\epsilon}_b[0])\big)(0 - t^n_b, \cdot)\big\|_{L_x^2} = \big\|\tilde{\epsilon}_b[0]\big\|_{\dot{H}^1_x \times L^2_x}, \quad b \geq 1,
\]
as well as the asymptotic orthogonality relation 
\[
 \big\| \tilde{\epsilon}^{n}[0] \big\|_{\dot{H}^1_x \times L^2_x}^2 = \sum_{b=1}^B \big\|\tilde{\epsilon}_b[0]\big\|_{\dot{H}^1_x \times L^2_x}^2 + \big\|\tilde{\epsilon}_B^n[0]\big\|_{\dot{H}^1_x\times L^2_x}^2 + o(1) \quad \text{as } n \to \infty.
\]
\end{prop}
\begin{proof} 
 Using Lemma~\ref{lem:soloftherightoneasymptotics}, the proof proceeds in direct analogy to the proof of Proposition 7.11 in~\cite{KS} or to the proof of Lemma 9.23 in~\cite{KL}.
\end{proof}

For each dyadic frequency $k \in \bZ$ we now consider the sequence $\bigl\{ \tilde{\epsilon}^{Kn,<K_*}_{unbounded,k}[0] \bigr\}_{n \geq 1}$ defined further above in~\eqref{eq:tildeepsilonKn<K} and extract concentration profiles $\tilde{\epsilon}^{K,<K_*}_{b, k}[0]$, $b \geq 1$. Applying Proposition~\ref{prop:linprofiles} each time and passing to a suitable subsequence, we may assume that we obtain the same time scales $\{ t^n_b \}_{n \geq 1}$. Importantly, by the construction of $\bigl\{ \tilde{\epsilon}^{Kn,<K_*}_{unbounded,k}[0] \bigr\}_{n \geq 1}$ in Lemma~\ref{lem:prelim profile}, the extracted concentration profiles $\tilde{\epsilon}^{K,<K_*}_{b, k}[0]$, $b \geq 1$, must all be temporally unbounded, i.e. we have $\lim_{n\to\infty} |t_b^n| = +\infty$ for every $b \geq 1$. In order to simplify the notation we set
\[
 \tilde{\epsilon}^{Kn,<K_*}_{b,k}(t,\cdot) := S_{A^n(t+t^n_b, \cdot)}\big( \tilde{\epsilon}_{b,k}^{K,<K_*}[0] \big)(t - t^n_b, \cdot).
\]
Moreover, for later reference, we observe that due to our definition \eqref{eq:tildeepsilonKn<K} and simple frequency considerations, we have that $ \tilde{\epsilon}^{Kn,<K_*}_{b,k}$ is essentially supported at frequency $\sim 2^k$ in the sense that 
\[
 \sum_{a\in \Z} 2^{\sigma |a|} \Big( \sum_{k\in \Z} \big\|P_{k+a} \tilde{\epsilon}^{Kn,<K_*}_{b,k} \big\|_{S_{k+a}}^2 \Big)^{\frac12} \lesssim \big\|\tilde{\epsilon}^{K,<K_*}_b[0]\big\|_{\dot{H}^1_x \times L^2_x}
\]
for an absolute constant $\sigma>0$.

\medskip 

Then it is natural to make the following ansatz for each dyadic frequency $k \in \bZ$,
\begin{equation} \label{eq:profilesfreqloc}
 \phi^{n,<K_*}_k = u^{n}_{k}  + \epsilon^{<K_*}_{bounded,k} + \bigl( U_{<k}^{(\phi^{n,<K_*})} \bigr)^\dagger \Big( \sum_{b=1}^B \tilde{\epsilon}^{Kn,<K_*}_{b,k} + \tilde{\epsilon}^{Kn,<K_*}_{B,k} \Big) + \eta^{Kn,<K_*}_{(k)}
\end{equation}
and without the frequency localization 
\begin{equation} \label{eq:profiles}
 \phi^{n,<K_*} = u^{n}  - p_* + \epsilon^{<K_*}_{bounded} + \sum_{k \in \bZ} \bigl( U_{<k}^{(\phi^{n,<K_*})} \bigr)^\dagger \Big( \sum_{b=1}^B \tilde{\epsilon}^{Kn,<K_*}_{b,k} + \tilde{\epsilon}^{Kn,<K_*}_{B,k} \Big) + \eta^{Kn,<K_*}.
\end{equation}
Here, $\epsilon^{<K_*}_{bounded}$ of course refers to the wave maps evolution of the data $\epsilon^{<K_*}_{bounded}[0] = (\epsilon^{<K_*}_{bounded}, \gamma^{<K_*}_{bounded})$ constructed in the preceding Lemma~\ref{lem:prelim profile}.

\medskip

There is a small subtlety in \eqref{eq:profilesfreqloc} on account of the fact that the expression 
\[
 \big( U_{<k}^{(\phi^{n,<K_*})} \big)^\dagger \Big( \sum_{b=1}^B \tilde{\epsilon}^{Kn,<K_*}_{b,k} + \tilde{\epsilon}^{Kn,<K_*}_{B,k} \Big)
\]
is not necessarily localized to frequency $\sim 2^k$ but only up to exponentially decaying tails, and hence the same applies to $ \eta^{Kn,<K_*}_{(k)}$, whence the different notation. Note, however, that by definition we have 
\begin{align*}
 \big( U_{<k}^{(\phi^{n,<K_*})} \big)^\dagger \Big( \sum_{b=1}^B\tilde{\epsilon}^{Kn,<K_*}_{b,k}+ \tilde{\epsilon}^{Kn,<K_*}_{B,k} \Big) + \eta^{Kn,<K_*}_{(k)} &= \tilde{P}_k \bigg( \big( U_{<k}^{(\phi^{n,<K_*})})^\dagger \Big( \sum_{b=1}^B\tilde{\epsilon}^{Kn,<K_*}_{b,k}+ \tilde{\epsilon}^{Kn,<K_*}_{B,k} \Big) + \eta^{Kn,<K_*}_{(k)} \bigg) \\
 & = \tilde{P}_k \bigg( \big( U_{<k}^{(\phi^{n,<K_*})} \big)^\dagger \Big( \sum_{b=1}^B \tilde{\epsilon}^{Kn,<K_*}_{b,k} + \tilde{\epsilon}^{Kn,<K_*}_{B,k} \Big) \bigg) + \tilde{P}_k \eta^{Kn,<K_*}_{(k)}
\end{align*}
for a suitable localizer $\tilde{P}_k$, and this is what we shall substitute in frequency localized terms. Then we can also replace \eqref{eq:profiles} by 
\begin{equation} \label{eq:profiles1}
 \phi^{n,<K_*} = u^{n} - p_* + \epsilon^{<K_*}_{bounded} + \sum_{k\in\bZ} \tilde{P}_k \bigg( \big( U_{<k}^{(\phi^{n,<K_*})} \big)^\dagger \Big( \sum_{b=1}^B \tilde{\epsilon}^{Kn,<K_*}_{b,k} + \tilde{\epsilon}^{Kn,<K_*}_{B,k} \Big) \bigg) + \sum_{k\in\bZ} \tilde{P}_k \eta_{(k)}^{Kn,<K_*}.
\end{equation}
More precisely, this ansatz is natural to make as long as we are in the regime where the $ \epsilon^{<K_*}_{bounded}$ is truly nonlinear, while in its asymptotic regime we would want to replace it by a solution to \eqref{eq:therightone}. Unfortunately, by comparison to \cite{KS} and \cite{KL}, it appears harder here to make a good {\it{global}} ansatz for the solution, as the previous Lemma~\ref{lem:prelim profile} has already indicated how delicate the correct choice of the bounded profile at time $t = 0$ was. In fact, we expect the correct choice of bounded profile at the next concentration time $t = t^n_1$ (assuming, as we may, that $t^n_1 \ll t^n_2 \ll \ldots$) to delicately hinge on $\phi^{n,<K_*}$ at that time. Our way out of this shall be a careful inductive procedure, first controlling the solution (in terms of the $ \eta^{Kn,<K_*}$) on the time slice $[0, t^n_1 -T_{1*}]$ for some sufficiently large but finite $T_{1*}$ (and as usual picking $n$ large enough), and then delicately modifying the ansatz \eqref{eq:profiles1} to track the solution on $[t^n_1 -T_{1*}, t^n_2 -T_{2*}]$ and so on. Here a crucial point shall be that {\it{our choice of $B$ shall be rather simple}}, and in fact only hinge on $\| u^{n} \|_{S[\bR]}$. This should be compared with the cruder arguments in \cite{KS} and \cite{KL}, where the choice of $B$ hinges on fine properties such as $S$ norm bounds of the profiles and their scattering behavior. In our situation, in light of the poor perturbation theory and the fact that we do not even know the later nonlinear profiles, we could not possibly define $B$ in this manner. 

\medskip 

We call the expressions
\[
 \epsilon^{<K_*}_{bounded}[0], \quad \sum_{k \in \bZ} \tilde{P}_k \Big( \big(U_{<k}^{(\phi^{n,<K_*})} \big)^\dagger \tilde{\epsilon}^{Kn,<K_*}_{b,k} \Big)[0], \quad b \geq 1,
\]
the \emph{profiles} (all evaluated at time $t = 0$). We note that the energy of these expressions as $n\rightarrow\infty$ is well-defined. Also, it follows from Lemma~\ref{lem:prelim profile} that we have the asymptotic orthogonality relation 
\begin{equation} \label{eq:asymptoticorthogonalityatcoordinates} 
 \begin{aligned}
  &\big\|\nabla_{t,x}\phi^{n,<K_*}(0, \cdot)\big\|_{L_x^2}^2 \\ 
  &= \big\|\nabla_{t,x} u^{n}(0, \cdot)\big\|_{L_x^2}^2 + \big\|\nabla_{t,x} \epsilon^{<K_*}_{bounded}(0, \cdot)\big\|_{L_x^2}^2 + \sum_{b=1}^B \big\|\nabla_{t,x} \sum_{k\in\bZ} \tilde{P}_k \Big( \big(U_{<k}^{(\phi^{n,<K_*})}\big)^\dagger \tilde{\epsilon}^{Kn,<K_*}_{b,k} \Big)(0, \cdot)\big\|_{L_x^2}^2 \\
  &\quad \quad + \big\|\nabla_{t,x} \sum_{k\in\bZ} \tilde{P}_k \Big( \big( U_{<k}^{(\phi^{n,<K_*})} \big)^\dagger \tilde{\epsilon}^{Kn,<K_*}_{B,k} \Big)(0, \cdot) \big\|_{L_x^2}^2 + o(1),
 \end{aligned} 
\end{equation}
where the error vanishes asymptotically as $K, n\rightarrow\infty$. The following theorem is the key result of this subsection.

\begin{thm} \label{thm:profiledecomp} 
Assume that the profiles all have asymptotically (as $K, K_*\rightarrow\infty$) energy strictly less than $E_{crit}$. In particular, this is the case when there are at least two profiles present for $K, K_*$ large enough. Then the data $\phi^{n,<K_*}[0]$ can be evolved globally in time and the resulting solution $\phi^{n,<K_*}$ satisfies uniformly for large $K_*$ and $n$ the bound 
\begin{equation} \label{equ:uniform_S_bounds_profile_decomp_thm}
 \big\| \phi^{n,<K_*} \big\|_{S} < \infty.
\end{equation}
In fact, given $\delta_2>0$, there exists $B = B(\|u^{n}\|_S, \delta_2)$ such that there are profiles $\tilde{\tilde{\epsilon}}^{n,<K_*}_{bounded}$ and $\tilde{\tilde{\epsilon}}^{Kn,<K_*}_{b}$, $b \geq 1$, satisfying for $K, K_*,n$ sufficiently large
\[
 E[\tilde{\tilde{\epsilon}}^{n,<K_*}_{bounded}] < E[\epsilon^{<K_*}_{bounded}] + \delta_2, \quad E[\tilde{\tilde{\epsilon}}^{Kn,<K_*}_{b}] < E[\tilde{\epsilon}^{Kn,<K_*}_{b}] + \delta_2, \quad b \geq 1,
\]
and also 
\[
 \big\|\tilde{\tilde{\epsilon}}^{n,<K_*}_{bounded}\big\|_{S} < \infty, \quad \big\|\tilde{\tilde{\epsilon}}^{Kn,<K_*}_{b}\big\|_{S}<\infty, \quad b \geq 1,
\]
uniformly for $n, K_*$ large and such that 
\[
 \phi^{n,<K_*} = u^{n} + \tilde{\tilde{\epsilon}}^{n,<K_*}_{bounded} + \sum_{b=1}^B \tilde{\tilde{\epsilon}}^{Kn,<K_*}_{b} + \sum_{k\in\bZ} \big( U_{<k}^{(\phi^{n,<K_*})} \big)^\dagger \tilde{\epsilon}^{Kn,<K_*}_{B,k} + \tilde{\tilde{\eta}}^{Kn,<K_*}
\]
with $\big\|\tilde{\tilde{\eta}}^{Kn,<K_*}\big\|_{S} < \delta_2$. The profiles $\tilde{\tilde{\epsilon}}^{n,<K_*}_{bounded}$ and $\tilde{\tilde{\epsilon}}^{Kn,<K_*}_{b}$ coincide with the nonlinear profile $\epsilon^{<K_*}_{bounded}$, respectively certain nonlinear profiles $\tilde{\tilde{\epsilon}}^{<K_*}_{b}$ near $t = 0$, respectively near $t = t^n_b$, and are of the form 
\[
 \tilde{\tilde{\epsilon}}^{Kn,<K_*}_{b} = \sum_{k \in \Z} \big( U_{<k}^{(\phi^{n,<K_*})} \big)^\dagger \tilde{\epsilon}^{Kn,<K_*}_{b,k},
\]
where the $\tilde{\epsilon}^{Kn,<K_*}_{b,k}$ are solutions to \eqref{eq:therightone} far away from these times. Also, $\tilde{\epsilon}^{Kn,<K_*}_{B,k}$ is a solution to~\eqref{eq:therightone}. 
\end{thm}

\begin{rem}
 We emphasize that the bound \eqref{equ:uniform_S_bounds_profile_decomp_thm} on the $S$ norms of the evolutions $\phi^{n, <K_*}$, which holds uniformly for all sufficently large $K_*$ and $n$, implies via the high-frequency perturbation Lemma~\ref{lem:localperturbative} that the data $\Pi_{\leq \mu_n^1 + \log(R_n)} \phi^n[0]$ can be globally evolved with uniform $S$ norm bounds for all sufficently large $n$. This conclusion is key for the next step in our induction on frequency process in the next subsection.
\end{rem}

Before we begin with the proof of Theorem~\ref{thm:profiledecomp} we first address some technical issues.

\subsubsection{Technical remarks about multilinear estimates using the decomposition \eqref{eq:profiles1}}

In the sequel, we shall estimate the terms on the right hand side of the basic gauged wave equation 
\begin{equation} \label{eq:againtheequation}
 \Box \big( U_{<k}^{(\phi^{n,<K_*})} \overline{\phi_{k}^{n,<K_*}} \big) = U_{<k}^{(\phi^{n,<K_*})}F_k(\phi^{n,<K_*}, \nabla_{t,x}\phi^{n,<K_*})
\end{equation}
or minor variations thereof by inserting the decompositions \eqref{eq:profiles1} and exploiting a priori bounds on $u^{n}$, $\epsilon^{<K_*}_{bounded}$, $\tilde{\epsilon}^{Kn,<K_*}_{b,k}$ and $\tilde{\epsilon}^{Kn,<K_*}_{B,k}$, where for the last two expressions we of course need control over all frequencies $k \in \bZ$. 

We recall that the structure of the renormalized nonlinearity $F_k(\cdot, \cdot)$ was carefully analyzed in Proposition~\ref{prop:nlw_for_renormalized_phi_bar_schematic_identities}. For what follows it will be useful to introduce the notion of ``perturbative factors'' and ``non-perturbative factors'' in the multilinear expressions constituting the nonlinearity $F_k(\cdot, \cdot)$. We call an input of a multilinear expression in $F_k(\cdot, \cdot)$ a ``non-perturbative factor'' if it can only be estimated in $L^\infty_t L^\infty_x$ in order to place the whole multilinear expression into $L^1_t L^2_x$. These ``non-perturbative factors'' can only occur in the quintilinear expressions\footnote{An example of a ``non-perturbative factor'' is the input $\phi_{\leq k-10}^{(1)}$ in the following quintilinear expression $P_k \bigg( \phi_{\leq k-10}^{(1)} \sum_{k_1 = k_2 + O(1)} P_{k_1} \Big( \phi_{\leq k_1-10}^{(2)} \sum_{k_3=k_4+O(1)} (\partial_\alpha \phi_{k_3}^{(3)} )^\dagger \partial^\alpha \phi_{k_4}^{(4)} \Big)^\dagger \phi_{k_2}^{(5)} \bigg)$ arising in the term $Q_k(\phi)$ in \eqref{equ:nlw_for_renormalized_phi_bar_term1}.} $Q_k(\phi)$ in \eqref{equ:nlw_for_renormalized_phi_bar_term1} as well as in \eqref{equ:schematic_form_nlw_for_renormalized_phi_bar_term3} and in \eqref{equ:schematic_form_nlw_for_renormalized_phi_bar_term4}. All other inputs of any multilinear term in $F_k(\cdot, \cdot)$ will be referred to as ``perturbative factors''.

Then we will be facing the technical difficulty that the ``gauged terms''
\[
 \big( U_{<k}^{(\phi^{n,<K_*})} \big)^\dagger \tilde{\epsilon}^{Kn,<K_*}_{b,k}, \quad \big( U_{<k}^{(\phi^{n,<K_*})} \big)^\dagger \tilde{\epsilon}^{Kn,<K_*}_{B,k}
\]
again involve the a priori uncontrolled function $\phi^{n,<K_*}$, although we of course have a priori control over $\big\| \big( U_{<k}^{(\phi^{n,<K_*})} \big)^\dagger \big\|_{L_{t,x}^\infty}$. Also, observe that the two preceding expressions actually are not localized sharply to frequency $\sim 2^k$, but only up to exponentially decaying tails. To deal with these issues we formulate
\begin{lem}\label{lem:tediousmultilinear} Let $\{c_k^{(b)}\}_{k\in \Z}$ be a sufficiently flat frequency envelope covering $\sum_{\tilde{k}}2^{\sigma|k - \tilde{k}|}\big\|P_k\big(\tilde{\epsilon}^{Kn,<K_*}_{b,\tilde{k}}\big)\big\|_{S_k}$ for some small $\sigma>0$. Then substituting 
\[
 \sum_{\tilde{k}}P_{\tilde{k}} \big( \big(U_{<\tilde{k}}^{(\phi^{n,<K_*})}\big)^\dagger \tilde{\epsilon}^{Kn,<K_*}_{b,\tilde{k}} \big)
\]
for $\phi^{n,<K_*}$ for the ``perturbative factors'' in the terms $F_k(\phi^{n,<K_*}, \nabla_{t,x}\phi^{n,<K_*})$, and calling the resulting functions $G_k$, we get under the bootstrap assumption $\big\|\sum_k\tilde{P}_k\eta_{(k)}^{Kn,<K_*}\big\|_S\lesssim 1$ the bound
\[
 \big\|G_k\big\|_{L_t^1 L_x^2}\leq  C \Big( \|u^{n}\|_{S}, \big\|\epsilon^{<K_*}_{bounded}\big\|_{S},\sum_k \Big( \sum_{\tilde{b}} c_k^{(\tilde{b})} \Big)^2, \sum_k \Big( \sum_{\tilde{k}}2^{\sigma|k - \tilde{k}|} \big\|P_k\tilde{\epsilon}^{Kn,<K_*}_{B,\tilde{k}}\big\|_{S}\Big)^2 \Big) c_k^{(b)}.
\]
Thus, one can estimate $G_k$ as if the factors were just $P_k \sum_{\tilde{k}}\tilde{\epsilon}^{Kn,<K_*}_{b,\tilde{k}}$. Similar inequalities can be obtained when one or more of the ``perturbative factors'' in $F_k(\cdot, \cdot)$ are occupied by $\sum_{\tilde{k}}P_{\tilde{k}} \big( \big(U_{<\tilde{k}}^{(\phi^{n,<K_*})} \big)^\dagger \tilde{\epsilon}^{Kn,<K_*}_{b,\tilde{k}} \big)$ while other ``perturbative factors'' are occupied by one or more of the remaining terms in \eqref{eq:profiles1}. 
\end{lem}
\begin{proof} We illustrate it by considering the specific term 
\[
 F_k(\phi, \nabla_{t,x}\phi) = \sum_{k_1<k}\sum_{k_2<k_1-10}\phi_{k_1} \partial_{\alpha}\phi_{k_2}^{\dagger}\partial^{\alpha}\phi_k.
\]
Thus we have to bound the expression 
\begin{equation} \label{eq:themess1} 
 \sum_{k_1<k}\sum_{k_2<k_1-10}P_{k_1} \big( \big(U_{<k_1}^{(\phi^{n,<K_*})}\big)^\dagger \tilde{\epsilon}^{Kn,<K_*}_{b,k_1} \big) \, \partial_{\alpha} P_{k_2} \big( \big(U_{<k_2}^{(\phi^{n,<K_*})}\big)^\dagger \tilde{\epsilon}^{Kn,<K_*}_{b,k_2} \big)^{\dagger} \, \partial^{\alpha} P_k \big( \big( U_{<k}^{(\phi^{n,<K_*})} \big)^\dagger \tilde{\epsilon}^{Kn,<K_*}_{b,k} \big).
\end{equation}
As usual, to simplify things, we reduce as we may to the case $k=0$. To begin with, we infer the general bound 
\begin{equation} \label{eq:L3Linftybound}
 2^{\frac{k}{3}} \big\| P_k \big( \big( U_{<k}^{(\phi^{n,<K_*})} \big)^\dagger \tilde{\epsilon}^{Kn,<K_*}_{b,k} \big) \big\|_{L_t^3 L_x^\infty} \lesssim c_k^{(b)}, 
\end{equation}
where the implied constant on the right is like the constant $C(\cdot)$ in the statement of the lemma. To see this, write 
\begin{equation} \label{eq:l3linftytrichotomy}
\begin{split}
 P_k \big( \big( U_{<k}^{(\phi^{n,<K_*})} \big)^\dagger \tilde{\epsilon}^{Kn,<K_*}_{b,k} \big) &= P_k \big( P_{<k-10} \big( U_{<k}^{(\phi^{n,<K_*})} \big)^\dagger P_{[k-10,k+10]} \tilde{\epsilon}^{Kn,<K_*}_{b,k} \big) \\
 &\quad + P_k \big( P_{[k-10,k+10]} \big( U_{<k}^{(\phi^{n,<K_*})} \big)^\dagger P_{<k-10} \tilde{\epsilon}^{Kn,<K_*}_{b,k} \big) \\
 &\quad + P_k \big( P_{\geq k-10} \big( U_{<k}^{(\phi^{n,<K_*})} \big)^\dagger P_{\geq k-10} \tilde{\epsilon}^{Kn,<K_*}_{b,k} \big).
\end{split}\end{equation}
The desired bound for the first term on the right is immediate. The bound for the third term on the right is also immediate, using 
\[
 \big\| P_{\geq k-10} \big( U_{<k}^{(\phi^{n,<K_*})} \big)^\dagger \big\|_{L_t^\infty L_x^2} \lesssim 2^{-k} \big\|\nabla_{t,x}\phi^{n,<K_*}\big\|_{L_t^\infty L_x^2}
\]
and Bernstein's inequality
\begin{align*}
\big\| P_k\big(P_{\geq k-10} \big( U_{<k}^{(\phi^{n,<K_*})} \big)^\dagger P_{\geq k-10} \tilde{\epsilon}^{Kn,<K_*}_{b,k} \big) \big\|_{L_t^3 L_x^\infty} &\lesssim \sum_{k_1>k-10} 2^k \big\|P_{k_1} \big( U_{<k}^{(\phi^{n,<K_*})} \big)^\dagger \big\|_{L_t^\infty L_x^2} \big\|P_{k_1}\tilde{\epsilon}^{Kn,<K_*}_{b,k}\big\|_{L_t^3 L_x^\infty} \\
&\lesssim2^k\sum_{k_1>k-10}2^{-\frac43 k_1}2^{-\sigma|k-k_1|}c_k^{(b)}\lesssim 2^{-\frac{k}{3}}c_k^{(b)}.
\end{align*}
For the second term on the right, we expand $P_{[k-10,k+10]} \big( U_{<k}^{(\phi^{n,<K_*})} \big)^\dagger$ schematically into 
\[
 - 2^{-k} P_{k} \big( \big( U_{<k}^{(\phi^{n,<K_*})} \big)^\dagger \sum_{k_1<k} \nabla_x \big( \phi_{<k_1} \phi_{k_1}^{\dagger} - \phi_{k_1} \phi_{<k_1}^{\dagger} \big) \big( U_{<k}^{(\phi^{n,<K_*})} \big)^\dagger \big). 
\]
Then, depending on where the derivative $\nabla_x$ lands, re-expand either $\nabla_x\phi_{<k_1}$ or $\nabla_x\phi_{k_1}$ using \eqref{eq:profiles1}, and repeat application of the above trichotomy. Then using that 
\[
 \big\| P_{<k-10} \tilde{\epsilon}^{Kn,<K_*}_{b,k} \big\|_{L_t^M L_x^\infty} \lesssim 2^{-\frac{k}{M}} c_k^{(b)}
\]
for $M$ large enough, one obtains the desired bound \eqref{eq:L3Linftybound} after finitely many steps, using H\"older's inequality. 

Now back to \eqref{eq:themess1}, consider first the case where all derivatives fall on a $\tilde{\epsilon}$-factor, which is the expression  
\begin{equation} \label{eq:themess2}
 \begin{split}
  &\sum_{k_1<0}\sum_{k_2<k_1-10} P_{k_1} \big( \big(U_{<k_1}^{(\phi^{n,<K_*})}\big)^\dagger \tilde{\epsilon}^{Kn,<K_*}_{b,k_1} \big) \, P_{k_2}\big( \big(U_{<k_2}^{(\phi^{n,<K_*})}\big)^\dagger \partial_{\alpha} \tilde{\epsilon}^{Kn,<K_*}_{b,k_2} \big)^{\dagger} \, P_0\big( \big( U_{<0}^{(\phi^{n,<K_*})} \big)^\dagger \partial^{\alpha}\tilde{\epsilon}^{Kn,<K_*}_{b,0}\big).
 \end{split}
\end{equation}
Then if in each of the expressions $P_{k_i}\big(\ldots\big), i = 1,2$ and $P_0\big(\ldots\big)$, we have a low-high interaction, we can simply bound the factors $\big( U_{<\cdot}^{(\phi^{n,<K_*})} \big)^\dagger$ in $L_{t,x}^\infty$ and the estimate follows. On the other hand, if at least one of the first two of these expressions has a high-low or high-high interaction, we place both factors inside it into $L_t^3 L_x^\infty$, using the observation before, and then also place the remaining low frequency term $P_{k_j}\big(\ldots\big)$ into $L_t^3 L_x^\infty$, while we simply place $P_0\big(\ldots\big)$ into $L_t^\infty L_x^2$. In case of a high-low or high-high interaction inside $P_0\big(\ldots\big)$, we place the lower frequency term into $L_t^3 L_x^\infty$ and the remaining low frequency terms $P_{k_j}\big(\ldots\big)$, $j = 1,2$, into $L_t^3 L_x^\infty$, while the high frequency term inside $P_0\big(\ldots\big)$ gets placed into $L_t^\infty L_x^2$. 

If at least one derivative $\partial_{\alpha}$ falls on a term $(U_{<\cdot}^{(\phi^{n,<K_*})})^\dagger$, say inside $P_{k_2}\big(\ldots\big)$, one can again close by placing three low frequency terms into $L_t^3 L_x^\infty$ as long as there is a low-high or high-high interaction inside $P_{k_2}\big(\ldots\big)$. In case of a high-low interaction, one uses re-expansion of $\partial_{\alpha} \big( U_{<k_2}^{(\phi^{n,<K_*})} \big)^\dagger$, using the equation for $U_{<k_2}^{(\phi^{n,<K_*})}$ and \eqref{eq:profiles1}, similar to the proof of \eqref{eq:L3Linftybound}. 
\end{proof}

\subsubsection{An important bilinear estimate} \label{subsubsec:gain_smallness_null_form_below_L2}

In the proof of Theorem~\ref{thm:profiledecomp}, an important role shall be played by the fact that certain source terms which are multilinear expressions whose factors are all essentially supported at frequency $\sim 1$, but which enjoy smallness of their $L^\infty_x$-norm, have small square-summed (over frequencies) $L_t^1 L_x^2$-norm. This smallness can be gained on account of our delicate bilinear null form estimate~\eqref{equ:null_form_estimate} which allows us to estimate the null form $\partial_\alpha \phi \partial^\alpha \psi$ ``below'' the $L_{t,x}^2$-space, in fact in $L^{\frac{3}{2}}_{t,x}$. This crucial fact was first pointed out by Klainerman-Machedon~\cite{KlMa1} provided both factors are free waves. We note that the corresponding improvements without the radiality assumption are due to Bourgain~\cite{Bourgain}, Wolff~\cite{Wolff} and Tao~\cite{Tao_bilinear_restriction} and played an important role in \cite{KS}. We quickly illustrate how to use this result with the following 
\begin{lem} \label{lem:bilinearnullformbelowL2use} 
Let $\{ \epsilon^{(n)} \}_{n \geq 1}$ be a family of functions essentially localized to frequency $\sim 1$ and with uniform bounds 
\[
 \sup_{n \geq 1} \| \epsilon^{(n)} \|_{S} \lesssim 1
\]
as well as the pointwise decay $\lim_{n\rightarrow\infty}\big\|P_k\nabla_{t,x}\epsilon^{(n)}\big\|_{L_{t,x}^\infty} = 0$ for all $k\in \Z$. Then denoting by $F_k(\epsilon^{(n)}, \nabla_{t,x}\epsilon^{(n)})$ any one of the source terms occurring in our equation 
\[
 \Box \big( U_{<k}^{(\phi)}\bar{\phi}_k \big) = U_{<k}^{(\phi)}F_k(\phi, \nabla_{t,x}\phi), 
\]
we have 
\[
 \lim_{n\to\infty} \sum_{k\in\bZ} \big\|F_k(\epsilon^{(n)}, \nabla_{t,x}\epsilon^{(n)})\big\|_{L_t^1L_x^2}^2 = 0.
\]
\end{lem}
\begin{proof} First, we may reduce the output frequency $k$ as well as the frequencies of all perturbative factors in $F_k(\cdot, \cdot)$ to size $O(1)$ due to the essential frequency localization of $\epsilon^{(n)}$. Then we consider a generic expression
\[
 \epsilon_{k_1}^{(n)}\partial_{\alpha}\epsilon^{(n)}_{k_2}\partial^{\alpha}\epsilon^{(n)}_{k_3}, \quad k_j = O(1).  
\]
We have 
\begin{align*}
 \big\|\epsilon_{k_1}^{(n)}\partial_{\alpha}\epsilon^{(n)}_{k_2}\partial^{\alpha}\epsilon^{(n)}_{k_3}\big\|_{L_t^1 L_x^2}&\lesssim \big\|\epsilon_{k_1}^{(n)}\big\|_{L_t^{2+}L_x^\infty}\big\|\partial_{\alpha}\epsilon^{(n)}_{k_2}\partial^{\alpha}\epsilon^{(n)}_{k_3}\big\|_{L_{t,x}^{2-}},
\end{align*}
where the implied constant depends on $k_{1,2,3}$ and we used Bernstein's inequality to pass from $L^1_t L^2_x$ to $L_t^1 L_x^{2-}$. Further, by interpolation and the null form estimate~\eqref{equ:null_form_estimate} we have 
\begin{align*}
 \big\|\partial_{\alpha}\epsilon^{(n)}_{k_2}\partial^{\alpha}\epsilon^{(n)}_{k_3}\big\|_{L_{t,x}^{2-}} &\lesssim \big\|\partial_{\alpha}\epsilon^{(n)}_{k_2}\partial^{\alpha}\epsilon^{(n)}_{k_3}\big\|_{L_{t,x}^{\frac32}}^{\frac{3}{4-}} \big\|\partial_{\alpha}\epsilon^{(n)}_{k_2}\partial^{\alpha}\epsilon^{(n)}_{k_3}\big\|_{L_{t,x}^{\infty}}^{\frac{1}{4+}} \lesssim \prod_{j=2,3} \big\|\epsilon^{(n)}_{k_j}\big\|_{S_{k_j}}^{\frac{3}{4-}} \big\|\partial_{\alpha}\epsilon^{(n)}_{k_2}\partial^{\alpha}\epsilon^{(n)}_{k_3}\big\|_{L_{t,x}^{\infty}}^{\frac{1}{4+}}.
\end{align*}
Then the last factor at the end vanishes asymptotically due to our assumption, which yields the claim.
\end{proof}

The preceding lemma, or simple variations thereof, will play an important role in Step~2 of Stage~1 of the proof of Theorem~\ref{thm:profiledecomp} below, and more specifically, in the proof of Lemma~\ref{lem:epsboundedlin} on an ``essentially unbounded'' time interval. Working on a bounded time interval as in Step 1 below, one can use much cruder bounds, as for example in {\it{(ii)}} of Step 1 below.

\subsubsection{The proof of Theorem~\ref{thm:profiledecomp}}

We now use the preceding preparations to provide 

\begin{proof}[Proof of Theorem~\ref{thm:profiledecomp}]
The strategy shall be to understand the bound as well as the profile decomposition on the interval $[0, t^n_1 - T_*]$ and then reveal in a second step how to adjust the profile decomposition to cover also the interval $[t^n_1 - T_*, t^n_2 - T_{1*}]$ for sufficiently large $T_*$ and $T_{1*}$ as well as for sufficiently large parameters $K$ and $n$. This will reveal that ensuring $\eta_{A^n}\big( \epsilon^{Kn,<K_*}_{B}[0] \big) < \delta_*(\|u^{n}\|_{S}, \delta_2)$ is indeed enough to hold the error term $\tilde{\tilde{\eta}}^{Kn,<K_*}<\delta_2$ even as the process gets continued beyond time~$t^n_2$. At every step, the values of $K$ and $n$ may have to be increased, but since there is only an a priori bounded number of steps, this process will not diverge. Finally, by their construction, the profiles $\tilde{\tilde{\epsilon}}^{n,<K_*}_{bounded}$ and $\tilde{\tilde{\epsilon}}^{Kn,<K_*}_{b}$ will satisfy uniform $S$ norm bounds with respect to $n$, $K$ and $K_*$ sufficiently large. We now proceed in two stages:

\medskip 
 
\noindent {\bf{Stage 1}}: {\it{Control of the solution on the first time slice $[0, t^n_1 - T_*]$ for sufficiently large $T_*$.}} In turn, we distinguish between the region where  $\epsilon^{<K_*}_{bounded}$ is nonlinear and the region where it scatters: 

\medskip 

\noindent {\bf{Step 1}}: {\it{Here we control the solution on any sufficiently large but fixed finite time interval starting at $t = 0$:}}

\begin{lem} \label{lem:epsboundednonlin} 
Given any $R>0$ (which we will eventually choose depending on the properties of $\epsilon^{<K_*}_{bounded}$) and arbitrary $\delta_3>0$, there are $K$ and $n$ sufficiently large such that defining $\sum_k\tilde{P}_k\eta^{Kn,<K_*}_{(k)}$ via \eqref{eq:profiles1} for a fixed choice of $B$, we have uniformly for all large $K_*$ that
\begin{align*}
 \Big\|\sum_k\tilde{P}_k\eta^{Kn,<K_*}_{(k)}\Big\|_{S([0,R])} < \delta_3.
\end{align*}
In particular, we can achieve that $\big\|\sum_k\tilde{P}_k\eta^{Kn,<K_*}_{(k)}\big\|_{S([0,R])} \ll \delta_2$, where $\delta_2 > 0$ is the constant prescribed in the statement of Theorem~\ref{thm:profiledecomp}. 
\end{lem}

Observe that this holds independently of the $B$ chosen in \eqref{eq:profiles}. The choice of $B$ will be important for the control later on and will be done depending only on the size of $\|u^{n}\|_{S}$ and $\delta_2$. For simplicity of notation, we shall from now on write $\eta^{Kn,<K_*} \equiv \sum_k\tilde{P}_k\eta^{Kn,<K_*}_{(k)}$ and $\eta^{Kn,<K_*}_k \equiv P_k\sum_{\tilde{k}}\tilde{P}_{\tilde{k}}\eta^{Kn,<K_*}_{({\tilde{k}})}$. 

\begin{proof}[Proof of Lemma \ref{lem:epsboundednonlin}]
We recall the frequency localized wave equation for $\phi^{n,<K_*}$, written schematically as
\[
 \Box \Big( U_{<k}^{(\phi^{n,<K_*})} \overline{\phi^{n,<K_*}_k} \Big) = U_{<k}^{(\phi^{n,<K_*})}F_k \big( \phi^{n,<K_*}, \nabla_{t,x}\phi^{n,<K_*} \big).
\]
For technical reasons, specifically the ability to pass from $\overline{\phi^{n,<K_*}_k}$ to $\phi^{n,<K_*}_k$, we shall in fact replace the former quantity by 
\begin{align*}
 \overline{\phi^{n,<K_*,R,M}_k} := \phi^{n,<K_*}_k &+ P_k \Big( \phi^{n,<K_*}_{<k-10} \sum_{k_{1,2}>k+\log R} \big( \phi^{n,<K_*}_{k_1} \big)^\dagger \phi^{n,<K_*}_{k_2} \Big) \\
 &+ P_k \Big(\phi^{n,<K_*}_{<k-10\log R} \sum_{k_{1,2}\in[k+\log M,k+\log R]} \big(\phi^{n,<K_*}_{k_1}\big)^\dagger \phi^{n,<K_*}_{k_2} \Big) \\
 &+ P_k \Big(\phi^{n,<K_*}_{<k-10\log R} \sum_{k<k_{1,2}\leq k+\log M}\chi_{\{r\geq 2^{-k} M\}} \big(\phi^{n,<K_*}_{k_1}\big)^\dagger \phi^{n,<K_*}_{k_2}\Big) \\
 &+ 2^{-k}P_k \Big(\sum_{k_1<k-\log R]} \phi^{n,<K_*}_{<k_1-10} \nabla \big( \phi^{n,<K_*}_{k_1} \big)^\dagger \phi^{n,<K_*}_k \Big) \\
 &+ 2^{-k}P_k \Big(\sum_{k_1\in [k-\log R, k - \log M]} \phi^{n,<K_*}_{<k_1-10\log R}\nabla \big( \phi^{n,<K_*}_{k_1} \big)^\dagger \phi^{n,<K_*}_k \Big) \\
 &+ 2^{-k}P_k \Big(\sum_{k_1\in [k-\log M,k]}\chi_{\{r\geq 2^{-k} M\}} \phi^{n,<K_*}_{<k_1-10\log R} \nabla\big(\phi^{n,<K_*}_{k_1}\big)^\dagger \phi^{n,<K_*}_k\Big).
\end{align*}
Here the parameter $M$ will be chosen large depending on the size of $\big\|\epsilon^{<K_*}_{bounded}\big\|_{S}$ and $\|u^{n}\|_{S}$, while the parameter $R$ shall be chosen large depending on more subtle properties of $\epsilon^{<K_*}_{bounded}$ such as its energy distribution among frequencies. In particular, we may assume $R\gg M$. We also need to modify the matrices $U_{<k}^{(\phi^{n,<K_*})}$ in order to render them ``less dependent'' on $\eta^{Kn,<K_*}$. Specifically, we replace them by $U_{<k}^{(\phi^{n,<K_*},R)}$, which is constructed analogously to $U_{<k}^{(\phi^{n,<K_*})}$ but with the matrix $B_h$ in its construction~\eqref{equ:matrix_gauge_transform} modified to 
\begin{align*}
 B_h^{(R)} &:= \frac{d}{dh} \sum_{\ell \in \Z} \eta(h - 10\log R - \ell) \big( \phi_{\leq \ell - 10}\phi^{\dagger}_\ell - \phi_\ell \phi_{\leq \ell-10}^{\dagger} \big) \\
 &\quad + \frac{d}{dh} \sum_{\ell \in \Z} \chi_{\{ r \geq 2^{-\ell} R^{40} \}} \big( \eta(h - \ell) - \eta(h - 10\log R - \ell \big) \big( \phi_{\leq \ell - 10} \phi^{\dagger}_\ell - \phi_\ell \phi_{\leq \ell-10}^{\dagger} \big).
\end{align*}

This results in the new equation 
\[
 \Box \Big( U_{<k}^{(\phi^{n,<K_*},R)} \overline{\phi^{n,<K_*,R,M}_k} \Big) = U_{<k}^{(\phi^{n,<K_*},R)}F_k^{(R,M)}\big( \phi^{n,<K_*}, \nabla_{t,x}\phi^{n,<K_*} \big).
\]
In order to pass to an equation for $\eta^{Kn,<K_*}$, we need to subtract off the bulk terms. This is straightforward for the low frequency term $u^{n}$ as well as the asymptotic terms 
\[
 \sum_k \tilde{P}_k \bigg( \big( U_{<k}^{(\phi^{n,<K_*})} \big)^\dagger \Big( \sum_{b=1}^B \tilde{\epsilon}^{Kn,<K_*}_{b,k} + \tilde{\epsilon}^{Kn,<K_*}_{B,k} \Big) \bigg),
\]
but more delicate for the bulk term $\epsilon_{bounded}^{<K_*}$. Here we make the following 

\medskip 

\noindent {\bf{Observation 1}}: {\it{For any fixed $k$ and $n$ there exists a (constant) matrix $G_k^{(n)} \in SO(m+1)$ with the property that as $n\rightarrow\infty$
\begin{align*}
 &\sum_{a\in \Z} 2^{\sigma|a|} \bigg( \sum_k \big\| P_{k+a} \big( U_{<k}^{(\phi^{n,<K_*},R)} \overline{\epsilon^{<K_*,R,M}_{bounded,k}} - G^{(n)}_k \cdot U_{<k}^{(\epsilon^{<K_*}_{bounded},R)} \overline{\epsilon^{<K_*,R,M}_{bounded,k}} \big) \big\|_{S_{k+a}([0,R])}^2 \bigg)^{\frac12} \\ 
 &= o(1) + c_R\big\|\eta^{Kn,<K_*}\big\|_{S([0,R]\times \R^2)} + \delta_4 e^{-R^{1000}-2^{K_*}}
\end{align*}
for suitable $\sigma>0$. Also, there is a partition of $[0,R]$ into $O(R^{100}+K_*^{100})$ many time intervals $I_j$ such that we have for all $j$
\begin{align*}
 &\bigg( \sum_k \big\| U_{<k}^{(\phi^{n,<K_*},R)}F_k^{(R,M)}(\epsilon^{<K_*}_{bounded}, \nabla\epsilon^{<K_*}_{bounded}) - G^{(n)}_k \cdot U_{<k}^{(\epsilon^{<K_*}_{bounded},R)}F_k^{(R,M)}(\epsilon^{<K_*}_{bounded}, \nabla\epsilon^{<K_*}_{bounded})\big)\big\|_{L_t^1 L_x^2[I_j]}^2 \bigg)^{\frac12} \\
 & = o(1) + c_R\big\|\eta^{Kn,<K_*}\big\|_{S}+ \delta_4 e^{-R^{1000}-2^{K_*}},
\end{align*}
where the error $o(1)$ is in the sense as $n\rightarrow\infty$. Also, $c_R$ is a coefficient with $\lim_{R\rightarrow\infty}c_R = 0$. In fact, the matrix $G_k^{(n)}$ does not depend on $M$ but it may implicitly depend on $R, K_*$.}}

\medskip 

We observe here that the factors $G^{(n)}_k$ are analogous to the exponential corrections $e^{i\gamma_{\delta_2nab}}$ in Proposition~9.24 in~\cite{KS}. We complement this with the following observation of a similar flavor:

\medskip 

\noindent {\bf{Observation 2}}: {\it{We have the asymptotic relations as $n \rightarrow \infty$ 
\begin{align*}
\sum_{a\in \Z} 2^{\sigma|a|} \bigg( \sum_k \big\| P_{k+a} \big( U_{<k}^{(\phi^{n,<K_*},R)} \overline{u^{n,R,M}_{k}} - U_{<k}^{(u^{n},R)} \overline{u^{n,R,M}_{k}} \big) \big\|_{S}^2 \bigg)^{\frac12} &= o(1), \\
\sum_{a\in \Z} 2^{\sigma|a|} \bigg( \sum_k \big\| P_{k+a} \big( U_{<k}^{(\phi^{n,<K_*},R)} F_k^{(R,M)}(u^{n}, \nabla u^{n}) - U_{<k}^{(u^{n},R)} F_k^{(R,M)}(u^{n}, \nabla u^{n}) \big) \big\|_{L_t^1 L_x^2}^2 \bigg)^{\frac12} &= o(1),
\end{align*}
where $\overline{u^{n,R,M}_{k}}$ is defined in analogy to $\overline{\phi^{n,<K_*,R,M}_k}$.  }}

\medskip 

We leave the proofs of these observations for later. Then the proof of Lemma~\ref{lem:epsboundednonlin} shall follow from the following steps:

\medskip 

\noindent {\bf Step 1 of proof of Lemma~\ref{lem:epsboundednonlin}}: {\it Fix a time slice $[0,R] \times \R^2$. Given $\delta_4>0$  (which we will later choose sufficiently small depending on $\delta_3, \|\epsilon^{<K_*}_{bounded}\|_{S}, \|u^{n}\|_{S}$), then we have for $K$ sufficiently large that
\begin{align*}
 &\limsup_{n\rightarrow\infty} \bigg( \sum_k \Big\| \Box \big( U_{<k}^{(\phi^{n,<K_*},R)} \overline{\phi^{n,<K_*,R,M}_k} - U_{<k}^{(u^{n},R)} \overline{u^{n,R,M}_{k}} - G^{(n)}_k \cdot U_{<k}^{(\epsilon^{<K_*},R)} \overline{\epsilon^{<K_*,R,M}_{bounded,k}} \big) \Big\|_{L_t^1 L_x^2([0,R])}^2 \bigg)^{\frac12} \\ 
 &\leq \delta_4 e^{-R^{1000}-2^{K_*}} C_{11}\big( \big\|\epsilon^{<K_*}_{bounded}\big\|_{S}, \|u^{n}\|_{S} \big) +  R^{41}(\log R+K_*)^{\frac12} \big\|\eta^{Kn,<K_*}\big\|_{S([0,R])} C_{10}\big( \big\|\epsilon^{<K_*}_{bounded}\big\|_{S}, \|u^{n}\|_{S} \big) \\
 &\quad + D\big( \big\|\epsilon^{<K_*}_{bounded}\big\|_{S}, \|u^{n}\|_{S} \big) \big\| \eta^{Kn,<K_*} \big\|_{S([0,R])}^3.
\end{align*}
}

To see this, we start by observing the identity 
\[
 \Box \big( G^{(n)}_k\cdot U_{<k}^{(\epsilon^{<K_*}_{bounded},R)}\overline{\epsilon^{<K_*,R,M}_{bounded,k}} \big) = G^{(n)}_k\cdot U_{<k}^{(\epsilon^{<K_*}_{bounded},R)}F_k^{(R,M)}(\epsilon^{<K_*}_{bounded}, \nabla\epsilon^{<K_*}_{bounded}),
\]
and so 
\begin{align*}
 &\Box \big( U_{<k}^{(\phi^{n,<K_*},R)}\overline{\phi^{n,<K_*,R,M}_k} - U_{<k}^{(u^{n},R)} \overline{u^{n,R,M}_k} - G^{(n)}_k\cdot U_{<k}^{(\epsilon^{<K_*}_{bounded},R)}\overline{\epsilon^{<K_*,R,M}_{bounded,k}}\big) \\
 & = U_{<k}^{(\phi^{n,<K_*},R)}F_k^{(R,M)}(\phi^{n,<K_*},\nabla \phi^{n,<K_*}) - G^{(n)}_k\cdot U_{<k}^{(\epsilon^{<K_*}_{bounded},R)}F_k^{(R,M)}(\epsilon^{<K_*}_{bounded}, \nabla\epsilon^{<K_*}_{bounded}) -U_{<k}^{(u^{n},R)} F_k^{(R,M)}(u^{n}, \nabla u^{n}) \\
 & = U_{<k}^{(\phi^{n,<K_*},R)} \Big( F_k^{(R,M)}(\phi^{n,<K_*},\nabla \phi^{n,<K_*}) - F_k^{(R,M)}(\epsilon^{<K_*}_{bounded}, \nabla \epsilon^{<K_*}_{bounded}) - F_k^{(R,M)}(u^{n}, \nabla u^{n}) \Big) + A_k^{(R,M, n)}, 
\end{align*}
where Observations 1 and 2 imply that after restriction to any one of $O(R^{100} + K_*^{100})$ many time intervals $I_j$ we get  
\[
 \lim_{R, n\rightarrow\infty} \Big( \sum_k \big\| A_k^{(R,M, n)}\big\|_{L_t^1 L_x^2}^2 \Big)^{\frac12} = 0. 
\]
This time restriction step will be carried out explicitly in {\it{Step 5}} below.

\medskip 

Thus, it now remains to bound the term 
\begin{equation} \label{eq:step1mainterm}
 U_{<k}^{(\phi^{n,<K_*},R)} \Big( F_k^{(R,M)}(\phi^{n,<K_*},\nabla \phi^{n,<K_*}) - F_k^{(R,M)}(\epsilon^{<K_*}_{bounded}, \nabla\epsilon^{<K_*}_{bounded}) - F_k^{(R,M)}(u^{n}, \nabla u^{n}) \Big).
\end{equation}
Here we distinguish between different interaction terms, first amongst the bulk terms in the decomposition~\eqref{eq:profiles1}. Before doing so, we quickly comment on the fact that the nonlinearity $F_k^{(R,M)}(\phi^{n,<K_*},\nabla \phi^{n,<K_*})$ now involves some additional terms on account of the fact that we modified $\overline{\phi^{n,<K_*}}$ to $\overline{\phi^{n,<K_*,R,M}}$ as well as $U_{<k}^{(\phi^{n,<K_*})}$ to $U_{<k}^{(\phi^{n,<K_*},R)}$. For the former modification these terms are schematically of the form 
\begin{equation} \label{eq:moretediouserrors}
 \begin{aligned}
  &P_k \big( \phi_{[k-10\log R, k]} \sum_{k_1 = k_2+O(1)\in[k,k+\log R]} \partial_{\alpha}\phi_{k_1}^\dagger \partial^{\alpha}\phi_{k_2} \big), \quad \sum_{k_1\in [k-10\log M, k]}P_k \big( \chi_{\{r\leq 2^{-k} M\}}\phi_{<k_1-10}\nabla\partial_\alpha\phi_{k_1}^\dagger \partial^{\alpha}\phi_k \big) \\
  &P_k \big( \phi_{<k}\sum_{k<k_1 = k_2+O(1)<k+\log M}\chi_{\{ r \leq 2^{-k} M \} } \partial_{\alpha}\phi_{k_1}^\dagger \partial^{\alpha}\phi_{k_2} \big), \quad P_k \big( \phi_{<k}\sum_{k<k_1 = k_2+O(1)<k+\log M}\chi_{\{r\leq 2^{-k} M\}}' \phi_{k_1}^\dagger \nabla_{t,x}\phi_{k_2} \big) \\
  &P_k \big( \sum_{k_1\in [k-10\log M, k]} \chi_{\{r\leq 2^{-k} M\}}' \phi_{<k_1-10} \nabla \phi_{k_1}^\dagger \partial^{\alpha}\phi_k \big)
 \end{aligned}
\end{equation}
as well as additional terms where $\chi'$ is replaced by $\chi''$ at the expense of the operator $\nabla_{t,x}$. For the latter modification we get additional terms of the schematic form 
\begin{align*}
 &\sum_{k_1\in[k-10\log R,k]}\phi_{<k_1 - 10}\chi_{\{ r\leq 2^{-k_1}R^{40} \}} \partial_{\alpha}\phi_{k_1}^{\dagger}\partial^{\alpha}\phi_k, \quad \sum_{k_1\in[k-10\log R,k]}\phi_{<k_1 - 10}\chi_{\{ r \leq 2^{-k_1}R^{40}\}}' \phi_{k_1}^{\dagger} \nabla_{t,x} \phi_k \\
 &\sum_{k_1\in[k-\log R,k]}\phi_{<k_1 - 10}\chi_{\{r\leq 2^{-k_1}R^{40}\}}''\phi_{k_1}^{\dagger}\phi_k.
\end{align*}

We claim that all of these terms admit good bounds in terms of the frequency square-summed $L_t^1 L_x^2$-norm, albeit possibly at the expense of a power of $R$ or $M$, and that they all have the important divisibility property. By this of course we mean that if one of the factors $\phi_{k_j}$ is replaced by $\eta_{k_j}$, then we can divide time into $O(M^K+R^K)$ many intervals for some explicit $K$ such that the corresponding square sum $L_t^1 L_x^2$-norm restricted to such a time slice is bounded by $\ll \|\eta\|_{S}$. The implicit constant in $O(M^K+R^K)$ depends only on~$\|\phi\|_{S}$. We quickly explain this for the first term in \eqref{eq:moretediouserrors}, the others being handled similarly. Thus write 
\begin{equation} \label{equ:illustrate_divisibility}
 \begin{aligned}
  &P_k \big( \phi_{[k-10\log R, k]} \sum_{k_1 = k_2+O(1)\in[k,k+\log R]}\partial_{\alpha}\phi_{k_1}^\dagger \partial^{\alpha}\phi_{k_2} \big) \\
  &= P_k \big( \chi_{\{r\leq R^{40}2^{-k}\}} \phi_{[k-10\log R, k]} \sum_{k_1 = k_2+O(1)\in[k,k+\log R]} \partial_{\alpha} \phi_{k_1}^\dagger \partial^{\alpha} \phi_{k_2} \big) \\
  &\quad + P_k \big( \chi_{\{r>R^{40} 2^{-k}\}} \phi_{[k-10\log R, k]} \sum_{k_1 = k_2+O(1)\in[k,k+\log R]} \partial_{\alpha} \phi_{k_1}^\dagger \partial^{\alpha} \phi_{k_2} \big).
 \end{aligned}
\end{equation}
Then to estimate the second term on the right-hand side we use the null form estimate~\eqref{equ:null_form_estimate} to infer that 
\begin{align*}
&\big\| P_k \big( \chi_{\{ r>R^{40}2^{-k}\}} \phi_{[k-10\log R, k]} \sum_{k_1 = k_2+O(1)\in[k,k+\log R]} \partial_{\alpha} \phi_{k_1}^\dagger \partial^{\alpha} \phi_{k_2} \big) \big\|_{L_t^1 L_x^2} \\ 
&\lesssim 2^{\frac{k}{3}}\big\|\chi_{\{ r>R^{40}2^{-k} \}} \phi_{[k-10\log R, k]} \big\|_{L_t^3 L_x^\infty} \big\|\sum_{k_1 = k_2+O(1)\in[k,k+\log R]}\partial_{\alpha}\phi_{k_1}^\dagger \partial^{\alpha} \phi_{k_2} \big\|_{L_{t,x}^{\frac32}}\\
&\lesssim R^{-1} \Big( \sum_{\ell \in [k-10\log R, k]} \|\phi_{\ell} \|_{S_\ell}^2 \Big)^{\frac12} \sum_{k_1} \|\phi_{k_1}\|_{S_{k_1}}^2.
\end{align*}
Here we have used an interpolate between the improved $L^\infty_x$-bound 
\[
 \big\| \chi_{\{ r\geq 2^{-k}R\}} P_k \phi \big\|_{L_{t,x}^\infty} \lesssim R^{-\frac12} \|P_k \phi \|_{\dot{H}^1_x}
\] 
and the $L_t^{2+} L_x^\infty$ Strichartz estimate to get improved control over $\big\| \chi_{\{ r \geq 2^{-k}R \}} P_k \phi \big\|_{L_t^3 L_x^\infty}$. Finally, square-summing over $k\in \Z$, we get 
\begin{align*}
\Big( \sum_{k \in \Z} \big\| P_k \big( \chi_{\{ r>R^{40}2^{-k}\}} \phi_{[k-10\log R, k]} \sum_{k_1 = k_2+O(1)\in[k,k+\log R]} \partial_{\alpha} \phi_{k_1}^\dagger \partial^{\alpha} \phi_{k_2} \big) \big\|_{L_t^1 L_x^2}^2 \Big)^{\frac12} \lesssim R^{-1} \log(R) \, \|\phi\|_{S}^3,
\end{align*}
which of course gains smallness for $R\gg 1$. To bound the first term on the right-hand side of \eqref{equ:illustrate_divisibility} further above, we use the local energy decay norm
\begin{align*}
&\big\| P_k \big( \chi_{\{ r \leq R^{40}2^{-k}\}} \phi_{[k-10\log R, k]} \sum_{k_1 = k_2+O(1)\in[k,k+\log R]} \partial_{\alpha} \phi_{k_1}^\dagger \partial^{\alpha} \phi_{k_2} \big) \big\|_{L_t^1 L_x^2} \\
&\lesssim R^{40} \big\| \phi_{[k-10\log R, k]}\big\|_{L_{t,x}^\infty} \Big( \sum_{k+\log R >k_1>k} \sum_{\ell < 40\log R - k} \big( 2^{\frac{k_1+\ell}{2}} R^{-20} \big\| \chi_{\{ r \sim 2^\ell\}} r^{-\frac12} \nabla_{t,x} \phi_{k_1} \big\|_{L_{t,x}^2} \big)^2 \Big) \\
&\lesssim  R^{40} \big\| \phi_{[k-10\log R, k]} \big\|_{L_{t,x}^\infty} \Big( \sum_{k+\log R >k_1>k} \sum_{\ell<40 \log R - k_1} \big( 2^{\frac{k_1+\ell}{2}} R^{-20} \big\| \chi_{\{r \sim 2^\ell \}} \nabla_{t,x} \phi_{k_1} \big\|_{L_{t,x}^2} \big)^2 \Big).
\end{align*}
Then square summing over $k\in \Z$ results in the bound 
\begin{align*}
\Big( \sum_{k\in \Z} \big\| P_k \big( \chi_{\{ r \leq R^{40}2^{-k}\}} \phi_{[k-10\log R, k]} \sum_{k_1 = k_2+O(1)>k} \partial_{\alpha} \phi_{k_1} \partial^{\alpha} \phi_{k_2} \big) \big\|_{L_t^1 L_x^2}^2 \Big)^{\frac12} \lesssim R^{40} \log(R) \|\phi\|_{S}^3. 
\end{align*}
Moreover, the quantity 
\begin{align*}
 \sum_{k_1} \sum_{\ell<40\log R - k_1} \Big( 2^{\frac{k_1+\ell}{2}} R^{-20} \big\| \chi_{\{ r \sim 2^\ell \}} r^{-\frac12} \nabla_{t,x} \phi_{k_1} \big\|_{L_{t,x}^2} \Big)^2 \lesssim \sum_{k_1} \sup_{\ell \in \Z} \big\| \chi_{\{ r \sim 2^\ell\}} r^{-\frac12} \nabla_{t,x} \phi_{k_1} \big\|_{L_{t,x}^2}^2 \lesssim \|\phi\|_{S}^2
\end{align*}
is easily seen to have the divisibility property. 

\medskip 

Back to controlling \eqref{eq:step1mainterm}, we now discuss how to bound the various interactions:

\medskip 

\noindent {\it{(i) Interactions between $u^{n}$ and $\epsilon^{<K_*}_{bounded}$.}} These are either terms in which $u^{n}$ appears in one of the lowest frequency ``non-perturbative'' factors such as 
\[
 P_k \bigg( (u^{n}_{<k-10} - p_*) \sum_{k_{1,2}>k+10} P_{k_1}\Big( P_{<k_1-10} \epsilon^{<K_*}_{bounded} \sum_{k_4 = k_5+O(1)} \big( \partial_{\alpha}\epsilon^{<K_*}_{bounded,k_4}\big)^\dagger \partial^{\alpha} \epsilon^{<K_*}_{bounded,k_5} \Big)^\dagger P_{k_2} \epsilon^{<K_*}_{bounded} \bigg),
\]
where one can only place $(u^{n}_{<k-10} - p_*)$ into $L^\infty_{t,x}$ or else they are terms where $u^{n}$ appears in a ``perturbative factor'' such as 
\[
 P_k \bigg( \epsilon^{<K_*}_{bounded, <k-10} \sum_{k_{1,2}>k+10} P_{k_1} \Big( P_{<k_1-10} \epsilon^{<K_*}_{bounded} \sum_{k_4 = k_5+O(1)} \partial_{\alpha}(u^{n}_{k_4})^\dagger \partial^{\alpha} u^{n}_{k_5} \Big)^\dagger P_{k_2} u^{n}_{<k-10} \bigg).
\]
For the former type of term one gets smallness for the square-summed $L_t^1 L_x^2$ norm from the smallness of $u^{n}_{<k-10} - p_*$ on the support of $P_k \epsilon^{<K_*}_{bounded}$ for $k \in \Z$ of moderate size, while for extremely small or large $k$ one exploits the smallness of $P_k\epsilon^{<K_*}_{bounded}$. For the second type of term, one may assume that the frequencies $2^{k_2}, 2^{k_4}, 2^{k_5}$ applied to $u^{n}$ are extremely small, in which case one gains by placing the nonlinearity into $L_t^\infty L_x^2$ instead of $L_t^1 L_x^2$ and making up by multiplying with $R$ (which is held fixed) and then letting $n\rightarrow\infty$. 

\medskip 

\noindent {\it{(ii) Interactions between $u^{n}$ and $\sum_k \tilde{P}_k \Big( \big( U_{<k}^{(\phi^{n,<K_*})} \big)^\dagger \Big( \sum_{b=1}^B \tilde{\epsilon}^{Kn,<K_*}_{b,k} + \tilde{\epsilon}^{Kn,<K_*}_{B,k} \Big) \Big)$.}} Here, if $u^{n}$ appears in a ``non-perturbative factor'' and not in a ``perturbative factor'', then we exploit the pointwise decay of the expression $\sum_k \tilde{P}_k \Big( \big( U_{<k}^{(\phi^{n,<K_*})} \big)^\dagger \Big( \sum_{b=1}^B \tilde{\epsilon}^{Kn,<K_*}_{b,k} + \tilde{\epsilon}^{Kn,<K_*}_{B,k} \Big) \Big)$. In fact, observe that due to the essential frequency localization (up to exponentially decaying tails) we infer 
\[
 \lim_{n\rightarrow\infty} \bigg\| \sum_k \big( U_{<k}^{(\phi^{n,<K_*})} \big)^\dagger \Big( \sum_{b=1}^B \tilde{\epsilon}^{Kn,<K_*}_{b,k} \Big) \bigg\|_{L_{t,x}^\infty([0,R]\times \R^2)} = 0. 
\]
The same holds for the remainder term $\sum_k \big( U_{<k}^{(\phi^{n,<K_*})} \big)^\dagger \tilde{\epsilon}^{Kn,<K_*}_{B,k}$ since $\tilde{\epsilon}^{Kn,<K_*}_{B,k}(0, \cdot)$ converges weakly to zero for all $k \in \Z$. To close the case in which all ``perturbative factors'' are of the form
\[
 \sum_k \tilde{P}_k \bigg( \big( U_{<k}^{(\phi^{n,<K_*})} \big)^\dagger \Big( \sum_{b=1}^B \tilde{\epsilon}^{Kn,<K_*}_{b,k} + \tilde{\epsilon}^{Kn,<K_*}_{B,k} \Big) \bigg),
\]
we achieve $L_t^1 L_x^2$ for the product by placing it into $L_t^\infty L_x^2$ and multiplying with $R$, and the former norm can be bounded by placing all factors into $L_t^\infty L_x^6$ (for cubic terms, and similarly for higher order terms) by interpolating between $L_t^\infty L_x^2$ and $L_{t,x}^\infty$, with the latter norm being asymptotically vanishing (as $n\rightarrow\infty$) due to the preceding remarks. 

If $u^{n}$ appears in at least one ``perturbative factor'' while $\sum_k \tilde{P}_k \Big( \big( U_{<k}^{(\phi^{n,<K_*})} \big)^\dagger \big( \sum_{b=1}^B \tilde{\epsilon}^{Kn,<K_*}_{b,k} + \tilde{\epsilon}^{Kn,<K_*}_{B,k} \big) \Big)$ appears in the ``non-perturbative factor'', we close again by using the asymptotic vanishing of this term in $L_{t,x}^\infty$ if not all ``perturbative factors'' are $u^{n}$. If all ``perturbative factors'' are frequency localized versions of $u^{n}$, we gain smallness on account of the frequency localization of 
\[
 \sum_k \tilde{P}_k \bigg( \big( U_{<k}^{(\phi^{n,<K_*})} \big)^\dagger \Big( \sum_{b=1}^B \tilde{\epsilon}^{Kn,<K_*}_{b,k} + \tilde{\epsilon}^{Kn,<K_*}_{B,k} \Big) \bigg)
\]
to frequencies $\geq \mu_n$ up to exponentially decaying errors. 

\medskip 

\noindent {\it{(iii) Interactions between $\epsilon^{<K_*}_{bounded}$ and $\sum_k \tilde{P}_k \Big( \big( U_{<k}^{(\phi^{n,<K_*})} \big)^\dagger \big( \sum_{b=1}^B \tilde{\epsilon}^{Kn,<K_*}_{b,k} + \tilde{\epsilon}^{Kn,<K_*}_{B,k} \big) \Big)$ as well as pure self-interactions of $\sum_k \tilde{P}_k \Big( \big( U_{<k}^{(\phi^{n,<K_*})} \big)^\dagger \Big( \sum_{b=1}^B \tilde{\epsilon}^{Kn,<K_*}_{b,k} + \tilde{\epsilon}^{Kn,<K_*}_{B,k} \Big) \Big)$}}. These are handled exactly like the preceding case. 

\medskip 

We emphasize that the discussion up to this point applies to all terms occurring inside $F_k^{(R,M)}$, also to the additional terms stemming from modifying $F_k$ to $F_k^{(R,M)}$, in light of the discussion preceding {\it{(i)}}. 

\medskip 

\noindent {\it{(iv) Interactions involving $\eta^{Kn,<K_*}$.}} These are the most delicate ones, but we can again get rid of most terms:

\medskip 

\noindent {\it{(iv.1) Interactions involving $\eta^{Kn,<K_*}$ and at least one factor $u^{n}$.}} Assume that the factor $u^{n}$ is in a non-perturbative position, and all other perturbative factors are occupied by (frequency localizations of) $\eta^{Kn,<K_*}$. Write this term schematically as $(u^{n}_{<k-10} - p_*) F_k^{(R,M)}(\eta^{Kn,<K_*})$. Then we get 
\[
 \bigg( \sum_k \big\| (u^{n}_{<k-10} - p_*) F_k^{(R,M)}(\eta^{Kn,<K_*}) \big\|_{L_t^1 L_x^2([0,R])}^2 \bigg)^{\frac12} \lesssim \big\| \eta^{Kn,<K_*} \big\|_{S}^3.
\]
On the other hand, if at least one of the other perturbative factors is occupied by a $\epsilon^{<K_*}_{bounded}$ or an asymptotically vanishing term $\sum_k \tilde{P}_k \Big( \big( U_{<k}^{(\phi^{n,<K_*})} \big)^\dagger \big( \sum_{b=1}^B \tilde{\epsilon}^{Kn,<K_*}_{b,k} + \tilde{\epsilon}^{Kn,<K_*}_{B,k} \big) \Big)$, we estimate it by 
\begin{align*}
&\bigg( \sum_k \Big\| u^{n}_{<k-10} F_k^{(R,M)}(\eta^{Kn,<K_*}, \epsilon^{<K_*}_{bounded}, \textstyle{\sum_k} \big( U_{<k}^{(\phi^{n,<K_*})} \big)^\dagger \Big( \sum_{b=1}^B \tilde{\epsilon}^{Kn,<K_*}_{b,k} + \tilde{\epsilon}^{Kn,<K_*}_{B,k}\big) \Big) \Big\|_{L_t^1 L_x^2([0,R])}^2 \bigg)^{\frac12} \\
&\quad \leq C_1 \big\|\eta^{Kn,<K_*}\big\|_{S}^2 + C_2 \big\|\eta^{Kn,<K_*}\big\|_{S}, 
\end{align*}
where $C_{1,2}$ depend only on $\big\|\epsilon^{<K_*}_{bounded}\big\|_{S}$ provided $n$ is sufficiently large to kill the effect of the asymptotically vanishing terms. 

Next, consider the terms where at least one perturbative factor is occupied by $u^{n}$ and where there is at least one factor $\eta^{Kn,<K_*}$. If this latter factor is in a perturbative position, then we bound the term by 
\[
 \lesssim C_3\Big( \|u^{n} \|_{S}, \big\| \epsilon^{<K_*}_{bounded} \big\|_{S} \Big) \big\| \eta^{Kn,<K_*} \big\|_{S}^2 + C_4 \Big( \|u^{n}\|_{S}, \big\| \epsilon^{<K_*}_{bounded} \big\|_{S} \Big) \big\|\eta^{Kn,<K_*} \big\|_{S}.
\]
On the other hand, if the factor $\eta^{Kn,<K_*}$ is in a non perturbative position and all perturbative factors are occupied by $u^{n}$, then we have to complement this argument with a pure energy type bound just as in the proof of the bound for the lowest frequency non-atomic part in Proposition~\ref{prop:bootstrap_bounds_to_next_level_in_evolving_lowest_frequency_nonatomic_part}. We provide this in the next item:

\medskip

\noindent {\it{(iv.2) Choosing $K$ and $n$ large enough, we have for all $k < \mu_n$ that
\begin{align*}
 &\big\| P_k \eta^{Kn,<K_*} \big\|_{L_t^\infty L_x^2([0,R])} \\
 &\ll \delta_4 e^{-R^{1000} - 2^{K_*}} 2^{-\sigma|k-b_L|} + R \, 2^{k} \bigg( C_5\Big( \|u^{n}\|_{S}, \big\|\epsilon^{<K_*}_{bounded}\big\|_{S} \Big) \big\|\eta^{nK}\big\|_S + C_6 \Big( \|u^{n}\|_{S}, \big\|\epsilon^{<K_*}_{bounded}\big\|_{S} \Big) \big\|\eta^{nK}\big\|^3_S \bigg) \\
 &\ll \delta_4 e^{-R^{1000} - 2^{K_*}} 2^{-\sigma|k-b_L|}
\end{align*}
for a suitable absolute constant $\sigma>0$. Here $\delta_4$ will be determined in terms of $\delta_3$ as well as $\|u^{n}\|_{S}$.}} In fact, this is much simpler than the proof of the corresponding bound for the lowest frequency non-atomic part in Lemma~\ref{lem:bootstrap_non_atomic_kinetic}. Replacing $\epsilon$ by $\eta^{Kn, <K_*}$ there and $\phi$ by 
\[
 u = u^{n}  - p_* +  \epsilon^{<K_*}_{bounded} + \sum_k \big( U_{<k}^{(\phi^{n,<K_*})} \big)^\dagger \Big( \sum_{b=1}^B \tilde{\epsilon}^{Kn,<K_*}_{b,k} + \tilde{\epsilon}^{Kn,<K_*}_{B,k} \Big),
\]
then fixing a very low frequency $k < \mu_n$ we have to bound integrals of the form  
\[
 \int_0^R \int_{\R^2} P_k \big( \eta^{Kn, <K_*} \partial_{\alpha}u^\dagger \partial^{\alpha} u \big)^\dagger \partial_t P_k \eta^{Kn, <K_*} \, dx \, dt
\]
as well as other, similar ones. But the preceding expression is bounded by 
\begin{align*}
&\bigg| \int_0^R \int_{\R^2} P_k \big( \eta^{Kn, <K_*} \partial_{\alpha}u^\dagger \partial^{\alpha} u \big)^\dagger \partial_t P_k \eta^{Kn, <K_*} \, dx \, dt \bigg| \\
&\leq R \, \big\|\eta^{Kn, <K_*}\big\|_{L_{t,x}^\infty([0,R])} \big\|\partial_{\alpha} u \partial^{\alpha} u \big\|_{L_t^\infty L_x^{1+}([0,R])} \big\| \partial_t P_k \eta^{Kn, <K_*} \big\|_{L_{t}^\infty L_x^{\infty-}([0,R])}
\end{align*}
and this in turn is bounded by the expression further above. Indeed, observe that the terms 
\[
 \sum_k \big( U_{<k}^{(\phi^{n,<K_*})} \big)^\dagger \Big( \sum_{b=1}^B \tilde{\epsilon}^{Kn,<K_*}_{b,k} + \tilde{\epsilon}^{Kn,<K_*}_{B,k} \Big)
\]
asymptotically do not contribute, on account of their asymptotic vanishing property. The term $\delta_4e^{-R^{1000}-2^{K_*}} 2^{-\sigma|k-b_L|}$ comes from the initial data $\eta^{Kn,<K_*}[0]$ at time $t = 0$, which can be made small by picking $K$ large enough on account of Lemma~\ref{lem:prelim profile}. The other integrals in the energy identity are handled similarly. 

\medskip 

We can now close case {\it{(iv.1)}}:

\medskip 

\noindent {\it{(iv.3) Conclusion of case (iv.1)}}. Consider now a term of the schematic form $\eta^{Kn,<K_*}_{<k-10} F_k^{(R,M)}(u^{n}, \nabla u^{n})$. We bound this by 
\begin{align*}
 \bigg( \sum_k \big\|\eta^{Kn,<K_*}_{<k-10}F_k^{(R,M)}(u^{n}, \nabla u^{n})\big\|_{L_t^1 L_x^2}^2 \bigg)^{\frac12} &\lesssim \bigg( \sum_{k\geq \mu_n} \big\| \eta^{Kn,<K_*}_{<k-10} F_k^{(R,M)}(u^{n}, \nabla u^{n}) \big\|_{L_t^1 L_x^2}^2 \bigg)^{\frac12} \\
 &\quad + \bigg( \sum_{k<\mu_n} \big\| \eta^{Kn,<K_*}_{<k-10} F_k^{(R,M)}(u^{n}, \nabla u^{n}) \big\|_{L_t^1 L_x^2}^2 \bigg)^{\frac12}.
\end{align*}
Then for the first term on the right we have 
\[
 \lim_{n\rightarrow\infty} \bigg( \sum_{k\geq \mu_n} \big\| \eta^{Kn,<K_*}_{<k-10} F_k^{(R,M)}(u^{n}, \nabla u^{n}) \big\|_{L_t^1 L_x^2}^2 \bigg)^{\frac12} = 0,
\]
while for the second term on the right, we use {\it{(iv.2)}} to bound it by 
\[
 \limsup_{n\rightarrow\infty} \bigg( \sum_{k<\mu_n} \big\| \eta^{Kn,<K_*}_{<k-10} F_k^{(R,M)}(u^{n}, \nabla u^{n}) \big\|_{L_t^1 L_x^2}^2 \bigg)^{\frac12} \lesssim \delta_4 e^{-R^{1000}-2^{K_*}} \|u^{n}\|_{S}^3. 
\]

\medskip 

\noindent {\it{(v) Interactions between $\eta^{Kn,<K_*}$ and $\epsilon^{<K_*}_{bounded}$.}} This is the most delicate case on account of the situations where $\eta^{Kn,<K_*}$ is in a non-perturbative position while the perturbative factors are all $\epsilon^{<K_*}_{bounded}$. In fact, it is in this situation where the very precise information about $\eta^{Kn,<K_*}$ at $t = 0$ from Lemma~\ref{lem:prelim profile} becomes really important. The basic idea is to partition the time slice $[0,R]\times \R^2$ into a cylinder $[0,R]\times B_{R_*}$ for some very large $R_*$ chosen such that $\epsilon^{<K_*}_{bounded}$ is small in the complement of $B_{R_*}$, and the set $[0,R]\times B_{R_*}^c$. More specifically, $R_*$ will be chosen depending on $R$ and $\delta_3$, and $R$ will then later be chosen to control the solution on all of $[0, t^n_1 - T_*]$. 

Now given $R$ and $\delta_3>0$, pick $R_*$ sufficiently large such that 
\[
 \bigg( \sum_k \big\| \chi_{B_{R_{*}}^c} F_k^{(R,M)}(\epsilon^{<K_*}_{bounded}) \big\|_{L_t^1 L_x^2([0,R]\times \R^2)}^2 \bigg)^{\frac12} \ll \delta_4 e^{-R^{1000}-2^{K_*}},
\]
where $\eta^{Kn,<K_*}_{<k-10}F_k^{(R,M)}(\epsilon^{<K_*}_{bounded})$ is a schematically written term in which $\eta^{Kn,<K_*}_{<k-10}$ is a non-perturbative factor while all the perturbative positions are occupied by $\epsilon^{<K_*}_{bounded}$. Then due to the simple $L^{\infty}_{t,x}$-bound $\big\|\eta^{Kn,<K_*}\big\|_{L_{t,x}^\infty}~\lesssim~1$ provided $K,n$ are sufficiently large, we get the bound 
\begin{align*}
 \bigg( \sum_k \big\|\chi_{B_{R_{*}}^c}\eta^{Kn,<K_*}_{<k-10}F_k^{(R,M)}(\epsilon^{<K_*}_{bounded})\big\|_{L_t^1 L_x^2([0,R]\times \R^2)}^2 \bigg)^{\frac12} \ll \delta_4 e^{-R^{1000}-2^{K_*}}.
\end{align*}
Thus, consider now the region $[0,R] \times B_{R_*}$ where $\eta^{Kn,<K_*}$ and $\epsilon^{<K_*}_{bounded}$ interact strongly. Then the smallness will have to come from the factor $\eta^{Kn,<K_*}_{<k-10}$. In fact, use the splitting 
\begin{align*}
\chi_{B_{R_{*}}}\eta^{Kn,<K_*}_{<k-10}F_k^{(R,M)}(\epsilon^{<K_*}_{bounded}) = \chi_{B_{R_{*}}}\eta^{Kn,<K_*}_{[-10\log R, k-10]}F_k^{(R,M)}(\epsilon^{<K_*}_{bounded}) + \chi_{B_{R_{*}}}\eta^{Kn,<K_*}_{<-10\log R}F_k^{(R,M)}(\epsilon^{<K_*}_{bounded}).
\end{align*}
Then we control the first term on the right with the estimate 
\begin{align*}
 &\bigg( \sum_k \big\| \chi_{B_{R_{*}}} \eta^{Kn,<K_*}_{[-10\log R, k-10]} F_k^{(R,M)}(\epsilon^{<K_*}_{bounded}) \big\|_{L_t^1L_x^2([0,R]\times\R^2)}^2 \bigg)^{\frac12} \\
 &\quad \lesssim R^{41}(\log R+K_*)^{\frac12} \big\|\eta^{Kn,<K_*}\big\|_{S} C_7\Big( \big\|\epsilon^{<K_*}_{bounded}\big\|_{S} \Big) \\
 &\quad \quad + \bigg( \sum_{k>K_*} \big\| \chi_{B_{R_{*}}}\eta^{Kn,<K_*}_{[-10\log R, k-10]}F_k^{(R,M)}(\epsilon^{<K_*}_{bounded})\big\|_{L_t^1 L_x^2([0,R]\times\R^2)}^2 \bigg)^{\frac12} \\
 &\quad \lesssim  R^{41}(\log R+K_*)^{\frac12}\big\|\eta^{Kn,<K_*}\big\|_{S} C_8\Big( \big\| \epsilon^{<K_*}_{bounded} \big\|_{S} \Big)
\end{align*}
on account of the fact that the norms $\big\|F_k^{(R,M)}(\epsilon^{<K_*}_{bounded})\big\|_{L_t^1L_x^2([0,R]\times\R^2)}$ are exponentially decreasing for $k\geq K_*$ by our enforced frequency localization. 
For the second term on the right we use that for $t\in [0,R]$ 
\[
 \big|\chi_{B_{R_{*}}}\eta^{Kn,<K_*}_{<-10\log R}(t, \cdot)\big|\leq R \, \big\|\partial_t\eta^{Kn,<K_*}_{<-10\log R}\big\|_{L^\infty_{t,x}([0,R]\times \R^2)} + \big\|\chi_{B_{R_{*}}}\eta^{Kn,<K_*}_{<-10\log R}(0, \cdot)\big\|_{L^\infty_x}.
\]
Since we may choose $R_*\gg R^{10}$ we may essentially move the cutoff $\chi_{B_{R_{*}}}$ past the frequency localizer (up to exponentially decaying tails) and so picking $K$ and $R_*$ large enough we may assume that 
\[
 \big\|\chi_{B_{R_{*}}}\eta^{Kn,<K_*}_{<-10\log R}(0, \cdot)\big\|_{L^\infty_x}\ll \delta_4 e^{-R^{1000}-2^{K_*}}. 
\]
The first term on the right on the other hand may be bounded by 
\[
 R \, \big\|\partial_t\eta^{Kn,<K_*}_{<-10\log R}\big\|_{L^\infty_{t,x}([0,R]\times \R^2)} \lesssim R^{-9} \big\|\eta^{Kn,<K_*}\big\|_{S},
\]
and so we obtain 
\begin{align*}
 \bigg( \sum_k \big\| \chi_{B_{R_{*}}} \eta^{Kn,<K_*}_{<-10\log R} F_k^{(R,M)}(\epsilon^{<K_*}_{bounded}) \big\|_{L_t^1L_x^2([0,R]\times\R^2)}^2 \bigg)^{\frac12} \lesssim \Big( \delta_4 e^{-R^{1000}-2^{K_*}} + R^{-9} \big\|\eta^{Kn,<K_*}\big\|_{S} \Big) \, C_9\Big( \big\|\epsilon^{<K_*}_{bounded}\big\|_{S}\Big).
\end{align*}
In total we have obtained the following bound
\begin{align*}
&\bigg( \sum_k\big\|\eta^{Kn,<K_*}_{<-10\log R}F_k^{(R,M)}(\epsilon^{<K_*}_{bounded})\big\|_{L_t^1L_x^2([0,R]\times\R^2)}^2 \bigg)^{\frac12} \\ 
&\quad \lesssim \delta_4 e^{-R^{1000}-2^{K_*}} C_9\Big( \big\|\epsilon^{<K_*}_{bounded}\big\|_{S} \Big) + R^{41} (\log R+K_*)^{\frac12} \big\|\eta^{Kn,<K_*}\big\|_{S} C_{10}\Big( \big\|\epsilon^{<K_*}_{bounded} \big\|_{S} \Big).
\end{align*}

\medskip 

Combining the cases {\it{(i)}} -- {\it{(v)}} and the remarks preceding them, we finally obtain the bound 
\begin{equation} \label{eq:eta0Ralmost}
 \begin{aligned}
  &\bigg( \sum_k \big\| \Box \big( U_{<k}^{(\phi^{n,<K_*},R)} \overline{\phi^{n,<K_*,M,R}_k} - U_{<k}^{(u^{n},R)} \overline{u^{n,M,R}_{k}} - G^{(n)}_k \cdot U_{<k}^{(\epsilon^{<K_*},R)}\overline{\epsilon^{<K_*,M,R}_{bounded,k}}\big)\big\|_{L_t^1 L_x^2([0,R])}^2 \bigg)^{\frac12} \\ 
  &\lesssim \delta_4 e^{-R^{1000}-2^{K_*}} C_{11}\Big( \big\|\epsilon^{<K_*}_{bounded} \big\|_{S}, \|u^{n}\|_{S} \Big) +  R^{41}(\log R+K_*)^{\frac12} \big\|\eta^{Kn,<K_*}\big\|_{S} C_{10}\Big( \big\|\epsilon^{<K_*}_{bounded} \big\|_{S}, \|u^{n}\|_{S}\Big) \\
  &\quad + D\Big( \big\|\epsilon^{<K_*}_{bounded}\big\|_{S}, \|u^{n}\|_{S} \Big) \big\| \eta^{Kn,<K_*} \big\|_{S}^3.
 \end{aligned} 
\end{equation}

The bound established thus far in Step 1 of the proof of Lemma~\ref{lem:epsboundednonlin} is not quite enough yet due to the fact that the terms $U_{<k}^{(\phi^{n,<K_*},R)} \overline{\phi^{n,<K_*,M,R}_k}$ etc. are not sharply frequency localized on account of the lack of frequency localizations of the gauge transformations $U_{<k}^{(\phi^{n,<K_*},R)}$ etc. However, a straightforward Littlewood-Paley trichotomy reveals that we get the somewhat sharper estimate 
\begin{equation} \label{eq:eta0Ralmostbetter} 
 \begin{aligned}
  &\sum_a 2^{\sigma|a|} \bigg( \sum_k \Big\| P_{k+a} \Box \big( U_{<k}^{(\phi^{n,<K_*},R)} \overline{\phi^{n,<K_*,M,R}_k} - U_{<k}^{(u^{n},R)} \overline{u^{n,M,R}_{k}} - G^{(n)}_k \cdot U_{<k}^{(\epsilon^{<K_*},R)} \overline{\epsilon^{<K_*,M,R}_{bounded,k}} \big) \Big\|_{L_t^1 L_x^2([0,R])}^2 \bigg)^{\frac12} \\
  &\lesssim \delta_4 e^{-R^{1000}-2^{K_*}} C_{11}\Big( \big\|\epsilon^{<K_*}_{bounded}\big\|_{S}, \|u^{n}\|_{S} \Big) +  R^{41}(\log R+K_*)^{\frac12} \big\|\eta^{Kn,<K_*}\big\|_{S} C_{10}\Big( \big\|\epsilon^{<K_*}_{bounded}\big\|_{S}, \|u^{n}\|_{S} \Big) \\
  &\quad + D\Big( \big\|\epsilon^{<K_*}_{bounded}\big\|_{S}, \|u^{n}\|_{S}\Big) \big\|\eta^{Kn,<K_*}\big\|_{S}^3
 \end{aligned} 
\end{equation}
for any $\sigma<1$. 

\medskip 

\noindent {\bf{Step 2 of the proof of Lemma~\ref{lem:epsboundednonlin}}}: {\it{The preceding step implies via the basic energy estimate~\eqref{equ:energy_estimate} that for $M, K, n$ large enough the following bound holds}}
\begin{equation} \label{eq:Step2lemma} 
 \begin{aligned}
  &\sum_a 2^{\sigma|a|} \bigg( \sum_k \Big\| P_{k+a} \Big( U_{<k}^{(\phi^{n,<K_*},R)} \big( \overline{(\eta^{Kn,<K_*}_k+\epsilon^{<K_*}_{bounded,k})^{M,R}} -  \overline{\epsilon^{<K_*,M,R}_{bounded,k}} \big) \Big) \Big\|_{S([0,R])}^2 \bigg)^{\frac12} \\
  &\lesssim \delta_4 e^{-R^{1000}-2^{K_*}} C_{11}\Big( \big\|\epsilon^{<K_*}_{bounded}\big\|_{S}, \|u^{n}\|_{S} \Big) +  R^{41} (\log R+K_*)^{\frac12} \big\|\eta^{Kn,<K_*}\big\|_{S([0,R])} C_{10} \Big(\big\|\epsilon^{<K_*}_{bounded}\big\|_{S}, \|u^{n}\|_{S}\Big) \\
  &\quad + E\Big(\big\|\epsilon^{<K_*}_{bounded}\big\|_{S}, \|u^{n}\|_{S}\Big) \big\|\eta^{Kn,<K_*}\big\|_{S([0,R])}^3 + c_R \big\|\eta^{Kn,<K_*}\big\|_{S}, 
 \end{aligned} 
\end{equation}
where
\[
 \lim_{R\rightarrow\infty} c_R = 0. 
\]
This follows essentially from the basic energy inequality~\eqref{equ:energy_estimate}. In fact, the self-interactions of $\epsilon^{<K_*}_{bounded}$ inside $\overline{\phi^{n,<K_*,M,R}_k}$ are cancelled by $G^{(n)}_k\cdot U_{<k}^{(\epsilon^{<K_*},R)}\overline{\epsilon^{<K_*,M,R}_{bounded,k}}$ via {\it{Observation 1}}.

The self-interactions of $u^{n}$ inside $\overline{\phi^{n,<K_*,M,R}_k}$ are cancelled by $U_{<k}^{(u^{n},R)} \overline{u^{n,M,R}_{k}}$ via {\it{Observation 2}}. Next, for the interactions between $\epsilon^{<K_*}_{bounded}$ and $u^{n,M}$ inside $\overline{\phi^{n,<K_*,M,R}_k}$ we use the smallness of $u^{n} - p_*$ on the support of $\epsilon^{<K_*}_{bounded}$ (i.e. on bounded sets $\{ r = O(1) \}$) provided all perturbative factors are occupied by $\epsilon^{<K_*}_{bounded}$, while in case that all perturbative factors are occupied by $u^{n}$ we gain smallness by exploiting the frequency localization of $\epsilon^{<K_*}_{bounded}$ to the range $k \geq \mu_n$ up to exponential tails. In case that there are perturbative factors occupied by both $u^{n}$ as well $\epsilon^{<K_*}_{bounded}$, the smallness also follows from the essentially disjoint frequency supports (up to exponentially decaying tails) of these functions. 

\medskip 

Furthermore all interactions of $u^{n,M}$, $\epsilon^{<K_*}_{bounded}$ or $\eta^{Kn,<K_*}$ with $\sum_k \tilde{P}_k \Big( \big( U_{<k}^{(\phi^n,<K_*} \big)^\dagger \big( \sum_{b=1}^B \tilde{\epsilon}^{Kn,<K_*}_{b,k} + \tilde{\epsilon}^{Kn,<K_*}_{B,k} \big) \Big)$ as well as all self-interactions of the latter term inside $\overline{\phi^{n,<K_*,M,R}_k}$ are seen to lead to terms vanishing with respect to $\| \cdot \|_{S([0,R])}$ as $n\rightarrow\infty$ on account of their pointwise decay. Moreover, all interactions of $\eta^{Kn,<K_*}$ with $u^{n}$ are seen to be either vanishing provided we let $n\rightarrow\infty$ using {\it{(iv.2)}}, or else of the form $O\Big(\big\|\eta^{Kn,<K_*}\big\|_{S}^2\Big)$.   

\medskip 

\noindent {\bf{Step 3 of the proof of Lemma~\ref{lem:epsboundednonlin}}}:  {\it{The preceding bound implies the following:}}
\begin{equation} \label{eq:Step3lemma} 
 \begin{aligned}
  &\bigg( \sum_k \big\| \overline{(\eta^{Kn,<K_*}_k + \epsilon^{<K_*}_{bounded,k})^{M,R}} - \overline{\epsilon^{<K_*,M,R}_{bounded,k}} \big\|_{S_k([0,R])}^2 \bigg)^{\frac12} \\
  &\lesssim \delta_4 e^{-R^{1000}-2^{K_*}} C_{12} \Big( \big\|\epsilon^{<K_*}_{bounded}\big\|_{S}, \|u^{n}\|_{S} \Big) + R^{41} (\log R+K_*)^{\frac12} \big\|\eta^{Kn,<K_*}\big\|_{S([0,R])} C_{13} \Big(\big\|\epsilon^{<K_*}_{bounded}\big\|_{S}, \|u^{n}\|_{S}\Big) \\
  &\quad + F\Big( \big\|\epsilon^{<K_*}_{bounded} \big\|_{S}, \big\|u^{n}\big\|_{S} \Big) \big\|\eta^{Kn,<K_*}\big\|_{S([0,R])}^3 + c_R \big\|\eta^{Kn,<K_*}\big\|_{S}. 
 \end{aligned} 
\end{equation}
In fact, we can use the more general observation that if $f_{(k)}$ is a family of functions satisfying 
\[
 \sum_{a\in \Z} 2^{\sigma|a|} \Big( \sum_{k \in \Z} \big\|P_{k+a} f_{(k)}\big\|_{S_{k+a}([0,R])}^2 \Big)^{\frac12} \leq A
\]
as well as 
\[
\lim_{n\rightarrow\infty} \sum_{a\in Z} 2^{\sigma|a|} \Big( \sum_{k<\mu_n} \big\|P_{k+a} f_{(k)}\big\|_{S_{k+a}([0,R])}^2 \Big)^{\frac12} = 0, 
\]
then we have 
\[
\big(\sum_{k\in Z}\big\|\tilde{P}_k\big[\big(U_{<k}^{(\phi^{n,<K_*},R)}\big)^{\dagger}f_{(k)}\big]\big\|_{S_k([0,R]\times \R^2}^2)\big)^{\frac12}\lesssim A + c_R\big\|\eta^{Kn,<K_*}\big\|_{S} + o(1)
\]
as $n\rightarrow\infty$. To see this, we consider the case $k = 0$ and apply the usual frequency trichotomy to the expression 
\[
 \tilde{P}_0 \Big( \big(U_{<0}^{(\phi^{n,<K_*},R)}\big)^{\dagger} f_{(0)} \Big). 
\]
Consider for example the low-high case 
\[
 \tilde{P}_0 \Big( P_{<-10} \big(U_{<0}^{(\phi^{n,<K_*},R)}\big)^{\dagger} P_0 f_{(0)} \Big). 
\]
Then the estimate for those parts of $\| \cdot \|_{S_0}$ not involving any derivatives is immediate. Next, consider the part of the norm involving the operator $(\partial_t + \partial_r)$. If it falls on $P_0 f_{(0)}$, then the estimate is again immediate. Thus consider now the expression 
\[
 \tilde{P}_0 \Big( (\partial_t + \partial_r) P_{<-10} \big( U_{<0}^{(\phi^{n,<K_*},R)} \big)^{\dagger} P_0 f_{(0)} \Big). 
\]
Schematically, this reduces to estimating terms of the form 
\[
 \tilde{P}_0 \Big( (\partial_t + \partial_r) B_{<-10}^{(\phi^{n,<K_*}, R)} P_0 f_{(0)} \Big). 
\]
Call the high-frequency factors in the bilinear expression defining $B_{<-10}^{(\phi^{n,<K_*}, R)}$ the ``perturbative factors''. Expanding out $B_{<-10}^{(\phi^{n,<K_*}, R)}$ using \eqref{eq:profiles1}, we can easily deal with the cases where the perturbative factor is either $u^{n}$ or a dispersive term 
\[
 \sum_k \tilde{P}_k \bigg( \big( U_{<k}^{(\phi^n,<K_*} \big)^\dagger \Big( \sum_{b=1}^B \tilde{\epsilon}^{Kn,<K_*}_{b,k} + \tilde{\epsilon}^{Kn,<K_*}_{B,k} \Big) \bigg).
\]
In the former case, we obtain a contribution $o(1)$ due to the diverging frequency supports of $f_{(0)}$ and $u^{n}$, and the same applies in the latter case, this time due to the vanishing property of the dispersive terms. To see this latter affirmation, consider a schematic term 
\begin{equation} \label{eq:reallytediousterm}
 \tilde{P}_0 \bigg( (\partial_t + \partial_r) \sum_{k<-10} \tilde{P}_k \Big( \big( U_{<k}^{(\phi^n,<K_*} \big)^\dagger \Big( \sum_{b=1}^B \tilde{\epsilon}^{Kn,<K_*}_{b,k} + \tilde{\epsilon}^{Kn,<K_*}_{B,k} \Big) \Big) P_0 f_{(0)} \bigg).
\end{equation}
Then if there is a low-high frequency interaction inside $\tilde{P}_k \big( \ldots \big)$ and the derivative $(\partial_t + \partial_r)$ falls on the high-frequency factor, we split the term as follows
\begin{align*}
 &\tilde{P}_0 \bigg( (\partial_t + \partial_r) \sum_{k<-10} \tilde{P}_k \Big( \big( U_{<k}^{(\phi^n,<K_*} \big)^\dagger \Big( \sum_{b=1}^B \tilde{\epsilon}^{Kn,<K_*}_{b,k} + \tilde{\epsilon}^{Kn,<K_*}_{B,k} \Big) \Big) P_0 f_{(0)} \bigg) \\
 &= \chi_{\{r<\Lambda\}} \tilde{P}_0 \bigg( (\partial_t + \partial_r) \sum_{k<-10} \tilde{P}_k \Big( \big( U_{<k}^{(\phi^n,<K_*} \big)^\dagger \Big( \sum_{b=1}^B \tilde{\epsilon}^{Kn,<K_*}_{b,k} + \tilde{\epsilon}^{Kn,<K_*}_{B,k} \Big) \Big) P_0 f_{(0)} \bigg) \\
 &\quad + \chi_{\{ r \geq \Lambda \}} \tilde{P}_0 \bigg( (\partial_t + \partial_r) \sum_{k<-10} \tilde{P}_k \Big( \big( U_{<k}^{(\phi^n,<K_*} \big)^\dagger \Big( \sum_{b=1}^B \tilde{\epsilon}^{Kn,<K_*}_{b,k} + \tilde{\epsilon}^{Kn,<K_*}_{B,k} \Big) \Big) P_0 f_{(0)} \bigg).
\end{align*}
Then we place the first term on the right into the $(-)$-component, which is easily seen to be $o(1)$ as $n\rightarrow\infty$ for any $\Lambda>0$, while we place the second term on the right into the $(\pm)$-component depending on the type of $(\partial_t + \partial_r)\tilde{\epsilon}^{Kn,<K_*}_{b,k}$ and $(\partial_t + \partial_r)\tilde{\epsilon}^{Kn,<K_*}_{b,B}$. If the latter is in the $(+)$-case, then we gain smallness by using $\big\| \chi_{\{ r>\Lambda \}} P_0 f_{(0)} \big\|_{L_{t,x}^\infty} \lesssim \Lambda^{-\frac12} \big\|f_{(0)}\big\|_{L_t^\infty \dot{H}^1_x}$. Similarly one deals with the $(-)$-case. Since we may let $\Lambda\rightarrow\infty$, these contributions are $o(1)$. The remaining frequency interactions inside \eqref{eq:reallytediousterm} are handled similarly. 

\medskip 

Finally, if the perturbative factor in $B_{<-10}^{(\phi^{n,<K_*}, R)}$ happens to be $\eta^{Kn,<K_*}$, then the contribution is seen to be bounded by $c_R \big\|\eta^{Kn,<K_*}\big\|_{S}$ with $\lim_{R\rightarrow\infty} c_R = 0$ due to the $R$-dependence of $B_{<-10}^{(\phi^{n,<K_*}, R)}$. 

\medskip 

\noindent {\bf{Step 4 of the proof of Lemma~\ref{lem:epsboundednonlin}}}: {\it{The preceding step implies the following:}}
\begin{equation} \label{eq:Step4lemma}
 \begin{aligned}
  &\big\|\eta^{Kn,<K_*}\big\|_{S([0,R])} \\
  &\lesssim \delta_4 e^{-R^{1000}-2^{K_*}} C_{13}\Big( \big\|\epsilon^{<K_*}_{bounded}\big\|_{S}, \|u^{n}\|_{S} \Big) + R^{41} (\log R+K_*)^{\frac12} \big\| \eta^{Kn,<K_*} \big\|_{S([0,R])} C_{14}\Big( \big\|\epsilon^{<K_*}_{bounded} \big\|_{S}, \|u^{n}\|_{S}\Big) \\
  &\quad + E \Big( \big\|\epsilon^{<K_*}_{bounded} \big\|_{S}, \|u^{n}\|_{S}\Big) \big\| \eta^{Kn,<K_*} \big\|_{S([0,R])}^2  + c_{R,M} \big\| \eta^{Kn,<K_*} \big\|_{S([0,R])},
 \end{aligned} 
\end{equation}
where 
\[
 \lim_{R, M\rightarrow\infty}c_{R,M} = 0.
\] 
To get this conclusion, we need to pass from 
\[
 \overline{(\eta^{Kn,<K_*}_k+\epsilon^{<K_*}_{bounded,k})^{M,R}} -  \overline{\epsilon^{<K_*,M,R}_{bounded,k}}
\]
to $\eta^{Kn,<K_*}_k$. For this it suffices to control 
\[
 \overline{(\eta^{Kn,<K_*}_k+\epsilon^{<K_*}_{bounded,k})^{M,R}} -  \overline{\epsilon^{<K_*,M,R}_{bounded,k}} - \eta^{Kn,<K_*}_k,
\]
which consists of cubic terms two of which (the high frequency and differentiated low frequency term) are ``perturbative''. But then, exploiting the precise definition of this expression, we easily see that its square-sum norm may be bounded by $c_{R,M} \big\|\eta^{Kn,<K_*}\big\|_{S}$. 

\medskip 

Finally, we have reached the point where we can get the conclusion of the lemma via a bootstrap argument: 

\medskip 

\noindent {\bf{Step 5 of the proof of Lemma~\ref{lem:epsboundednonlin}}}: {\it{Conclusion of the proof via time localization}}. Observe that we can choose $R$ larger than any function of $\big\|\epsilon^{<K_*}_{bounded}\big\|_{S}$ and $\|u^{n}\|_{S}$. In particular, all terms on the right hand side of \eqref{eq:Step4lemma} except the second one are perturbative. However, this term first arose in Step~1 of the proof of Lemma~\ref{lem:epsboundednonlin} on account of the norms of certain source terms in $F^{(R,M)}$ with good divisibility properties, and propagated through the following steps. In particular, we can replace the bad factor $R^{41}(\log R + K_*)^{\frac12} C_{10}\Big( \big\|\epsilon^{<K_*}_{bounded} \big\|_{S}\Big)$ by a good one, such as $R^{-1}$, by subdividing the time interval $[0,R]$ into 
\[
 M_1 := C_{10}\Big(\big\|\epsilon^{<K_*}_{bounded}\big\|_{S}\Big) R^{100}(\log R+K_*)^{100}
\]
 many sub-intervals $I_1, I_2,\ldots, I_{M_1}$, say, and replacing $[0,R]\times \R^2$ by $I_j\times \R^2$. Doing so then allows us to infer the better estimate 
\begin{equation} \label{eq:Step5lemma}
 \begin{aligned}
  \big\|\eta^{Kn,<K_*}\big\|_{S[I_1]} &\lesssim \delta_4 e^{-R^{1000}-2^{K_*}} C_{15}\Big( \big\|\epsilon^{<K_*}_{bounded}\big\|_{S}, \|u^{n}\|_{S}\Big) + R^{-1} \big\|\eta^{Kn,<K_*}\big\|_{S[I_1]} \\
  &\quad + F\Big( \big\|\epsilon^{<K_*}_{bounded}\big\|_{S}, \|u^{n}\|_{S} \Big) \big\|\eta^{Kn,<K_*}\big\|_{S[I_1]}^2 + c_{R,M} \big\|\eta^{Kn,<K_*}\big\|_{S[I_1]},
 \end{aligned} 
\end{equation}
from which we infer 
\[
 \big\| \eta^{Kn,<K_*} \big\|_{S[I_1]} \lesssim \delta_4 e^{-R^{1000}-2^{K_*}} C_{16} \Big( \big\|\epsilon^{<K_*}_{bounded}\big\|_{S}, \|u^{n}\|_{S} \Big).
\]
Progressing inductively through the intervals $I_j$, $j = 2,3,\ldots, M_1$, and each time using the improved bound for the previous interval for the initial data, we finally infer the bound
\begin{align*}
 \big\|\eta^{Kn,<K_*}\big\|_{S([0,R])} \lesssim \delta_4 e^{C_{17} \big( \|\epsilon^{<K_*}_{bounded} \|_{S}, \|u^{n}\|_{S} \big) R^{100} (\log R+K_*)^{100}} e^{-R^{1000}-2^{K_*}} C_{18}\Big( \big\|\epsilon^{<K_*}_{bounded} \big\|_{S}, \|u^{n}\|_{S} \Big) \lesssim \delta_4 \ll \delta_3,
\end{align*}
if we choose $R$ sufficiently large and $\delta_4$ sufficiently small compared to $\delta_3$. This then proves Lemma~\ref{lem:epsboundednonlin} up to establishing {\it{Observation 1}} and {\it{Observation 2}}, which we now do. 
\end{proof}

\medskip 

\begin{proof}[Proof of Observation~1]
Write $\epsilon \equiv \epsilon^{<K_*}_{bounded}$ and $\phi \equiv \phi^{n,<K_*}$. We consider the quantity $U_{<s}^{(\phi,R)} \big( U_{<s}^{(\epsilon,R)} \big)^\dagger$. Differentiating it with respect to $s$, we find 
\begin{align*}
 \partial_s \Big( U_{<s}^{(\phi,R)} \big( U_{<s}^{(\epsilon,R)} \big)^\dagger \Big) = U_{<s}^{(\phi,R)} (B_s^{(\phi,R)} - B_s^{(\epsilon,R)}) \big( U_{<s}^{(\epsilon,R)} \big)^\dagger
\end{align*}
and so since $\lim_{s\rightarrow-\infty} U_{<s}^{(\phi,R)} \big( U_{<s}^{(\epsilon,R)} \big)^\dagger = Id$, we get 
\[
 U_{<s}^{(\phi,R)} \big( U_{<s}^{(\epsilon,R)} \big)^\dagger = \int_{-\infty}^s U_{<\tilde{s}}^{(\phi)} (B_{\tilde{s}}^{(\phi,R)} - B_{\tilde{s}}^{(\epsilon,R)}) \big( U_{<\tilde{s}}^{(\epsilon,R)} \big)^\dagger \, d\tilde{s} + Id. 
\]
We shall then simply set 
\[
 G^{(n)}_k := U_{<k}^{(\phi,R)} \big( U_{<k }^{(\epsilon,R)} \big)^\dagger(0, 0).
\]
To see that this works, note that by our definition of $U_{<k}^{(\phi, R)}$, we get 
\begin{align*}
 \Big( U_{<k}^{(\phi,R)} \big( U_{<k}^{(\epsilon,R)} \big)^\dagger - G^{(n)}_k \Big) U_{<k}^{(\epsilon,R)} \epsilon_k &= \bigg( \int_0^R \int_{-\infty}^k \partial_t \Big( U_{<\tilde{s}}^{(\phi)}(B_{\tilde{s}}^{(\phi,R)} - B_{\tilde{s}}^{(\epsilon,R)}) \big( U_{<\tilde{s}}^{(\epsilon,R)} \big)^\dagger \Big) \, d\tilde{s} \, dt \bigg) U_{<k}^{(\epsilon,R)} \epsilon_k \\
 &\quad + \Big( U_{<k}^{(\phi,R)}(0, \cdot) \big( U_{<k}^{(\epsilon,R)} \big)^\dagger(0, \cdot) - G^{(n)}_k \Big) U_{<k}^{(\epsilon,R)} \epsilon_k.
\end{align*}
Then by expanding the difference $(B_{\tilde{s}}^{(\phi,R)} - B_{\tilde{s}}^{(\epsilon,R)})$ one checks that for the first term on the right we have 
\begin{align*}
 \sum_{a\in \Z} 2^{\sigma|a|} \Big( \sum_{k\in Z} \big\| P_{k+a}(\cdot) \big\|_{S_{k+a}}^2 \Big)^{\frac12} \leq o(1) + c_R \big\|\eta^{Kn,<K_*}\big\|_{S}
\end{align*}
as $n\rightarrow\infty$, where $\lim_{R\rightarrow\infty}c_R = 0$. Furthermore, split
\[
\epsilon_k = \chi_{B_{R_*}}\epsilon_k + (1-\chi_{B_{R_*}})\epsilon_k 
\]
for some $R_*\gg R$ as in case {\it{(v)}} of Step 1 of the proof of Lemma~\ref{lem:epsboundednonlin} and write
\begin{align*} 
 &\Big( U_{<k}^{(\phi,R)}(0, r) \big( U_{<k}^{(\epsilon,R)} \big)^\dagger(0, r) - G^{(n)}_k \Big) U_{<k}^{(\epsilon,R)} \chi_{B_{R_*}} \epsilon_k \\ 
 &= \bigg( \int_0^r \int_{-\infty}^k \partial_{\tilde{r}} \Big( U_{<\tilde{s}}^{(\phi)} (B_{\tilde{s}}^{(\phi,R)} - B_{\tilde{s}}^{(\epsilon,R)}) \big( U_{<\tilde{s}}^{(\epsilon,R)} \big)^\dagger \Big)(0,\tilde{r}) \, d\tilde{s} \, d\tilde{r} \bigg) U_{<k}^{(\epsilon,R)} \chi_{B_{R_*}} \epsilon_k.
\end{align*}
Then arguing as in {\it{(v)}} of Step 1 of the proof of Lemma~\ref{lem:epsboundednonlin}, one gets that the preceding expression satisfies
\begin{align*}
 \sum_{a \in \Z} 2^{\sigma|a|} \bigg( \sum_{k\in \Z} \big\|P_{k+a}(\cdot)\big\|_{S_{k+a}}^2 \bigg)^{\frac12} \leq  o(1) + \delta_4 e^{-R^{1000}-2^{K_*}}
\end{align*}
for $K, n$ large enough, and the same conclusion applies to the contribution of $(1-\chi_{B_{R_*}})\epsilon_k $ for $R_*$ large enough relative to $R, K_*$. 
This proves the first part of {\it{Observation 1}}, and the second part is obtained similarly, using divisibility for $F^{(R,M)}(\epsilon)$. 
\end{proof}

\begin{proof}[Proof of Observation 2]
We write 
\[
 U_{<s}^{(\phi,R)} = \int_{-\infty}^s U_{<\tilde{s}}^{(\phi,R)} B_{\tilde{s}}^{(\phi,R)} \, d\tilde{s}, \quad U_{<s}^{(u^{n},R)} = \int_{-\infty}^s U_{<\tilde{s}}^{(u^{n},R)} B_{\tilde{s}}^{(u^{n},R)} \, d\tilde{s},
\]
whence we get (specializing to frequency $k = 0$)
\begin{align*}
 \big( U_{<0}^{(\phi,R)} - U_{<0}^{(u^{n},R)} \big) \overline{u^{n,M,R}_{0}} &= \int_{-\infty}^0 U_{<\tilde{s}}^{(\phi,R)} (B_{\tilde{s}}^{(\phi,R)} - B_{\tilde{s}}^{(u^{n}_{\Lambda},R)}) \overline{u^{n,M,R}_{0}} \, d\tilde{s} \\
 &\quad + \int_{-\infty}^0 \big( U_{<\tilde{s}}^{(\phi,R)} - U_{<\tilde{s}}^{(u^{n},R)} \big) B_{\tilde{s}}^{(u^{n},R)} \overline{u^{n,M,R}_{0}} \,d\tilde{s}.
\end{align*}
One then reiterates in the second expression on the right, and so it suffices to bound the first term on the right. Then recalling that $\mu_n$ denotes the frequency dividing between the lowest frequency non-atomic constituent and the first large frequency atom, then if $0 < \mu_n$ we can close by exploiting the exponential decay of all of the $\epsilon_{bounded,s_1}^K$ etc. as well as well as the bound in {\it{(iv.2)}} to control $\eta^{nK}_{s_1}$ for $s_1 < \mu_n$. On the other hand, if $0 \geq \mu_n$, we split
\begin{align*}
 \int_{-\infty}^0 U_{<\tilde{s}}^{(\phi,R)} (B_{\tilde{s}}^{(\phi,R)} - B_{\tilde{s}}^{(u^{n},R)}) \overline{u^{n,M,R}_{0}} \, d\tilde{s} &= \int_{-\infty}^{\mu_n} U_{<\tilde{s}}^{(\phi,R)}(B_{\tilde{s}}^{(\phi,R)} - B_{\tilde{s}}^{(u^{n},R)}) \overline{u^{n,M,R}_{0}} \, d\tilde{s} \\
 &\quad + \int_{\mu_n}^{0} U_{<\tilde{s}}^{(\phi,R)} (B_{\tilde{s}}^{(\phi,R)} - B_{\tilde{s}}^{(u^{n},R)}) \overline{u^{n,M,R}_{0}} \, d\tilde{s}.
\end{align*}
The first term on the right is bounded by exploiting the exponential decay of each term in 
\[
 \epsilon_{bounded,s_1}^{<K_*} + \sum_{b=1}^B \tilde{P}_{s_1} \Big( \big(U_{<s_1}^{(\phi^{n,<K_*})} \big)^\dagger \tilde{\epsilon}^{Kn,<K_*}_{b, s_1} + \big( U_{<s_1}^{(\phi^{Kn,<K_*})} \big)^\dagger \tilde{\epsilon}^{Kn,<K_*}_{B, s_1} \Big) + \eta^{Kn,<K_*}_{s_1}
\]
with respect to $s_1-\mu_n$, while in the second term one loses $\mu_n$ due to the integration over $\tilde{s}$, but gains $o(1) \, 2^{-\sigma |\mu_n|}$ due to the smallness of $\overline{u^{n}_{0}}$. 
\end{proof}

\medskip 

Continuing with the Stage 1 of the proof of Theorem~\ref{thm:profiledecomp}, we pass to 

\medskip 

\noindent {\bf{Step 2}}: {\it{We adjust the decomposition \eqref{eq:profiles1} slightly and thereby achieve control on the interval $[R, t^n_1 - T_*]$.}} Specifically, we replace \eqref{eq:profiles1} by 
\begin{equation} \label{eq:profilesmodified}
 \phi^{n,<K_*} = u^{n} +  \sum_k \tilde{P}_k \bigg( (U_{<k}^{(\phi^{n,<K_*})} \big)^\dagger \tilde{\epsilon}^{Kn,<K_*}_{k} \bigg) + \sum_k \tilde{P}_k \bigg( \big( U_{<k}^{(\phi^{n,<K_*})} \big)^\dagger \Big( \sum_{b=1}^B \tilde{\epsilon}^{Kn,<K_*}_{b,k} + \tilde{\epsilon}^{Kn,<K_*}_{B,k} \Big) \bigg) + \tilde{\eta}^{Kn,<K_*},
\end{equation}
where we let $\tilde{\epsilon}^{Kn,<K_*}_{k}$ be the evolution with respect to \eqref{eq:therightone} of the initial data 
\[
 \Big( \big( U_{<k}^{(\phi^{n,<K_*})} \big) P_{[-K,\infty)} \epsilon^{<K_*}_{bounded,k}, \partial_t \big( U_{<k}^{(\phi^{n,<K_*})} \big) P_{[-K,\infty)} \epsilon^{<K_*}_{bounded,k} + \big( U_{<k}^{(\phi^{n,<K_*})} \big) \partial_t P_{[-K,\infty)} \epsilon^{<K_*}_{bounded,k} \Big) \Big|_{t = R}.
\]
Observe that we have included an extra low frequency cutoff $P_{[-K,\infty)}$ for the data and we absorb the corresponding correction into $\tilde{\eta}^{Kn,<K_*}$. Then it is straightforward to check that 
\[
 \lim_{K\rightarrow\infty} \limsup_{n\rightarrow\infty} \big\| \nabla_{t,x} \big( \tilde{\eta}^{Kn,<K_*} - \eta^{Kn,<K_*} \big)(R, \cdot) \big\|_{L_x^2} = 0. 
\]
Now we can formulate the following important 
\begin{lem} \label{lem:epsboundedlin} 
There is $T_*$ sufficiently large (depending on the first temporally unbounded profile) such that if $K, R, n$ are sufficiently large and $B = B\big( \|u^{n} \|_{S}, \delta_2\big)$ is sufficiently large, then we have 
\[
 \big\| \tilde{\eta}^{Kn,<K_*} \big\|_{S([R, t^n_1 - T_*])} < \delta_2. 
\]
Furthermore, there exist (not necessarily frequency localized) functions $H_{(k)}^{(1)}\big( \tilde{\eta}^{Kn,<K_*}, \tilde{\epsilon}^{Kn,<K_*}_{B}, \phi^{n,<K_*} \big)$, \\
$H_{(k)}^{(2)} \big( \tilde{\epsilon}^{Kn,<K_*}_{B}, \phi^{n,<K_*} \big)$ and $g_{(k)}\big( \tilde{\eta}^{Kn,<K_*}, \tilde{\epsilon}^{Kn,<K_*}_{B}, \phi^{n,<K_*} \big)$, $k\in \Z$, which are multilinear expressions in its frequency localized inputs as well as in the gauge transformations $U_{<h}^{(\phi^{n,<K_*})}$ and $\big( U_{<h}^{(\phi^{n,<K_*})} \big)^{\dagger}$, and with 
\[
 \sum_{a\in \Z} 2^{\sigma |a|} \bigg( \sum_k \big\| P_{k+a} H_{(k)}^{(2)}\big( \tilde{\epsilon}^{Kn,<K_*}_{B}, \phi^{n,<K_*}, u^{n} \big) \big\|_{L_t^1 L_x^2}^2 \bigg)^{\frac12} \leq c\big( \eta_{A^n}(\tilde{\epsilon}^{Kn,<K_*}_{B}) \big)
\]
with $c(\gamma)\rightarrow 0 $ as $\gamma\rightarrow 0$, as well as with
\begin{align*}
 \sum_{a\in \Z}2^{\sigma |a|} \bigg( \sum_k \big\| P_{k+a} H_{(k)}^{(1)}\big( \tilde{\eta}^{Kn,<K_*}, \tilde{\epsilon}^{Kn,<K_*}_{B}, \phi^{n,<K_*} \big) \big\|_{L_t^1 L_x^2}^2 \bigg)^{\frac12} &\lesssim \Bigl( \big\|\tilde{\eta}^{Kn,<K_*}\big\|_{S} + \big\|\tilde{\eta}^{Kn,<K_*}\big\|_{S}^{100} \Bigr) F\Big( \big\| \tilde{\epsilon}^{Kn,<K_*}_{B} \big\|_{S}\Big), \\
 \sum_{a\in \Z} 2^{\sigma |a|} \bigg( \sum_k \big\| P_{k+a} g_{(k)}\big( \tilde{\eta}^{Kn,<K_*}, \tilde{\epsilon}^{Kn,<K_*}_{B}, \phi^{n,<K_*}\big) \big\|_{S_{k+a}}^2 \bigg)^{\frac12} &\lesssim \big\| \tilde{\epsilon}^{Kn,<K_*}_{B}\big\|_{S}^2 + \big\|\tilde{\eta}^{Kn,<K_*}\big\|^2,
\end{align*}
and correction terms $\kappa_k$, $k\in \Z$ with 
\[
 \sum_{a \in \Z} 2^{\sigma |a|} \Big( \sum_k \big\| P_{k+a} \kappa_k \big\|_{S}^2 \Big)^{\frac12} \leq o(1)
\]
as $R, T_*, n\rightarrow\infty$, and such that 
\begin{equation} \label{eq:tiletaonlargeint}
 \tilde{\Box}_{A^n} \big( U_{<k-10}^{(\phi^{n,<K_*})} \tilde{\eta}^{Kn,<K_*}_k + g_{(k)} + \kappa_k \big) = H^{(1)}_{(k)}\big( \tilde{\eta}^{Kn,<K_*}, \tilde{\epsilon}^{Kn,<K_*}_{B}, \phi^{n,<K_*} \big) + H_{(k)}^{(2)}\big( \tilde{\epsilon}^{Kn,<K_*}_{B}, \phi^{n,<K_*} \big) + o^{(k)}(1),
\end{equation}
where the error satisfies the relation 
\[
 \lim_{R,T_*,n\rightarrow\infty} \sum_{a \in \Z} 2^{\sigma |a|} \Big( \sum_k \big\| P_{k+a} o^{(k)}(1) \big\|_{L_t^1 L_x^2([R, t^n_1 - T_*])}^2 \Big)^{\frac12} = 0. 
\]
\end{lem}
\begin{proof}
Consider 
\begin{equation} \label{eq:lem:epsboundedlin} 
 \tilde{\Box}_{A^n} \bigg( U_{<k-10}^{(\phi^{n,<K_*})} \overline{\phi^{n,<K_*}_k} - U_{<k-10}^{(u^{n})} \overline{u^{n}_{k}} - \tilde{\epsilon}^{Kn,<K_*}_{k} - \Big( \sum_{b=1}^B \tilde{\epsilon}^{Kn,<K_*}_{b,k} + \tilde{\epsilon}^{Kn,<K_*}_{B,k} \Big) \bigg).
\end{equation}
Then on the one hand this equals 
\[
 \tilde{\Box}_{A^n} \Big( U_{<k-10}^{(\phi^{n,<K_*})} \tilde{\eta}^{Kn,<K_*}_k + \zeta_k \Big),
\]
where (setting $\tilde{\epsilon}^{Kn,<K_*}_{0,k}:= \tilde{\epsilon}^{Kn,<K_*}_{k}$ for simplicity)
\begin{align*}
 \zeta_k &= \big( U_{<k-10}^{(\phi^{n,<K_*})} - U_{<k-10}^{(u^{n})} \big) \overline{u^{n}_{k}} \\ 
 &\quad + U_{<k-10}^{(\phi^{n,<K_*})} \tilde{P}_k \bigg( \big( U_{<k-10}^{(\phi^{n,<K_*})} \big)^{\dagger} \Big( \sum_{b=0}^B \tilde{\epsilon}^{Kn,<K_*}_{b,k} + \tilde{\epsilon}^{Kn,<K_*}_{B,k} \Big) \bigg) - \Big( \sum_{b=0}^B \tilde{\epsilon}^{Kn,<K_*}_{b,k} + \tilde{\epsilon}^{Kn,<K_*}_{B,k} \Big) \\
 &\quad + U_{<k-10}^{(\phi^{n,<K_*})} \bigg( \overline{\phi^{n,<K_*}_k} - \overline{u^{n}_{k}} - \tilde{P}_k \bigg( \big( U_{<k-10}^{(\phi^{n,<K_*})} \big)^{\dagger} \Big( \sum_{b=0}^B \tilde{\epsilon}^{Kn,<K_*}_{b,k} + \tilde{\epsilon}^{Kn,<K_*}_{B,k} \Big) \bigg) - \tilde{\eta}^{Kn,<K_*}_k \bigg),
\end{align*}
is of the form $\kappa_k + g_{(k)}$ as claimed. In fact, for the first difference term on the right, it is of type $\kappa_k$ due to a small variation on {\it{Observation 2}} in the preceding step, while for the second difference term on the right, i.e. 
\begin{equation} \label{eq:tediouscommutator}
 U_{<k-10}^{(\phi^{n,<K_*})} \tilde{P}_k \bigg( \big( U_{<k-10}^{(\phi^{n,<K_*})} \big)^{\dagger} \Big( \sum_{b=0}^B \tilde{\epsilon}^{Kn,<K_*}_{b,k} + \tilde{\epsilon}^{Kn,<K_*}_{B,k} \Big) \bigg) - \Big( \sum_{b=0}^B \tilde{\epsilon}^{Kn,<K_*}_{b,k} + \tilde{\epsilon}^{Kn,<K_*}_{B,k} \Big),
\end{equation}
one observes that it coincides with 
\begin{equation} \label{eq:tediouscommutator2}
 \Big( \sum_{b=0}^B \tilde{P}_k \tilde{\epsilon}^{Kn,<K_*}_{b,k} + \tilde{P}_k\tilde{\epsilon}^{Kn,<K_*}_{B,k} \Big) - \Big( \sum_{b=0}^B \tilde{\epsilon}^{Kn,<K_*}_{b,k} + \tilde{\epsilon}^{Kn,<K_*}_{B,k} \Big),
\end{equation}
up to the sum of a term satisfying the desired bound for $\kappa_k$ and a term of type $g_{(k)}\big( \tilde{\eta}^{Kn,<K_*}, \tilde{\epsilon}^{Kn,<K_*}_{B}, \phi^{n,<K_*} \big)$. To see this, note that one can write the difference of the term \eqref{eq:tediouscommutator} and the immediately preceding term schematically as
\[
 2^{-k}\nabla_x \big( U_{<k-10}^{(\phi^{n,<K_*})} \big)^{\dagger} \Big( \sum_{b=0}^B \tilde{\epsilon}^{Kn,<K_*}_{b,k} + \tilde{\epsilon}^{Kn,<K_*}_{B,k} \Big)
\]
and expanding out $\nabla_x \big( U_{<k-10}^{(\phi^{n,<K_*})} \big)^{\dagger}$ using the equation for $U_{<k-10}^{(\phi^{n,<K_*})}$, one retains only those terms involving $\tilde{\eta}^{Kn,<K_*}$ and $\tilde{\epsilon}^{Kn,<K_*}_{B}$ in combination with the factor $\tilde{\epsilon}^{Kn,<K_*}_{B,k}$ to be included into the term $g_{(k)}$. 

The preceding difference \eqref{eq:tediouscommutator2} also satisfies the desired bound for $\kappa_k$ as one easily sees after applying the operator $\tilde{P}_k$ to the equation \eqref{eq:therightone}. 
Finally, for the third term in the above formula for $\zeta_k$ we note that it arises due to the nonlinear interactions inside $\overline{\phi^{n,<K_*}_k}$, and its contribution to the norm controlling $\kappa_k$ is easily seen to be bounded by $O \Big( \big\|\tilde{\eta}^{Kn,<K_*}\big\|_{S}^2 + \big\|\tilde{\epsilon}^{Kn,<K_*}_{B} \big\|_{S}^2 \Big)$ in the limit $R, T, n\rightarrow\infty$ . 

Getting back to \eqref{eq:lem:epsboundedlin}, on the other hand we can write this as
\begin{equation} \label{eq:anothermess111} 
 \begin{aligned}
  &\tilde{\Box}_{A^n} \bigg( U_{<k-10}^{(\phi^{n,<K_*})} \overline{\phi^{n,<K_*}_k} - U_{<k-10}^{(u^{n})} \overline{u^{n}_{k}} - \Big( \sum_{b=0}^B \tilde{\epsilon}^{Kn,<K_*}_{b,k} + \tilde{\epsilon}^{Kn,<K_*}_{B,k} \Big) \bigg) \\
  &= \tilde{\Box}_{A^n} \Big( U_{<k-10}^{(\phi^{n,<K_*})} \overline{\phi^{n,<K_*}_k} - U_{<k-10}^{(u^{n})} \overline{u^{n}_{k}} \Big) \\
  &= \Box \Big( U_{<k-10}^{(\phi^{n,<K_*})} \overline{\phi^{n,<K_*}_k} \Big) - \Box \Big( U_{<k-10}^{(u^{n})} \overline{u^{n}_k} \Big) + U_{<\mu_n}^{(u^{n})} \cA_{\alpha, low}(u^{n})  \big( U_{<\mu_n}^{(u^{n})} \big)^\dagger \partial^{\alpha} \Big( U_{<k-10}^{(u^{n})} \overline{u^{n}_{k}} - U_{<k-10}^{(\phi^{n,<K_*})}\overline{\phi^{n,<K_*}_k} \Big).
 \end{aligned}
\end{equation}
As for the last term on the right, expand it out as 
\begin{align*}
 &U_{<\mu_n}^{(u^{n})} \cA_{\alpha, low}(u^{n})  \big( U_{<\mu_n}^{(u^{n})} \big)^\dagger \partial^{\alpha} \Big( U_{<k-10}^{(u^{n})} \overline{u^{n}_{k}} - U_{<k-10}^{(\phi^{n,<K_*})}\overline{\phi^{n,<K_*}_k} \Big) \\
 &= U_{<\mu_n}^{(u^{n})} \cA_{\alpha, low}(u^{n}) \big( U_{<\mu_n}^{(u^{n})} \big)^\dagger \partial^{\alpha} \Big( \big( U_{<k-10}^{(u^{n})} - U_{<k-10}^{(\phi^{n,<K_*})} \big)  \overline{u^{n}_{k}} \Big) \\
 &\quad - U_{<\mu_n}^{(u^{n})} \cA_{\alpha, low}(u^{n}) \big( U_{<\mu_n}^{(u^{n})} \big)^\dagger \partial^{\alpha} \Big( U_{<k-10}^{(\phi^{n,<K_*})} \big( \overline{\phi^{n,<K_*}_k}-\overline{u^{n}_{k}} \big) \Big).
\end{align*}
Here the first term on the right is easily seen to be of the form
\begin{equation} \label{eq:errorterm11} 
 \begin{aligned}
  &U_{<\mu_n}^{(u^{n})} \cA_{\alpha, low}(u^{n}) \big( U_{<\mu_n}^{(u^{n})} \big)^\dagger \partial^{\alpha} \Big( \big( U_{<k-10}^{(u^{n})} - U_{<k-10}^{(\phi^{n,<K_*})} \big) \overline{u^{n}_{k}} \Big) \\
  &= U_{<\mu_n}^{(u^{n})} \cA_{\alpha, low}(u^{n}) \big( U_{<\mu_n}^{(u^{n})} \big)^\dagger \partial^{\alpha} \Big( \big( U_{<k-10}^{(u^{n})} - U_{<k-10}^{(u^{n} + \tilde{\eta}^{Kn,<K_*})} \big) \overline{u^{n}_{k}} \Big) + o^{(k)}(1),
 \end{aligned}
\end{equation}
where the error $o^{(k)}(1)$ has the claimed vanishing property provided $T_*, R, n\rightarrow\infty$, and similarly we have 
 \begin{equation} \label{eq:errorterm12}
  \begin{aligned}
   &U_{<\mu_n}^{(u^{n})} \cA_{\alpha, low}(u^{n}) \big( U_{<\mu_n}^{(u^{n})} \big)^\dagger \partial^{\alpha} \Big( U_{<k-10}^{(\phi^{n,<K_*})} \big( \overline{\phi^{n,<K_*}_k} - \overline{u^{n}} \big) \Big) \\
   &= U_{<\mu_n}^{(u^{n})} \cA_{\alpha, low}(u^{n}) \big( U_{<\mu_n}^{(u^{n})} \big)^\dagger \partial^{\alpha} \Big( U_{<k-10}^{(\phi^{n,<K_*})} \tilde{\eta}^{Kn,<K_*}_k \Big) \\ 
   &\quad + U_{<\mu_n}^{(u^{n})} \cA_{\alpha, low}(u^{n}) \big( U_{<\mu_n}^{(u^{n})} \big)^\dagger \partial^{\alpha} \bigg( U_{<k-10}^{(\phi^{n,<K_*})} P_k \Big( \sum_{b=0}^B \big( U_{<k-10}^{(\phi^{n,<K_*})} \big)^\dagger \tilde{\epsilon}^{Kn,<K_*}_{b,k} + \big( U_{<k-10}^{(\phi^{n,<K_*})} \big)^\dagger \tilde{\epsilon}^{Kn,<K_*}_{B,k} \Big) \bigg) \\
   &\quad + o^{(k)}(1).
  \end{aligned}
 \end{equation}
Of the preceding two expressions \eqref{eq:errorterm11} and \eqref{eq:errorterm12}, the most delicate case appears when the derivative $\partial^{\alpha}$ in \eqref{eq:errorterm11} falls on the term $u^{n}_{k}$, since then smallness has to come from essentially a factor of the form $\tilde{\eta}^{Kn,<K_*}_{<k-10}$ arising essentially from $U_{<k-10}^{(u^{n})} - U_{<k-10}^{(u^{n} + \tilde{\eta}^{Kn,<K_*})}$. Dealing with this situation requires an energy-type bound for $\tilde{\eta}^{Kn,<K_*}_k$ for frequencies $k \leq \mu_n$, analogous to Lemma~\ref{lem:bootstrap_non_atomic_kinetic}. In fact, making a suitable bootstrap hypothesis for the very low frequency part of $\tilde{\eta}^{Kn,<K_*}_k$ and proceeding as in the proof of Lemma~\ref{lem:bootstrap_non_atomic_kinetic}, we infer the following bound 
\[
 \big\|\nabla_{t,x}\tilde{\eta}^{Kn,<K_*}_{k}\big\|_{L_t^\infty L_x^2}\leq o(1)\cdot 2^{-\sigma|k-\mu_n|},
\]
where $o(1)$ vanishes as $n\rightarrow\infty$. Using this and, as usual, splitting into the cases $k \leq \mu_n$ and $k>\mu_n$, one easily infers that as $n\rightarrow\infty$,
\begin{align*}
 \sum_{a\in\Z} 2^{\sigma |a|} \bigg( \sum_k \Big\| P_{k+a} \Big( U_{<\mu_n}^{(u^{n})} \cA_{\alpha, low}(u^{n}) \big( U_{<\mu_n}^{(u^{n})} \big)^\dagger \partial^{\alpha} \big( \big( U_{<k-10}^{(u^{n})} -U_{<k-10}^{(u^{n} + \tilde{\eta}^{Kn,<K_*})} \big) u^{n}_{k} \big) \Big) \Big\|_{L_t^1L_x^2([R, t^n_1 - T_*])}^2 \bigg)^{\frac12} = o(1)
\end{align*}
On the other hand, the first two lines on the right in \eqref{eq:errorterm12} cancel against a corresponding term obtained when expanding 
\[
 \Box \big( U_{<k-10}^{(\phi^{n,<K_*})} \overline{\phi^{n,<K_*}_k} \big) - \Box \big( U_{<k-10}^{(u^{n})} \overline{u^{n}_{k}} \big) = U_{<k-10}^{(\phi^{n,<K_*})} F_k\big( \phi^{n,<K_*} \nabla_{t,x} \phi^{n,<K_*} \big) - U_{<k-10}^{(u^{n})} F_k \big( u^{n}, \nabla_{t,x} u^{n} \big).
\]
In fact, we get 
\begin{equation} \label{eq:expansion}
 \begin{aligned}
  &U_{<k-10}^{(\phi^{n,<K_*})} F_k \big( \phi^{n,<K_*}, \nabla_{t,x} \phi^{n,<K_*}) - U_{<k-10}^{(u^{n})} F_k \big( u^{n}, \nabla_{t,x} u^{n} \big) \\
  &= U_{<\mu_n}^{(u^{n})} \cA_{\alpha, low}(u^{n}) \big( U_{<\mu_n}^{(u^{n})} \big)^\dagger \partial^{\alpha} \big( U_{<k-10}^{(\phi^{n,<K_*})} \tilde{\eta}^{Kn,<K_*}_k \big) \\
  &\quad + U_{<\mu_n}^{(u^{n})} \cA_{\alpha, low}(u^{n}) \big( U_{<\mu_n}^{(u^{n})} \big)^\dagger \partial^{\alpha} P_k \Big( \sum_{b=0}^B \big( U_{<k-10}^{(\phi^{n,<K_*})} \big)^{\dagger} \tilde{\epsilon}^{Kn,<K_*}_{b,k} + \big( U_{<k-10}^{(\phi^{n,<K_*})} \big)^{\dagger} \tilde{\epsilon}^{Kn,<K_*}_{B,k} \Big) \\
  &\quad + H^{(1)}_{(k)}\big( \tilde{\eta}^{Kn,<K_*}, \tilde{\epsilon}^{Kn,<K_*}_{B}, \phi^{Kn,<K_*} \big) + H^{(2)}_{(k)}\big( \tilde{\epsilon}^{Kn,<K_*}_{B}, \phi^{Kn,<K_*} \big) + o^{(k)}(1),
 \end{aligned} 
\end{equation}
where the error $o^{(k)}(1)$ is again as in the statement of Lemma~\ref{lem:epsboundedlin}. In fact, with the exception of the second term on the right in \eqref{eq:expansion}, all terms on the left in \eqref{eq:expansion} involving at least one perturbative factor of the form 
\[
 \sum_{b=0}^B \big( U_{<k-10}^{(\phi^{n,<K_*})} \big)^{\dagger} \tilde{\epsilon}^{Kn,<K_*}_{b,k}
\]
are seen to be of the form $o^{(k)}(1)$ by using Lemma~\ref{lem:bilinearnullformbelowL2use}. The same applies for all terms involving at least one perturbative factor $u^{n}$. For the remaining interactions, the terms of the form $H^{(1)}\big( \tilde{\eta}^{Kn,<K_*}, \tilde{\epsilon}^{Kn,<K_*}_{B}, \phi^{Kn,<K_*} \big)$ arise when at least one of the perturbative terms in the source is of the form $\tilde{\eta}^{Kn,<K_*}$ while all other perturbative factors are of the form
\[
 \sum_k \tilde{P}_k \Big( \big( U_{<k-10}^{(\phi^{n,<K_*})} \big)^{\dagger} \tilde{\epsilon}^{Kn,<K_*}_{B,k} \Big),
\] 
while the terms of the form $H^{(2)}\big( \tilde{\epsilon}^{Kn,<K_*}_{B}, \phi^{Kn,<K_*} \big)$ arise when all perturbative factors are of the form 
\[
 \sum_k \big( U_{<k-10}^{(\phi^{n,<K_*})} \big)^\dagger \tilde{\epsilon}^{Kn,<K_*}_{Bk}.
\]
Now a priori in expressions of the latter kind the norm of the expression will also depend on $\big\| \phi^{Kn,<K_*} \big\|_{S}$, provided a derivative falls on the factor $U_{<k-10}^{(\phi^{n,<K_*})}$. In this case we re-express $\phi^{Kn,<K_*}$ using \eqref{eq:profilesmodified}. On the other hand, all terms in $H^{(2)}\big( \tilde{\epsilon}^{Kn,<K_*}_{B}, \phi^{Kn,<K_*} \big)$ where no derivative lands on $U_{<k-10}^{(\phi^{n,<K_*})}$ are left unchanged. This process can be continued, absorbing terms either into $H^{(1)}\big( \tilde{\eta}^{Kn,<K_*}, \tilde{\epsilon}^{Kn,<K_*}_{B}, \phi^{Kn,<K_*} \big)$ or into $H^{(2)}\big( \tilde{\epsilon}^{Kn,<K_*}_{B}, \phi^{Kn,<K_*} \big)$ or into $o^{(k)}(1)$, until all remaining terms in $H^{(2)}\big( \tilde{\epsilon}^{Kn,<K_*}_{B}, \phi^{Kn,<K_*} \big)$ in which a derivative can still potentially fall on a factor $\big( U_{<k-10}^{(\phi^{n,<K_*})} \big)^\dagger$ can be estimated purely in terms of Strichartz norms of its inputs. Then it is easy to see that it suffices to control $\big\| \nabla_{t,x} \big( U_{<k-10}^{(\phi^{n,<K_*})} \big)^\dagger \big\|_{L_t^\infty L_x^{2+}}$, which we do by energy conservation. This proves the relation \eqref{eq:tiletaonlargeint} and a straightforward bootstrap/divisibility argument then yields the bound 
\[
 \big\| \tilde{\eta}^{Kn,<K_*} \big\|_{S([R, t^n_1 - T_*])} < \delta_2,
\]
provided we have picked $R, T_*$, and $n$ sufficiently large. In fact, to pass from the bound on $\big\{ U_{<k-10}^{(\phi^{n,<K_*})} \tilde{\eta}^{Kn,<K_*}_k \big\}_{k\in\Z}$ to the bound for $\tilde{\eta}^{Kn,<K_*}$, one argues exactly as in Step 3 of the proof of the preceding Lemma~\ref{lem:epsboundednonlin}. 
\end{proof}

At this point, we have shown that under the hypotheses of Theorem~\ref{thm:profiledecomp}, we control the solution $\phi^{n,<K_*}$ for $n$ sufficiently large on the full interval $[0, t^n_1 - T_*]$ for a fixed but sufficiently large $T_*$, and that we in fact obtain bounds on $\big\| \phi^{n,<K_*} \big\|_{S}$ which are independent of $K_*$ and $n$. 

\medskip

Our aim now is to continue the preceding process on the next large time interval $[t^n_1 - T_*, t^n_2 - T_{1*}]$ for $T_{1*}$ sufficiently large. 

\medskip 

\noindent {\bf{Stage 2}}: {\it{Control of the solution on $[t^n_1 - T_*, t^n_2 - T_{1*}]$.}} Here we want to repeat the method of Stage 1 by replacing the initial time $t = 0$ by the initial time $t = t^n_1 - T_*$. A basic technical difficulty we face here is that we first have to introduce the analogue of $\epsilon^{<K_*}_{bounded}$, say $\epsilon^{<K_*}_{1bounded}$, by suitably modifying the function $\sum_k \big( U_{<k-10}^{(\phi^{n,<K_*})} \big)^\dagger \tilde{\epsilon}^{Kn,<K_*}_{1,k}$, i.e. the profile with $b = 1$, which concentrates at time $t = t^n_1 - T_*$. For this we will also have to take into account the perturbation $\tilde{\eta}^{Kn,<K_*}$, which, while small in energy, may lead to a significant effect on the $L^\infty_x$-norm. To construct the profile, we more or less proceed as in the proof of Lemma~\ref{lem:prelim profile}, but the following lemma simplifies things.
\begin{lem} \label{lem:weakvanishingofeta} 
The function $\nabla_{t,x} \tilde{\eta}^{Kn,<K_*}(t^n_1 - T_*, \cdot)$ vanishes weakly in $L^2_x(\R^2)$ as $T_*, K, n\rightarrow\infty$. 
\end{lem}
\begin{proof}
This is a consequence of the identity \eqref{eq:tiletaonlargeint} in the preceding Lemma~\ref{lem:epsboundedlin}. Writing the flow of \eqref{eq:therightone} associated with the data $(0, f)$ at $t = 0$ in the form 
\[
 U_{A^n}(t)(f), 
\]
we have 
\begin{align*}
 &\Big( U_{<k-10}^{(\phi^{n,<K_*})} \tilde{\eta}^{Kn,<K_*}_k + g_{(k)}\big( \tilde{\eta}^{Kn,<K_*}, \tilde{\epsilon}^{Kn,<K_*}_{B}, \phi^{n,<K_*} \big) \Big) \Big|_{t^n_1 - T_*} \\
 &= S_{A^n} \Big( \big( U_{<k-10}^{(\phi^{n,<K_*})} \tilde{\eta}^{Kn,<K_*}_k + g_{(k)}\big( \tilde{\eta}^{Kn,<K_*}, \tilde{\epsilon}^{Kn,<K_*}_{B}, \phi^{n,<K_*} \big) \big)[0] \Big) \Big|_{t^n_1 - T_*} \\ 
 &\quad + \int_0^{t} U_{A^n}(t-s)\Big( H^{(1)}_{(k)}\big( \tilde{\eta}^{Kn,<K_*}, \tilde{\epsilon}^{Kn,<K_*}_{B}, \phi^{n,<K_*} \big) + H_{(k)}^{(2)}\big( \tilde{\epsilon}^{Kn,<K_*}_{B}, \phi^{n,<K_*}\big) \Big)(s, \cdot) \, ds \Big|_{t^n_1 - T_*}  + o(1),
\end{align*}
where the error vanishes (in the energy sense) upon letting $T_*, K, n\rightarrow\infty$. Consider the principal contribution to the last integral, which is the expression 
\[
 \int_0^{t} U_{A^n}(t-s) \Big( H_{(k)}^{(2)}\big( \tilde{\epsilon}^{Kn,<K_*}_{B}, \phi^{n,<K_*}\big) \Big)(s, \cdot) \, ds \Big|_{t^n_1 - T_*}.
\]
Recall that $H_{(k)}^{(2)}\big( \tilde{\epsilon}^{Kn,<K_*}_{B}, \phi^{n,<K_*} \big)$ is a multilinear expression with $\sum_k \big( U_{<k-10}^{(\phi^{n,<K_*})} \big)^\dagger \tilde{\epsilon}^{Kn,<K_*}_{B,k}$ in its ``perturbative factors''. In turn, $\tilde{\epsilon}^{Kn,<K_*}_{B,k}$ can be expanded as a sum of free waves associated to \eqref{eq:therightone} up to an arbitrarily small error $\tilde{\epsilon}^{Kn,<K_*}_{B_1}$ and each of these free waves concentrates at a time scale diverging from $t^n_1 - T_*$. This means that up to an arbitrarily small error and upon letting $n\rightarrow\infty$, the function 
\[
 \int_0^{t} U_{A^n}(t-s)\Big( H_{(k)}^{(2)}\big( \tilde{\epsilon}^{Kn,<K_*}_{B}, \phi^{n,<K_*} \big) \Big)(s, \cdot) \, ds \Big|_{t^n_1 - T_*}
\]
converges in the $S$-norm to a solution of \eqref{eq:therightone} concentrating away from $t = t^n_1 - T_*$, and in particular we have 
\[
 \nabla_{t,x} \int_0^{t} U_{A^n}(t-s) \Big( H_{(k)}^{(2)}\big( \tilde{\epsilon}^{Kn,<K_*}_{B}, \phi^{n,<K_*} \big) \Big)(s, \cdot) \, ds \Big|_{t^n_1 - T_*} \rightharpoonup 0.
\]
As for the remaining integral 
\[
 \int_0^{t} U_{A^n}(t-s) \Big( H^{(1)}_{(k)}\big( \tilde{\eta}^{Kn,<K_*}, \tilde{\epsilon}^{Kn,<K_*}_{B}, \phi^{n,<K_*} \big) \Big)(s, \cdot) \, ds \Big|_{t^n_1 - T_*}, 
\]
inserting the principal term for $\tilde{\eta}^{Kn,<K_*}$, which equals 
\[
 \int_0^{t} U_{A^n}(t-s) \Big( H_{(k)}^{(2)}\big( \tilde{\epsilon}^{Kn,<K_*}_{B}, \phi^{n,<K_*} \big) \Big)(s, \cdot) \, ds,
\]
results in a similar term, again vanishing weakly at $t =  t^n_1 - T_*$ by the same kind of argument, and one winds up with another error term stemming from substituting 
\[
 \int_0^{t} U_{A^n}(t-s) \Big( H^{(1)}_{(k)}\big( \tilde{\eta}^{Kn,<K_*}, \tilde{\epsilon}^{Kn,<K_*}_{B}, \phi^{n,<K_*} \big) \Big)(s, \cdot) \, ds 
\]
for $\tilde{\eta}^{Kn,<K_*}$. In light of the smallness of $\big\|\tilde{\epsilon}^{Kn,<K_*}_{B}\big\|_{S}$ and $\big\|\tilde{\eta}^{Kn,<K_*}\big\|_{S}$, this term has much smaller $S$-norm than 
\[
 \int_0^{t} U_{A^n}(t-s) \Big( H^{(1)}_{(k)}\big( \tilde{\eta}^{Kn,<K_*}, \tilde{\epsilon}^{Kn,<K_*}_{B}, \phi^{n,<K_*} \big) \Big)(s, \cdot) \, ds.
\]
Repeating this argument a finite number of times leads to a sequence of terms vanishing weakly at $t =  t^n_1 - T_*$ up to an arbitrarily small error, which yields the claim.
\end{proof}

We now construct a good bounded profile $\epsilon^{<K_*}_{1bounded}$ at time $t = t^n_1 - T_*$. Departing from \eqref{eq:profilesmodified}, we get upon restriction to a suitable subsequence with respect to $n$,
\begin{align*}
 \bigg( \sum_k \tilde{P}_k \Big( \big( U_{<k}^{(\phi^{n,<K_*})} \big)^\dagger \tilde{\epsilon}^{Kn,<K_*}_{1,k} \Big) + \tilde{\eta}^{Kn,<K_*} + p_{1*} \bigg) \bigg|_{t^n_1 - T_*} = \phi^{n,<K_*}\big|_{t^n_1 - T_*} + o_{L^\infty_{loc}}(1)
\end{align*}
with the error vanishing as $n\rightarrow\infty$ in the sense of $L_{loc}^{\infty}$. Here $p_{1*} \in \mathbb{S}^m$ may be set equal to $\lim_{n\rightarrow\infty} u^{n}(t^n_1 - T_*, 0)$, the limit existing upon passing to a suitable subsequence. In fact, the remaining terms 
\[
 \sum_k \tilde{P}_k \bigg( \big( U_{<k}^{(\phi^{n,<K_*})} \big)^\dagger \Big( \sum_{b=0,\,b\neq 1}^B \tilde{\epsilon}^{Kn,<K_*}_{b,k} + \tilde{\epsilon}^{Kn,<K_*}_{B,k} \Big) \bigg) \bigg|_{t^n_1 - T_*}
\]
all converge to zero in $L^\infty_{loc}$ as $n\rightarrow\infty$ due to the fact that we have essentially (up to exponentially decaying tails) localized their frequency supports to compact intervals, and each of them converges weakly to zero at time $t^n_1 - T_*$ as $n\rightarrow\infty$.

Now due to the essentially compact frequency support of the expression, upon letting $K, n\rightarrow\infty$ along a suitable subsequence, we can find $\tilde{\epsilon}^{<K_*}_{1bounded}$ such that 
\[
 \bigg( \sum_k \tilde{P}_k \Big( \big( U_{<k}^{(\phi^{n,<K_*})} \big)^\dagger \tilde{\epsilon}^{Kn,<K_*}_{1,k} \Big) + \tilde{\eta}^{Kn,<K_*} \bigg) \bigg|_{t^n_1 - T_*} \rightharpoonup \,  \tilde{\epsilon}^{<K_*}_{1bounded} \big|_{t^n_1 - T_*}
\]
and also strongly in $H^{1+}_{loc}$, in particular in $L^\infty_{loc}$. In particular, it follows that  $\tilde{\epsilon}^{<K_*}_{1bounded} + p_{1*}$ maps into $\mathbb{S}^m$. Thus it is natural to set 
\[
 \epsilon^{<K_*}_{1bounded} \big|_{t^n_1 - T_*} := \tilde{\epsilon}^{<K_*}_{1bounded} \big|_{t^n_1 - T_*} + p_{1*},
\]
and to pass from \eqref{eq:profilesmodified} to a new decomposition with a new error term $\tilde{\tilde{\eta}}^{Kn,<K_*}$, where now $\epsilon^{<K_*}_{1bounded}$ denotes the nonlinear evolution associated with the data $\epsilon^{<K_*}_{1bounded}\big|_{t^n_1 - T_*}$, 
\begin{equation} \label{eq:profilesmodified1}
 \phi^{n,<K_*} = u^{n} - p_{1*} + \epsilon^{<K_*}_{1bounded} + \sum_k \tilde{P}_k \bigg( \big( U_{<k}^{(\phi^{n,<K_*})} \big)^\dagger \Big( \sum_{b=0,\,b\neq 1}^B \tilde{\epsilon}^{Kn,<K_*}_{b,k} + \tilde{\epsilon}^{Kn,<K_*}_{B,k} \Big) \bigg) + \tilde{\tilde{\eta}}^{Kn,<K_*}.
\end{equation}
This is the analogue of \eqref{eq:profiles} which we expect to be valid on a large but finite time interval $[t^n_1 - T_*, t^n_1 - T_* + R_1]$, say, where $R_1$ will now depend on the scattering properties of  $\epsilon^{<K_*}_{1bounded}$, just as $R$ before was chosen depending on the scattering properties of the evolution of $\epsilon^{<K_*}_{bounded}$. Also, comparing \eqref{eq:profilesmodified1} and \eqref{eq:profilesmodified}, we have 
\[
 \tilde{\tilde{\eta}}^{Kn,<K_*} \big|_{t^n_1 - T_*} = \sum_k \big( U_{<k}^{(\phi^{n,<K_*})} \big)^\dagger \tilde{\epsilon}^{Kn,<K_*}_{1,k} \big|_{t^n_1 - T_*} + \tilde{\eta}^{Kn,<K_*} \big|_{t^n_1 - T_*} - \tilde{\epsilon}^{<K_*}_{1bounded}.
\]
Then by definition of $\tilde{\epsilon}^{<K_*}_{1bounded}$ we have $\lim_{n\rightarrow\infty}\big\|\tilde{\tilde{\eta}}^{Kn,<K_*}\big|_{t^n_1 - T_*}\big\|_{L_{loc}^\infty} = 0$, which is analogous to the property of $\eta^{Kn,<K_*}$ in Lemma~\ref{lem:prelim profile}. Also, note that for any $R_*>0$ we have 
\[
 \lim_{n\rightarrow\infty} \big\| \nabla_{t,x} \tilde{\tilde{\eta}}^{Kn,<K_*} \big|_{t^n_1 - T_*} \big\|_{L_x^2(B_{R_*})} = 0.
\]
We further observe that due to Lemma~\ref{lem:weakvanishingofeta} and its proof we have the following. Denoting by $\tilde{\eta}_{(1)k}^{Kn,<K_*}$ the solution of the approximation to \eqref{eq:tiletaonlargeint} given by 
\begin{equation} \label{eq:tiletaapprox} 
 \begin{aligned}
  &\tilde{\Box}_{A^n} \Big( U_{<k-10}^{(\phi^{n,<K_*})} \tilde{\eta}^{Kn,<K_*}_{(1)k} + g_{(k)}\big( \tilde{\eta}^{Kn,<K_*}_{(1)}, \tilde{\epsilon}^{Kn,<K_*}_{B}, \phi^{n,<K_*} \big) \Big) \\
  &= H^{(1)}_{(k)}\big( \tilde{\eta}^{Kn,<K_*}_{(1)}, \tilde{\epsilon}^{Kn,<K_*}_{B}, \phi^{n,<K_*} \big) + H_{(k)}^{(2)}\big( \tilde{\epsilon}^{Kn,<K_*}_{B}, \phi^{n,<K_*} \big)
 \end{aligned}
\end{equation}
with the same initial data as $\tilde{\eta}^{Kn,<K_*}$ at $t = R$, then $\tilde{\eta}^{Kn,<K_*}_{(1)}$ obeys the same bound as $\tilde{\eta}^{Kn,<K_*}$, and moreover we have 
\[
 \lim_{R, T_*, K,n\rightarrow\infty} \big\| \nabla_{t,x} \big( \tilde{\eta}^{Kn,<K_*}_{(1)}\big|_{t^n_1 - T_*} - \tilde{\tilde{\eta}}^{Kn,<K_*}\big|_{t^n_1 - T_*} \big) \big\|_{L_x^2} = 0.
\]
We also note that 
\[
\lim_{R, T_*, K,n\rightarrow\infty}\big\|\tilde{\eta}^{Kn,<K_*}_{(1)}- \tilde{\tilde{\eta}}^{Kn,<K_*})\big\|_{S([R, t_1^n - T_*])} = 0.
\]

At this point, we briefly pause to make an important 
\begin{rem} 
Note that at this stage of the proof of Theorem~\ref{thm:profiledecomp} we have already introduced three different perturbation terms $\eta^{Kn,<K_*}$, $\tilde{\eta}^{Kn,<K_*}$, and $\tilde{\tilde{\eta}}^{Kn,<K_*}$. As we intend to iterate the preceding procedure on each of the time intervals $[t^n_b - T_{b*}, t^n_{b+1} - T_{b+1*}]$, $b = 1,2,\ldots, B$, we have to be careful that this process does not result in accruing more and more errors, i.e. we have to ensure uniform smallness of these perturbations. The idea here is that we have the parameters $R, T_*,K, n$ and analogous ones for the later stages ($B$ in total) at our disposal, while $B$ will be chosen as already mentioned only depending on an absolute small constant $\delta_2$ as well as on $\|u^{n}\|_{S}$. Thus, by eventually picking $R, T_*,K, n$ etc. sufficienty large, we will be able to ensure that passing from one error to the next will be small even when re-iterated $B$ times. 
\end{rem}

At this stage, we have the tools to address the preceding remark and continue the solution, thus far constructed locally in time, in a global fashion. In fact, working with the presentation \eqref{eq:profilesmodified1} on an interval $[t_1^{n} - T_*, t_1^n + R_1]$ for some very large $R_1$ (whose choice will depend on the scattering properties of $\epsilon^{<K_*}_{1bounded}$) and essentially repeating the arguments in the proof of Lemma~\ref{lem:epsboundednonlin}, picking $n$ large enough, we can extend the solution in the form \eqref{eq:profilesmodified1} such that $\big\| \tilde{\tilde{\eta}}^{Kn,<K_*}\big\|_{S([t_1^{n} - T_*, t_1^n + R_1])} \ll \delta_2$. Moreover, we still have the relation 
\[
 \lim_{R, T_*, K,n\rightarrow\infty} \big\| \nabla_{t,x}\big( \tilde{\eta}^{Kn,<K_*}_{(1)} \big|_{t^n_1 +R_1} - \tilde{\tilde{\eta}}^{Kn,<K_*} \big|_{t^n_1 + R_1} \big) \big\|_{L_x^2} = 0.
\]

Then we construct the solution on the interval $[t_1^n+R_1, t_2^n - T_{1*}]$ for $R_1, T_{1*}$ sufficiently large, by repeating the procedure in Step 2 of Stage 1 before. More precisely, we modify \eqref{eq:profilesmodified1} to 
\begin{equation} \label{eq:profilesmodified2}
 \phi^{n,<K_*} = u^{n} + \sum_k \tilde{P}_k \bigg( \big( U_{<k}^{(\phi^{n,<K_*})} \big)^\dagger \Big( \sum_{b=0}^B \tilde{\epsilon}^{Kn,<K_*}_{b,k} + \tilde{\epsilon}^{Kn,<K_*}_{B,k} \Big) \bigg) + \tilde{\tilde{\tilde{\eta}}}^{Kn,<K_*}
\end{equation}
on $[t_1^n+R_1, t_2^n - T_{1*}]$, where now $\epsilon^{<K_*}_{1bounded}$ has been replaced by the expression $\sum_k \big( U_{<k}^{(\phi^{n,<K_*})} \big)^\dagger \tilde{\epsilon}^{Kn,<K_*}_{1,k}$, where $\tilde{\epsilon}^{Kn,<K_*}_{1,k}$ solves \eqref{eq:therightone} with data at time $t = t_1^n+R_1$ given by 
\[
 \tilde{\epsilon}^{Kn,<K_*}_{1,k}[t_1^n+R_1] = \Big( \big( U_{<k}^{(\phi^{n,<K_*})}\epsilon^{<K_*}_{1bounded} \big)\big|_{t_1^n+R_1}, \, \partial_t \big( U_{<k}^{(\phi^{n,<K_*})} \epsilon^{<K_*}_{1bounded} \big) \big|_{t_1^n+R_1} \Big).
\]
Then repeating the arguments in the proof of Lemma~\ref{lem:epsboundedlin}, we infer the global bound 
\[
 \big\| \tilde{\tilde{\tilde{\eta}}}^{Kn,<K_*} \big\|_{S([t_1^n+R_1, t_2^n - T_{1*}])} \ll \delta_2.
\]
Moreover, importantly, we also obtain that 
\[
 \lim_{R, T_*, R_1, T_{1*}, K, n\rightarrow\infty} \big\|\tilde{\tilde{\tilde{\eta}}}^{Kn,<K_*} - \tilde{\eta}^{Kn,<K_*}_{(1)}\big\|_{S([t_1^n+R_1, t_2^n - T_{1*}])} = 0.
\]
Combining this with our earlier considerations, we observe that the error function consisting of $\tilde{\eta}^{Kn,<K_*}$ on $[R, t_1^n - T_*]$, $\tilde{\tilde{\eta}}^{Kn,<K_*}$ on $[t_1^n - T_*, t_1^n + R_1]$ and $\tilde{\tilde{\tilde{\eta}}}^{Kn,<K_*}$ on $[ t_1^n + R_1, t_2^n - T_{1*}]$ differs from $ \tilde{\eta}^{Kn,<K_*}_{(1)}$ (which we recall solves \eqref{eq:tiletaapprox} and had its data prescribed at time $t = R$) by an error with respect to $\| \cdot \|_{S}$, which vanishes as the parameters $R, T_*, R_1, T_{1*}, K, n\rightarrow\infty$. 

\medskip 

However, on account of the fact that $\tilde{\epsilon}_B^{Kn,<K_*}$ satisfies uniform $S$-bounds (its energy being uniformly bounded), the equation \eqref{eq:tiletaapprox} can be solved for $\tilde{\eta}^{Kn,<K_*}_{(1)}$ via a simple divisibility argument {\it{on any time interval}} on which $\phi^{n,<K_*}$ is defined, and picking $B = B\big( \|u^{n}\|_{S}, \delta_2 \big)$ large enough, we can ensure that 
\[
 \big\| \tilde{\eta}^{Kn,<K_*}_{(1)} \big\|_{S} \ll \delta_2. 
\]
We can then repeat the preceding procedure and prolong the solution beyond $t_2^n - T_{1*}$, until after $B$ steps we have a global solution. Of course at each of the $B$ many steps, we have to adjust the parameters $R, T_*, R_1, T_{1*}, K, n$ etc. to preserve the required smallness of the error. This proves Theorem~\ref{thm:profiledecomp}. 
\end{proof}

\subsection{Conclusion of the induction on frequency process} \label{subsec:conclusion}

In the previous subsection we established that the data $\Pi_{\leq \mu_n^1 + \log(R_n)} \phi^n[0]$ can be globally evolved with uniform $S$-norm bounds under the assumption that all profiles have energy strictly less than $E_{crit}$. We may now continue this induction on frequency process and by proceeding as in Subsection~\ref{subsec:lowest_freq_non_atomic} obtain that the data $\Pi_{\leq \mu_n^2 - \log(R_n)} \phi^n[0]$ can be globally evolved with uniform $S$-norm bounds. Then we ``add in'' the second frequency atom $P_{[\mu_n^2 - \log(R_n), \mu_n^2 + \log(R_n)]} \phi^n[0]$ in the sense that by proceeding analogously to Subsection~\ref{subsec:adding_in_first_atom} we may establish the global evolution of the corresponding geometric data $\Pi_{\leq \mu_n^2 + \log(R_n)} \phi^n[0]$ with uniform $S$-norm bounds under the assumption that the associated profiles all have energy strictly less than $E_{crit}$.

We may continue this procedure $\Lambda_0$ many times and establish the global evolution with uniform $S$-norm bounds of the essentially singular sequence of data $\{ \phi^n[0] \}_{n \geq 1}$, which would however be a contradiction, unless the sequence $\{ \phi^n[0] \}_{n \geq 1}$ is composed of exactly one frequency atom that consists of exactly one profile of asymptotic energy $E_{crit}$. Thus, in view of Theorem~\ref{thm:profiledecomp}, we may assume after rescaling so that $\phi^n[0]$ is essentially supported at frequency $\sim 1$, that we can either write 
\begin{equation} \label{equ:min_blowup_sol_bounded}
 \Pi_{< K_*} \phi^n[0] = \phi_{bounded}^{<K_*}[0] + \eta^{n, <K_*}[0]
\end{equation}
with $\lim_{K_*, n \to \infty} \bigl\| \eta^{n, <K_*}[0] \bigr\|_{\dot{H}^1_x \times L^2_x} = 0$, or else we have
\begin{equation} \label{equ:min_blowup_sol_unbounded}
 \Pi_{< K_*} \phi^n[0] = \sum_{k \in \bZ} \big( U_{<k}^{(\phi^{n, <K_*})} \big)^\dagger  S(\cdot - t_n) \big( \phi^{<K_*}_k[0] \big)[0] + \eta^{n, <K_*}[0]
\end{equation}
with $\lim_{n\to\infty} |t_n| = + \infty$, where $S(\cdot)$ now simply denotes the free wave propagator associated with the standard d'Alembertian on $\R^{1+2}$, and we again have $\lim_{K_*,n \to \infty} \big\|\eta^{n,<K_*}[0]\big\|_{\dot{H}^1_x \times L^2_x} = 0$. 

\medskip 

In any case we infer that the limits
\[
 \phi_{bounded}[0] = \lim_{K_*\rightarrow\infty}\phi^{<K_*}_{bounded}[0], \quad  \phi[0] = \lim_{K_*\rightarrow\infty}\phi^{<K_*}[0]
\]
exist and satisfy the ``criticality condition''
\[
 E[\phi_{bounded}] = E[\phi] = E_{crit}.
\]
We also observe that we have $\eta^{n,<K_*} \rightarrow 0$ in $L^\infty_{loc}$ as $n\rightarrow\infty$ after passing to subsequences.

\medskip 

In the next two propositions we infer that we can actually extract a minimal blowup solution from either scenario.
\begin{prop} \label{prop:minblowup1} 
Assume the situation given by \eqref{equ:min_blowup_sol_bounded}. Then the wave maps evolution $\phi_{bounded}(t, x)$ with lifespan $I$ associated with the energy class data $\phi_{bounded}[0] \colon \R^2 \rightarrow T\mathbb{S}^m$ satisfies 
\[
 \sup_{[T_1,T_2]\subset I} \big\|\phi_{bounded}\big\|_{S([T_1,T_2])} = +\infty. 
\]
In particular, $\phi_{bounded}(t, x)$ is a minimal blowup solution. 
\end{prop}
\begin{proof} We must show that it holds that
\[
 \lim_{K_*\rightarrow\infty} \sup_{[T_1, T_2] \subset I} \big\| \phi^{<K_*}_{bounded} \big\|_{S([T_1,T_2])} = + \infty. 
\]
Suppose not. Then all evolutions $\phi^{<K_*}_{bounded}$ exist globally in time and we have that 
\[
 \limsup_{K_* \rightarrow \infty} \, \big\| \phi^{<K_*}_{bounded} \big\|_{S[\R]} < \infty. 
\]
This uniform bound on the evolutions $\phi^{<K_*}_{bounded}$ implies by Lemma~\ref{lem:soloftherightoneasymptotics} that they have uniform scattering behavior. Given $\gamma>0$, there exists a decomposition
\[
 \phi^{<K_*}_{bounded} = \big( \phi^{<K_*}_{bounded} \big)_1 + \big( \phi^{<K_*}_{bounded} \big)_2
\]
and some $R_\gamma > 0$ such that
\[
 \big\| \big( \phi^{<K_*}_{bounded} \big)_1 \big\|_{S[\R]} < \gamma, \quad \big\| \big( \phi^{<K_*}_{bounded} \big)_2 \big\|_{L^\infty_t L^\infty_x[(-\infty, -R_\gamma] \cup [R_\gamma, +\infty)]} < \gamma. 
\]
We can then show, following Stage 1 in the proof of Theorem~\ref{thm:profiledecomp}, that we obtain uniform (in large enough $n$ and $K_*$) bounds on $\big\| \eta^{n,<K_*} \big\|_{S[\R]}$. This in turn contradicts the fact that for any $K_*$ large enough, we must have $\lim_{n\rightarrow\infty} \big\| \phi^{n,<K_*} \big\|_{S[\R]} = +\infty$ on account of Lemma~\ref{lem:localperturbative} and the fact that $\{ \phi^n\}_{n\geq 1}$ is an essentially singular sequence which is essentially supported at frequency $\sim 1$. 

In order to obtain the desired control over $\eta^{n,<K_*}$, we split $\R^{1+2}$ into the regions $[-R_\gamma, R_\gamma]^c \times \R^2$ and $[-R_\gamma, R_\gamma] \times \R^2$. To control $\eta^{n,<K_*}$ on the former region, we use a (much simplified) variant of the argument for Lemma~\ref{lem:epsboundedlin}. To handle the latter region, we argue as in (v) of Step 1 of the proof of Lemma~\ref{lem:epsboundednonlin}, replacing $\epsilon^{<K_*}_{bounded}$ there by $\phi^{<K_*}_{bounded}$.
\end{proof}

To conclude, we have 
\begin{prop} \label{prop:minblowup2} 
Assume the situation given by \eqref{equ:min_blowup_sol_unbounded}. Then there exists an energy class data pair $\tilde{\phi}[0] \colon \R^2 \to T \mathbb{S}^m$ with $E[\tilde{\phi}] = E_{crit}$ and such that denoting its wave maps evolution with lifespan $I$ by $\tilde{\phi}(t, x)$, we have 
\[
 \sup_{[T_1,T_2] \subset I} \| \tilde{\phi} \|_{S([T_1,T_2])} = +\infty. 
\]
In particular, $\tilde{\phi}(t,x)$ is a minimal blowup solution.
\end{prop}
\begin{proof} Assume without loss of generality that $t_n \to +\infty$ as $n \to \infty$. We reduce to the situation in the preceding Proposition~\ref{prop:minblowup1}. To begin with, using the argument of the proof of Lemma~\ref{lem:epsboundedlin}, we see that there is a large $C>0$ such that the wave maps evolution of the data $\phi^{n,<K_*}[0] \equiv \Pi_{<K_*} \phi^n[0]$ exists on $(-\infty, t^n-C]$ for any large enough $n$. Then by a simple compactness argument, exploiting the frequency localization up to exponentially decaying tails, we have that 
\[
 \phi^{n,<K_*}[t^n - C] \longrightarrow \tilde{\phi}^{<K_*}[0]
\]
as $n\rightarrow\infty$ in $\big( \dot{H}^{1}_{loc} \cap L^\infty_{loc} \big) \times L^2_{loc}$. We also have that the limit
\[
 \tilde{\phi}[0] = \lim_{K_*\rightarrow\infty} \tilde{\phi}^{<K_*}[0]
\]
exists in $\dot{H}^1_x \times L^2_x$. We now claim that this is the desired minimal blowup solution. In fact, we can write 
\[
 \phi^{n,<K_*}[t^n - C] =  \tilde{\phi}^{<K_*}[0] + \tilde{\eta}^{n,<K_*}[t^n - C], 
\]
where $ \tilde{\eta}^{n,<K_*}(t^n - C)\rightarrow 0$ both in $L^\infty_{loc}$ as well as $\dot{H}^1_{loc}$, and also $ \partial_t\tilde{\eta}^{n,<K_*}(t^n - C)\rightarrow 0$ in $L^2_{loc}$. But then assuming that the evolutions of the $ \tilde{\phi}^{<K_*}[0]$ satisfy uniform (in $K_*$) global-in-time $S$-norm bounds, we obtain a contradiction exactly as in the proof of the preceding Proposition~\ref{prop:minblowup1}.
\end{proof}

Finally, we observe that the minimal blowup solution constructed in the preceding Proposition~\ref{prop:minblowup1}, respectively in Proposition~\ref{prop:minblowup2}, has the following crucial compactness property whose proof follows exactly as in \cite{Kenig_Merle1, Kenig_Merle2}, see also Corollary~9.36 in \cite{KS}. 
\begin{cor} \label{cor:compactness property}
 There exists a non-trivial, radially symmetric, energy class, minimal blowup solution $\phi^\infty \colon I \times \bR^2 \to \bS^{m}$ to (WM) of energy $E_{crit}$ and with maximal interval of existence $I$. Moreover, there exists a continuous function $\lambda \colon I \to (0, \infty)$ so that the family of functions
 \begin{equation*}
  \Bigl\{ \bigl( \phi^\infty(t, \lambda(t)^{-1} \cdot ), \lambda(t)^{-1} \partial_t \phi^\infty(t, \lambda(t)^{-1} \cdot \bigr) \colon t \in I \Bigr\}
 \end{equation*}
 is pre-compact in $\dot{H}^1_x \times L^2_x$.
\end{cor}

\section{Rigidity argument} \label{sec:rigidity_argument}

In this last section we rule out the existence of a minimal blowup solution $\phi^\infty$ as in Corollary~\ref{cor:compactness property}. To this end we closely mimic the rigidity argument of Kenig-Merle~\cite{Kenig_Merle1} as implemented in~\cite{KS} for energy critical wave maps with $\mathbb{H}^2$ target. We emphasize that the momentum vanishing property of the minimal blowup solution is here an immediate consequence of the hypothesis of radial symmetry.

\subsection{Preliminary properties of minimal blowup solutions}
We begin by stating some standard properties of the minimal blowup solution $\phi^\infty$.
As in Corollary~\ref{cor:compactness property}, $I$ denotes the lifespan of $\phi^\infty$ and we set $I^+:= I\cap[0,\infty)$. The finite speed of propagation for wave maps gives the following lemma.
\begin{lem}\label{lem:finite_speed_of_propagation}
For given $\varepsilon>0$, let $M>0$ be such that
\begin{equation}
 \int_{|x|\geq M} \frac{1}{2} \sum_{\alpha=0}^2 |\partial_{\alpha} \phi^\infty(0,x)|^2 \, dx <\varepsilon.
\end{equation}
Then
\begin{equation}\label{equ:finite_speed_of_propagation}
 \int_{|x|\geq M+t} \frac{1}{2} \sum_{\alpha=0}^2 |\partial_{\alpha} \phi^\infty(t,x)|^2 \, dx <\varepsilon.
\end{equation}
for all $t\in I^+$.
\end{lem}

Next, we recall the following conservation laws for wave maps.
\begin{prop}
 Let $\phi$ be an energy class solution to (WM). Then the following conservation laws hold.
\begin{itemize} [leftmargin=*]
 \item Energy conservation
\begin{equation} \label{equ:cons_energy}
 \frac{d}{dt} \int_{\R^2} \frac{1}{2} \sum_{\alpha=0}^2 |\partial_{\alpha} \phi|^2 \, dx \, =0
\end{equation}
 \item Momentum conservation
\begin{equation} \label{equ:cons_momentum}
 \frac{d}{dt} \int_{\R^2} \partial_{t} \phi^\dagger  \partial_i \phi \, dx \, =0 \qquad i=1,2
\end{equation}
 \item Weighted energy
\begin{equation} \label{equ:cons_weighted_energy}
 \frac{d}{dt} \int_{\R^2}  \sum_{\alpha=0}^2 x_i \varphi (x/R) \frac{1}{2} \left|\partial_{\alpha} \phi(t,x)\right|^2 \, dx \, =\, - \int_{\R^2} \partial_{t} \phi^\dagger \partial_i \phi \, dx +O( r(R)) \qquad i=1,2
\end{equation}
 \item Weighted momentum monotonicity
\begin{equation} \label{equ:cons_weighted_momentum}
 \frac{d}{dt} \int_{\R^2}  \sum_{i=1}^2 x_i \varphi (x/R) \partial_{t} \phi^\dagger \partial_i \phi\, dx \, =\, - \int_{\R^2} |\partial_{t} \phi|^2 \, dx +O( r(R))
\end{equation}
\end{itemize}
where $\varphi$ is a fixed bump function which is equal to one on $\{ |x|\leq 1 \}$ and 
$$r(R):= \int_{|x|\geq R} \sum_{\alpha=0}^2 |\partial_{\alpha} \phi^\infty(0,x)|^2 \, dx. $$
\end{prop}
\begin{proof}
 For classical solutions to (WM) these identities follow from Stokes' theorem together with the fact that the stress-energy tensor is divergence free. By Definition~\ref{defn:energy_class_evolution} one can then pass to the limit and obtain the result for energy class solutions.
\end{proof}

By virtue of Corollary~\ref{cor:compactness property}, the energy class, minimal blowup solution $\phi^\infty$ satisfies some important properties, which we present in the next lemmas.
In the case that $I^+$ is finite, we have the following lower bound for the continuous function $\lambda(t)$ of Corollary~\ref{cor:compactness property}.
\begin{lem} \label{lem:lower_bound_lambda}
Assume that $I^+$ is finite. After suitable rescaling, we may assume that $I^+=[0,1)$. Let $\lambda: I^+\rightarrow (0, \infty)$ be as in Corollary~\ref{cor:compactness property}
and let us denote the pre-compact set in $\dot{H}^1_x \times L^2_x$ therein by $K$. Then, there exists a constant $C_0(K)>0$ such that
\begin{equation}\label{equ:bound_lambda}
 0<\frac{C_0(K)}{1-t}\leq \lambda(t)
\end{equation}
for all $0\leq t<1$.
\end{lem}
\begin{proof}
 The proof follows along the lines of Lemma 10.4 in \cite{KS}
\end{proof}

In addition, under the same assumptions of Lemma~\ref{lem:lower_bound_lambda} we have the following support property of $\phi^\infty$.

\begin{lem} \label{lem:support_property}
Let $\phi^\infty$ as in Corollary~\ref{cor:compactness property} with finite lifespan. Then there exists $x_0\in \R^2$ such that
$$\mathrm{supp}(\phi^\infty(t, \cdot) )\subset B(x_0, 1-t)$$
for all $0\leq t<1$.
\end{lem}
\begin{proof}
 The proof follows as in Lemma 4.8 of \cite{Kenig_Merle1} by use of Lemmas~\ref{lem:finite_speed_of_propagation} and ~\ref{lem:lower_bound_lambda}.
\end{proof}

We finally state the vanishing momentum property which we get here for free by the very definition of radially symmetric wave maps.

\begin{prop} \label{prop:vanishing_momentum}
 Let $\phi^\infty$ as in Corollary~\ref{cor:compactness property}. Then 
\begin{equation} \label{equ:vanishing_momentum}
\int_{\R^2}\partial_{t} {\phi^\infty}^\dagger \partial_i \phi^\infty   \, dx \, =0 \qquad i=1,2
\end{equation}
for all times in $I^+$.
\end{prop}

\subsection{Rigidity I: Infinite time interval and reduction to the self-similar case for finite time intervals}

The final goal of the rigidity argument consists in proving the following theorem.

\begin{prop}\label{prop:rigidity}
 Let $\phi^\infty$ as in Corollary~\ref{cor:compactness property} with lifespan $I=(-T_0, T_1)$. Then, one cannot have $T_1$ or $T_0$ finite. Moreover, if $\lambda(t)\geq \lambda_0>0$ for all $t\in \R$, one necessarily has 
$\partial_\alpha \phi^\infty=0$ for $\alpha=0,1,2$.
\end{prop}
In this subsection, we will provide the proof of Proposition~\ref{prop:rigidity} for the infinite time case and then reduce to the self-similar situation for the finite time case. The proof of Proposition~\ref{prop:rigidity}
in the finite time case will be concluded in the next subsection.

\begin{proof}{[Proposition ~\ref{prop:rigidity}: infinite time case]} 
We will first treat the case $T_1=\infty$. Assuming that $\partial_\alpha \phi^\infty$ do not all vanish, we extract a non-trivial finite energy radially symmetric harmonic map into the sphere,
leading to a contradiction. Following the proof of Lemma 10.9 in \cite{KS}, which crucially hinges on the geometry of the target, we show that there exists $\alpha\in \R$ such that
\begin{equation} \label{equ:no_harmonic_map_invested}
 \int_I \int_{\bR^2} | \partial_t \phi^\infty(t,x)|^2 \, dx \, dt \geq \alpha > 0
\end{equation}
for all intervals $I$ of unit length. 
Suppose not, then there exists a sequence of intervals $I_n:= [t_n, t_n+1]$ with the property that $t_n\rightarrow \infty$ and
\begin{equation} \label{equ:contradiction_hyp}
 \int_{I_n}\int_{\R^2} |\partial_t \phi^\infty(t,x)|^2 \, dx \, dt \leq \frac{1}{n}.
\end{equation}
Then there exists a sequence of times $s_n\in I_n$ such that $||\partial_t \phi^\infty(s_n, \cdot)||_2\rightarrow 0$ as $n\rightarrow \infty$.
By the compactness property of Corollary~\ref{cor:compactness property}, the family of functions
$$ \Bigl\{ \bigl( \phi^\infty(s_n, \lambda(s_n)^{-1} \cdot ), \lambda(s_n)^{-1} \partial_t \phi^\infty(s_n, \lambda(s_n)^{-1} \cdot \bigr) \Bigr\}_{n=0}^\infty$$
is pre-compact in $\dot{H}^1_x\times L^2_x$. Up to a subsequence, there exist limit functions $\Phi^\ast$ such that
$$\lambda(s_n)^{-1} \partial_\alpha \phi^\infty(s_n, \lambda(s_n)^{-1} \cdot \bigr) \rightarrow \partial_\alpha\Phi^\ast(\cdot)$$
strongly in $L^2_x$ for $\alpha=0,1,2$.
Pre-compactness and standard perturbative arguments ensure that there exists some non-empty interval $I^\ast$ around zero such that
$$\lambda(s_n)^{-1} \partial_\alpha \phi^\infty(s_n + t\lambda(s_n)^{-1}, \lambda(s_n)^{-1} \cdot \bigr) \rightarrow \partial_\alpha\Phi^\ast(t, \cdot)$$
in $L^\infty_{loc}(I^\ast; L^2(\R^2))$ with $\Phi^\ast$ being a radially symmetric weak solution of (WM) on $I^\ast \times \R^2$.
We now distinguish two cases:
\begin{itemize}
 \item there exists a sequence of times $s_n\in I_n$ such that $\{\lambda(s_n)\}$ is bounded: since $\lambda(t)\geq  \lambda_0>0$, there exists a non-empty time interval
$\tilde{I}\subset I^\ast$ such that $s_n+\lambda(s_n)^{-1} \tilde{I} \subset I_n$ 
for each $n$; then \eqref{equ:contradiction_hyp} implies that $ \int_{\tilde{I}}\int_{\R^2} |\partial_t \Phi^\ast(t,x)|^2 \, dx \, dt=0$ whence $\partial_t \Phi^\ast(t, \cdot)=0$
for all $t\in \tilde{I}$.
\item $\{\lambda(s_n)\}$ is unbounded for every sequence $\{s_n\}$ with $s_n\in I_n$: in this case we use a crucial Vitali covering argument as in \cite{KS} to arrive at the same conclusion; we write for each $n$
$$I_n= \bigcup_{s\in I_n} \left[s-\lambda(s)^{-1}, s+ \lambda^{-1}(s) \right]\cap I_n$$
and by the Vitali covering lemma, we may pick a disjoint subcollection of intervals $\{I_s\}_{s\in J_n}$ with $I_s:=\left[s-\lambda(s)^{-1}, s+ \lambda^{-1}(s) \right]\cap I_n$ for some subset
$J_n\subset I_n$ with the property that
$$\sum_{s\in J_n} 5\left| I_s\right| \geq 1;$$
it follows that, by definition of $I_n$, we may pick a sequence of times $s_n \in I_n$ with the property that 
$\int_{I_{s_n}} \|\partial_t \phi^\infty(t, \cdot)\|^2_{L^2_x} \, dt = o(\lambda^{-1}(s_n))$ and in particular 
$$\int_{-1}^1 \|\left(\chi_{I_n}\partial_t \phi^\infty\right)(s_n+t\lambda(s_n)^{-1}, \cdot)\|^2_{L^2_x} \, dt = o(1),$$
but then, passing to a subsequence, we can extract a limiting function $\Phi^\ast$ from 
$$\lambda(s_n)^{-1} \partial_t \phi^\infty(s_n +t \lambda(s_n)^{-1},  \lambda(s_n)^{-1} \cdot)$$ with the property that $\partial_t \Phi^\ast(t, \cdot)=0$ for all $t$ 
in the lifespan of $\Phi^\ast$.
\end{itemize}
Hence, in both cases we have deduced the existence of a radially symmetric weak wave map $\Phi^\ast: \R^{2+1}\rightarrow \bS^2$ with the following properties
$$\partial_t \Phi^\ast=0, \quad t\in I^\ast$$
and
$$\sum_{\alpha=1}^2 \| \partial_\alpha \Phi^\ast\|_{L^2_x}^2\neq 0.$$
This leads to a contradiction since there do not exist non-vanishing finite energy radial harmonic maps into the sphere, see for instance \cite{Struwe_radial1}.
We can then conclude that \eqref{equ:no_harmonic_map_invested} holds.
The remainder of the argument hinges upon the weighted momentum monotonicity \eqref{equ:cons_weighted_momentum} combined with the following property:
for any given $\varepsilon > 0$, by pre-compactness of the orbit of $\phi^\infty$ in $\dot{H}^1_x \times L^2_x$ there exists $R(\varepsilon)>0$ such that for all $t \in I$,
\[
 \int_{ \{ |x| \geq R(\varepsilon) \} } \bigl| \lambda(t)^{-1} \nabla_{t,x} \phi^\infty(t, \lambda(t)^{-1} x) \bigr|^2 \, dx \leq \varepsilon.
\]
Changing variables and using the key assumption $\lambda(t) \geq \lambda_0 > 0$, we can conclude that for all $t \in I$,
\[
 \int_{ \{ |x| \geq \frac{ R(\varepsilon) }{\lambda_0} \} } | \nabla_{t,x} \phi^\infty(t,x) |^2 \, dx  \leq \varepsilon.
\]
Using the notation $\tilde{R}(\varepsilon) := \frac{ R(\varepsilon) }{\lambda_0}$, then by the weighted momentum monotonicity identity \eqref{equ:cons_weighted_momentum} integrated over a sufficiently large time interval $[0,T]$, $T\gg 1$,
we get for arbitrary $\varepsilon > 0$ and corresponding $\tilde{R} \equiv \tilde{R}(\varepsilon)$ that
\begin{equation} \label{equ:integrated_in_time}
 \int_{\bR^2} x^j \varphi(x / \tilde{R})  {(\partial_t \phi^\infty)}^\dagger \partial_j \phi^\infty \, dx \bigg|_{0}^{T}  = - \int_0^T \int_{\bR^2} | \partial_t \phi^\infty (t,x) |^2 \, dx \, dt + \int_0^T O( r (\tilde{R}) ) \, dt,
\end{equation}
where 
\[
 r(\tilde{R}) = \int_{|x| \geq \tilde{R} } |\nabla_{t,x} \phi^\infty(t,x) |^2 \, dx \leq \varepsilon
\]
by our choice of $\tilde{R}$. But then using \eqref{equ:no_harmonic_map_invested}, the right-hand side of \eqref{equ:integrated_in_time} can be bounded from above by
\begin{align*}
 - \int_0^T \int_{\bR^2} | \partial_t \phi^\infty (t,x) |^2 \, dx \, dt + \int_0^T O( r (\tilde{R}) ) \, dt &\leq - (T-1) \alpha + C T \varepsilon \longrightarrow - \infty
\end{align*}
as $T \to \infty$, if we just choose $\varepsilon > 0$ sufficiently small so that $C \varepsilon \ll \alpha$. At this point $\varepsilon > 0$ and correspondingly $\tilde{R}$ are fixed. But then we see that the left-hand side of \eqref{equ:integrated_in_time} is bounded by
$ \tilde{R} E_{crit}$ 
and so we get a contradiction by taking $T \gg 1$ sufficiently large. 
\end{proof}

Having concluded the infinite time case, we turn now to the proof of Proposition~\ref{prop:rigidity} in the case $T_1<\infty$. Similarly to \cite{Kenig_Merle1}, \cite{KS} and \cite{KL}, the first 
step consists in reducing to a self-similar blow-up scenario.
As before, we can assume without loss of generality that $T_1=1$. We recall from Lemma~\ref{lem:lower_bound_lambda} that
\begin{equation}
\lambda(t) \geq \frac{C_0(K)}{1-t}, \quad 0<t<1,
\end{equation}
and from Lemma~\ref{lem:support_property} that 
\begin{equation} \label{equ:support_bis}
 \mathrm{supp}(\phi^\infty(t, \cdot) )\subset B(0, 1-t)\quad \text{and}\quad \mathrm{supp}(\partial_t\phi^\infty(t, \cdot) )\subset B(0, 1-t).
\end{equation}
Then we may prove the following upper bound for $\lambda(t)$.

\begin{lem} \label{lem:upper_bound_lambda}
Let $\phi^\infty$ as above with $T_1=1$. Then there exists $C_1(K)>0$ such that
\begin{equation} \label{equ:upper_bound_lambda}
\lambda(t) \leq \frac{C_1(K)}{1-t} 
\end{equation}
for all $0 \leq t<1$.
\end{lem}
\begin{proof}
We adopt the same strategy as in the proof of Lemma 10.11 in \cite{KS} (see also Lemma 8.12 in \cite{KL}).
Suppose that \eqref{equ:upper_bound_lambda} fails.
Define, for $0\leq t<1$ the following functional
$$z(t):=  \sum_{j=1}^2 \int_{\R^2} x_j (\partial_j \phi^\infty (t,x))^\dagger \partial_t \phi^\infty(t,x) \, dx.$$
The weighted momentum monotonicity
\eqref{equ:cons_weighted_momentum} allows to compute that
$$z'(t)=-\int_{\R^2} \left|\partial_t \phi^\infty(t,x) \right|^2\, dx.$$
Next by \eqref{equ:support_bis}, we see that $z(t) \rightarrow 0$ as $t\rightarrow 1$, hence we can write
$$z(t)= \int_t^1 \int_{\R^2} \left|\partial_t \phi^\infty(s,x) \right|^2\, dx \, ds.$$
At this point, we need to distinguish two cases: either there exists $\alpha>0$ such that
$$\int_t^1 \int_{\R^2} \left|\partial_t \phi^\infty(s,x) \right|^2\, dx\, ds \geq \alpha (1-t),\quad 0\leq t<1$$
or else, there exists a sequence $\{t_n\}_n \subset [0,1)$ with $t_n\rightarrow 1$ such that, denoting $J_n=(t_n,1)$, it holds
$$\left|J_n\right|^{-1} \int_{J_n} \int_{\R^2}  \left|\partial_t \phi^\infty(s,x) \right|^2\, dx\, ds \rightarrow 0 \quad \text{as}\, n\rightarrow \infty. $$
In the first case, we get a contradiction by proceeding as in the proof of Lemma 5.6 of \cite{Kenig_Merle1} and using therein the vanishing momentum property of Proposition~\ref{prop:vanishing_momentum}.
In the second case, we argue similarly as in the proof of Proposition~\ref{prop:rigidity} for the infinite time case: by a Vitali argument and by pre-compactness one can conclude
the existence of a non-vanishing finite energy radially symmetric harmonic map into the sphere which gives a contradiction (cf. \cite{Struwe_radial1}).
\end{proof}
Combining the previous lemmas we can reduce to the self--similar scenario. Indeed, we have the following.
\begin{cor} \label{cor:self-similar}
 Let $\phi^\infty$ be as above with $T_1=1$, then the set
 \begin{equation*}
  \Bigl\{  (1-t)\partial_\alpha \phi^\infty(t, (1-t) \cdot \bigr) : t \in I, \, \alpha = 0, 1, 2 \Bigr\}
 \end{equation*}
 is pre-compact in $L^2_x(\R^2)$.
\end{cor}
\begin{proof}
For the proof we refer to the proof of Proposition 5.7 in \cite{Kenig_Merle1}.
\end{proof}

\subsection{Rigidity II: The self-similar case}
In this section we deal with the finite time case and we carry out the last step of the Kenig-Merle rigidity argument by excluding the possibility of self-similar blow-up. In particular, we rule out the existence of a minimal blow-up
solution $\phi^\infty$ as in Corollary~\ref{cor:self-similar}.
For the sake of simplicity we drop the superscript $\infty$ and denote the minimal blowup solution from Corollary~\ref{cor:self-similar} just by $\phi$.
 As in \cite{Merle_Zaag}, \cite{Kenig_Merle1}, \cite{KS}, we introduce the self-similar variables
$$y= \frac{x}{1-t}, \quad  s=- \log(1-t), \quad 0\leq t <1$$
and 
$$\widetilde{\phi}(s,y, 0):= \phi(t,x) = \phi(1- e^{-s}, e^{-s} y), \quad 0\leq s< \infty. $$
By construction, $\nabla_{s,y} \widetilde{\phi} (s, \cdot, 0)$ is supported in $\{y\in \R^2:|y|\leq 1 \}$. Next, for small $\delta>0$, we also define
$$y= \frac{x}{1+ \delta-t}, \quad  s=- \log(1+\delta-t), \quad 0\leq t <1$$
and set
\begin{equation} \label{equ:self_similar_delta_phi}
 \widetilde{\phi}(s,y,\delta):= \phi(t,x) = \phi(1+ \delta- e^{-s}, e^{-s} y), \quad 0\leq s< \infty.
\end{equation}
Note that $\widetilde{\phi}(s,y,\delta)$ is defined for $- \log(1+\delta) \leq s< -\log \delta$. 
By standard computations, we can write the wave maps equation in self-similar variables as follows
\begin{equation} \label{equ:equation_selfsimilar}
 \partial_s^2 \widetilde{\phi}= \frac{1}{\rho} \text{div} \left( \rho \nabla_y \widetilde{\phi}- \rho ( y \cdot \nabla_y \widetilde{\phi}) y\right) - 2 y \cdot \nabla_y \partial_s \widetilde{\phi}
- \partial_s \widetilde{\phi}+ \widetilde{\phi} \left(\left|\partial_s \widetilde{\phi}+ y\cdot \nabla_y \widetilde{\phi}\right|^2 - \left|\nabla_y\widetilde{\phi}\right|^2\right)
\end{equation}
where $\rho= (1- |y|^2)^{-1/2}$.
The following lemma collects some basic properties of $\widetilde{\phi}$ inherited from $\phi$.

\begin{lem}\label{lem:properties_phitilde}
 For fixed $\delta>0$, we have for all $0\leq s<-\log \delta$ that
\begin{itemize}
 \item $\mathrm{supp} (\partial_\alpha \widetilde{\phi}(s,\cdot, \delta) ) \subset \{y \in \R^2:|y|\leq 1- \delta\} \quad \alpha= 0,1,2$
 \item $ \int_{\R^2} \left(\left|\nabla_y \widetilde{\phi}(s,y, \delta)\right|^2+ \left|\partial_s \widetilde{\phi}(s,y, \delta)\right|^2 \right) \, dy \leq C$
 \item $\sum_{\alpha=0}^2 \int_{\R^2} \left| \partial_\alpha \widetilde{\phi}(s,y, \delta) \right|^2 \log\left( \frac{1}{1- |y|^2} \right)\, dy \leq C \log \frac{1}{\delta}$
 \item $\sum_{\alpha=0}^2 \int_{\R^2} \left| \partial_\alpha \widetilde{\phi}(s,y, \delta) \right|^2 \left( 1- |y|^2 \right)^{- \frac{1}{2}}\, dy \leq C \delta^{-1/2}$.
\end{itemize}
\end{lem}
\begin{proof}
 The properties are obtained by direct calculation. See also \cite{Kenig_Merle1} and \cite{KL}.
\end{proof}
As in \cite{KS}, we introduce the following Lyapunov functional
$$\widetilde{E}\left( \widetilde{\phi} \right) (s):= \frac{1}{2} \int_{B_1} \left[\left|\partial_s \widetilde{\phi}\right|^2
+ \left|\nabla_y \widetilde{\phi}\right|^2 - \left|y \cdot \nabla_y \widetilde{\phi}\right|^2  \right] \left( 1- |y|^2 \right)^{- \frac{1}{2}} \, dy, $$
which satisfies a suitable monotonicity property stated in the next proposition.
\begin{prop} \label{prop:monotonicity_phitilde}
 For $0\leq s_1 <s_2<\log(1/\delta)$, the following identities holds
\begin{enumerate}
 \item[(1)] $\widetilde{E}\left( \widetilde{\phi} \right) (s_2)- \widetilde{E}\left( \widetilde{\phi} \right) (s_1) 
= \int_{s_1}^{s_2} \int_{B_1} \frac{|\partial_s \widetilde{\phi}| ^2}{\left( 1- |y|^2 \right)^{3/2}} \, dy \, ds,$
 \item[(2)] $\lim _{s\rightarrow \log (1/\delta)} \widetilde{E}\left( \widetilde{\phi} \right) (s)\leq E_{crit}.$
\end{enumerate}
\end{prop}
\begin{proof}
 For (1) see the proof of Lemma 2.1 in \cite{Merle_Zaag}, while (2) can be proven as in \cite{Kenig_Merle1}.
\end{proof}

By direct application of Proposition~\ref{prop:monotonicity_phitilde} we obtain the following corollary.

\begin{lem} \label{lem:s_delta}
 For all $\delta>0$, there exists $\bar{s}_\delta \in \left(\frac{|\log \delta|}{2}, |\log \delta| \right)$ such that
$$\int_{\bar{s}_\delta}^{\bar{s}_\delta+ |\log \delta|^{\frac{1}{2}}}\int_{B_1} \frac{|\partial_s \widetilde{\phi}| ^2}{\left( 1- |y|^2 \right)^{3/2}}\, dy \, ds\leq \frac{E_{crit}}{|\log \delta|^{\frac{1}{2}}}.$$
\end{lem}

By using once more pre-compactness together with the self-similar behavior, we aim to extract a limit $\phi^\ast$ as $\delta \rightarrow 0$ and to show that $\phi^\ast$ is a stationary solution to \eqref{equ:equation_selfsimilar} and finally obtain a contradiction as in previous sections.
In order to achieve this, let $\bar{t}_{\delta}= 1+ \delta - e^{-\bar{s}_\delta}$ for $\bar{s}_{\delta}$ as in Lemma~\ref{lem:s_delta}. By Corollary~\ref{cor:self-similar} we may select $\delta_j\rightarrow 0$ as $j\rightarrow \infty$ 
such that, for each $\alpha=0,1,2$
$$(1- \bar{t}_{\delta_j})  \partial_\alpha \phi^\infty(\bar{t}_{\delta_j}, (1-\bar{t}_{\delta_j})x) \rightarrow \partial_\alpha \Phi^\ast(x)$$
strongly in $L^2$ as $\delta_j\rightarrow 0$. In fact, we may also achieve that
\begin{equation} \label{equ:self_similar_data}
(1+ \delta_j - \bar{t}_{\delta_j})  \partial_\alpha \phi^\infty(\bar{t}_{\delta_j}, (1+\delta_j-\bar{t}_{\delta_j})x) \rightarrow \partial_\alpha \Phi^\ast(x) 
\end{equation}
strongly in $L^2$. Next, we consider the evolution, in the sense of Definition~\ref{defn:energy_class_evolution}, of the energy class data given by the left hand side of \eqref{equ:self_similar_data}. We denote these evolutions by $\phi^{j_\ast}$
and we remark that, due to pre-compactness and standard perturbative arguments, these evolutions exist on some joined fixed life-span $[0, T^\ast]$, where we may assume $0<T^\ast<1$. Furthermore, on $[0, T^\ast]$ we have
$$\partial_\alpha \phi^{j_\ast} (t,x)= (1+\delta_j - \bar{t}_{\delta_j}) \partial_\alpha \phi^\infty ( \bar{t}_{\delta_j} + (1+\delta_j-\bar{t}_{\delta_j})t, (1+\delta_j-\bar{t}_{\delta_j})x)$$
and for each $\alpha=0,1,2$
$$\partial_\alpha \phi^{j_\ast} (t,\cdot)\rightarrow \partial_\alpha \Phi^\ast(t,\cdot)$$
strongly in $L^2$ as $j\rightarrow \infty$ and uniformly for all $0\leq t \leq T^\ast$, where $\Phi^\ast$ is a weak wave map on $[0, T^\ast] \times \R^2$. We also remark that, due to the previous identities, it holds
$$\mathrm{supp} \left(\phi^{j_\ast} (t,\cdot) \right) \subset \left\{x\in \R^2: \, |x| \leq \frac{1-\bar{t}_{\delta_j}}{1+\delta_j - \bar{t}_{\delta_j}}-t<1-t \right\}$$
for $0\leq t\leq T^\ast$.
Next, we switch to self-similar variables 
$$y= \frac{x}{1-t}, \quad  s=- \log(1-t), \quad 0\leq t \leq T^\ast,$$
and define
$$\widetilde{\phi^{j_\ast}}(s,y):= \phi^{j_\ast}(1- e^{-s}, e^{-s} y)$$
and analogously for $\widetilde{\Phi^\ast}$. Then, as in \cite{KS}, we infer that
$$ \partial_\alpha \widetilde{\phi^{j_\ast}}(s,\cdot)\rightarrow \partial_\alpha \widetilde{\Phi^\ast}(s,\cdot)$$
strongly in $L^2_y(\R^2)$ as $j\rightarrow \infty$ and uniformly for all $0\leq s\leq -\log (1- T^\ast/2)=: \widetilde{T}$ and for $\alpha=0,1,2$.
Moreover, with $\widetilde{\phi}$ as in \eqref{equ:self_similar_delta_phi}, we have
$$ \widetilde{\phi^{j_\ast}}(s,y)= \widetilde{\phi} (\bar{s}_{\delta_j}+s , y, \delta_j)$$
and hence for each $\alpha=0,1,2$
\begin{equation} \label{equ:self_similar_convergence}
 \partial_\alpha \widetilde{\phi} (\bar{s}_{\delta_j}+s , y, \delta_j) \rightarrow \partial_\alpha \widetilde{\Phi^\ast} (s, \cdot)
\end{equation}
strongly in $L^2_y$ and uniformly in $0\leq s\leq \widetilde{T}$ as $j\rightarrow  \infty$. Therefore, $\Phi^\ast$ is a solution of \eqref{equ:equation_selfsimilar} and
$$\mathrm{supp} \left(\partial_\alpha \widetilde{\Phi}^\ast(\cdot, s)\right)\subset \{y\in \R^2: \,|y|\leq 1\}$$
and moreover
$$\mathrm{trace} \left(\widetilde{\Phi}^\ast(\cdot, s) \right)= const$$
where ``$\mathrm{trace}$'' is the $L^2-$ trace.

\begin{lem} \label{lem:limit_constant_in_time}
 Let $\widetilde{\Phi}^\ast$ be as above. Then we have 
 $$\widetilde{\Phi}^\ast(s,y)= \widetilde{\Phi}^\ast(y)\quad \text{and}\quad \widetilde{\Phi}^\ast\nequiv const.$$
\end{lem}

\begin{proof}
 The proof works exactly as in \cite{KS}: by \eqref{equ:self_similar_convergence} and Lemma~\ref{lem:s_delta} one shows that $\widetilde{\Phi}^\ast(s,y)= \widetilde{\Phi}^\ast(y)$ while
 the fact that $ \widetilde{\Phi}^\ast\nequiv const$ follows as in \cite{Kenig_Merle1}.
\end{proof}

By this procedure we have obtained a stationary, nonconstant, distributional solution to \eqref{equ:equation_selfsimilar} with finite energy (relative to the $y$ variable).
As in \cite{KS}, the following proposition leads to a contradiction.

\begin{prop}\label{prop:final_contradiction}
 Let $\Phi$ be a distributional solution to \eqref{equ:equation_selfsimilar} of finite energy
 $$\int_{B_1} |\nabla \Phi(y)|^2 \, dy < \infty.$$
 Then $\Phi^\ast= const$.
 \end{prop}
\begin{proof}
 For the proof we refer the reader to the proof of Proposition 10.17 in \cite{KS} where the argument of Section 7.5.1 in \cite{Shatah_Struwe_book} is implemented and the conclusion is reached by Lemaire's uniqueness theorem~\cite{Lemaire}.
\end{proof}

Thanks to Proposition~\ref{prop:final_contradiction} we can conclude the proof of Proposition~\ref{prop:rigidity} in the finite time case.

\begin{proof}{[Proposition ~\ref{prop:rigidity}: finite time case]} 
In the finite time case, thanks to Corollary~\ref{cor:self-similar} we have been able to construct a nonconstant self--similar solution $\Phi^\ast$ of \eqref{equ:equation_selfsimilar} of finite energy, but this is in contradiction
with Proposition~\ref{prop:final_contradiction} and this completes the proof of Proposition~\ref{prop:rigidity}.
\end{proof}

In order to close the rigidity argument, we notice that we need to justify the additional condition $\lambda(t)\geq \lambda_0>0$ for all $t\in \R$ assumed in Proposition~\ref{prop:rigidity}, but this follows as in Lemma 10.18 of \cite{KS}.
This concludes the rigidity argument and hence finishes the proof of Theorem~\ref{thm:main_theorem}.

\bibliographystyle{amsplain}
\bibliography{references}

\end{document}